\documentclass[11pt,oneside,english]{amsart}
\usepackage[T1]{fontenc}
\usepackage[utf8]{inputenc}
\usepackage[a4paper]{geometry}
\usepackage{color}
\usepackage{babel}
\usepackage{mathtools}
\usepackage{enumitem}
\usepackage{amstext}
\usepackage{amsthm}
\usepackage{amssymb}
\usepackage{stmaryrd}
\usepackage{esint}
\PassOptionsToPackage{normalem}{ulem}
\usepackage{ulem}
\usepackage[unicode=true,pdfusetitle,
 bookmarks=true,bookmarksnumbered=false,bookmarksopen=false,
 breaklinks=false,pdfborder={0 0 1},backref=false,colorlinks=false]
 {hyperref}

\makeatletter
\numberwithin{equation}{section}
\theoremstyle{plain}
\newtheorem{thm}{\protect\theoremname}[section]
\theoremstyle{definition}
\newtheorem{defn}[thm]{\protect\definitionname}
\theoremstyle{remark}
\newtheorem{rem}[thm]{\protect\remarkname}
\theoremstyle{definition}
\newtheorem{assumption}[thm]{\protect\assumptionname}
\theoremstyle{plain}
\newtheorem{prop}[thm]{\protect\propositionname}
\theoremstyle{plain}
\newtheorem{lem}[thm]{\protect\lemmaname}
\theoremstyle{plain}
\newtheorem{cor}[thm]{\protect\corollaryname}

\makeatother

\providecommand{\assumptionname}{Assumption}
\providecommand{\corollaryname}{Corollary}
\providecommand{\definitionname}{Definition}
\providecommand{\lemmaname}{Lemma}
\providecommand{\propositionname}{Proposition}
\providecommand{\remarkname}{Remark}
\providecommand{\theoremname}{Theorem}

\begin{document}
\global\long\def\bbC{\mathbb{C}}%
\global\long\def\bbN{\mathbb{N}}%
\global\long\def\bbQ{\mathbb{Q}}%
\global\long\def\bbS{\mathbb{S}}%
\global\long\def\bbR{\mathbb{R}}%
\global\long\def\bbZ{\mathbb{Z}}%

\global\long\def\bfD{{\bf D}}%
\global\long\def\bfF{{\bf F}}%
\global\long\def\bfG{{\bf G}}%

\global\long\def\bfd{{\bf d}}%
\global\long\def\bfv{{\bf v}}%

\global\long\def\rmp{\mathrm{p}}%

\global\long\def\calA{\mathcal{A}}%
\global\long\def\calB{\mathcal{B}}%
\global\long\def\calC{\mathcal{C}}%
\global\long\def\calD{\mathcal{D}}%
\global\long\def\calE{\mathcal{E}}%
\global\long\def\calF{\mathcal{F}}%
\global\long\def\calG{\mathcal{G}}%
\global\long\def\calH{\mathcal{H}}%
\global\long\def\calI{\mathcal{I}}%
\global\long\def\calJ{\mathcal{J}}%
\global\long\def\calK{\mathcal{K}}%
\global\long\def\calL{\mathcal{L}}%
\global\long\def\calM{\mathcal{M}}%
\global\long\def\calN{\mathcal{N}}%
\global\long\def\calU{\mathcal{U}}%
\global\long\def\calO{\mathcal{O}}%
\global\long\def\calP{\mathcal{P}}%
\global\long\def\calQ{\mathcal{Q}}%
\global\long\def\calR{\mathcal{R}}%
\global\long\def\calS{\mathcal{S}}%
\global\long\def\calT{\mathcal{T}}%
\global\long\def\calU{\mathcal{U}}%
\global\long\def\calV{\mathcal{V}}%
\global\long\def\calW{\mathcal{W}}%
\global\long\def\calX{\mathcal{X}}%
\global\long\def\calY{\mathcal{Y}}%
\global\long\def\calZ{\mathcal{Z}}%

\global\long\def\frkc{\mathfrak{c}}%
\global\long\def\frkd{\mathfrak{d}}%
\global\long\def\frkp{\mathfrak{p}}%
\global\long\def\frkr{\mathfrak{r}}%
\global\long\def\frks{\mathfrak{s}}%

\global\long\def\alp{\alpha}%
\global\long\def\bt{\beta}%
\global\long\def\dlt{\delta}%
\global\long\def\Dlt{\Delta}%
\global\long\def\eps{\epsilon}%
\global\long\def\gmm{\gamma}%
\global\long\def\Gmm{\Gamma}%
\global\long\def\kpp{\kappa}%
\global\long\def\tht{\theta}%
\global\long\def\lmb{\lambda}%
\global\long\def\Lmb{\Lambda}%
\global\long\def\vphi{\varphi}%
\global\long\def\omg{\omega}%
\global\long\def\Omg{\Omega}%

\global\long\def\rd{\partial}%
\global\long\def\aleq{\lesssim}%
\global\long\def\ageq{\gtrsim}%
\global\long\def\aeq{\simeq}%

\global\long\def\peq{\mathrel{\phantom{=}}}%
\global\long\def\To{\longrightarrow}%
\global\long\def\weakto{\rightharpoonup}%
\global\long\def\embed{\hookrightarrow}%
\global\long\def\impmi{\Longleftrightarrow}%
\global\long\def\Re{\mathrm{Re}}%
\global\long\def\Im{\mathrm{Im}}%
\global\long\def\chf{\mathbf{1}}%
\global\long\def\td#1{\widetilde{#1}}%
\global\long\def\br#1{\overline{#1}}%
\global\long\def\ul#1{\underline{#1}}%
\global\long\def\wh#1{\widehat{#1}}%
\global\long\def\rng#1{\mathring{#1}}%
\global\long\def\tint#1#2{{\textstyle \int_{#1}^{#2}}}%
\global\long\def\tsum#1#2{{\textstyle \sum_{#1}^{#2}}}%

\global\long\def\lan{\langle}%
\global\long\def\ran{\rangle}%
\global\long\def\blan{\big\langle}%
\global\long\def\bran{\big\rangle}%
\global\long\def\Blan{\Big\langle}%
\global\long\def\Bran{\Big\rangle}%

\global\long\def\Mod{\mathrm{Mod}}%
\global\long\def\rmex{\mathrm{ex}}%
\global\long\def\rmin{\mathrm{in}}%
\global\long\def\loc{\mathrm{loc}}%
\global\long\def\sg{\mathrm{sg}}%
\global\long\def\NL{\mathrm{NL}}%
\global\long\def\rad{\mathrm{rad}}%
\global\long\def\dec{\mathrm{d}}%
\global\long\def\dsc{\mathrm{d}}%
\global\long\def\coer{\mathrm{c}}%
\global\long\def\ext{\mathrm{ext}}%
\global\long\def\setJ{\llbracket J\rrbracket}%

\global\long\def\alptube{\alp_{\mathrm{tb}}}%
\global\long\def\Ttbexit{T^{\mathrm{tb}}}%
\global\long\def\alpbtstrap{\alp^{\ast}}%
\global\long\def\dltbtstrap{\dlt^{\ast}}%
\global\long\def\dltprofile{\dlt_{\rmp}}%
\global\long\def\etatube{\eta_{\mathrm{tb}}}%
\global\long\def\rough{\mathrm{rgh}}%
\global\long\def\secmax{\mathrm{max2}}%

\title[Rigidity in multi-bubble dynamics]{Rigidity results in multi-bubble dynamics for \\
non-radial energy-critical heat equation}
\author{Kihyun Kim}
\email{kihyun.kim@snu.ac.kr}
\address{Department of Mathematical Sciences and Research Institute of Mathematics,
Seoul National University, 1 Gwanak-ro, Gwanak-gu, Seoul 08826, Republic
of Korea}
\author{Frank Merle}
\email{merle@ihes.fr}
\address{IHES, 35 route de Chartres, 91440 Bures-sur-Yvette, France and AGM,
CY Cergy Paris Université, 2 av.\ Adolphe Chauvin, 95302 Cergy-Pontoise
Cedex, France}
\thanks{\emph{Acknowledgements.} K. Kim was supported by the New Faculty Startup
Fund from Seoul National University, the POSCO Science Fellowship
of POSCO TJ Park Foundation, and the National Research Foundation
of Korea (NRF) grant funded by the Korea government (MSIT) RS-2025-00523523.
This work was initiated when K. Kim was supported by Huawei Young
Talents Programme at IHES and part of this work was done when he was
visiting IHES. He gratefully acknowledges its kind hospitality.}
\begin{abstract}
This paper concerns the classification of asymptotic behaviors in
multi-bubble dynamics for the energy-critical nonlinear heat equations
in large dimensions $N\geq7$ without symmetry. This multi-bubble
dynamics appears naturally at least for a sequence of times in view
of soliton resolution. We assume each bubble is given by the scalings
and translations of $\pm W$ with (localized) non-colliding conditions
for a sequence of times, where $W$ is the ground state. The case
of one soliton was previously established and in particular there
is no blow-up. We consider the case of $J\geq2$ solitons, where we
expect only infinite-time blow-up. 

We are able to identify three different scenarios, where we have a
continuous-in-time resolution with an unexpected universal blow-up
speed. The first one is when one scaling is much larger than the others.
In this case, one bubble does not concentrate (hence stabilize) and
the other bubbles concentrate with the universal blow-up speed $t^{-2/(N-6)}$
together with strong sign constraints. Next, assuming we are not in
the first scenario, we establish a non-degenerate condition on the
positions of bubbles to obtain that all bubbles concentrate with the
universal blow-up speed $t^{-1/(N-4)}$. The last case we consider
is a degenerate, but not too much degenerate, scenario. Here again,
we obtain that all bubbles concentrate with the universal blow-up
speed $t^{-1/(N-3)}$. This last rate has not been discovered before.
Our theorem covers the case of four or less bubbles and we provide
the construction of examples. To our knowledge, this is the first
classification result in the non-radial multi-bubble dynamics, where
both the scales, positions, and signs enter the dynamics nontrivially.
\end{abstract}

\maketitle
\tableofcontents{}

\section{Introduction}

In this paper, we consider the \emph{energy-critical nonlinear heat
equation} 
\begin{equation}
\left\{ \begin{aligned}\rd_{t}u & =\Dlt u+|u|^{\frac{4}{N-2}}u,\\
u(0,x) & =u_{0}(x)\in\bbR,
\end{aligned}
\quad(t,x)\in[0,T)\times\bbR^{N}\right.\tag{NLH}\label{eq:NLH}
\end{equation}
in dimensions $N\geq3$. This equation has \emph{scaling symmetry}:
if $u(t,x)$ is a solution to \eqref{eq:NLH}, then so is 
\begin{equation}
u_{\lmb}(t,x)=\lmb^{-\frac{N-2}{2}}u(\lmb^{-2}t,\lmb^{-1}x),\qquad\lmb\in(0,\infty),\label{eq:scaling}
\end{equation}
with initial data $\lmb^{-\frac{N-2}{2}}u_{0}(\lmb^{-1}x)$. \eqref{eq:NLH}
is a gradient flow associated to the\emph{ energy} 
\begin{equation}
E[u(t)]=\int_{\bbR^{N}}\Big(\frac{1}{2}|\nabla u(t,x)|^{2}-\frac{N-2}{2N}|u(t,x)|^{\frac{2N}{N-2}}\Big)dx.\label{eq:energy}
\end{equation}
Thus for any two times $t_{1}\geq t_{0}\geq0$ one formally has the
\emph{energy identity}
\begin{equation}
E[u(t_{0})]=E[u(t_{1})]+\int_{t_{0}}^{t_{1}}\int_{\bbR^{N}}|\rd_{t}u|^{2}dxdt.\label{eq:energy-identity}
\end{equation}
The fact that the energy is invariant under (the spatial part of)
the scaling transform \eqref{eq:scaling} manifests the \emph{energy-critical}
nature. It is known from \cite{Weissler1980,BrezisCazenave1996} (see
also \cite{CollotMerleRaphael2017CMP}) that \eqref{eq:NLH} is well-posed
in the \emph{energy space} $\dot{H}^{1}=\dot{H}^{1}(\bbR^{N})$.

In the description of asymptotic dynamics of solutions to \eqref{eq:NLH},
a key role is played by stationary solutions (we will often call them
\emph{solitons} or \emph{bubbles}), i.e., solutions to the elliptic
PDE 
\begin{equation}
\Dlt Q+|Q|^{\frac{4}{N-2}}Q=0\qquad\text{in }\bbR^{N}.\label{eq:static-eqn}
\end{equation}
There is a whole zoo of solutions to \eqref{eq:static-eqn}; see for
instance \cite{Ding1986CMP,delPinoMussoPacardPistoia2011JDE,delPinoMussoPacardPistoia2013Pisa}.
The complete classification of solutions to \eqref{eq:static-eqn}
is widely open. However, the following \emph{Aubin--Talenti bubble}
(or the \emph{ground state}) 
\begin{equation}
W(x)=\Big(1+\frac{|x|^{2}}{N(N-2)}\Big)^{-\frac{N-2}{2}}\label{eq:Aubin-Talenti-bubble}
\end{equation}
is a stationary solution distinguished from the others; (i) $W$ attains
the best constant of the critical Sobolev inequality $\|u\|_{L^{2^{\ast}}}\aleq\|u\|_{\dot{H}^{1}}$
($2^{\ast}=\frac{2N}{N-2}$) \cite{Aubin1976,Talenti1976}, (ii) it
is the unique \emph{radially symmetric} classical solution to \eqref{eq:static-eqn}
up to scaling and sign, and (iii) it is the unique \emph{nonnegative}
classical solution (hence for $\dot{H}^{1}$ by elliptic regularity)
to \eqref{eq:static-eqn} up to scaling and translation \cite{Obata1972,GidasNiNirenberg1979CMP,CaffarelliGidasSpruck1989CPAM}.

If a $\dot{H}^{1}$-solution $u(t)$ remains \emph{$\dot{H}^{1}$-bounded}
on its maximal lifespan $[0,T_{+})$ (note that any time-global $H^{1}$-solution
remains $\dot{H}^{1}$-bounded \cite{Ishiwata2025preprint}), then
the integral of $\int_{\bbR^{N}}|\rd_{t}u(t)|^{2}dx$ over $[0,T_{+})$
remains finite by the energy identity \eqref{eq:energy-identity}.
Without being rigorous, this integrability implies that $\rd_{t}u$
(after renormalized in space depending on time) goes to zero in time
averaged sense. As $\rd_{t}u=0$ means $u(t)$ is static, we may say
that $u(t)$ becomes ``almost static.'' However, $u(t)$ being almost
static does not necessarily imply that $u(t)$ is globally in space
close to a stationary solution because the sum of sufficiently decoupled
stationary solutions can also be almost static with bounded energy.
One only has that $u(t)$ asymptotically approaches a decoupled sum
of stationary solutions, depending on time, in time averaged sense.
This phenomenon is known as \emph{bubbling} in the context of critical
harmonic map heat flows; see for example \cite{Struwe1985,Qing1995CommAnalGeom,DingTian1995CommAnalGeom,Wang1996HoustonJM,QingTian1997CPAM,LinWangCVPDE,Topping2004MathZ},
and a recent work \cite{JendrejLawrieSchlag2025ForumPi} and references
therein.

Since a convergence in time averaged sense implies a convergence for
a sequence of times, one is motivated to \emph{sequential soliton
resolution} (or sequential bubble decomposition); any $\dot{H}^{1}$-bounded
solutions $u(t)$ admit a sequence $t_{n}\to T_{+}$ such that $u(t_{n})$
is asymptotically decomposed into a sum of ``decoupled'' sum of stationary
solutions depending on $t_{n}$ (and a ``radiation'' if $T_{+}<+\infty$).
Then, the celebrated \emph{soliton resolution conjecture} asks whether
this resolution into stationary solutions and a radiation holds \emph{pointwisely
in time} as $t\to T_{+}$, not sequentially in time $t_{n}\to T_{+}$. 

In fact, soliton resolution is not confined to critical parabolic
flows. It appears rather \emph{universally }for a wide range of evolutionary
PDEs, perhaps most famously in the context of integrable PDEs. To
mention only a few, see for instance \cite{FermiPastaUlamTsingou1955,ZabuskyKruskal1965}
(for numerical results) and \cite{EckhausSchuur1983,DeiftVenakidesZhou1994CPAM,BorgheseJenkinsMcLaughlin2018,JenkinsLiuPerrySulem2018CMP,ChenLiu2021AIHP,KimKwon2024arXiv,GassotGerardMiller2026arXiv},
where soliton resolution is established by integrable methods (except
\cite{KimKwon2024arXiv}). There are still very few non-integrable
examples where soliton resolution is resolved, but we may mention
some non-integrable critical PDEs (mostly heat or wave equations),
where soliton resolution is established within symmetry \cite{DuyckaertsKenigMerle2013CambJMath,DuyckaertsKenigMerle2023Acta,DuyckaertsKenigMartelMerle2022CMP,CollotDuyckaertsKenigMerle2024VietJM,JendrejLawrie2025JAMS,JendrejLawrie2023AnnPDE,KimKwonOh2025AJM,JendrejLawrie2023CVPDE,Aryan2024arXiv}.
Note that, for non-radial critical parabolic flows, a weaker version
of continuous-in-time soliton resolution is established in \cite{JendrejLawrieSchlag2025ForumPi,Aryan2025arXiv}
(see also \cite{Ishiwata2025preprint}). For the non-radial critical
wave equation, sequential soliton resolution is established in \cite{DuyckaertsJiaKenigMerle2017GAFA}.
See also \cite{CoteMartelYuan2021ARMA}.

We seek a classification result for critical PDEs, in a setting where
soliton resolution is settled, with the goal of proving rigidity in
the asymptotic behavior of the parameters of solitons. This classification
problem is extremely challenging in the non-radial case due to the
wide variety of possible situations. We decided to concentrate on
\eqref{eq:NLH}, where we proved in our previous work \cite{KimMerle2025CPAM}
the full classification theorem under radial symmetry in dimensions
$N\geq7$. 

As an intermediate goal before a potential full classification, we
assume that the soliton profiles appearing in \emph{sequential} resolution
are given by the \emph{ground state} $W$. This assumption in particular
covers both radial solutions and nonnegative solutions by the uniqueness
of $W$ within these classes. These considerations motivate the following
definition.
\begin{defn}[$W$-bubbling sequence]
\label{def:bubble-seq}Let $J\in\bbN$; let $\vec{\iota}=(\iota_{1},\dots\iota_{J})\in\{\pm\}^{J}$;
let $u_{n}\in\dot{H}^{1}(\bbR^{N})$, $\vec{\lmb}_{n}=(\lmb_{1,n},\dots,\lmb_{J,n})\in(0,\infty)^{J}$,
and $\vec{z}_{n}=(z_{1,n},\dots,z_{J,n})\in(\bbR^{N})^{J}$ be sequences
indexed by $n\in\bbN$. We call $\{u_{n},\vec{\iota},\vec{\lmb}_{n},\vec{z}_{n}\}$
a \emph{$W$-bubbling sequence} if 
\begin{equation}
\Big\| u_{n}-\sum_{i=1}^{J}\iota_{i}W_{\lmb_{i,n},z_{i,n}}\Big\|_{\dot{H}^{1}}+\sum_{1\leq i\neq i'\leq J}\Big\{\sqrt{\frac{\lmb_{i',n}}{\lmb_{i,n}}}+\sqrt{\frac{\lmb_{i,n}}{\lmb_{i',n}}}+\frac{|z_{i,n}-z_{i',n}|}{\sqrt{\lmb_{i,n}\lmb_{i',n}}}\Big\}^{-1}\to0\quad\text{as }n\to\infty.\label{eq:def-W-bubbling}
\end{equation}
Here, $\iota W_{\lmb,z}$ denotes 
\[
\iota W_{\lmb,z}(x)=\frac{\iota}{\lmb^{\frac{N-2}{2}}}W\Big(\frac{x-z}{\lmb}\Big).
\]
The latter term of \eqref{eq:def-W-bubbling} going to zero is called
as \emph{asymptotic orthogonality} of parameters.
\end{defn}

\begin{rem}
\label{rem:1.2}Let $u(t)$ be a $\dot{H}^{1}$-bounded solution to
\eqref{eq:NLH} with the maximal time of existence $T_{+}\in(0,+\infty]$.
If $T_{+}=+\infty$, we set $u_{n}=u(t_{n})$ for some $t_{n}\to+\infty$
in the definition above; if $T_{+}<+\infty$, we set $u_{n}=u(t_{n})-u^{\ast}$
for some $u^{\ast}\in\dot{H}^{1}$ and $t_{n}\to T_{n}$; and we call
such a solution $u(t)$ a \emph{$W$-bubbling solution}. In the global
case $T_{+}=+\infty$, we do not have a residual part because linear
solutions converge to $0$ in $\dot{H}^{1}$. Note also that such
$W$-bubbling solutions avoid stationary solutions constructed in
\cite{delPinoMussoPacardPistoia2011JDE} as $t_{n}\to T_{+}$, for
example.
\end{rem}

The near soliton dynamics in dimensions $N\geq7$ was completely classified
by Collot--Merle--Raphaël \cite{CollotMerleRaphael2017CMP}. In
particular, the dynamics of $W$-bubbling dynamics in the simplest
case $J=1$ is fully understood; any $W$-bubbling solution $u(t)$
is global and converges to $\iota W_{\lmb^{\ast},z^{\ast}}$ in $\dot{H}^{1}$
as $t\to+\infty$ for some $(\lmb^{\ast},z^{\ast})\in(0,\infty)\times\bbR^{N}$.
In particular, there is no type-II blow-up with only one $W$, and
the scale and spatial center stabilize. 

Thus for multi-bubble solutions ($J\geq2$) we can expect that the
asymptotic dynamics of parameters is governed solely by bubble interactions,
which can be computed explicitly and yield universal rates depending
on the strength of interaction. In this paper, we justify this heuristics
in ``not too much'' degenerate situations which will include the case
of small $J$. We also expect that this fact can be extended to other
equations (e.g., wave, Schrödinger, and so on) with similar picture.

As mentioned earlier, the non-radial multi-bubble dynamics is extremely
complex with a wide variety of asymptotic behaviors of the parameters.
To set up a reasonable intermediate goal for full classification,
we assume in addition that we are in the time-sequential, non-colliding,
and localized scenario for $W$-bubbles.
\begin{assumption}[Sequential assumptions on parameters]
\label{assumption:sequential-on-param}Let an integer $J\geq2$,
signs $\iota_{1},\dots,\iota_{J}\in\{\pm1\}$, and \emph{distinct}
points $z_{1}^{\ast},\dots,z_{J}^{\ast}\in\bbR^{N}$ be given. Let
$u(t)$ be a global $\dot{H}^{1}$-solution to \eqref{eq:NLH} together
with sequences $t_{n}\to+\infty$, $\vec{\lmb}_{n}\in(0,\infty)^{J}$,
and $\vec{z}_{n}\in(\bbR^{N})^{J}$ such that $\{u(t_{n}),\vec{\iota},\vec{\lmb}_{n},\vec{z}_{n}\}$
is a $W$-bubbling sequence,
\begin{align}
\lmb_{i,n}\aleq1 & \quad\text{for all }i\in\setJ,\quad\text{and}\label{eq:lmb_j,n-are-bdd}\\
z_{i,n}\to z_{i}^{\ast} & \quad\text{for all }i\in\setJ\quad\text{as }n\to\infty.\label{eq:z_j,n-converges-to-z_j^ast}
\end{align}
Here, $\setJ\coloneqq\{1,2,\dots,J\}$.
\end{assumption}

\begin{rem}
\label{rem:1.4}Using a scaling argument in the proof, one can deal
with slightly less restrictive assumptions that renormalizes to Assumption~\ref{assumption:sequential-on-param}
after rescaling. Without introducing $z_{1}^{\ast},\dots,z_{J}^{\ast}\in\bbR^{N}$,
one may consider scale invariant assumptions
\begin{align*}
\min_{1\leq i\neq j\leq J}|z_{i,n}-z_{j,n}| & \aeq\max_{1\leq i\neq j\leq J}|z_{i,n}-z_{j,n}|,\\
\min_{1\leq i\neq j\leq J}|z_{i,n}-z_{j,n}| & \ageq\lmb_{\max,n},
\end{align*}
where $\lmb_{\max,n}$ is the maximum of $\lmb_{i,n}$s. By rescaling
and taking a further subsequence, one can define $z_{1}^{\ast},\dots,z_{J}^{\ast}$
in a renormalized sense and almost recovers Assumption~\ref{assumption:sequential-on-param}.
\end{rem}

It is worthwhile to mention that finite-time blow-up of non-colliding
$W$-bubbles can be excluded easily by localizing the argument of
\cite{CollotMerleRaphael2017CMP}; see Proposition~\ref{prop:finite-time-non-existence}.
Henceforth, we focus on the dynamics of \emph{global} solutions.

With the above assumption and assuming $N\geq7$ as in our previous
work \cite{KimMerle2025CPAM}, we establish various scenarios, where
we have classification. 

In order to achieve classification for each scenario mentioned below,
we reduce PDE dynamics into finite-dimensional dynamics of parameters
and completely classify their asymptotic behaviors. This process is
delicate for several reasons; we begin with quite general time-sequential
non-colliding $W$-bubbling solutions, the reduction process introduces
errors, and we need to formulate appropriate a priori estimates for
the dynamics (or bootstrap hypotheses) and close all such assumed
estimates.

Assume $u(t)$ is a sum of asymptotically decoupled $\iota_{i}W_{\lmb_{i}(t),z_{i}(t)}$
for $1\leq i\leq J$, where asymptotic decoupling means \eqref{eq:def-W-bubbling}.
Denoting $\vec{\lmb}=(\lmb_{1},\dots,\lmb_{J})$ and $\vec{z}=(z_{1},\dots,z_{J})$,
the formal ODE system for $(\vec{\lmb},\vec{z})$ in the main context
becomes 
\begin{align}
\lmb_{i}\lmb_{i,t} & =\kpp_{0}\kpp_{\infty}\sum_{j\neq i}\iota_{i}\iota_{j}\frac{\lmb_{i}^{\frac{N-2}{2}}\lmb_{j}^{\frac{N-2}{2}}}{|z_{j}-z_{i}|^{N-2}}+\text{(error)},\label{eq:intro-lmb-eqn}\\
z_{i,t} & =\kpp_{1}\kpp_{\infty}\sum_{j\neq i}\iota_{i}\iota_{j}\frac{\lmb_{i}^{\frac{N-2}{2}}\lmb_{j}^{\frac{N-2}{2}}}{|z_{j}-z_{i}|^{N}}(z_{j}-z_{i})+\text{(error)},\label{eq:intro-z-eqn}
\end{align}
for some positive universal constants $\kpp_{0}$, $\kpp_{1}$, and
$\kpp_{\infty}$ (See \eqref{eq:def-kpp}--\eqref{eq:def-kpp_infty}).
It is easy to see that this dynamical system has formal conservation
laws (i.e., without error terms)
\begin{equation}
z_{1}+\cdots+z_{J}\label{eq:intro-formal-cons-1}
\end{equation}
and
\begin{equation}
(\lmb_{1}^{2}+\cdots+\lmb_{J}^{2})+\frac{2\kpp_{0}}{\kpp_{1}}(|z_{1}|^{2}+\cdots+|z_{J}|^{2}),\label{eq:intro-formal-cons-2}
\end{equation}
but these are unfortunately not well-controlled with error terms.

Our first result treats the scenario that one $\lmb_{i_{0}}$ is much
larger than the others for a sequence of times. We completely classify
this case; not only one has (i) continuous-in-time resolution with
$z_{i}(t)\to z_{i}^{\ast}$ for all $i$, we also conclude that (ii)
$\lmb_{i_{0}}$ converges to a \emph{nonzero} limit and the other
$\lmb_{i}(t)$s decay to zero with the rate $\lmb_{i}(t)\aeq t^{-2/(N-6)}$,
and (iii) the signs necessarily satisfy $\iota_{i}=-\iota_{i_{0}}$
for $i\neq i_{0}$; see Theorem~\ref{thm:main-one-bubble-tower-classification}.

We are now left with the situation where $\lmb_{\max}$ (the maximum
of $\lmb_{i}$s) and $\lmb_{\secmax}$ (the second largest among $\lmb_{i}$s)
have comparable size. Then, only using the information on $z_{i}^{\ast}$s,
more precisely, introducing the matrix 
\[
A_{ij}^{\ast}=\chf_{i\neq j}\kpp_{0}\kpp_{\infty}\frac{\iota_{i}\iota_{j}}{|z_{i}^{\ast}-z_{j}^{\ast}|^{N-2}},
\]
($\chf_{i\neq j}$ is $1$ if $i\neq j$ and is $0$ otherwise) and
assuming $z_{i}\approx z_{i}^{\ast}$, we can rewrite \eqref{eq:intro-lmb-eqn}
as 
\[
\lmb_{i}\lmb_{i,t}=\sum_{j\neq i}A_{ij}^{\ast}\lmb_{i}^{\frac{N-2}{2}}\lmb_{j}^{\frac{N-2}{2}}+\text{(error)}.
\]
If all $\iota_{i}$ are the same, $A^{\ast}$ corresponds to the matrix
formed by Green's functions (up to multiplication by a constant) in
the whole domain $\bbR^{N}$; see e.g., \cite{BahriCoron1988CPAM,BahriLiRey1995CVPDE,CortazarDelPinoMusso2020JEMS}.
We claim that in a less degenerate scenario, which will be called
``(totally) non-degenerate'' scenario, only the matrix $A^{\ast}$
encodes the blow-up speed in our classification theorem with the rate
$\lmb_{i}\aeq t^{-1/(N-4)}$; see the first item of Theorem~\ref{thm:main-lmb-to-zero-classification}. 

Although $z_{i}^{\ast}$s ``generically'' form a less degenerate configuration,
not all $z_{i}^{\ast}$ need to do so. Such a configuration begins
to exist when $J\geq4$. We then introduce the concept of ``minimal
degeneracy'' (to be defined later) of a configuration. In that situation,
simply replacing $z_{i}(t)$ by $z_{i}^{\ast}$ is not sufficient
to determine the rates; one further needs to investigate the time
variation of the matrix 
\begin{equation}
A_{ij}[\vec{z}]\coloneqq\chf_{i\neq j}\kpp_{0}\kpp_{\infty}\frac{\iota_{i}\iota_{j}}{|z_{i}-z_{j}|^{N-2}},\label{eq:def-A_jk}
\end{equation}
(or the first variation of $A[\vec{z}]$ with respect to $\vec{z}$)
to obtain a classification result with the rate $\lmb_{i}(t)\aeq t^{-1/(N-3)}$;
see the second item of Theorem~\ref{thm:main-lmb-to-zero-classification}
assuming a non-degenerate condition on the equations of $z_{i,t}$.

The rate \emph{$\lmb_{i}(t)\aeq t^{-1/(N-3)}$ is a completely new
rate} of different nature that can arise only in the degenerate regime. 

We expect \emph{all these regimes to exist} and we detailed the construction
of (a minimal example of) the last scenario, which is the most delicate.
See Proposition~\ref{prop:Sec9-main}. The other cases would follow
with a simpler proof. 

We also prove that the case of $J\leq4$ (only under $N\geq7$ and
Assumption~\ref{assumption:sequential-on-param}) is \emph{completely
classified} by our main theorems. 

Up to our knowledge, this work provides the first classification result
in the \emph{non-radial} multi-bubbles dynamics, where both the scales
and translations enter the dynamics crucially. All these types of
assumptions where classification can be done were not expected and
come to us as a surprise. 

Finally, let us mention some difficulties from the radial case to
the current non-radial case. First, the configuration of spatial centers
$z_{i}^{\ast}$ is important as we observe a posteriori different
blow-up rates that we did not expect. In fact, more exotic rates can
arise; see Remark~\ref{rem:optimality-non-degen} below. This complicates
very much the problem and the asymptotic behavior of solutions after
finite-dimensional reduction. Second, there might be some configuration
that gives rise to a stationary solution due to the delicate balance
within soliton interactions, which is not the case for radial bubble
towers. We avoid all such configurations in our main theorems. Note
also that stationary solutions such as in \cite{delPinoMussoPacardPistoia2011JDE}
are avoided by sequential $W$-bubbling assumption. Finally, analysis
with non-radial modified multi-bubble profile (if restricted to the
elliptic setting, this is similar to \cite{DengSunWei2025Duke}) adds
an extra layer of difficulty from the radial case to the non-radial
case.

\subsection{Main results}

We are now ready to state our classification results. As mentioned
before, we assume $N\geq7$ and Assumption~\ref{assumption:sequential-on-param}.
Our first result treats the case when one $\lmb_{i_{0}}$ are much
larger than the others, for some sequence of times.
\begin{thm}
\label{thm:main-one-bubble-tower-classification}Let $N\geq7$ and
assume Assumption~\ref{assumption:sequential-on-param}. If 
\begin{equation}
\lmb_{\max,n}\gg\lmb_{\secmax,n}\label{eq:main1-lmb_max>>lmb_secmax}
\end{equation}
for some subsequence in $n$, then the following hold. 
\begin{itemize}
\item (Continuous-in-time resolution) There exist continuous curves $\vec{\lmb}(t)\in(0,\infty)^{J}$
and $\vec{z}(t)\in(\bbR^{N})^{J}$ such that 
\begin{equation}
\Big\| u(t)-\sum_{i=1}^{J}\iota_{i}W_{\lmb_{i}(t),z_{i}(t)}\Big\|_{\dot{H}^{1}}+\sum_{i=1}^{J}|z_{i}(t)-z_{i}^{\ast}|\to0\qquad\text{as }t\to+\infty.\label{eq:main1-conti-time-res-with-z-conv}
\end{equation}
\item (Determination of signs) There exists $i_{0}\in\setJ$ such that $\iota_{i}=-\iota_{i_{0}}$
for all $i\neq i_{0}$.
\item (Asymptotics of scales) There exists $\lmb_{i_{0}}^{\ast}\in(0,\infty)$
such that as $t\to+\infty$ 
\begin{align}
\lmb_{i_{0}}(t) & =\lmb_{i_{0}}^{\ast}+o(1),\label{eq:main1-lmb_i_0}\\
\lmb_{i}(t) & =(\ell_{i}+o(1))\cdot t^{-2/(N-6)}\quad\text{for all }i\neq i_{0}.\label{eq:main1-lmb_i}
\end{align}
Here, the constants $\ell_{i}$ are uniquely determined by $\lmb_{i_{0}}^{\ast}$
and $z_{1}^{\ast},\dots,z_{J}^{\ast}$, whose explicit formula can
be found in \eqref{eq:case2-lmb-ratio}. 
\end{itemize}
\end{thm}

Next, we turn to the situation where Theorem~\ref{thm:main-one-bubble-tower-classification}
does not apply, i.e., $\lmb_{\max,n}\aeq\lmb_{\secmax,n}$. The matrix
$A[\vec{z}]$ introduced earlier comes into play and we need to give
precise definitions of (total) non-degeneracy and (minimal) degeneracy
of configurations to state our second result.
\begin{defn}[(Non-)degeneracy of configurations]
\label{def:non-degeneracy-of-configurations}Let $J\geq1$; let signs
$\iota_{1},\dots,\iota_{J}\in\{\pm\}$ and \emph{distinct} points
$z_{1}^{\ast},\dots,z_{J}^{\ast}\in\bbR^{N}$ be given. We define
a $J\times J$ matrix by 
\[
A_{ij}^{\ast}\coloneqq A_{ij}[\vec{z}^{\ast}]\coloneqq\chf_{i\neq j}\kpp_{0}\kpp_{\infty}\frac{\iota_{i}\iota_{j}}{|z_{i}^{\ast}-z_{j}^{\ast}|^{N-2}},\qquad i,j\in\setJ,
\]
where $\kpp_{0}$ and $\kpp_{\infty}$ are positive universal constants
defined in \eqref{eq:def-kpp} and \eqref{eq:def-kpp_infty}. For
a subset $\emptyset\neq\calI\subseteq\setJ$, we denote by $A_{\calI}^{\ast}$
the $|\calI|\times|\calI|$ submatrix of $A^{\ast}$ defined using
the indices in $\calI$. We use the same definition and notation for
$A[\vec{z}]$ defined by simply removing {*}.

Now let $J\geq2$. We call the configuration $\{(\iota_{1},z_{1}^{\ast}),\dots,(\iota_{J},z_{J}^{\ast})\}$
\emph{non-degenerate} if $A^{\ast}$ has no nonzero kernel element
in $[0,\infty)^{J}$. Otherwise, we call it \emph{degenerate}. We
call it \emph{totally non-degenerate} if every sub-configuration of
size at least two is non-degenerate. We call it \emph{minimally degenerate}
if it is degenerate but any proper sub-configurations of size at least
two are non-degenerate. Note that our definition of (non-)degeneracy
is invariant under symmetries, i.e., under scalings, translations,
and flipping all the signs $\iota_{i}$ simultaneously.
\end{defn}

\begin{thm}
\label{thm:main-lmb-to-zero-classification}Let $N\geq7$, assume
Assumption~\ref{assumption:sequential-on-param}, and suppose Theorem~\ref{thm:main-one-bubble-tower-classification}
does not apply.
\begin{itemize}
\item If the configuration $\{(\iota_{1},z_{1}^{\ast}),\dots,(\iota_{J},z_{J}^{\ast})\}$
is totally non-degenerate, then there exist continuous curves $\vec{\lmb}(t)\in(0,\infty)^{J}$
and $\vec{z}(t)\in(\bbR^{N})^{J}$ such that the resolution \eqref{eq:main1-conti-time-res-with-z-conv}
holds and 
\begin{equation}
\lmb_{1}(t)\aeq\dots\aeq\lmb_{J}(t)\aeq t^{-1/(N-4)}.\label{eq:main-thm-nondegen-lmb-rate}
\end{equation}
Moreover, the set 
\begin{equation}
\Big\{\vec{\mu}\in(0,\infty)^{J}:\sum_{j=1}^{J}A_{ij}^{\ast}\mu_{j}^{\frac{N-2}{2}}+\frac{1}{(N-4)\mu_{i}^{\frac{N-6}{2}}}=0\text{ for all }i\in\setJ\Big\}\label{eq:main-thm-nondegen-ratio-set}
\end{equation}
is non-empty and $t^{1/(N-4)}\vec{\lmb}(t)$ converges to a connected
component of this set.
\item If the configuration $\{(\iota_{1},z_{1}^{\ast}),\dots,(\iota_{J},z_{J}^{\ast})\}$
is minimally degenerate and satisfies \eqref{eq:case3-v_i-not-all-zero},
then there exist continuous curves $\vec{\lmb}(t)\in(0,\infty)^{J}$
and $\vec{z}(t)\in(\bbR^{N})^{J}$ such that the resolution \eqref{eq:main1-conti-time-res-with-z-conv}
holds and 
\begin{equation}
\lmb_{i}(t)=(\ell_{i}+o(1))\cdot t^{-1/(N-3)}\quad\text{as }t\to+\infty.\label{eq:main-thm-degenerate-lmb-rate}
\end{equation}
Here, the constants $\ell_{i}$ are uniquely determined by the configuration,
whose explicit formula can be found in \eqref{eq:case3-lmb-rate}.
\end{itemize}
\end{thm}

\begin{rem}[New scenarios]
The minimally degenerate rate $\lmb_{i}(t)\aeq t^{-1/(N-3)}$ is
\emph{entirely new} and we also \emph{construct} a nonlinear solution
exhibiting this scenario in Section~\ref{sec:Construction-of-minimal-examples}.
The non-degenerate rate $\lmb_{i}(t)\aeq t^{-1/(N-4)}$ appeared in
a different setting \cite{CortazarDelPinoMusso2020JEMS}. Note that
both rates of Theorem~\ref{thm:main-lmb-to-zero-classification}
do not exist in the radial case. On the other hand, the rates in Theorem~\ref{thm:main-one-bubble-tower-classification}
have a bubble tower nature and essentially appeared in the radial
case \cite{delPinoMussoWei2021AnalPDE,KimMerle2025CPAM} (with different
constants in front). 
\end{rem}

\begin{rem}[On construction of scenarios]
All the scenarios of our main theorems exist, although we only provide
a detailed construction of a minimal $J=4$ example for the minimally
degenerate scenario. We chose this minimally degenerate scenario because
it was not discovered before and it is the most \emph{delicate} (unstable)
scenario. For the construction of minimal examples of other scenarios,
let us refer to Remark~\ref{rem:other-examples}. A similar construction
for general non-degenerate or minimally degenerate configurations
might be possible, but this might require additional technical assumptions
on the configuration. We do not pursue this direction further as constructions
are not the main focus of this paper.
\end{rem}

\begin{rem}[Complete classification for small $J$]
Our main results cover \emph{any distinct points} $z_{1}^{\ast},\dots,z_{J}^{\ast}$
if $J\leq4$, and hence provide a \emph{complete classification} of
the dynamics for $J\leq4$ under $N\geq7$ and Assumption~\ref{assumption:sequential-on-param}.

If $J\in\{2,3\}$, there can be only \emph{two scenarios}: either
the one of Theorem~\ref{thm:main-one-bubble-tower-classification}
or the totally non-degenerate scenario of Theorem~\ref{thm:main-lmb-to-zero-classification}.
This is due to Lemma~\ref{lem:almost-same-sign-implies-totally-non-degen};
every configuration is totally non-degenerate in this case.

If $J=4$, there can be only\emph{ three scenarios}: one scenario
from Theorem~\ref{thm:main-one-bubble-tower-classification} and
two scenarios from Theorem~\ref{thm:main-lmb-to-zero-classification}.
This is due to Lemma~\ref{lem:almost-same-sign-implies-totally-non-degen}
and Lemma~\ref{lem:J=00003D4-implies-not-too-degenerate}; every
configuration is either totally non-degenerate or minimally degenerate
satisfying \eqref{eq:case3-v_i-not-all-zero}. We also remark that
there exists a large family of minimally degenerate configurations
even in case of $J=4$ (Lemma~\ref{lem:large-family-degen}). In
Section~\ref{sec:Construction-of-minimal-examples}, we also construct
a nonlinear solution exhibiting this degenerate scenario.
\end{rem}

\begin{rem}[On optimality of total non-degeneracy]
\label{rem:optimality-non-degen}The optimality of total non-degeneracy
in the first item of Theorem~\ref{thm:main-lmb-to-zero-classification}
can be explained by another exotic example $J=5$ satisfying 
\[
\lmb_{1}\aeq\dots\aeq\lmb_{4}\aeq t^{-\frac{1}{N-3}}\gg\lmb_{5}\aeq t^{-\frac{N-4}{(N-3)(N-6)}}.
\]
This example might be constructed when $\{(\iota_{1},z_{1}^{\ast}),\dots,(\iota_{4},z_{4}^{\ast})\}$
is minimally degenerate and the last point $z_{5}^{\ast}$ is located
rather generically (but with the correct sign $\iota_{5}$ depending
on $z_{5}^{\ast}$). The rigorous proof might be done by extending
the argument in Section~\ref{sec:Construction-of-minimal-examples},
but it would more intricate due to the absence of symmetry and the
addition of the point $z_{5}^{\ast}$. The full configuration $\{(\iota_{1},z_{1}^{\ast}),\dots,(\iota_{5},z_{5}^{\ast})\}$
is generically non-degenerate, but it is not totally non-degenerate
due to the degenerate sub-configuration $\{(\iota_{1},z_{1}^{\ast}),\dots,(\iota_{4},z_{4}^{\ast})\}$.
This example explains the necessity of total non-degeneracy in the
first item of Theorem~\ref{thm:main-lmb-to-zero-classification}. 
\end{rem}

\begin{rem}[On assumption \eqref{eq:case3-v_i-not-all-zero} in the minimally
degenerate scenario]
The vectors $v_{i}^{\ast}\in\bbR^{N}$ appeared in the assumption
\eqref{eq:case3-v_i-not-all-zero} are related to the time variation
of $z_{i}(t)$. For $J\leq4$, \eqref{eq:case3-v_i-not-all-zero}
is always satisfied by Lemma~\ref{lem:J=00003D4-implies-not-too-degenerate}.
We do not have any example of minimally degenerate configurations
violating \eqref{eq:case3-v_i-not-all-zero}. Let us also mention
that for large $J$, there exists a configuration $\{z_{1}^{\ast},\dots,z_{J}^{\ast}\}$
such that $A^{\ast}$ has many linearly independent nonnegative kernel
elements (Lemma~\ref{lem:large-dim-kernel}). 
\end{rem}

\begin{rem}[Sign-changing solutions]
Our solutions are necessarily\emph{ sign-changing}; $\iota_{1},\dots,\iota_{J}$
cannot be all positive or all negative in every scenario of our main
theorems. For Theorem~\ref{thm:main-one-bubble-tower-classification},
this is clear. For Theorem~\ref{thm:main-lmb-to-zero-classification},
since every configuration with all the same signs is totally non-degenerate,
the first item of Theorem~\ref{thm:main-lmb-to-zero-classification}
applies, but the fact that the set \eqref{eq:main-thm-nondegen-ratio-set}
is non-empty (the left hand side is strictly positive) excludes this
scenario to occur. In view of soliton resolution for nonnegative data
(cf.~\cite{Aryan2025arXiv,Ishiwata2025preprint}), any nonnegative
global solution either escapes our Assumption~\ref{assumption:sequential-on-param}
(we hope to address this situation in the future), converges to $\iota W_{\lmb^{\ast},z^{\ast}}$,
or converges to $0$ in $\dot{H}^{1}$. 

If only one sign is positive and the others are all negative (say
$-\iota_{1}=\iota_{2}=\dots=\iota_{J}$), then the configuration $\{z_{1}^{\ast},\dots,z_{J}^{\ast}\}$
is always totally non-degenerate (Lemma~\ref{lem:almost-same-sign-implies-totally-non-degen}).
Therefore, our main theorems fully cover this case as well.
\end{rem}

\begin{rem}[On dimensions]
\label{rem:dimensions}We assume $N\geq7$ to have a classification
result for solutions whose initial data lies in the critical space
$\dot{H}^{1}$. As mentioned in \cite{CollotMerleRaphael2017CMP,KimMerle2025CPAM},
the classification for $\dot{H}^{1}$ solutions would change drastically
in lower dimensions. The dimension $N=6$ is critical in the sense
that one loses the stabilization property ($\lmb_{i_{0}}(t)\to\lmb_{i_{0}}^{\ast}\in(0,\infty)$
of Theorem~\ref{thm:main-one-bubble-tower-classification}); see
\cite{Harada2025arXiv,Harada2024arXiv}. We also note that there is
a type-II blow-up solution \cite{Harada2020AnnPDE} when $N=6$, which
is not $\dot{H}^{1}$-bounded. The dynamics in lower dimensions $N\leq5$
is much less rigid. There exists type-II blow-up \cite{Schweyer2012JFA,delPinoMussoWei2019ActaSinica,delPinoMussoWeiZhou2020DCDS,delPinoMussoWeiZhangZhou2020arXiv,Harada2020AIHP}
with earlier formal construction \cite{FilippasHerreroVelazquez2000}.
There is no rigidity for global solutions as well; see for example
\cite{FilaKing2012,delPinoMussoWei2020AnalPDE,WeiZhangZhou2024JDE,LiWeiZhangZhou2024NonlinearAnal}. 

Nevertheless, the construction of various scenarios discussed in this
paper needs less information compared to classification, so we expect
that some of them can be constructed in lower dimensions. For example,
the non-degenerate scenario of Theorem~\ref{thm:main-lmb-to-zero-classification}
in the whole domain $\bbR^{N}$ might be constructed for $N\geq5$
in analogy with \cite{CortazarDelPinoMusso2020JEMS,delPinoMussoWeiZheng2020Pisa}
and the construction of the degenerate scenario of Theorem~\ref{thm:main-lmb-to-zero-classification}
might be possible even for lower dimensions. 
\end{rem}

\begin{rem}[Finite-time blow-up case]
As mentioned before, finite-time blow-up of time-sequential non-colliding
$W$-bubbling solutions is easily excluded by localizing the argument
of \cite{CollotMerleRaphael2017CMP}; see Proposition~\ref{prop:finite-time-non-existence}
below. 
\end{rem}

\begin{rem}[On Assumption~\ref{assumption:sequential-on-param}]
First, we assume sequential boundedness of $\lmb_{i,n}$ and $z_{i,n}$.
(Note, however, that one can work with slightly weaker assumptions
as in Remark~\ref{rem:1.4} that can include unbounded $\lmb_{i,n}$
and $z_{i,n}$.) Such boundedness is suggested by the formal conservation
law \eqref{eq:intro-formal-cons-2}, but unfortunately we do not have
such controls globally in time for the full problem with error terms.
On the other hand, we also have the time-sequential non-colliding
assumption since $z_{i}^{\ast}$ are distinct. As we defined $W$-bubbling
in terms of scaling invariant $\dot{H}^{1}$-norm, the colliding regime
(two $z_{i,n}$ converging to the same point) and the unboundedness
of $\lmb_{\max,n}$ are linked directly in the proof through rescaling.
They have to be treated differently and this situation is beyond the
scope of this paper. We hope to address these issues in the future
building on the present work and \cite{KimMerle2025CPAM}. 
\end{rem}

\begin{rem}[Related constructions for the non-degenerate rate]
In a smooth bounded domain $\Omg\subset\bbR^{N}$ with $N\geq5$,
Cortázar, del Pino, and Musso~\cite{CortazarDelPinoMusso2020JEMS}
constructed positive multi-bubble solutions with the rates $\lmb_{i}(t)\aeq t^{-1/(N-4)}$.
The construction there utilizes the boundary effect on Green's function,
which enables the construction of positive solutions with $z_{i}^{\ast}$
sufficiently close to the boundary $\rd\Omg$. In $\bbR^{N}$, however,
positive solutions of this type are excluded by our classification.
We note that the matrix $A^{\ast}$ is never positive-definite, whereas
the construction in \cite{CortazarDelPinoMusso2020JEMS} is done for
$z_{i}^{\ast}$s whose corresponding matrix ($\calG(q)$ in \cite{CortazarDelPinoMusso2020JEMS})
is positive-definite. Another related construction is \cite{delPinoMussoWeiZheng2020Pisa},
where the authors used a different bubbling profile: the sign-changing
one of \cite{delPinoMussoPacardPistoia2011JDE}.
\end{rem}

\begin{rem}[Related constructions and classifications]
This type of problem is not limited to \eqref{eq:NLH}. Let us mention
some constructions and classification results (that were not mentioned
before) from other critical PDEs. As there is a huge literature especially
on constructions, we can mention only a few results.

For near-soliton dynamics and type-II single-bubble finite-time blow-up
(with radiation), see for example \cite{BourgainWang1997,MerleRaphael2005AnnMath,MerleRaphael2003GAFA,Raphael2005MathAnn,MerleRaphael2004InventMath,MerleRaphael2006JAMS,MerleRaphael2005CMP}
for $L^{2}$-critical NLS, \cite{MartelMerleRaphael2014Acta} for
$L^{2}$-critical gKdV, \cite{GustafsonKangTsai2008Duke,GustafsonNakanishiTsai2010CMP,MerleRaphaelRodnianski2013InventMath,RaphaelSchweyer2013CPAM,RaphaelSchweyer2014AnalPDE}
for 2D Landau--Lifshitz--Gilbert flows, \cite{KriegerSchlagTataru2008Invent,RaphaelRodnianski2012Publ.Math.,KriegerSchlagTataru2009Duke,HillairetRaphael2012AnalPDE,JendrejLawrieRodriguez2022ASENS}
for critical wave equations, \cite{CollotGhoulMasmoudiNguyen2022CPAM,Mizoguchi2020CPAM}
for 2D Keller--Segel model, \cite{Kim2025JEMS} for the equivariant
self-dual Chern--Simons--Schrödinger equation, and very recent \cite{JeongKimKimKwon2026arXiv}
for the Calogero--Moser derivative NLS. Here, the latter three results
\cite{Mizoguchi2020CPAM,Kim2025JEMS,JeongKimKimKwon2026arXiv} prove
classifications for large data in the single-bubble case. Many of
these results rely on the forward construction scheme more systematically
developed in \cite{MerleRaphaelRodnianski2015CambJMath,Collot2018MemAMS}.

With more constraints such as threshold mass or energy, see \cite{Merle1993Duke,RaphaelSzeftel2011JAMS,Dodson2023AnnPDE,Dodson2024AnalPDE,DuyckaertsMerle2008IMRP,DuyckaertsMerle2009GAFA,MartelMerleRaphael2015JEMS,JendrejLawrie2018Invent},
and also \cite{Martel2005AJM,JendrejLawrie2023CPAM}. 

For multi-solitons (or multi-bubbles), see \cite{Merle1990CMP,Martel2005AJM,MartelRaphael2018AnnSci,DaviladelPinoWei2020Invent,CollotGhoulMasmoudiNguyen2024arXiv}
for non-radial multi-solitons, and see \cite{Jendrej2017AnalPDE,Jendrej2019AJM,JendrejLawrie2023CPAM,JendrejKrieger2025arXiv,delPinoMussoWei2021AnalPDE}
for radial bubble towers. Let us also mention \cite{MerleZaag2012AJM,CoteZaag2013CPAM,JendrejLawrie2024arXiv}
for non-critical PDEs involving multi-solitons.
\end{rem}

\begin{rem}[Other types of classification results for semilinear heat equation]
Let us consider the nonlinear heat equation 
\[
\rd_{t}u=\Dlt u+|u|^{p-1}u
\]
with general $p>1$ in this remark. Our work is closely related to
those on type-II blow-up in the critical case $p=p_{s}=\frac{N+2}{N-2}$.
Note that there is a finite-time ODE blow-up for any $p>1$, called
\emph{type-I} ($\|u(t)\|_{L^{\infty}}\aeq(T_{+}-t)^{-\frac{1}{p-1}}$).
In the subcritical range $p<p_{s}$, every finite-time blow-up is
known to be of type-I \cite{GigaKohn1985CPAM,GigaKohn1987Indiana,GigaMatsuiSasayama2004Indiana}.
(A similar result holds for sub-conformal wave equations \cite{MerleZaag2003AJM,MerleZaag2005MathAnn}.)
On the other hand, when $p>p_{JL}$, where $p_{JL}$ is the Joseph--Lundgren
exponent, a type-II (single-bubble) blow-up was found by Herrero and
Velázquez \cite{HerreroVelazquez-preprint,HerreroVelazquezCRASP1994}
(see also \cite{Mizoguchi2004AdvDE,Collot2017AnalPDE}) together with
a sequence of polynomial rates for $\|u(t)\|_{L^{\infty}}$. Then,
Mizoguchi \cite{Mizoguchi2007MathAnn,Mizoguchi2011TAMS} proved the
following\emph{ }classification: radial nonnegative (with more assumptions
on $u_{0}$) type-II blow-up solutions should exhibit the previously
constructed rates. Finally, in the intermediate range of $p$, Matano
and Merle \cite{MatanoMerle2004CPAM} (see also \cite{Mizoguchi2011JDE})
excluded type-II blow-up solutions under radial symmetry in the range
$p_{s}<p<p_{JL}$, and under $u_{0}\geq0$ and radial symmetry in
the critical case $p=p_{s}$. In particular, any radial type-II blow-up
solutions in the critical case must be sign-changing. See also \cite{WangWei2021arXiv},
where type-II blow-up is excluded for $u_{0}\geq0$, non-radial, $N\geq7$
in the critical case.
\end{rem}

\begin{prop}[Non-existence of non-colliding finite-time blow-up]
\label{prop:finite-time-non-existence}Let $N\geq7$. There does
not exist a $\dot{H}^{1}$-solution $u(t)$ defined on $[0,T_{+})$
with the following property.
\begin{itemize}
\item $T_{+}<+\infty$ and $\limsup_{t\to T_{+}}\|\chi u(t)\|_{\dot{H}^{1}}<+\infty$.
\item There exists $u^{\ast}\in\dot{H}^{1}$ such that 
\begin{equation}
\lim_{t\to T_{+}}\|(\chi-\chi_{r})(u(t)-u^{\ast})\|_{\dot{H}^{1}}=0\qquad\forall r\in(0,1).\label{eq:finite-time-assumption-2}
\end{equation}
\item There exist sequences of times $t_{n}\to T_{+}$, scales $\lmb_{n}\to0$,
and spatial centers $z_{n}\to0$ in $\bbR^{N}$ such that 
\begin{equation}
\lim_{n\to\infty}\|\chi\{u(t_{n})-u^{\ast}-W_{\lmb_{n},z_{n}}\}\|_{\dot{H}^{1}}=0.\label{eq:finite-time-assumption-3}
\end{equation}
\end{itemize}
Here, $\chi$ denotes a smooth radially symmetric function such that
$0\leq\chi\leq1$, $\chi(x)=1$ if $|x|\leq1$, and $\chi(x)=0$ if
$|x|\geq2$. We denoted $\chi_{R}\coloneqq\chi(\frac{\cdot}{R})$
for $R\in(0,\infty)$.
\end{prop}

We sketch the proof of this proposition (with many details) in Section~\ref{subsec:finite-time-non-existence}.

\subsection{Notation}

Unless stated otherwise, we assume throughout the paper 
\[
N\geq7.
\]

\noindent \emph{\uline{Basic notation}}.
\begin{itemize}
\item $\bbN$ is the set of positive integers, $\bbN_{0}\coloneqq\bbN\cup\{0\}$. 
\item For $A\in\bbR$ and $B\geq0$, we denote $A\aleq B$ or $A=\calO(B)$
if $|A|\leq CB$ for some universal constant $C>0$. If $A,B\geq0$,
we denote $A\aeq B$ if $A\aleq B$ and $A\ageq B$. The dependence
on parameters is written in subscripts. Dependence on $N$ or $J$
is ignored.
\item We use the small $o$-notation such as $o_{n\to\infty}(1)$ and $o_{R\to\infty}(1)$,
which denote quantities going to zero as $n\to\infty$ and $R\to\infty$,
respectively. Typically, $o(1)$ denotes a function in time of size
$o_{t\to+\infty}(1)$.
\item The integral sign $\int$ without any specification denotes $\int f=\int_{\bbR^{N}}f(x)dx$.
\item $\lan\cdot,\cdot\ran$ means the following. When $u$ and $v$ are
real-valued functions, $\lan u,v\ran=\int uv$. When $\vec{a},\vec{b}\in\bbR^{J}$
are vectors ($J\in\bbN$ denotes the number of bubbles), $\lan\vec{a},\vec{b}\ran=\sum_{i=1}^{J}a_{i}b_{i}$. 
\item For $x\in\bbR^{N}$, we use both the notation $x=(x_{1},x_{2},\dots,x_{N})=(x^{1},x^{2},\dots,x^{N})$. 
\item For $x,y\in\bbR^{N}$, the scalar product of $x$ and $y$ is denoted
by $x\cdot y=\sum_{a=1}^{N}x_{a}y_{a}$.
\item We use the notation $\chf_{A}$. If $A$ is a statement, it means
that it is $1$ if $A$ is true and $0$ otherwise. If $A$ is a set,
then it denotes the indicator function on $A$; $\chf_{A}(x)=1$ if
$x\in A$ and $0$ otherwise.
\item $f_{\lmb,z}(x)\coloneqq\lmb^{-\frac{N-2}{2}}f(\lmb^{-1}(x-z))$ for
a function $f$ on $\bbR^{N}$, $\lmb\in(0,\infty)$, and $z\in\bbR^{N}$.
$f_{\lmb}\coloneqq f_{\lmb,0}$.
\item For $s\in\bbR$, $\Lmb_{s}\coloneqq x\cdot\nabla+\frac{N}{2}-s$ denotes
the generator of $\dot{H}^{s}$-scaling. We define $\Lmb\coloneqq\Lmb_{1}$. 
\item For $x\in\bbR$, $\lan x\ran\coloneqq(1+|x|^{2})^{\frac{1}{2}}$.
\item We fix $\chi\in C_{c}^{\infty}(\bbR^{N})$ a smooth radially symmetric
function on $\bbR^{N}$ such that $0\leq\chi\leq1$, $\chi(x)=1$
if $|x|\leq1$, and $\chi(x)=0$ if $|x|\geq2$. We denote $\chi_{R}\coloneqq\chi(\frac{\cdot}{R})$
and $\chi_{R,z}\coloneqq\chi(\frac{\cdot-z}{R})$ for $R\in(0,\infty)$
and $z\in\bbR^{N}$. 
\item We denote $2^{\ast}=\frac{2N}{N-2}$, $2^{\ast\ast}=\frac{2N}{N-4}$,
$(2^{\ast})'=\frac{2N}{N+2}$, and $(2^{\ast\ast})'=\frac{2N}{N+4}$. 
\item We denote $p=\frac{N+2}{N-2}$, $f(u)=|u|^{p-1}u$, and $f'(u)=p|u|^{p-1}$
for the nonlinearity of \eqref{eq:NLH}.
\item $\calL_{U}\coloneqq\Dlt+f'(U)$ denotes a linearized operator around
a function $U$.
\item $\NL_{U}(g)\coloneqq f(U+g)-f(U)-f'(U)g$ denotes the nonlinear term
in $g$ in $f(U+g)$.
\item Motivated by our previous work \cite{KimMerle2025CPAM}, it is convenient
to introduce 
\[
D=\frac{N-2}{2}.
\]
\item The following universal constants naturally appear: 
\begin{align}
\kpp_{0} & \coloneqq\frac{N-2}{2}\frac{\int_{\bbR^{N}}W^{p}dx}{\int_{\bbR^{N}}(\Lmb W)^{2}dx}=-\frac{\lan f'(W),\Lmb W\ran}{\|\Lmb W\|_{L^{2}}^{2}}>0,\label{eq:def-kpp}\\
\kpp_{1} & \coloneqq(N-2)\frac{\int_{\bbR^{N}}W^{p}dx}{\int_{\bbR^{N}}(\rd_{1}W)^{2}dx}>0,\label{eq:def-kpp_1}\\
\kpp_{\infty} & \coloneqq\lim_{|x|\to\infty}|x|^{N-2}W(x)=(N(N-2))^{\frac{N-2}{2}}.\label{eq:def-kpp_infty}
\end{align}
\end{itemize}
\emph{\uline{Notation for multi-bubbles}}. 
\begin{itemize}
\item Let $J\in\bbN$ denote the number of bubbles. We denote $\setJ\coloneqq\{1,\dots,J\}$. 
\item The sign, scale, and spatial center of each bubble may be described
by $(\iota_{i},\lmb_{i},z_{i})\in\{\pm\}\times(0,\infty)\times\bbR^{N}$
for $i\in\setJ$; these can collectively be denoted by 
\[
(\vec{\iota},\vec{\lmb},\vec{z})\in\calP_{J}\coloneqq\{\pm\}^{J}\times(0,\infty)^{J}\times(\bbR^{N})^{J}.
\]
\item When $(\vec{\iota},\vec{\lmb},\vec{z})\in\calP_{J}$ is given, for
a function $f$ on $\bbR^{N}$ we denote 
\[
f_{;i}(x)\coloneqq\frac{\iota_{i}}{\lmb_{i}^{\frac{N-2}{2}}}f(y_{i}),\quad y_{i}\coloneqq\frac{x-z_{i}}{\lmb_{i}},\quad\text{and}\quad f_{\ul{;i}}\coloneqq\frac{1}{\lmb_{i}^{2}}f_{;i}.
\]
\item For $(\vec{\iota},\vec{\lmb},\vec{z})\in\calP_{J}$, the sum of $W$
bubbles is denoted by 
\[
\calW(\vec{\iota},\vec{\lmb},\vec{z})\coloneqq\sum_{i=1}^{J}\iota_{i}W_{\lmb_{i},z_{i}}=\sum_{i=1}^{J}W_{;i}.
\]
\item We denote $z_{i}^{0}=\lmb_{i}$, most frequently used in the expression
$\lmb_{i}\rd_{z_{i}^{a}}$ with $a\in\{0,1,\dots,N\}$. 
\item As in our $W$-bubbling assumption, we will work with bubbles that
are asymptotically orthogonal (at least sequentially in time). Let
$\dlt>0$ and introduce the set 
\[
\calP_{J}(\dlt)\coloneqq\{(\vec{\iota},\vec{\lmb},\vec{z})\in\calP_{J}:R^{-1}<\dlt\},
\]
where 
\begin{equation}
\begin{aligned}R_{ij} & \coloneqq\max\Big\{\sqrt{\frac{\lmb_{i}}{\lmb_{j}}},\sqrt{\frac{\lmb_{j}}{\lmb_{i}}},\frac{|z_{i}-z_{j}|}{\sqrt{\lmb_{i}\lmb_{j}}}\Big\}\quad\text{for }i\neq j,\\
R_{i} & \coloneqq\min_{j\in\setJ\setminus\{i\}}R_{ij},\quad\text{and}\quad R\coloneqq\min_{i\in\setJ}R_{i}.
\end{aligned}
\label{eq:def-R}
\end{equation}
The smaller $R^{-1}$ is, the smaller the interaction between bubbles
is. The asymptotic orthogonality of parameters $(\vec{\iota}_{n},\vec{\lmb}_{n},\vec{z}_{n})$
means $R_{n}^{-1}\to0$ as $n\to\infty$. Finally, we define their
$\dlt$-tubular neighborhood 
\begin{equation}
\calT_{J}(\dlt)\coloneqq\{u\in\dot{H}^{1}:\exists(\vec{\iota},\vec{\lmb},\vec{z})\in\calP_{J}(\dlt)\text{ such that }\|u-\calW(\vec{\iota},\vec{\lmb},\vec{z})\|_{\dot{H}^{1}}<\dlt\}.\label{eq:def-tube-J}
\end{equation}
\item Let $\vec{\lmb}=(\lmb_{1},\dots,\lmb_{J})\in(0,\infty)^{J}$. We denote
by $\lmb_{\max},\lmb_{\secmax},\lmb_{\min}$ the largest, second largest,
and the smallest $\lmb_{i}$ among $\lmb_{1},\dots,\lmb_{J}$, respectively.
For any $q\in\bbR$, we also use the notation $\vec{\lmb}^{q}\coloneqq(\lmb_{1}^{q},\dots,\lmb_{J}^{q})\in(0,\infty)^{J}$.
\end{itemize}
\emph{\uline{Some additional notations}}.
\begin{itemize}
\item For $\calZ_{a}$, $\calV_{a}$, $\calY$, see Section~\ref{subsec:Linear-coercivity-estimates}.
\item $f(a,b)=\min\{|a|^{p-1}|b|,|b|^{p-1}|a|\}$; see \eqref{eq:def-f(a,b)}.
\item For $U$, $\td U$, and $\frkr_{a,i}$, see Proposition~\ref{prop:Modified-Profiles}.
\item For $\vec{\frkc}$ and $v_{i}^{\ast}$ when $\{(\iota_{1},z_{1}^{\ast}),\dots,(\iota_{J},z_{J}^{\ast})\}$
is minimally degenerate, see Lemma~\ref{lem:min-degen-implies-unique-pos-vec}
and Lemma~\ref{lem:case3-non-degen-quantity}, respectively.
\item For $\Ttbexit_{n}(\alp)$ and $T_{n}(K_{0},\dlt_{2},\alp)$, see \eqref{eq:def-Ttbexit_n}
and \eqref{eq:def-T_n(K,delta,alpha)}, respectively.
\end{itemize}

\subsection{Strategy of the proof and organization of the paper}

We use \emph{modulation analysis}; we effectively reduce the PDE dynamics
into finite-dimensional dynamics of parameters and aim to completely
classify their asymptotic behaviors. One of the main challenges is
of course the classification of the asymptotic behaviors of scaling
parameters $\lmb_{i}(t)$ due to the large number of equations and
time-sequential assumptions. \textcolor{red}{\smallskip{}
}

\textbf{\emph{Step~1.}}\emph{ Modified multi-bubble profiles, modulation,
and spacetime control} (Section~\ref{sec:Modified-multi-bubble-profiles}--\ref{sec:Modulation-in-multi-bubble}).
The goal of this step is to set the ground for the modulation analysis
in the non-radial case, by generalizing the basic modulation setup
in the radial case \cite{KimMerle2025CPAM} to the non-radial case
(more precisely, for solutions close to $W$ multi-bubbles). We develop
this basic modulation theory \emph{in a general setting including
collisions}. Some parts (e.g., the modified multi-bubble profile construction)
might not be entirely new. We hope to use this basic but general setup
as preliminaries in our future studies.

One of the crucial observations in the radial case \cite{KimMerle2025CPAM}
is that, by constructing suitable (radial) modified multi-bubble profiles
$U$, one can obtain a \emph{spacetime (monotonicity) estimate} near
any (radial) multi-bubbles, justifying that the error term $g(t)=u(t)-U(t)$
does not contribute to the modulation dynamics. We hope to generalize
this to the non-radial setting for $u$ close to $W$ multi-bubbles.

The construction of modified multi-bubble profiles is an elliptic
problem (no time dependence). A very similar construction is done
in Deng--Sun--Wei \cite{DengSunWei2025Duke}, but we need some modifications
and additional estimates to be used in dynamical analysis. For $\dlt>0$
small and $(\vec{\iota},\vec{\lmb},\vec{z})\in\calP_{J}(\dlt)$, we
construct a \emph{modified multi-bubble profile} $U=U(\vec{\iota},\vec{\lmb},\vec{z})$
as a perturbation of $\calW(\vec{\iota},\vec{\lmb},\vec{z})$ such
that 
\begin{equation}
\Dlt U+f(U)=\sum_{j=1}^{J}\Big\{\frac{\frkr_{0,j}}{\lmb_{j}^{2}}[\Lmb W]_{;j}+\sum_{b=1}^{N}\frac{\frkr_{b,j}}{\lmb_{j}^{2}}[\rd_{b}W]_{;j}\Big\}\label{eq:strategy-U-eqn}
\end{equation}
for some constants $\frkr_{b,j}$, $b\in\{0,1,\dots,N\}$, $j\in\setJ$.
The specific choice of the profiles $[\Lmb W]_{;j}$ and $[\rd_{b}W]_{;j}$
is crucial in the monotonicity estimate (in the next paragraph) and
the construction of $U$ is done by a standard Lyapunov--Schmidt
reduction method as in \cite{DengSunWei2025Duke}. We will only estimate
$\|\td U\|_{\dot{H}^{1}}$, $\|\td U\|_{\dot{H}^{2}}$, and $\|\lmb_{i}\rd_{z_{i}^{a}}\td U\|_{\dot{H}^{1}}$,
where $\td U=U-\calW$, showcasing that these estimates are sufficient
to prove our classification results. Note that the estimate for $\lmb_{i}\rd_{z_{i}^{a}}\td U$
is necessary in the dynamical analysis due to the time variation of
parameters. 

Having constructed modified multi-bubble profiles, we obtain a spacetime
control by generalizing the radial case \cite{KimMerle2025CPAM}.
For a $\calW$-bubbling solution $u(t)$ with maximal lifespan $[0,T_{+})$,
the energy identity \eqref{eq:energy-identity} implies a \emph{spacetime
control} (or \emph{monotonicity estimate}) 
\begin{equation}
\int_{0}^{T_{+}}\|\Dlt u+f(u)\|_{L^{2}}^{2}dt<+\infty.\label{eq:strategy-dissipation}
\end{equation}
The goal is then to obtain a useful coercivity estimate for the dissipation
term $\|\Dlt u+f(u)\|_{L^{2}}$, at each time. As the bubbling assumption
suggests, we consider this when $u$ is sufficiently close to $W$
multi-bubbles, or $u\in\calT_{J}(\dlt)$ for some small $\dlt>0$.
Simply decomposing $u=U+g$ with suitable orthogonality conditions
on $g$, the fact that $[\Lmb W]_{;j}$ and $[\rd_{b}W]_{;j}$ in
the right hand side of \eqref{eq:strategy-U-eqn} are almost kernel
elements of the linearized operator $\calL_{\calW}=\Dlt+f'(\calW)$
implies 
\[
\|\Dlt u+f(u)\|_{L^{2}}^{2}\aeq\|g\|_{\dot{H}^{2}}^{2}+\sum_{b,j}\frac{\frkr_{b,j}^{2}}{\lmb_{j}^{2}}.
\]
Substituting this control into \eqref{eq:strategy-dissipation} gives
a spacetime control for $g$ and $\frkr_{b,j}/\lmb_{j}$ on time intervals
where they are defined.

It is worthwhile to mention that the modulation parameters for the
decomposition $u=U(\vec{\iota},\vec{\lmb},\vec{z})+g$ in $\calT_{J}(\dlt)$
are only determined up to permutation of indices; see Lemma~\ref{lem:static-modulation}. 

\smallskip{}

\textbf{\emph{Step~2.}}\emph{ Solutions under Assumption~\ref{assumption:sequential-on-param}
and introduction of bootstrap times} $T_{n}$. From now on, we consider
solutions $u(t)$ satisfying Assumption~\ref{assumption:sequential-on-param}.
Together with a decomposition lemma, we are in the situation where
$u(t)$ is a global solution with $t_{n}\to+\infty$ such that $\lmb_{\max}(t_{n})\aleq1$,
$z_{i}(t_{n})\to z_{i}^{\ast}$, and $\|g(t_{n})\|_{\dot{H}^{1}}+R^{-1}(t_{n})\to0$,
where $R$ is defined in \eqref{eq:def-R}. 

Our assumption is \emph{only sequential}. Thus we may introduce a
bootstrap for Assumption~\ref{assumption:sequential-on-param}; say
$\lmb_{\max}(t)\leq K_{0}$, $|z_{i}-z_{i}^{\ast}|\leq\dlt_{2}$,
$\|g(t_{n})\|_{\dot{H}^{1}}+R^{-1}(t_{n})\leq\alp$ for some large
constant $K_{0}$, small constants $\dlt_{2}$ and $\alp>0$. As these
controls are true at $t=t_{n}$, we can define the first exit time
$T_{n}=T_{n}(K_{0},\dlt_{2},\alp)$. To keep the presentation simple,
let us omit mentioning these parameters from now on. 

Now, we turn our attention to the ODE system of the parameters $\lmb_{i}$
and $z_{i}$ and hope to classify their asymptotic behaviors with
simple sequential assumptions on $u(t_{n})$ with $t_{n}\to+\infty$.
In various scenarios considered in this paper, (i) the primary goal
is to close this bootstrap (hence Assumption~\ref{assumption:sequential-on-param}
is justified continuous-in-time), and (ii) the next goal is to classify
the asymptotic behavior of $\lmb_{i}$ and $z_{i}$.

The first remarkable case is when one $\lmb_{i}$ is much larger than
the others, without additional conditions on $z_{1}^{\ast},\dots,z_{J}^{\ast}$.
This case is considered in the next step. After treating that case,
we assume some conditions on $z_{1}^{\ast},\dots,z_{J}^{\ast}$. We
first consider the totally non-degenerate case for the equation of
the scales $\lmb_{i}$. The last case corresponds to the situation
where the equation of the scales $\lmb_{i}$ degenerates but we still
have non-degeneracy for the equation of the spatial centers $z_{i}$.\smallskip{}

\textbf{\emph{Step~3.}}\emph{ Classification when $\lmb_{\max,n}\gg\lmb_{\secmax,n}$:
proof of Theorem~\ref{thm:main-one-bubble-tower-classification}}
(Section~\ref{sec:Case2}). Assume we are in the case $\lmb_{\max}(t_{n})\gg\lmb_{\secmax}(t_{n})$;
let $i_{0}\in\setJ$ be such that $\lmb_{i_{0}}(t_{n})=\lmb_{\max}(t_{n})$.
Let 
\[
\calI_{+}=\{i\in\setJ:\iota_{i}=\iota_{i_{0}}\}.
\]
We need to prove $\calI_{+}=\{i_{0}\}$ and the asymptotic dynamics
of $\lmb_{i}$ and $z_{i}$ are described as in Theorem~\ref{thm:main-one-bubble-tower-classification}\emph{.}
Roughly speaking, the dynamics in this situation is governed by the
influence from the bubble $W_{;i_{0}}$ of maximum scale (thanks to
the fact that $z_{i}^{\ast}$ are distinct and $\lmb_{\max}\aleq1$);
for $i\neq i_{0}$, $\lmb_{i}$ is decreasing if $i\notin\calI_{+}$,
and $\lmb_{i}$ is increasing if $i\in\calI_{+}$. 

Closely related to the fact that $\lmb_{i}$ is not static in time
is the observation that $\max_{b,j}|\frkr_{b,j}|/\lmb_{j}$ enjoys
a certain lower bound, which combined with the spacetime control $\int_{t_{n}}^{T_{n}}\{\tsum{b,j}{}\frkr_{b,j}^{2}/\lmb_{j}^{2}+\|g(t)\|_{\dot{H}^{2}}^{2}\}dt=o_{n\to\infty}(1)$
allows us to propagate the smallness $\|g(t_{n})\|_{\dot{H}^{1}}+R^{-1}(t_{n})=o_{n\to\infty}(1)$
forward in time via the energy identity \eqref{eq:energy-identity}.
This leads to $\|g(t)\|_{\dot{H}^{1}}+R^{-1}(t)=o_{n\to\infty}(1)$,
closing the bootstrap for $\|g(t)\|_{\dot{H}^{1}}+R^{-1}(t)$. Thus
we can focus on the dynamics of $\lmb_{i}$ and $z_{i}$, which is
finite-dimensional. 

If $\calI_{+}\supsetneq\{i_{0}\}$, we need to derive a contradiction.
For $i\in\calI_{+}$ not equal to $i_{0}$, as $W_{;i}$ and $W_{;i_{0}}$
have the same sign, both $\lmb_{i}$ and $\lmb_{i_{0}}$ are growing.
Note that the variation of $z_{i}$ can be controlled roughly by the
amount of change of $\lmb_{i}$, so we will not discuss $z_{i}$ further
in this step. One can justify that $\lmb_{i}$ can grow much so that
$\lmb_{i}(t)/\lmb_{i_{0}}(t)\ageq1$, contradicting to the asymptotic
orthogonality of parameters $R(t)=o_{n\to\infty}(1)$. Thus $\calI_{+}=\{i_{0}\}$.

Having established $\calI_{+}=\{i_{0}\}$, we need to show that $\lmb_{i_{0}}(t)\to\lmb_{i_{0}}^{\ast}$
and $\lmb_{i}(t)$ for $i\neq i_{0}$ satisfy the asymptotics in Theorem~\ref{thm:main-one-bubble-tower-classification}.
For $i\neq i_{0}$, as $\iota_{i}=-\iota_{i_{0}}$, $\lmb_{i}$ is
decreasing. This forces the ratio $\lmb_{i}/\lmb_{i_{0}}$ to decrease
as well, and we conclude $\lmb_{i}(t)/\lmb_{i_{0}}(t)=o(1)$. With
this information and the fact that $z_{i}^{\ast}$ are all distinct,
it is not difficult to conclude the asymptotic dynamics as in Theorem~\ref{thm:main-one-bubble-tower-classification}.

\smallskip{}

\textbf{\emph{Step~4.}}\emph{ Dynamical case separation} (Section~\ref{sec:Dynamical-case-separation}).
From now on, we consider the dynamics where Theorem~\ref{thm:main-one-bubble-tower-classification}
is not applicable. In particular, $\lmb_{\max,n}\aeq\lmb_{\secmax,n}$.
The almost orthogonality of parameters and $z_{i}\approx z_{i}^{\ast}$
force 
\[
\lmb_{\max,n}\aeq\lmb_{\secmax,n}\ll1.
\]
In this situation, the modulation dynamics is described by \eqref{eq:intro-lmb-eqn}
and \eqref{eq:intro-z-eqn} at the leading order. As mentioned in
Step~2, our primary goal is to close the bootstrap ($T_{n}=+\infty$),
which will be done when the configuration $\{z_{1}^{\ast},\dots,z_{J}^{\ast}\}$
is not too much degenerate as in the assumption of Theorem~\ref{thm:main-lmb-to-zero-classification}.

Perhaps of independent interest, we begin with a \emph{general} configuration
of \emph{distinct} points $z_{1}^{\ast},\dots,z_{J}^{\ast}$, and
prove that the dynamics of $u(t)$ satisfying Assumption~\ref{assumption:sequential-on-param}
falls into one of the three scenarios to be described in this paragraph.
It is important to introduce right \emph{dynamical assumptions} to
separate the cases. We consider the quantities 
\[
\min_{\emptyset\neq\calI\subseteq\setJ}\Big(\frac{|A_{\calI}^{\ast}\vec{\lmb}_{\calI}^{D}(t)|}{|\vec{\lmb}_{\calI}^{D}(t)|}+\sum_{j\notin\calI}\frac{\lmb_{j}(t)}{\lmb_{\max}(t)}\Big)\qquad\text{and}\qquad\frac{\lmb_{\secmax}(t)}{\lmb_{\max}(t)}.
\]
The motivation for the second one is clear from Theorem~\ref{thm:main-one-bubble-tower-classification};
if we have $\lmb_{\max}(t_{n}')\gg\lmb_{\secmax}(t_{n}')$ for some
sequence of times $t_{n}'\in[t_{n},T_{n})$, possibly different from
the original $t_{n}$, then we can still apply Theorem~\ref{thm:main-one-bubble-tower-classification}
to get a classification result. (Here, we do not discuss whether $z_{i}$,
$g$, and $R$ are well-controlled also at $t_{n}'$.) The first choice
is motivated by the fact (Lemma~\ref{lem:distance-from-degen}) that
this quantity being away from zero implies that the leading terms
in the equations for $(\lmb_{i})_{t}$ \eqref{eq:intro-lmb-eqn} \emph{do
not degenerate}. However, a lower bound on this quantity does not
propagate forward in time (recall that $z_{1}^{\ast},\dots,z_{J}^{\ast}$
here are general distinct points), so as the first case (\textbf{Case~1})
we assume it for all time in the following sense: 
\begin{equation}
\min_{\emptyset\neq\calI\subseteq\setJ}\Big(\frac{|A_{\calI}^{\ast}\vec{\lmb}_{\calI}^{D}(t)|}{|\vec{\lmb}_{\calI}^{D}(t)|}+\sum_{j\notin\calI}\frac{\lmb_{j}(t)}{\lmb_{\max}(t)}\Big)\ageq1\qquad\text{for all }t\in[t_{n}',T_{n})\label{eq:intro-case1}
\end{equation}
for some sequence of times $t_{n}'\in[t_{n},T_{n})$. The other case
(\textbf{Cases~2--3}) becomes 
\[
\frac{|A_{\calI}^{\ast}\vec{\lmb}_{\calI}^{D}(t_{n}')|}{|\vec{\lmb}_{\calI}^{D}(t_{n}')|}+\sum_{j\notin\calI}\frac{\lmb_{j}(t_{n}')}{\lmb_{\max}(t_{n}')}\to0
\]
for some $\emptyset\neq\calI\subseteq\setJ$. Choose \emph{minimal}
such $\calI$. If $|\calI|=1$ (\textbf{Case~2}), then $A_{\calI}^{\ast}=0$
so this is equivalent to $\lmb_{\secmax}(t_{n}')\ll\lmb_{\max}(t_{n}')$,
which is already classified by Theorem~\ref{thm:main-one-bubble-tower-classification}.
If $|\calI|\geq2$ (\textbf{Case~3}), then this implies that the
sub-configuration $\{\vec{z}_{i}:i\in\calI\}$ is \emph{degenerate}
(while keeping $\lmb_{\max}(t)\aeq\lmb_{\secmax}(t)$ for all $t$
by minimality of $\calI$). \smallskip{}

\textbf{\emph{Step~5.}}\emph{ Classification in Case 1: proof of
the first item of Theorem~\ref{thm:main-lmb-to-zero-classification}}
(Section~\ref{sec:Case1}). The dynamics in Case~1 (hence assuming
\eqref{eq:intro-case1}) can be classified. This case includes the
first item of Theorem~\ref{thm:main-lmb-to-zero-classification}. 

As mentioned above, an essential consequence of \eqref{eq:intro-case1}
is that the leading terms in the equations for $(\lmb_{i})_{t}$ \eqref{eq:intro-lmb-eqn}
do not degenerate. This fact allows us to close the bootstrap for
$\|g(t)\|_{\dot{H}^{1}}+R^{-1}(t)$ via the energy identity and the
spacetime control. Since $\lmb_{\max}(t)\aeq\lmb_{\secmax}(t)$, the
asymptotic orthogonality of parameters imply $\lmb_{\max}(t)=o_{n\to\infty}(1)$
as well. Using the non-degeneracy, one can also justify that the variation
of $z$ is controlled by $\lmb_{\max}$, and hence $z$ is also well-controlled.
This closes the bootstrap and in particular the continuous-in-time
resolution with $z_{i}(t)\to z_{i}^{\ast}$. 

It then remains to prove the asymptotic behavior of $\lmb_{i}$. This
is then achieved by several small steps: proving the lower bound for
$\lmb_{\max}$, sequential upper bound for $\lmb_{\max}$, upper bound
for $\lmb_{\max}$, and then for each $\lmb_{i}$. Some of these steps
utilize backward integration (from $t=+\infty$) of the arguments
used in closing bootstrap. From a sequential bound to the continuous-in-time
bound requires additional arguments to exclude oscillations. Having
$\lmb_{\max}(t)\aeq t^{-1/(N-4)}$, one can integrate the equation
for each $(\lmb_{i})_{t}$ to conclude $\lmb_{i}(t)\aeq t^{-1/(N-4)}$
as well. 

We also show that all $\lmb_{i}(t)\aeq t^{-1/(N-4)}$ can happen only
when the configuration $\{z_{1}^{\ast},\dots,z_{J}^{\ast}\}$ is non-degenerate.
Finally, we show that the ratio vector $t^{1/(N-4)}\lmb_{i}(t)$ converges
to a connected component of the set \eqref{eq:main-thm-nondegen-ratio-set}
by observing a monotonicity functional (in self-similar variable)
at the leading order, similar to the one in \cite{CortazarDelPinoMusso2020JEMS}.\smallskip{}

\textbf{\emph{Step~6.}}\emph{ Proof of the second item of Theorem~\ref{thm:main-lmb-to-zero-classification}}
(Section~\ref{sec:Case3}). We classify the dynamics of the minimally
degenerate case under \eqref{eq:case3-v_i-not-all-zero}. This case
exhibits \emph{the most delicate dynamics} in this paper. (Recall
also that this scenario gives a new blow-up rate of different nature
in our theorem.) 

As the configuration $\{z_{1}^{\ast},\dots,z_{J}^{\ast}\}$ degenerates,
the leading terms in the $(\lmb_{i})_{t}$ equation \eqref{eq:intro-lmb-eqn}
may degenerate. This forces us to look at the variation of $(z_{i})_{t}$
to search for another non-degeneracy in the parameter dynamics. Our
assumption of minimal degeneracy with \eqref{eq:case3-v_i-not-all-zero}
guarantees, by Lemma~\ref{lem:case3-non-degen-quantity}, that the
equations for $(z_{i})_{t}$ \eqref{eq:intro-z-eqn} do not degenerate
whenever the $(\lmb_{i})_{t}$ equations \eqref{eq:intro-lmb-eqn}
degenerate. 

Dealing with this scenario requires a further generalization of the
non-degenerate case. Due to the possibility that the equations for
$(\lmb_{i})_{t}$ can degenerate, one cannot replace $A[\vec{z}]$
by $A^{\ast}$ as in the non-degenerate case. We basically deal with
\eqref{eq:intro-lmb-eqn}--\eqref{eq:intro-z-eqn} as written without
further simplifications, use weaker non-degeneracy arising from Lemma~\ref{lem:case3-non-degen-quantity}
to close bootstrap. Again, by a monotonicity formula, we are able
to prove that the ratio vector $\vec{\lmb}^{D}/|\vec{\lmb}^{D}|$
converges to the kernel of $A^{\ast}$, which is one-dimensional by
minimal degeneracy.\smallskip{}

\textbf{\emph{Step~7.}}\emph{ Construction of a nonlinear solution
of minimally degenerate scenario} (Section~\ref{sec:Construction-of-minimal-examples}).
We construct a solution exhibiting the minimally degenerate scenario.
We take a minimal example $J=4$, where $\iota_{1}=\iota_{2}=+1$,
$\iota_{3}=\iota_{4}=-1$, and $z_{1}^{\ast},z_{2}^{\ast},z_{3}^{\ast},z_{4}^{\ast}$
are vertices of a rectangle. We impose reflection symmetry conditions
on solutions to simplify the presentation. The construction is done
by a standard forward construction method with a Brouwer argument
(to treat unstable directions arising after linearization; see e.g.,
\cite{MerleZaag1997Duke,CoteMartelMerle2011RMI}). 

\section{\label{sec:Modified-multi-bubble-profiles}Modified multi-bubble
profiles}

In this preliminary section, we record some properties (especially
coercivity estimates) of the linearized operators $\calL_{\calW}$
around multi-bubbles, estimate the multi-bubble interaction term $\Psi$,
and construct the modified multi-bubble profiles $U$. 

\subsection{\label{subsec:Linear-coercivity-estimates}Linearized operators around
multi-bubbles and their coercivity}

In this subsection, we discuss linearized operators around multi-bubbles
and their coercivity estimates. For a function $U$, we denote 
\[
\calL_{U}g\coloneqq\Dlt g+f'(U)g.
\]

\uline{One-bubble linearized operator \mbox{$\calL_{W}$}}. We
begin by recalling some well-known facts (see e.g., \cite[Section 2.2]{CollotMerleRaphael2017CMP})
on the one-bubble linearized operator $\calL_{W}$. First, $\calL_{W}:\dot{H}^{2}\to L^{2}$,
$\calL_{W}:\dot{H}^{1}\to(\dot{H}^{1})^{\ast}$, and $\calL_{W}$
is \emph{non-degenerate}, i.e., $\ker\calL_{W}=\mathrm{span}\{\calV_{a}:a=0,1,\dots,N\}$,
where the \emph{modulation vectors} $\calV_{a}$ are defined by 
\begin{equation}
\calV_{a}\coloneqq\begin{cases}
\Lmb W & \text{if }a=0,\\
\rd_{a}W & \text{if }a\in\{1,\dots,N\}.
\end{cases}\label{eq:def-calV}
\end{equation}
We also recall that $\calL_{W}$ has precisely one positive eigenvalue
$e_{0}>0$ of multiplicity $1$ with the $L^{2}$-normalized eigenfunction
$\calY$ that is positive, radially symmetric, and exponentially decaying
as well as its derivatives. 

Now, we fix the \emph{orthogonality profiles} 
\begin{equation}
\calZ=(\calZ_{0},\calZ_{1},\dots,\calZ_{N}),\qquad\calZ_{a}=\begin{cases}
\calZ_{0} & \text{if }a=0,\\
\calV_{a}/\|\rd_{1}W\|_{L^{2}}^{2} & \text{if }a\in\{1,\dots,N\},
\end{cases}\label{eq:def-calZ}
\end{equation}
where $\calZ_{0}\in C_{c}^{\infty}(\bbR^{N})$ is any radially symmetric
function such that $\lan\calZ_{0},\Lmb W\ran=1$ and $\lan\calZ_{0},\calY\ran=0$.
With these orthogonality profiles, we have \emph{transversality} (of
$\ker\calL_{W}$ and $\{\calZ_{0},\dots,\calZ_{N}\}^{\perp_{L^{2}}}$):
\begin{equation}
\lan\calZ_{a},\calV_{b}\ran=\chf_{a=b}\quad\text{and}\quad\lan\calV_{a},\calV_{b}\ran=\begin{cases}
\|\Lmb W\|_{L^{2}}^{2} & \text{if }a=b=0,\\
\|\rd_{1}W\|_{L^{2}}^{2} & \text{if }a=b\neq0,\\
0 & \text{otherwise}.
\end{cases}\label{eq:transversality}
\end{equation}
By $\calZ_{0},\dots,\calZ_{N}\in(\dot{H}^{2})^{\ast}$ and \eqref{eq:transversality},
we have the following \emph{$\dot{H}^{2}$-coercivity estimate}: there
exist $\dlt,C>0$ such that 
\begin{equation}
\|\calL_{W}g\|_{L^{2}}\geq\dlt\|g\|_{\dot{H}^{2}}-C\sum_{a=0}^{N}|\lan\calZ_{a},g\ran|,\qquad\forall g\in\dot{H}^{2}.\label{eq:calL-one-bubble-coer}
\end{equation}
By $\calZ_{0},\dots,\calZ_{N}\in(\dot{H}^{1})^{\ast}$, \eqref{eq:transversality},
and $\lan\calZ_{a},\calY\ran=0$, we have the following \emph{$\dot{H}^{1}$-coercivity
estimate}: there exist $\dlt,C>0$ such that 
\begin{equation}
-\lan\calL_{W}g,g\ran=\int(|\nabla g|^{2}-f'(W)g^{2})\geq\dlt\|g\|_{\dot{H}^{1}}^{2}-C\Big(\lan\calY,g\ran^{2}+\sum_{a=0}^{N}\lan\calZ_{a},g\ran^{2}\Big),\qquad\forall g\in\dot{H}^{1}.\label{eq:calL-one-bubble-coer-H1-quad-form}
\end{equation}
By duality and $\calL_{W}\calY=e_{0}\calY$, this also implies the
following: there exist $\dlt,C>0$ such that 
\begin{equation}
\|\calL_{W}g\|_{(\dot{H}^{1})^{\ast}}\geq\dlt\|g\|_{\dot{H}^{1}}-C\sum_{a=0}^{N}|\lan\calZ_{a},g\ran|,\qquad\forall g\in\dot{H}^{1}.\label{eq:calL-one-bubble-coer-H1}
\end{equation}

\uline{Multi-bubble linearized operator \mbox{$\calL_{\calW}$}}.
We turn to consider the linearized operator $\calL_{\calW}$ around
pure multi-bubbles $\calW=\calW(\vec{\iota},\vec{\lmb},\vec{z})$
with $(\vec{\iota},\vec{\lmb},\vec{z})\in\calP_{J}(\dlt)$ for some
small $\dlt>0$. When $\dlt>0$ is small, $\calL_{\calW}=\Dlt+f'(\calW)$
can be approximated by $-\Dlt+f'(W_{;i})$ on regions where $W_{;i}$
is dominant so that we can apply one-bubble coercivity estimates discussed
above. 
\begin{lem}
For $(\vec{\iota},\vec{\lmb},\vec{z})\in\calP_{J}$, we have 
\begin{align}
\lan\lmb_{i}^{-1}\calZ_{a;i},\lmb_{j}^{-1}\calV_{b;j}\ran & =\chf_{(a,i)=(b,j)}+\calO(\chf_{i\neq j}R_{ij}^{-(N-4)}),\label{eq:calZ-calV-inner-prod}\\
\lan\lmb_{i}^{-1}\calV_{a;i},\lmb_{j}^{-1}\calV_{b;j}\ran & =\chf_{(a,i)=(b,j)}\|\calV_{a}\|_{L^{2}}^{2}+\calO(\chf_{i\neq j}R_{ij}^{-(N-4)}).\label{eq:calV-calV-inner-prod}
\end{align}
\end{lem}

\begin{proof}
If $i=j$, then by scaling and \eqref{eq:transversality}, we get
\begin{align*}
\lan\lmb_{i}^{-1}\calZ_{a;i},\lmb_{i}^{-1}\calV_{b;i}\ran & =\lan\calZ_{a},\calV_{b}\ran=\chf_{a=b},\\
\lan\lmb_{i}^{-1}\calV_{a;i},\lmb_{i}^{-1}\calV_{b;i}\ran & =\lan\calV_{a},\calV_{b}\ran=\chf_{a=b}\|\calV_{a}\|_{L^{2}}^{2}.
\end{align*}
if $i\neq j$, as $|\calZ_{a}|+|\calV_{a}|\aleq W$, it suffices to
show $\int\lmb_{i}^{-1}W_{\lmb_{i},z_{i}}\lmb_{j}^{-1}W_{\lmb_{j},z_{j}}=\calO(R_{ij}^{-(N-4)})$.
We may assume $\lmb_{i}\leq\lmb_{j}$. Introducing $\lmb=\lmb_{i}/\lmb_{j}\leq1$
and $|z|=|z_{i}-z_{j}|/\lmb_{j}$, we obtain by scaling 
\[
\tint{}{}\lmb_{i}^{-1}W_{\lmb_{i},z_{i}}\lmb_{j}^{-1}W_{\lmb_{j},z_{j}}=\tint{}{}\lmb^{-1}W_{\lmb,z}W\aeq\lmb^{(N-4)/2}(1+|z|)^{-(N-4)}\aeq R_{ij}^{-(N-4)}.\qedhere
\]
\end{proof}
\begin{prop}[Linear coercivity]
\label{prop:linear-coer}Let $J\geq1$. There exist $\dlt,C>0$ with
the following property. Denote $\calW=\calW(\vec{\iota},\vec{\lmb},\vec{z})$
for $(\vec{\iota},\vec{\lmb},\vec{z})\in\calP_{J}$.
\begin{itemize}
\item For any $(\vec{\iota},\vec{\lmb},\vec{z})\in\calP_{J}$, we have 
\begin{align}
\|\calL_{\calW}g\|_{(\dot{H}^{1})^{\ast}} & \leq C\|g\|_{\dot{H}^{1}},\qquad\forall g\in\dot{H}^{1},\label{eq:calL-H1-bdd}\\
\|\calL_{\calW}g\|_{L^{2}} & \leq C\|g\|_{\dot{H}^{2}},\qquad\forall g\in\dot{H}^{2}.\label{eq:calL-H2-bdd}
\end{align}
\item For any $(\vec{\iota},\vec{\lmb},\vec{z})\in\calP_{J}(\dlt)$, we
have 
\begin{align}
-\lan\calL_{\calW}g,g\ran & \geq\dlt\|g\|_{\dot{H}^{1}}^{2}-C(\tsum i{}\lan\calY_{\ul{;i}},g\ran^{2}+\tsum{a,i}{}\lan\calZ_{a\ul{;i}},g\ran^{2}),\qquad\forall g\in\dot{H}^{1}.\label{eq:calL-H1-coer-quad-form}\\
\|\calL_{\calW}g\|_{(\dot{H}^{1})^{\ast}} & \geq\dlt\|g\|_{\dot{H}^{1}}-C\tsum{a,i}{}|\lan\calZ_{a\ul{;i}},g\ran|,\qquad\forall g\in\dot{H}^{1},\label{eq:calL-H1-coer}\\
\|\calL_{\calW}g\|_{L^{2}} & \geq\dlt\|g\|_{\dot{H}^{2}}-C\tsum{a,i}{}\lmb_{i}^{-1}|\lan\calZ_{a\ul{;i}},g\ran|,\qquad\forall g\in\dot{H}^{2}.\label{eq:calL-H2-coer}
\end{align}
\end{itemize}
\end{prop}

\begin{proof}
See Appendix~\ref{subsec:proof-of-lin-coer}.
\end{proof}

\subsection{Interaction between bubbles}

Recall the nonlinearity $f(u)=|u|^{p-1}u$. We denote the \emph{multi-bubble
interaction term} by 
\begin{equation}
\Psi\coloneqq f(\calW)-\tsum{i=1}Jf(W_{;i}).\label{eq:def-Psi}
\end{equation}
To estimate this interaction term and other remainder terms that appear
later, we record useful pointwise estimates regarding the nonlinearity
$f(u)$. It is convenient to introduce 
\begin{equation}
f(a,b)\coloneqq\min\{|a|^{p-1}|b|,|a||b|^{p-1}\},\qquad\forall a,b\in\bbR.\label{eq:def-f(a,b)}
\end{equation}

\begin{lem}[Pointwise estimates]
\label{lem:f(a,b)-est}Let $p\in(1,2]$. We have the following pointwise
estimates: 
\begin{align}
|f(\tsum{i=1}Ka_{i})-\tsum{i=1}Kf(a_{i})| & \aleq_{K,p}\tsum{i,j:i\neq j}{}f(a_{i},a_{j}),\qquad K\in\bbN,\label{eq:f(a,b)-1}\\
|a\{f'(a+b)-f'(a)\}| & \aleq_{p}f(a,b),\label{eq:f(a,b)-2}\\
|f(a+b)-f(a)-f'(a)b| & \aleq_{p}\min\{|b|^{p},\chf_{a\neq0}|a|^{p-2}|b|^{2}\},\label{eq:f(a,b)-3}\\
|a\{f(a+b+c)-f(a+b)-f'(a)c\}| & \aleq_{p}\{f(a,b)+f(a,c)\}|c|.\label{eq:f(a,b)-4}
\end{align}
\end{lem}

\begin{proof}
Let us omit the elementary proofs of \eqref{eq:f(a,b)-1}, \eqref{eq:f(a,b)-2},
and \eqref{eq:f(a,b)-3}, and only prove \eqref{eq:f(a,b)-4}. The
implicit constants here depend on $p$. Write 
\[
a\{f(a+b+c)-f(a+b)-f'(a)c\}=a\{f'(a+b)-f'(a)\}c+a\NL_{a+b}(c).
\]
where $\NL_{a+b}(c)=f(a+b+c)-f(a+b)-f'(a+b)c$. Applying \eqref{eq:f(a,b)-2}
to the first term of the right hand side gives the bound $\calO(f(a,b)|c|)$,
which is acceptable. It remains to estimate the second term $a\NL_{a+b}(c)$.
If $|a+b|>\frac{1}{2}|a|$, then \eqref{eq:f(a,b)-3} and $p-2\leq0$
give $|a\NL_{a+b}(c)|\aleq|a|\min\{|c|^{p},|a+b|^{p-2}|c|^{2}\}\aleq|a|\min\{|c|^{p},|a|^{p-2}|c|^{2}\}\aeq f(a,c)|c|$.
If $|c|>\frac{1}{2}|a|$, then we use \eqref{eq:f(a,b)-3} to have
$|a\NL_{a+b}(c)|\aleq|a||c|^{p}\aeq f(a,c)|c|$. The remaining case
is when $\max\{|a+b|,|c|\}\leq\frac{1}{2}|a|$, we use $|b|\geq\frac{1}{2}|a|$
and \eqref{eq:f(a,b)-3} to have $|a\NL_{a+b}(c)|\aleq|a||c|^{p}\aleq|a|^{p}|c|\aleq|b|^{p-1}|a||c|\aeq f(a,b)|c|$.
\end{proof}
By \eqref{eq:f(a,b)-1}, we have $|\Psi|\aleq\tsum{i\neq j}{}f(W_{;i},W_{;j})$.
We need integral estimates adapted to $f(W_{;i},W_{;j})$. 
\begin{lem}[Estimates for $f(W_{;1},W_{;2})$]
\label{lem:f(W1,W2)-est}We have the following estimates. 
\begin{itemize}
\item ($(\dot{H}^{1})^{\ast}$-estimate for $f(W_{;1},W_{;2})$) For $\eps\in\bbR$,
we have 
\begin{equation}
\int_{\{|W_{;1}|\geq|W_{;2}|\}}|W_{;1}|^{\frac{N-\eps}{N-2}}|W_{;2}|^{\frac{N+\eps}{N-2}}\aleq_{\eps}\begin{cases}
R_{12}^{-N} & \text{if }\eps>0,\\
R_{12}^{-N}(1+\log R_{12}) & \text{if }\eps=0,\\
R_{12}^{-(N+\eps)} & \text{if }\eps<0.
\end{cases}\label{eq:W1W2-Hdot1}
\end{equation}
In particular, we have 
\begin{equation}
\|f(W_{;1},W_{;2})\|_{L^{(2^{\ast})'}}\aleq R_{12}^{-\frac{N+2}{2}}.\label{eq:f(W1,W2)-H1dual-est}
\end{equation}
\item (Estimates for $f(W_{;1},W_{;2})$ adapted to other scalings) Assume
$\lmb_{1}\leq\lmb_{2}$. We have 
\begin{align}
\|f(W_{;1},W_{;2})\|_{L^{2}} & \aleq\lmb_{1}^{-1}\{\chf_{N=7}R_{12}^{-(N-2)}+\chf_{N=8}R_{12}^{-(N-2)}\lan\log R_{12}\ran^{\frac{1}{2}}+\chf_{N\geq9}R_{12}^{-\frac{N}{2}-2}\},\label{eq:f(W1,W2)-L2-est}\\
\|f(W_{;1},W_{;2})\|_{L^{(2^{\ast\ast})'}} & \aleq\chf_{N=7}\lmb_{1}^{\frac{3}{4}}\lmb_{2}^{\frac{1}{4}}R_{12}^{-\frac{N}{2}}+\chf_{N=8}\lmb_{1}R_{12}^{-\frac{N}{2}}\lan\log(\tfrac{\lmb_{2}}{\lmb_{1}})\ran^{\frac{1}{(2^{\ast\ast})'}}+\chf_{N\geq9}\lmb_{1}R_{12}^{-\frac{N}{2}},\label{eq:f(W1,W2)-H2dual-est}
\end{align}
where we recall $(2^{\ast\ast})'=\frac{2N}{N+4}$.
\end{itemize}
\end{lem}

\begin{proof}
See Appendix~\ref{subsec:Proof-of-lemma-integrals}.
\end{proof}
With the help of the above lemma, let us record the estimates for
$\Psi$ and $E[\calW]-JE[W]$.
\begin{lem}[{Estimates for $\Psi$ and $E[\calW]-JE[W]$}]
We have the following estimates. 
\begin{itemize}
\item (Rough estimate for $E[\calW]-JE[W]$) We have 
\begin{equation}
|E[\calW]-JE[W]|\aleq R^{-(N-2)}.\label{eq:E(calW)-JE(W)-est}
\end{equation}
\item (Size estimates for $\Psi$) We have 
\begin{align}
\|\Psi\|_{L^{(2^{\ast})'}} & \aleq R^{-\frac{N+2}{2}},\label{eq:Psi-Hdot1-dual}\\
\|\Psi\|_{L^{2}} & \aleq\tsum i{}\lmb_{i}^{-1}\frkp(R_{i}),\label{eq:Psi-L2}
\end{align}
where 
\begin{equation}
\frkp(R)\coloneqq\chf_{N=7}R^{-5}+\chf_{N=8}R^{-6}(1+\log R)^{\frac{1}{2}}+\chf_{N\geq9}R^{-\frac{N}{2}-2}.\label{eq:def-frkp}
\end{equation}
\item (Inner product estimates for $\Psi$) For any $\eps>0$, we have 
\begin{align}
\lan\calV_{a;i},\Psi\ran & =\calO_{\eps}(R^{-\eps}R_{i}^{-(N-2-\eps)}),\label{eq:calV-Psi-estimate}\\
\lan\calV_{a;i},\Psi\ran & =\tsum{j\neq i}{}\lan\calV_{a;i},f'(W_{;i})W_{;j}\ran+\calO_{\eps}(R^{-(2+\eps)}R_{i}^{-(N-2-\eps)}).\label{eq:calV-Psi-leading}
\end{align}
\item (The inner products $\lan\calV_{a;i},f'(W_{;i})W_{;j}\ran$) We have
\begin{align}
 & \lan[\Lmb W]_{;i},f'(W_{;i})W_{;j}\ran\nonumber \\
 & =\begin{cases}
{\displaystyle -\tfrac{N-2}{2}\tint{}{}f(W)\cdot\iota_{i}\iota_{j}(\tfrac{\lmb_{i}}{\lmb_{j}})^{\frac{N-2}{2}}W(\tfrac{z_{j}-z_{i}}{\lmb_{j}})+\calO\big(R_{ij}^{-N}\tfrac{\lmb_{i}}{\lmb_{j}}\lan\log(\tfrac{\lmb_{j}+|z_{i}-z_{j}|}{\lmb_{i}})\ran\big)} & \text{if }\lmb_{i}\leq\lmb_{j},\\
{\displaystyle \tint{}{}f(W)\cdot\iota_{i}\iota_{j}(\tfrac{\lmb_{j}}{\lmb_{i}})^{\frac{N-2}{2}}\Lmb W(\tfrac{z_{i}-z_{j}}{\lmb_{i}})+\calO\big(R_{ij}^{-N}\tfrac{\lmb_{j}}{\lmb_{i}}\lan\log(\tfrac{\lmb_{i}+|z_{i}-z_{j}|}{\lmb_{j}})\ran\big)} & \text{if }\lmb_{i}\geq\lmb_{j},
\end{cases}\label{eq:2.48}
\end{align}
and 
\begin{align}
 & \lan[\nabla W]_{;i},f'(W_{;i})W_{;j}\ran\nonumber \\
 & =\begin{cases}
{\displaystyle -\tint{}{}f(W)\cdot\iota_{i}\iota_{j}(\tfrac{\lmb_{i}}{\lmb_{j}})^{\frac{N}{2}}[\nabla W](\tfrac{z_{i}-z_{j}}{\lmb_{j}})+\calO\big(R_{ij}^{-(N+1)}(\tfrac{\lmb_{i}}{\lmb_{j}})^{\frac{3}{2}}\lan\log(\tfrac{\lmb_{j}+|z_{i}-z_{j}|}{\lmb_{i}})\ran\big)} & \text{if }\lmb_{i}\leq\lmb_{j},\\
{\displaystyle \tint{}{}f(W)\cdot\iota_{i}\iota_{j}(\tfrac{\lmb_{j}}{\lmb_{i}})^{\frac{N-2}{2}}[\nabla W](\tfrac{z_{j}-z_{i}}{\lmb_{i}})+\calO\big(R_{ij}^{-(N+1)}(\tfrac{\lmb_{j}}{\lmb_{i}})^{\frac{1}{2}}\lan\log(\tfrac{\lmb_{i}+|z_{i}-z_{j}|}{\lmb_{j}})\ran\big)} & \text{if }\lmb_{i}\geq\lmb_{j}.
\end{cases}\label{eq:2.49}
\end{align}
\end{itemize}
\end{lem}

\begin{rem}
The estimate \eqref{eq:calV-Psi-leading} in a slightly weaker form
appears in \cite[Lemma 2.1]{DengSunWei2025Duke}. The term $R_{i}^{-(N-2-\eps)}$
should not be replaced by $R^{-(N-2-\eps)}$. Nevertheless, the proof
is similar to \cite[Lemma 2.1]{DengSunWei2025Duke}. See also \cite{Bahri1989book}.
\end{rem}

\begin{rem}
The error $\calO_{\eps}(R^{-(2+\eps)}R_{i}^{-(N-2-\eps)})$ of \eqref{eq:calV-Psi-leading}
can be replaced by $\calO(R^{-2}(1+\log R)^{\frac{2}{N}}\cdot R_{i}^{-(N-2)}(1+\log R_{i})^{\frac{N-2}{N}})$
if one expresses the triple product \eqref{eq:triple} with the choice
$\eps'=0$, $\alp=\frac{2}{N}$, $\beta=\frac{2}{N}$, and $\gmm=\frac{N-4}{N}$
in the proof. The error $\calO_{\eps}(R^{-\eps}R_{i}^{-(N-2-\eps)})$
of \eqref{eq:calV-Psi-estimate} can then be replaced by $\calO(R_{i}^{-(N-2)}(1+\log R_{i})^{\frac{N-2}{N}}(1+\log R)^{-\frac{N-2}{N}})$.
\end{rem}

\begin{proof}
\uline{Proof of \mbox{\eqref{eq:E(calW)-JE(W)-est}}}. Using $\Dlt W_{;i}+f(W_{;i})=0$
for each $i$, observe 
\begin{align*}
E[\calW]-JE[W] & =\tfrac{1}{2}\tsum{i,j:i\neq j}{}\tint{}{}\nabla W_{;i}\cdot\nabla W_{;j}-\tfrac{1}{p+1}\tint{}{}(|\calW|^{p+1}-\tsum i{}|W_{;i}|^{p+1})\\
 & =\tfrac{1}{2}\tsum{i,j:i\neq j}{}\tint{}{}f(W_{;i})W_{;j}-\tfrac{1}{p+1}\tint{}{}(|\calW|^{p+1}-\tsum i{}|W_{;i}|^{p+1}).
\end{align*}
Applying $||\sum_{i}a_{i}|^{p+1}-\sum_{i}|a_{i}|^{p+1}|\aleq\sum_{i,j:i\neq j}|a_{i}|^{p}|a_{j}|$
and \eqref{eq:W1W2-Hdot1} with $\eps=-2$, we obtain 
\[
|E[\calW]-JE[W]|\aleq\tsum{i,j:i\neq j}{}\tint{}{}|W_{;i}|^{p}|W_{;j}|\aleq\tsum{i,j:i\neq j}{}R_{ij}^{-(N-2)}\aleq R^{-(N-2)}.
\]

\uline{Proof of \mbox{\eqref{eq:Psi-Hdot1-dual}} and \mbox{\eqref{eq:Psi-L2}}}.
By \eqref{eq:f(a,b)-1}, we have $|\Psi|\aleq\tsum{i\neq j}{}f(W_{;i},W_{;j})$.
\eqref{eq:Psi-Hdot1-dual} then follows from \eqref{eq:f(W1,W2)-H1dual-est}.
Next, \eqref{eq:Psi-L2} follows from \eqref{eq:f(W1,W2)-L2-est}
and \eqref{eq:def-frkp}: 
\[
\|\Psi\|_{L^{2}}\aleq\sum_{i,j:i\neq j}\frac{\frkp(R_{ij})}{\min\{\lmb_{i},\lmb_{j}\}}\aleq\sum_{i}\frac{\frkp(R_{i})}{\lmb_{i}}.
\]

\uline{Proof of \mbox{\eqref{eq:calV-Psi-estimate}} assuming \mbox{\eqref{eq:calV-Psi-leading}}}.
By \eqref{eq:W1W2-Hdot1} with $\eps=-2$, $\lan\calV_{a;i},f'(W_{;i})W_{;j}\ran=\calO(R_{ij}^{-(N-2)})$
for $j\neq i$. Substituting this into \eqref{eq:calV-Psi-leading}
gives \eqref{eq:calV-Psi-estimate}.

\uline{Proof of \mbox{\eqref{eq:calV-Psi-leading}}}. We may assume
$\eps>0$ is sufficiently small. Define $\eps'>0$ by the relation
$N\cdot\frac{N-2}{N+\eps'}=N-2-\eps$. We begin with 
\begin{align*}
 & \big|\lan\calV_{a;i},\Psi\ran-\tsum{j\neq i}{}\lan\calV_{a;i},f'(W_{;i})W_{;j}\ran\big|\\
 & \leq\Big\{\int_{\{|W_{;i}|\geq\max_{j\neq i}|W_{;j}|\}}+\int_{\{|W_{;i}|<\max_{j\neq i}|W_{;j}|\}}\Big\}|\calV_{a;i}||\Psi-\tsum{j\neq i}{}f'(W_{;i})W_{;j}|.
\end{align*}
In the region $|W_{;i}|\geq\max_{j\neq i}|W_{;j}|$, we have 
\[
\Psi-\tsum{j\neq i}{}f'(W_{;i})W_{;j}=\NL_{W_{;i}}(\tsum{j\neq i}{}W_{;j})-\tsum{j\neq i}{}f(W_{;j})=\calO(\tsum{j\neq i}{}|W_{;j}|^{p}),
\]
so 
\[
\chf_{\{|W_{;i}|\geq\max_{j\neq i}|W_{;j}|\}}|\Psi-\tsum{j\neq i}{}f'(W_{;i})W_{;j}|\aleq\tsum{j\neq i}{}\chf_{\{|W_{;i}|\geq|W_{;j}|\}}|W_{;j}|^{p}.
\]
By \eqref{eq:W1W2-Hdot1}, we get 
\begin{equation}
\int_{\{|W_{;i}|\geq\max_{j\neq i}|W_{;j}|\}}|\calV_{a;i}||\Psi-\tsum{j\neq i}{}f'(W_{;i})W_{;j}|\aleq\sum_{j\neq i}\int_{\{|W_{;i}|\geq|W_{;j}|\}}|W_{;i}||W_{;j}|^{p}\aleq R_{i}^{-N}.\label{eq:2.47-1}
\end{equation}
In the region $|W_{;i}|<\max_{j\neq i}|W_{;j}|$, we simply bound
\begin{equation}
|\Psi-\tsum{j\neq i}{}f'(W_{;i})W_{;j}|\aleq\tsum{j,k:j\neq k}{}f(W_{;j},W_{;k})+\tsum{j\neq i}{}|W_{;i}|^{p-1}|W_{;j}|.\label{eq:2.32-1}
\end{equation}
Now, we note an inequality for scalars $a_{1},\dots,a_{J}\geq0$ and
any $i\in\setJ$ with $a_{i}<\max_{j\neq i}a_{j}$: 
\begin{equation}
\tsum{j,k:j\neq k}{}f(a_{j},a_{k})+\tsum{j\neq i}{}a_{i}^{p-1}a_{j}\aleq\tsum{j,k:j,k,i\text{ distinct}}{}\chf_{a_{j}\geq a_{k}\geq a_{i}}a_{j}^{p-1}a_{k}+\tsum{j\neq i}{}\chf_{a_{j}\geq a_{i}}a_{j}a_{i}^{p-1}.\label{eq:2.32-2}
\end{equation}
This inequality is obvious when $J=2$. When $J\geq3$, choose $a_{j_{0}}=\max_{j\neq i}a_{j}>a_{i}$
and $k_{0}\neq i,j_{0}$ by $a_{k_{0}}=\max_{k\neq i,j_{0}}a_{k}$.
Then, 
\begin{align*}
\tsum{j,k:j\neq k}{}f(a_{j},a_{k})+\tsum{j\neq i}{}a_{i}^{p-1}a_{j} & \aleq a_{j_{0}}^{p-1}\max\{a_{i},a_{k_{0}}\}+a_{j_{0}}a_{i}^{p-1}\\
 & \aleq\chf_{a_{k_{0}}\geq a_{i}}a_{j_{0}}^{p-1}a_{k_{0}}+a_{j_{0}}a_{i}^{p-1}\aleq\text{RHS}\eqref{eq:2.32-2}
\end{align*}
 as desired. Substituting $a_{j}=|W_{;j}|$ into \eqref{eq:2.32-2}
and using \eqref{eq:2.32-1}, we get
\begin{align*}
 & \chf_{\{|W_{;i}|<\max_{j\neq i}|W_{;j}|\}}|\Psi-\tsum{j\neq i}{}f'(W_{;i})W_{;j}|\\
 & \aleq\tsum{j,k:j,k,i\text{ distinct}}{}\chf_{\{|W_{;j}|\geq|W_{;k}|\geq|W_{;i}|\}}|W_{;j}|^{p-1}|W_{;k}|+\tsum{j\neq i}{}\chf_{\{|W_{;j}|\geq|W_{;i}|\}}|W_{;j}||W_{;i}|^{p-1}.
\end{align*}
This implies 
\begin{align*}
 & \int_{\{|W_{;i}|<\max_{j\neq i}|W_{;j}|\}}|\calV_{a;i}||\Psi-\tsum{j\neq i}{}f'(W_{;i})W_{;j}|\\
 & \aleq\sum_{j,k:j,k,i\text{ distinct}}\int_{\{|W_{;j}|\geq|W_{;k}|\geq|W_{;i}|\}}|W_{;j}|^{p-1}|W_{;k}||W_{;i}|+\sum_{j\neq i}\int_{\{|W_{;j}|\geq|W_{;i}|\}}|W_{;j}||W_{;i}|^{p}.
\end{align*}
For the second term, applying \eqref{eq:W1W2-Hdot1} gives 
\[
\sum_{j\neq i}\int_{\{|W_{;j}|\geq|W_{;i}|\}}|W_{;j}||W_{;i}|^{p}\aleq R_{i}^{-N}.
\]
For the first term, as $\eps'>0$ is sufficiently small, we can express
\begin{equation}
|W_{;j}|^{\frac{4}{N-2}}|W_{;k}||W_{;i}|=(|W_{;j}|^{\frac{N-\eps'}{N-2}}|W_{;k}|^{\frac{N+\eps'}{N-2}})^{\alp}(|W_{;j}|^{\frac{N-\eps'}{N-2}}|W_{;i}|^{\frac{N+\eps'}{N-2}})^{\beta}(|W_{;k}|^{\frac{N-\eps'}{N-2}}|W_{;i}|^{\frac{N+\eps'}{N-2}})^{\gmm}\label{eq:triple}
\end{equation}
for some $0\leq\alp,\beta,\gmm\leq1$ with $\beta+\gmm=\frac{N-2}{N+\eps'}$
and $\alp=1-\beta-\gmm$. Applying Hölder's inequality, \eqref{eq:W1W2-Hdot1},
and the relation $N\cdot\frac{N-2}{N+\eps'}=N-2-\eps$, we get 
\begin{align*}
 & \int_{\{|W_{;j}|\geq|W_{;k}|\geq|W_{;i}|\}}|W_{;j}|^{p-1}|W_{;k}||W_{;i}|\\
 & \aleq_{\eps}R_{jk}^{-N\alp}R_{ji}^{-N\beta}R_{ki}^{-N\gmm}\aleq_{\eps}R^{-N\alp}R_{i}^{-N(\beta+\gmm)}\aeq_{\eps}R^{-(2+\eps)}R_{i}^{-(N-2-\eps)}.
\end{align*}
Therefore, we have proved 
\begin{equation}
\int_{\{|W_{;i}|<\max_{j\neq i}|W_{;j}|\}}|\calV_{a;i}||\Psi-\tsum{j\neq i}{}f'(W_{;i})W_{;j}|\aleq_{\eps}R^{-(2+\eps)}R_{i}^{-(N-2-\eps)}.\label{eq:2.47-2}
\end{equation}
Gathering \eqref{eq:2.47-1} and \eqref{eq:2.47-2} completes the
proof of \eqref{eq:calV-Psi-leading}.

\uline{Proof of \mbox{\eqref{eq:2.48}} and \mbox{\eqref{eq:2.49}}}.
See Appendix~\ref{sec:Proof-of-inner-prod}.
\end{proof}

\subsection{Construction of modified multi-bubble profiles}

The goal of this subsection is to construct modified multi-bubble
profiles. This is done by solving a nonlinear elliptic PDE (see \eqref{eq:U-eqn}
below) using a standard Lyapunov--Schmidt method. See \cite{DengSunWei2025Duke}
for a very similar construction with $C^{0}$-estimates.
\begin{prop}[Modified multi-bubble profiles]
\label{prop:Modified-Profiles}Let $J\in\bbN$. There exists a constant
$\dlt_{\rmp}>0$ with the following property. For each $(\vec{\iota},\vec{\lmb},\vec{z})\in\calP_{J}(\dlt_{\rmp})$,
there exist unique profile $U=\calW+\td U\in\dot{H}^{1}$ and constants
$\frkr_{b,j}\in\bbR$ that satisfy the equation 
\begin{equation}
\Dlt U+f(U)=\sum_{b,j}\frkr_{b,j}\calV_{b\ul{;j}},\label{eq:U-eqn}
\end{equation}
the orthogonality conditions 
\[
\lan\calZ_{a;j},\td U\ran=0,\qquad\forall a\in\{0,1,\dots,N\},\ i\in\setJ,
\]
and smallness 
\[
\|\td U\|_{\dot{H}^{1}}\leq\dlt_{\rmp}.
\]
Moreover, $U\in\dot{H}^{1}\cap\dot{H}^{2}$ and $U$ is differentiable
in $\vec{\lmb}$ and $\vec{z}$ with the estimates 
\begin{align}
\|\td U\|_{\dot{H}^{1}} & \aleq R^{-\frac{N+2}{2}},\label{eq:tdU-Hdot1-est}\\
\|\td U\|_{\dot{H}^{2}} & \aleq\tsum i{}\lmb_{i}^{-1}\frkp(R_{i}),\label{eq:tdU-Hdot2-est}\\
\|\lmb_{k}\rd_{z_{k}^{c}}\td U\|_{\dot{H}^{1}} & \aleq\min\{R^{-\frac{N+2}{2}},R_{k}^{-\frac{N+2}{2}}+\lmb_{k}\|\td U\|_{\dot{H}^{2}}+\max_{b}|\frkr_{b,k}|\},\label{eq:tdU-lmb-deriv-est}
\end{align}
where we recall $\frkp(R_{i})$ from \eqref{eq:def-frkp}. The coefficients
$\frkr_{b,j}$ satisfy 
\begin{align}
\max_{b,j}|\frkr_{b,j}| & \aleq R^{-(N-2)},\label{eq:frkr-est}\\
\max_{b,j}|\lmb_{j}^{-1}\frkr_{b,j}| & \aleq\tsum i{}\lmb_{i}^{-1}R_{i}^{-\frac{N+2}{2}}R^{-\frac{N-6}{2}}.\label{eq:frkr/lmb-est}
\end{align}
\end{prop}

\begin{rem}
By the uniqueness statement, the profile $U$ remains the same under
the permutation of indices in $(\vec{\iota},\vec{\lmb},\vec{z})$.
Moreover, if $(\vec{\iota},\vec{\lmb},\vec{z})$ is symmetric (with
respect to reflections or rotations, for example), then $U$ is also
symmetric accordingly.
\end{rem}

\begin{rem}
The $\dot{H}^{1}$-estimate \eqref{eq:tdU-Hdot1-est} recovers the
same $\dot{H}^{1}$-bound of \cite[Proposition 6.1]{DengSunWei2025Duke},
where very similar modified multi-bubble profiles are constructed
with a different choice of profiles in the right hand side of \eqref{eq:U-eqn}
and in the case of $\iota_{1}=\dots=\iota_{J}=+1$. For the modulation
analysis, the choice of profiles in the right hand side of \eqref{eq:U-eqn}
is crucial. We also remark that \cite[Proposition 5.4]{DengSunWei2025Duke}
obtains a pointwise estimate for $\td U$ (which is expected to hold
in our setting as well) and it might be a substitute of our $\dot{H}^{2}$-estimate
for $\td U$. For dynamical estimates, we also need additional estimates
for the parameter-derivatives $\lmb_{k}\rd_{z_{k}^{c}}\td U$ \eqref{eq:tdU-lmb-deriv-est}
and $\dot{H}^{2}$-estimate \eqref{eq:tdU-Hdot2-est} for $\td U$.
It is essential to\emph{ }keep certain powers of $R_{i}$ and $\lmb_{i}$
(in other words, one should not replace every $R_{i}^{-1}$ by $R^{-1}$).
\end{rem}

\begin{rem}
In each dynamical scenario of our main theorems, we will compute the
leading term of $\frkr_{a,i}$. See for example Lemma~\ref{lem:case1-leading-frkr_a,i}.
\end{rem}

We begin with an inversion lemma for the associated linear problem
formulated in $\dot{H}^{1}$ and $\dot{H}^{2}$ topologies.
\begin{lem}[Inversion of $\calL_{\calW}$]
\label{lem:inversion}Let $J\in\bbN$. There exists $\dlt>0$ with
the following property. For any $(\vec{\iota},\vec{\lmb},\vec{z})\in\calP(\dlt)$
and $h\in(\dot{H}^{1})^{\ast}$, there exist unique $g\in\dot{H}^{1}$
and scalars $c_{b,j}\in\bbR$ such that 
\begin{equation}
\calL_{\calW}g=h+\sum_{b,j}c_{b,j}\calV_{b\ul{;j}}\quad\text{and}\quad\lan\calZ_{a\ul{;i}},g\ran=0\text{ for all }a,i.\label{eq:inversion-eqn}
\end{equation}
Moreover, the mapping $h\mapsto(g,c_{a,i})$ is linear and satisfies
\begin{equation}
\|g\|_{\dot{H}^{1}}\aleq\|h\|_{(\dot{H}^{1})^{\ast}}\qquad\text{and}\qquad\max_{a,i}|c_{a,i}|\aleq R^{-\frac{N+2}{2}}\|h\|_{(\dot{H}^{1})^{\ast}}+\max_{b,j}|h(\calV_{b;j})|.\label{eq:inversion-H1-bound}
\end{equation}
If in addition $h\in L^{2}$, then $g\in\dot{H}^{2}$ and 
\begin{equation}
\|g\|_{\dot{H}^{2}}\aleq\|h\|_{L^{2}}.\label{eq:inversion-H2-bound}
\end{equation}
\end{lem}

\begin{proof}
We proceed similarly as in the proof of \cite[Proposition 4.1]{delPinoFelmerMusso2003CVPDE}
and \cite[Proposition 5.3]{DengSunWei2025Duke}. Let us denote 
\[
\calZ^{\perp}\coloneqq\{g\in\dot{H}^{1}:\lan\calZ_{a\ul{;i}},g\ran=0\quad\forall a,i\}
\]
so that \eqref{eq:calZ-calV-inner-prod} implies the decomposition
\begin{equation}
\dot{H}^{1}=\calZ^{\perp}\oplus\mathrm{span}\{\calV_{a;i}\}_{a,i}.\label{eq:Hdot1-direct-sum-decom}
\end{equation}

\uline{Step 1: Unique existence of \mbox{$g\in\calZ^{\perp}$}
and \mbox{$c_{a,i}$}s}. The weak formulation of \eqref{eq:inversion-eqn}
is 
\begin{equation}
-\lan\nabla\psi,\nabla g\ran+\lan\psi,f'(\calW)g\ran=h(\psi)+\tsum{b,j}{}\lan\psi,\calV_{b\ul{;j}}\ran c_{b,j},\qquad\forall\psi\in\dot{H}^{1}.\label{eq:3.3}
\end{equation}
By \eqref{eq:Hdot1-direct-sum-decom}, this weak formulation is equivalent
to 
\begin{align}
-\lan\nabla\psi,\nabla g\ran+\lan\psi,f'(\calW)g\ran & =h(\psi)+\tsum{b,j}{}\lan\psi,\calV_{b\ul{;j}}\ran c_{b,j},\qquad\forall\psi\in\calZ^{\perp},\label{eq:3.4}\\
\tsum{b,j}{}\lan\calV_{a;i},\calV_{b\ul{;j}}\ran c_{b,j} & =\lan\calL_{\calW}\calV_{a;i},g\ran-h(\calV_{a;i}),\qquad\forall a,i.\label{eq:3.5}
\end{align}

We first investigate \eqref{eq:3.5}. For the matrix in the left hand
side, \eqref{eq:calV-calV-inner-prod} implies 
\[
\lan\calV_{a;i},\calV_{b\ul{;j}}\ran=\chf_{(a,i)=(b,j)}\|\calV_{a}\|_{L^{2}}^{2}+\calO(\chf_{i\neq j}\frac{\lmb_{i}}{\lmb_{j}}R_{ij}^{-(N-4)})=\chf_{(a,i)=(b,j)}\|\calV_{a}\|_{L^{2}}^{2}+\calO(R^{-(N-6)}),
\]
so the matrix $\lan\calV_{a;i},\calV_{b\ul{;j}}\ran$ is diagonally
dominant with uniformly bounded inverse. Hence, the vector $(c_{b,j})_{b,j}$
is uniquely determined and takes the form 
\begin{equation}
c_{b,j}=c_{b,j}'(g)+c_{b,j}''(h),\label{eq:3.6}
\end{equation}
where $c_{b,j}'$ and $c_{b,j}''$ are linear functionals satisfying
the estimates 
\begin{align}
\max_{b,j}|c_{b,j}'(g)| & \aleq\max_{a,i}\|\calL_{\calW}\calV_{a;i}\|_{(\dot{H}^{1})^{\ast}}\|g\|_{\dot{H}^{1}}\aleq R^{-\frac{N+2}{2}}\|g\|_{\dot{H}^{1}},\label{eq:3.9}\\
\max_{b,j}|c_{b,j}''(h)| & \aleq\max_{a,i}|h(\calV_{a;i})|\aleq\|h\|_{(\dot{H}^{1})^{\ast}}.\label{eq:3.10}
\end{align}
Here, in the last inequality of \eqref{eq:3.9} we used \eqref{eq:f(a,b)-2}
and \eqref{eq:f(W1,W2)-H1dual-est}: 
\[
\|\calL_{\calW}\calV_{a;i}\|_{L^{(2^{\ast})'}}=\|\{f'(\calW)-f'(W_{;i})\}\calV_{a;i}\|_{L^{(2^{\ast})'}}\aleq\tsum{j\neq i}{}\|f(W_{;i},W_{;j})\|_{L^{(2^{\ast})'}}\aleq R_{i}^{-\frac{N+2}{2}}.
\]

Substituting \eqref{eq:3.6} into \eqref{eq:3.4}, the weak formulation
\eqref{eq:3.3} becomes
\begin{equation}
-\lan\nabla\psi,\nabla g\ran+\lan\psi,f'(\calW)g\ran-\tsum{b,j}{}\lan\psi,\calV_{b\ul{;j}}\ran c_{b,j}'(g)=h'(\psi),\qquad\forall\psi\in\calZ^{\perp},\label{eq:3.7}
\end{equation}
where $h'\in(\dot{H}^{1})^{\ast}$ is defined by $h'(\psi)=h(\psi)+\tsum{b,j}{}\lan\psi,\calV_{b\ul{;j}}\ran c_{b,j}''(h)$.

It suffices to prove unique solvability of \eqref{eq:3.7} for any
given $h'\in(\calZ^{\perp})^{\ast}$. By the standard argument with
Fredholm's alternative, it suffices to show that any solution $g\in\calZ^{\perp}$
to \eqref{eq:3.7} with $h'=0$ is trivial. Let $g\in\calZ^{\perp}$
be a solution to 
\begin{equation}
-\lan\nabla\psi,\nabla g\ran+\lan\psi,f'(\calW)g\ran-\tsum{b,j}{}\lan\psi,\calV_{b\ul{;j}}\ran c_{b,j}'(g)=0\qquad\forall\psi\in\calZ^{\perp}.\label{eq:3.8}
\end{equation}
Due to the definition of $c_{b,j}'(g)$s, \eqref{eq:3.8} holds for
any $\psi=\calV_{a;i}$ as well. By the decomposition $\dot{H}^{1}=\calZ^{\perp}\oplus\mathrm{span}\{\calV_{a;i}\}_{a,i}$,
\eqref{eq:3.8} holds for all $\psi\in\dot{H}^{1}$. In other words,
$\calL_{\calW}g=\tsum{b,j}{}c_{b,j}'(g)\calV_{b\ul{;j}}$ in the weak
sense. Thus
\[
\|\calL_{\calW}g\|_{(\dot{H}^{1})^{\ast}}=\|\tsum{b,j}{}c_{b,j}'(g)\calV_{b\ul{;j}}\|_{(\dot{H}^{1})^{\ast}}\aleq\tsum{b,j}{}|c_{b,j}'(g)|.
\]
Applying \eqref{eq:calL-H1-coer} with $g\in\calZ^{\perp}$ to the
left hand side and \eqref{eq:3.9} to the right hand side, we obtain
\[
\|g\|_{\dot{H}^{1}}\aleq\|\calL_{\calW}g\|_{(\dot{H}^{1})^{\ast}}\aleq\tsum{b,j}{}|c_{b,j}'(g)|\aleq R^{-\frac{N+2}{2}}\|g\|_{\dot{H}^{1}}.
\]
Since $R^{-1}<\dlt$ is small, we conclude $g=0$ as desired. 

\uline{Step 2: Proof of \mbox{\eqref{eq:inversion-H1-bound}}}.
Let $g$ and $c_{b,j}$ satisfy \eqref{eq:inversion-eqn}. By \eqref{eq:calL-H1-coer},
we have 
\[
\|g\|_{\dot{H}^{1}}\aleq\|\calL_{\calW}g\|_{(\dot{H}^{1})^{\ast}}\aleq\|h\|_{(\dot{H}^{1})^{\ast}}+\tsum{b,j}{}|c_{b,j}|.
\]
Recall $c_{b,j}=c_{b,j}'(g)+c_{b,j}''(h)$ from \eqref{eq:3.6}. Applying
the estimates \eqref{eq:3.9} and \eqref{eq:3.10}, we get 
\[
\tsum{b,j}{}|c_{b,j}|\aleq R^{-\frac{N+2}{2}}\|g\|_{\dot{H}^{1}}+\tsum{a,i}{}|h(\calV_{a;i})|.
\]
The previous two displays give \eqref{eq:inversion-H1-bound}.

\uline{Step 3: Proof of \mbox{\eqref{eq:inversion-H2-bound}} when
\mbox{$h\in L^{2}$}}. Assume in addition $h\in L^{2}$ and let $g$
and $c_{a,i}$ be the solution to \eqref{eq:inversion-eqn}. By elliptic
regularity, $g\in\dot{H}^{2}$. It remains to prove \eqref{eq:inversion-H2-bound}.

By \eqref{eq:calL-H2-coer}, we have 
\[
\|g\|_{\dot{H}^{2}}\aleq\|\calL_{\calW}g\|_{L^{2}}\aleq\|h\|_{L^{2}}+\tsum{b,j}{}\frac{|c_{b,j}|}{\lmb_{j}}.
\]
We investigate \eqref{eq:3.5} again to estimate the scalars $c_{b,j}$s.
We rearrange \eqref{eq:3.5} as 
\[
\tsum{b,j}{}\lan\lmb_{i}^{-1}\calV_{a;i},\lmb_{j}^{-1}\calV_{b;j}\ran\frac{c_{b,j}}{\lmb_{j}}=\frac{1}{\lmb_{i}}\{\lan\calL_{\calW}\calV_{a;i},g\ran-\lan\calV_{a;i},h\ran\}.
\]
In view of \eqref{eq:calV-calV-inner-prod}, the matrix $\lan\lmb_{i}^{-1}\calV_{a;i},\lmb_{j}^{-1}\calV_{b;j}\ran$
has uniformly bounded inverse, so 
\begin{align*}
\max_{b,j}\frac{|c_{b,j}|}{\lmb_{j}} & \aleq\max_{a,i}\frac{1}{\lmb_{i}}|\lan\calL_{\calW}\calV_{a;i},g\ran-\lan\calV_{a;i},h\ran|\\
 & \aleq\max_{a,i}\frac{1}{\lmb_{i}}\|\calL_{\calW}\calV_{a;i}\|_{L^{(2^{\ast\ast})'}}\|g\|_{L^{2^{\ast\ast}}}+\|h\|_{L^{2}}\aleq o_{R\to\infty}(1)\|g\|_{\dot{H}^{2}}+\|h\|_{L^{2}}.
\end{align*}
Here, in the last inequality we used \eqref{eq:f(W1,W2)-H2dual-est}.
This gives \eqref{eq:inversion-H2-bound}.
\end{proof}
With the above inversion lemma, we are now ready to prove Proposition~\ref{prop:Modified-Profiles}.
\begin{proof}[Proof of Proposition~\ref{prop:Modified-Profiles}]
Let $\dlt>0$ be a small constant that can shrink in the course of
the proof; the proposition will hold by setting $\dlt_{\rmp}=\dlt$
at the end of the proof.

\uline{Step 1: Proof of unique existence and \mbox{\eqref{eq:tdU-Hdot1-est}}}. 

The profile $\td U$ and the constants $\frkr_{b,j}$ need to solve
\begin{equation}
\calL_{\calW}\td U=-\Psi-\NL_{\calW}(\td U)+\sum_{b,j}\frkr_{b,j}\calV_{b\ul{;j}}\quad\text{and}\quad\lan\calZ_{a;i},\td U\ran=0,\label{eq:2.52}
\end{equation}
where $\Psi=f(\calW)-\tsum{i=1}Jf(W_{;i})$ and $\NL_{\calW}(\td U)=f(\calW+\td U)-f(\calW)-f'(\calW)\td U$. 

We formulate a fixed point problem to prove the existence. With the
notation in Lemma~\ref{lem:inversion}, let us denote the solution
operator associated to \eqref{eq:inversion-eqn} by $g=g(h)$ and
$c_{b,j}=c_{b,j}(h)$. Then, problem \eqref{eq:2.52} is equivalent
to 
\begin{equation}
\td U=g(-\Psi-\NL_{\calW}(\td U)),\label{eq:tdU-fixed-point-prob}
\end{equation}
and once $\td U$ is determined, the coefficients $\frkr_{b,j}$ of
\eqref{eq:2.52} are determined by $\frkr_{b,j}=c_{b,j}(-\Psi-\NL_{\calW}(\td U))$.

We solve \eqref{eq:tdU-fixed-point-prob} for $\td U$ within the
class $\|\td U\|_{\dot{H}^{1}}\leq\dlt$ by contraction principle.
First, applying \eqref{eq:Psi-Hdot1-dual} for $\Psi$ and \eqref{eq:f(a,b)-3}
for $\NL_{\calW}(\td U)$, we get 
\begin{align*}
\|\Psi\|_{(\dot{H}^{1})^{\ast}} & \aleq R^{-\frac{N+2}{2}}\aleq\dlt^{\frac{N+2}{2}},\\
\|\NL_{\calW}(\td U)\|_{(\dot{H}^{1})^{\ast}} & \aleq\||\td U|^{p}\|_{L^{(2^{\ast})'}}=\|\td U\|_{L^{2^{\ast}}}^{p}\aleq\dlt^{p}.
\end{align*}
Combining these bounds with \eqref{eq:inversion-H1-bound} gives 
\[
\|g(-\Psi-\NL_{\calW}(\td U))\|_{\dot{H}^{1}}\leq C(\dlt^{\frac{N+2}{2}}+\dlt^{p})\leq\frac{1}{2}\dlt.
\]
Next, for any $\td U,\td V\in\dot{H}^{1}$ with $\|\td U\|_{\dot{H}^{1}},\|\td V\|_{\dot{H}^{1}}\leq\dlt$,
we have 
\begin{align*}
\NL_{\calW}(\td U)-\NL_{\calW}(\td V) & =\{f(\calW+\td U)-f(\calW+\td V)\}-f'(\calW)(\td U-\td V)\\
 & =\tint 01\{f'(\calW+t\td U+(1-t)\td V)-f'(\calW)\}dt\cdot(\td U-\td V)\\
 & =\calO(|\td U|^{p-1}+|\td V|^{p-1})\cdot(\td U-\td V),
\end{align*}
so 
\[
\|\NL_{\calW}(\td U)-\NL_{\calW}(\td V)\|_{L^{(2^{\ast})'}}\aleq(\|\td U\|_{L^{2^{\ast}}}^{p-1}+\|\td V\|_{L^{2^{\ast}}}^{p-1})\|\td U-\td V\|_{L^{2^{\ast}}}\aleq\dlt^{p-1}\|\td U-\td V\|_{\dot{H}^{1}}.
\]
This implies 
\[
\|g(\NL_{\calW}(\td U)-\NL_{\calW}(\td V))\|_{\dot{H}^{1}}\leq C\dlt^{p-1}\|\td U-\td V\|_{\dot{H}^{1}}\leq\frac{1}{2}\|\td U-\td V\|_{\dot{H}^{1}}.
\]
By contraction principle, there exists a unique solution $\td U\in\dot{H}^{1}$
with $\|\td U\|_{\dot{H}^{1}}\leq\dlt$ to problem \eqref{eq:tdU-fixed-point-prob}.
This concludes the unique existence of $\td U$ and $\frkr_{b,j}$
within the class $\|\td U\|_{\dot{H}^{1}}\leq\dlt$. 

Note that the estimates in the previous paragraph give 
\[
\|\td U\|_{\dot{H}^{1}}\aleq\|\Psi\|_{(\dot{H}^{1})^{\ast}}+\|\NL_{\calW}(\td U)\|_{(\dot{H}^{1})^{\ast}}\aleq R^{-\frac{N+2}{2}}+\dlt^{p-1}\|\td U\|_{\dot{H}^{1}}.
\]
As $\dlt$ is small, we conclude $\|\td U\|_{\dot{H}^{1}}\aleq R^{-\frac{N+2}{2}}$
as in \eqref{eq:tdU-Hdot1-est}. 

\uline{Step 2: Proof of \mbox{$\td U\in\dot{H}^{2}$} and \mbox{\eqref{eq:tdU-Hdot2-est}}}. 

We revisit contraction principle applied to the fixed point problem
\eqref{eq:tdU-fixed-point-prob}, with additional $\dot{H}^{2}$ controls
this time. Applying \eqref{eq:Psi-L2} for $\Psi$ and proceeding
similarly (as before) for $\NL_{\calW}(\td U)$, we get 
\begin{align*}
\|\Psi\|_{L^{2}} & \aleq\tsum i{}\lmb_{i}^{-1}\frkp(R_{i}),\\
\|\NL_{\calW}(\td U)\|_{L^{2}} & \aleq\||\td U|^{p}\|_{L^{2}}\aleq\|\td U\|_{L^{2^{\ast}}}^{p-1}\|\td U\|_{L^{2^{\ast\ast}}}\aleq\dlt^{p-1}\|\td U\|_{\dot{H}^{2}}.
\end{align*}
Combining these bounds with \eqref{eq:Psi-L2} gives 
\[
\|g(-\Psi-\NL_{\calW}(\td U))\|_{\dot{H}^{2}}\aleq\tsum i{}\lmb_{i}^{-1}\frkp(R_{i})+\dlt^{p-1}\|\td U\|_{\dot{H}^{2}}.
\]
On the other hand, for $\td U,\td V\in\dot{H}^{1}\cap\dot{H}^{2}$
with $\|\td U\|_{\dot{H}^{1}},\|\td V\|_{\dot{H}^{1}}\leq\dlt$, we
have 
\begin{align*}
\|g(\NL_{\calW}(\td U)-\NL_{\calW}(\td V))\|_{\dot{H}^{2}} & \aleq\|\NL_{\calW}(\td U)-\NL_{\calW}(\td V)\|_{L^{2}}\\
 & \aleq(\|\td U\|_{L^{2^{\ast}}}^{p-1}+\|\td V\|_{L^{2^{\ast}}}^{p-1})\|\td U-\td V\|_{L^{2^{\ast\ast}}}\aleq\dlt^{p-1}\|\td U-\td V\|_{\dot{H}^{2}}.
\end{align*}
By contraction principle, we obtain $\td U\in\dot{H}^{1}\cap\dot{H}^{2}$
with the additional estimate \eqref{eq:tdU-Hdot2-est}. Note that
this $\td U$ agrees with the one constructed in Step~1 by the uniqueness
of $\td U$ in $\dot{H}^{1}$.

\uline{Step 3: Proof of \mbox{$\frkr_{b,j}$} estimates \mbox{\eqref{eq:frkr-est}}
and \mbox{\eqref{eq:frkr/lmb-est}}}. Testing \eqref{eq:2.52} against
$\calV_{a;i}$, we have 
\begin{equation}
\sum_{b,j}\lan\calV_{a;i},\calV_{b\ul{;j}}\ran\frkr_{b,j}=\lan\calV_{a;i},\Psi\ran+\lan\calV_{a;i},f(\calW+\td U)-f(\calW)-f'(W_{;i})\td U\ran.\label{eq:2.54-1}
\end{equation}
We estimate the first term of the right hand side using \eqref{eq:calV-Psi-estimate}:
\begin{equation}
|\lan\calV_{a;i},\Psi\ran|\aleq\min\{R^{-(N-2)},R_{i}^{-\frac{N+2}{2}}R^{-\frac{N-6}{2}}\}.\label{eq:2.54-3}
\end{equation}
We estimate the second term of the right hand side using \eqref{eq:f(a,b)-4}:
\begin{align}
 & |\lan\calV_{a;i},f(\calW+\td U)-f(\calW)-f'(W_{;i})\td U\ran|\nonumber \\
 & \aleq\|\{\tsum{j\neq i}{}f(W_{;i},W_{;j})+f(W_{;i},\td U)\}\td U\|_{L^{1}}\nonumber \\
 & \aleq\tsum{j\neq i}{}\|f(W_{;i},W_{;j})\|_{L^{(2^{\ast})'}}\|\td U\|_{L^{2^{\ast}}}+\|\lmb_{i}^{-2}\lan y_{i}\ran^{-4}|\td U|^{2}\|_{L^{1}}\nonumber \\
 & \aleq R_{i}^{-\frac{N+2}{2}}\|\td U\|_{\dot{H}^{1}}+\min\{\|\td U\|_{\dot{H}^{1}}^{2},\lmb_{i}\|\td U\|_{\dot{H}^{1}}\|\td U\|_{\dot{H}^{2}},\lmb_{i}^{2}\|\td U\|_{\dot{H}^{2}}^{2}\}.\label{eq:2.54-2}
\end{align}

For the proof of \eqref{eq:frkr-est}, the first bounds of \eqref{eq:2.54-3}
and \eqref{eq:2.54-2} and \eqref{eq:tdU-Hdot1-est} imply 
\[
\tsum{b,j}{}\lan\calV_{a;i},\calV_{b\ul{;j}}\ran\frkr_{b,j}=\calO(R^{-(N-2)})+\calO(R^{-N-2})=\calO(R^{-(N-2)}).
\]
As the matrix $\lan\calV_{a;i},\calV_{b\ul{;j}}\ran$ has uniformly
bounded inverse, we get \eqref{eq:frkr-est}. 

For the proof of \eqref{eq:frkr/lmb-est}, the second bounds of \eqref{eq:2.54-3}
and \eqref{eq:2.54-2} and \eqref{eq:tdU-Hdot1-est} imply 
\[
\tsum{b,j}{}\lan\lmb_{i}^{-1}\calV_{a;i},\lmb_{j}^{-1}\calV_{b;j}\ran\frac{\frkr_{b,j}}{\lmb_{j}}=\calO\Big(\frac{R_{i}^{-\frac{N+2}{2}}}{\lmb_{i}}R^{-\frac{N-6}{2}}+\|\td U\|_{\dot{H}^{2}}R^{-\frac{N+2}{2}}\Big).
\]
As the matrix $\lan\lmb_{i}^{-1}\calV_{a;i},\lmb_{j}^{-1}\calV_{b;j}\ran$
has uniformly bounded inverse, we get
\[
\max_{b,j}\frac{|\frkr_{b,j}|}{\lmb_{j}}\aleq\sum_{i}\frac{R_{i}^{-\frac{N+2}{2}}}{\lmb_{i}}R^{-\frac{N-6}{2}}+\|\td U\|_{\dot{H}^{2}}R^{-\frac{N+2}{2}}\aleq\sum_{i}\frac{R_{i}^{-\frac{N+2}{2}}}{\lmb_{i}}R^{-\frac{N-6}{2}},
\]
where in the last inequality we used \eqref{eq:tdU-Hdot2-est} and
$\frkp(R_{i})\aleq R_{i}^{-\frac{N+2}{2}}$. 

\uline{Step 4: Proof of \mbox{$\lmb_{k}\rd_{z_{k}^{c}}\td U\in\dot{H}^{1}$}
and \mbox{\eqref{eq:tdU-lmb-deriv-est}}}. Denote $\td U_{c,k}\coloneqq\lmb_{k}\rd_{z_{k}^{c}}\td U$.
We first derive an equation for $\td U_{c,k}$ by taking $\lmb_{k}\rd_{z_{k}^{c}}$
to the equation $\Dlt\td U+f(U)-\tsum j{}f(W_{;j})=\tsum{b,j}{}\frkr_{b,j}\calV_{b\ul{;j}}$
for $\td U$ and using $\lmb_{k}\rd_{z_{k}^{c}}W_{;k}=-\calV_{c;k}$:
\begin{equation}
\{\Dlt+f'(U)\}\td U_{c,k}-\{f'(U)-f(W_{;k})\}\calV_{c;k}=\tsum{b,j}{}(\lmb_{k}\rd_{z_{k}^{c}}\frkr_{b,j})\calV_{b\ul{;j}}+\tsum b{}\frkr_{b,k}\lmb_{k}\rd_{z_{k}^{c}}\calV_{b\ul{;k}}.\label{eq:2.43-1}
\end{equation}
On the other hand, differentiating the orthogonality condition $\lan\calV_{a\ul{;i}},\td U\ran=0$
gives $\lan\calV_{a\ul{;i}},\td U_{c,k}\ran=-\chf_{i=k}\lan\lmb_{k}\rd_{z_{k}^{c}}\calV_{a\ul{;k}},\td U\ran$.
Since this inner product is nonzero in general, we further decompose
\begin{equation}
\td U_{c,k}=\tsum{b,j}{}\frks_{c,k,b,j}\calV_{b;j}+\wh U_{c,k}\quad\text{so that}\quad\lan\calV_{a\ul{;i}},\wh U_{c,k}\ran=0\text{ for all }a,i,\label{eq:2.43-2}
\end{equation}
where the constants $\frks_{c,k,b,j}$ are uniquely determined by
the relation 
\begin{equation}
\tsum{b,j}{}\lan\calV_{a\ul{;i}},\calV_{b;j}\ran\frks_{c,k,b,j}=-\chf_{i=k}\lan\lmb_{k}\rd_{z_{k}^{c}}\calV_{a\ul{;k}},\td U\ran.\label{eq:2.43-3}
\end{equation}
Substituting the decomposition \eqref{eq:2.43-2} into \eqref{eq:2.43-1}
and using $\Dlt\calV_{b;j}=-f'(W_{;j})\calV_{b;j}$ as well, we obtain
the equation for $\wh U_{c,k}$:
\begin{align}
\calL_{\calW}\wh U_{c,k} & =-\{f'(U)-f'(\calW)\}\wh U_{c,k}+\wh{\Psi}_{c,k}+\tsum{b,j}{}(\lmb_{k}\rd_{z_{k}^{c}}\frkr_{b,j})\calV_{b\ul{;j}},\label{eq:2.43-4}\\
\wh{\Psi}_{c,k} & \coloneqq\{f'(U)-f(W_{;k})\}\calV_{c;k}-\tsum{b,j}{}\frks_{c,k,b,j}\{f'(U)-f'(W_{;j})\}\calV_{b;j}+\tsum b{}\frkr_{b,k}\lmb_{k}\rd_{z_{k}^{c}}\calV_{b\ul{;k}}.\label{eq:2.43-7}
\end{align}

The first term of the right hand side of \eqref{eq:2.43-3} can be
estimated by 
\[
\|\{f'(U)-f'(\calW)\}\wh U_{c,k}\|_{L^{(2^{\ast})'}}\aleq\||U-\calW|^{p-1}|\wh U_{c,k}|\|_{L^{(2^{\ast})'}}\aleq\|\td U\|_{L^{2^{\ast}}}^{p-1}\|\wh U_{c,k}\|_{L^{2^{\ast}}}\aleq R^{-\frac{2(N+2)}{N-2}}\|\wh U_{c,k}\|_{\dot{H}^{1}}.
\]
Since $\dlt>0$ is small, this estimate combined with Lemma~\ref{lem:inversion}
implies that \eqref{eq:2.43-4} has a unique solution $(\wh U_{c,k},\lmb_{k}\rd_{z_{k}^{c}}\frkr_{b,j})$
and $\wh U_{c,k}$ satisfies 
\begin{equation}
\|\wh U_{c,k}\|_{\dot{H}^{1}}\aleq\|\wh{\Psi}_{c,k}\|_{(\dot{H}^{1})^{\ast}}.\label{eq:2.43-6}
\end{equation}
By the derivation of \eqref{eq:2.43-4} and the uniqueness of $\td U$,
we conclude that $\td U$ is differentiable in $\vec{\lmb}$ and $\vec{z}$
and $\lmb_{k}\rd_{z_{k}^{c}}\td U=\tsum{b,j}{}\frks_{c,k,b,j}\calV_{b;j}+\wh U_{c,k}$
as well. 

The rest is devoted to the proof of the quantitative estimate \eqref{eq:tdU-lmb-deriv-est}
for $\td U_{c,k}=\lmb_{k}\rd_{z_{k}^{c}}\td U$. In view of \eqref{eq:2.43-2},
we first investigate the coefficients $\frks_{c,k,a,i}$. Substituting
\eqref{eq:calV-calV-inner-prod} and \eqref{eq:tdU-Hdot1-est} into
\eqref{eq:2.43-3}, we get 
\[
\frks_{c,k,a,i}+\tsum{b,j}{}\calO(\chf_{i\neq j}\frac{\lmb_{j}}{\lmb_{i}}R_{ij}^{-(N-4)})\frks_{c,k,b,j}=\calO(\chf_{i=k}\min\{R^{-\frac{N+2}{2}},\lmb_{k}\|\td U\|_{\dot{H}^{2}}\}).
\]
By the mapping properties 
\begin{align*}
\tsum{b,j}{}\chf_{i\neq j}\frac{\lmb_{j}}{\lmb_{i}}R_{ij}^{-(N-4)}\cdot1 & \aleq\tsum{b,j}{}\chf_{i\neq j}R_{ij}^{-(N-6)}\aleq R^{-(N-6)},\\
\tsum{b,j}{}\chf_{i\neq j}\frac{\lmb_{j}}{\lmb_{i}}R_{ij}^{-(N-4)}\cdot\frac{1}{\lmb_{j}} & \aleq\tsum{b,j}{}\chf_{i\neq j}R_{ij}^{-(N-4)}\cdot\frac{1}{\lmb_{i}}\aleq R^{-(N-4)}\cdot\frac{1}{\lmb_{i}},
\end{align*}
we conclude that 
\begin{equation}
|\frks_{c,k,a,i}|\aleq\min\{R^{-\frac{N+2}{2}},\lmb_{k}\|\td U\|_{\dot{H}^{2}}\}.\label{eq:2.43-5}
\end{equation}

Next, we claim 
\begin{equation}
\|\wh{\Psi}_{c,k}\|_{(\dot{H}^{1})^{\ast}}\aleq\min\{R^{-\frac{N+2}{2}},R_{k}^{-\frac{N+2}{2}}+\lmb_{k}\|\td U\|_{\dot{H}^{2}}+\max_{b}|\frkr_{b,k}|\}.\label{eq:2.43-8}
\end{equation}
We estimate each term of $\wh{\Psi}_{c,k}$ given in \eqref{eq:2.43-7}.
First, by \eqref{eq:f(a,b)-2}, \eqref{eq:f(W1,W2)-H1dual-est}, and
\eqref{eq:tdU-Hdot1-est}, we have 
\begin{align*}
\|\{f'(U)-f(W_{;k})\}\calV_{c;k}\|_{L^{(2^{\ast})'}} & \aleq\tsum{\ell\neq k}{}\|f(W_{;k},W_{;\ell})\|_{L^{(2^{\ast})'}}+\||W_{;k}|^{p-1}\td U\|_{L^{(2^{\ast})'}}\\
 & \aleq\tsum{\ell\neq k}{}R_{k\ell}^{-\frac{N+2}{2}}+\min\{\|\td U\|_{L^{2^{\ast}}},\lmb_{k}\|\td U\|_{L^{2^{\ast\ast}}}\}\\
 & \aleq\min\{R^{-\frac{N+2}{2}},R_{k}^{-\frac{N+2}{2}}+\lmb_{k}\|\td U\|_{\dot{H}^{2}}\}.
\end{align*}
Combining this with \eqref{eq:2.43-5}, the second term of \eqref{eq:2.43-7}
can be bounded by
\[
\|\tsum{b,j}{}\frks_{c,k,b,j}\{f'(U)-f'(W_{;j})\}\calV_{b;j}\|_{L^{(2^{\ast})'}}\aleq\min\{R^{-\frac{N+2}{2}},\lmb_{k}\|\td U\|_{\dot{H}^{2}}\}\cdot R^{-\frac{N+2}{2}}.
\]
Finally, we simply have 
\[
\|\tsum b{}\frkr_{b,k}\lmb_{k}\rd_{z_{k}^{c}}\calV_{b\ul{;k}}\|_{(\dot{H}^{1})^{\ast}}\aleq\max_{b}|\frkr_{b,k}|.
\]
Gathering the previous three estimates and using \eqref{eq:frkr-est}
with $N-2>\frac{N+2}{2}$ for the $\frkr_{b,k}$, we obtain \eqref{eq:2.43-8}. 

We are now ready to complete the proof of \eqref{eq:tdU-lmb-deriv-est}.
Recall \eqref{eq:2.43-2} so that 
\[
\|\td U_{c,k}\|_{\dot{H}^{1}}\aleq\tsum{b,j}{}|\frks_{c,k,b,j}|+\|\wh U_{c,k}\|_{\dot{H}^{1}}.
\]
For the first term, we apply \eqref{eq:2.43-5}. For the second term,
we apply \eqref{eq:2.43-6} and \eqref{eq:2.43-8}. This completes
the proof of \eqref{eq:tdU-lmb-deriv-est}. 
\end{proof}

\section{\label{sec:Modulation-in-multi-bubble}Modulation in multi-bubble
tube}

This section sets the ground for the dynamical analysis of solutions
$u(t)$. In particular, for solutions $u(t)$ lying in a neighborhood
of multi-bubbles (say, in a multi-bubble tube $\calT_{J}(\dlt)$),
we will decompose 
\[
u(t)=U(\vec{\iota},\vec{\lmb}(t),\vec{z}(t))+g(t)\qquad\text{with}\qquad\lan\calZ_{a\ul{;i}}(t),g(t)\ran=0,
\]
where $U(\vec{\iota},\vec{\lmb},\vec{z})$ is the modified multi-bubble
profile constructed in Proposition~\ref{prop:Modified-Profiles}.
We then prove a crucial coercivity estimate for the dissipation term
$\|\Dlt u+f(u)\|_{L^{2}}^{2}$ that leads to the \emph{monotonicity
estimate}. We also prove rough modulation estimates for $\lmb_{i,t}$,
$z_{i,t}$ and some basic properties for time-global $W$-bubbling
solutions $u(t)$. Finally, we sketch the proof of the non-existence
of non-colliding finite-time blow-up $W$-bubbling solutions (Proposition~\ref{prop:finite-time-non-existence}).

\subsection{Distance between pure multi-bubbles}

Let $J\in\bbN$ and $\dlt>0$.
\begin{align*}
B(\vec{\iota},\vec{\lmb},\vec{z};\dlt) & \coloneqq\{(\vec{\iota}',\vec{\lmb}',\vec{z}')\in\calP_{J}:d_{\calP}((\vec{\iota},\vec{\lmb},\vec{z}),(\vec{\iota}',\vec{\lmb}',\vec{z}'))<\dlt\},\\
d_{\calP}((\vec{\iota},\vec{\lmb},\vec{z}),(\vec{\iota}',\vec{\lmb}',\vec{z}')) & \coloneqq\max_{i\in\setJ}\Big(\min\Big\{1,|\iota_{i}-\iota_{i}'|+\Big|\log\Big(\frac{\lmb_{i}}{\lmb_{i}'}\Big)\Big|+\frac{|z_{i}-z_{i}'|}{\lmb_{i}}+\frac{|z_{i}-z_{i}'|}{\lmb_{i}'}\Big\}\Big).
\end{align*}
Note that $d_{\calP}$ is a metric on $\calP_{J}$. This can be used
to measure distances between pure multi-bubbles in view of the following
proposition. 
\begin{prop}[Distance between pure multi-bubbles]
\label{prop:distance-multi-bubbles}Let $J\geq1$. 
\begin{itemize}
\item For $(\vec{\iota},\vec{\lmb},\vec{z}),(\vec{\iota}',\vec{\lmb}',\vec{z}')\in\{\pm\}^{J}\times(0,\infty)^{J}\times(\bbR^{N})^{J}$,
we have 
\begin{equation}
\|\calW(\vec{\iota},\vec{\lmb},\vec{z})-\calW(\vec{\iota}',\vec{\lmb}',\vec{z}')\|_{\dot{H}^{1}}\aleq d_{\calP}((\vec{\iota},\vec{\lmb},\vec{z}),(\vec{\iota}',\vec{\lmb}',\vec{z}')).\label{eq:2.4}
\end{equation}
\item There exist $\dlt>0$ and $C\geq1$ with the following property. If
$(\vec{\iota},\vec{\lmb},\vec{z}),(\vec{\iota}',\vec{\lmb}',\vec{z}')\in\calP_{J}(\dlt)$
are such that $\|\calW(\vec{\iota},\vec{\lmb},\vec{z})-\calW(\vec{\iota}',\vec{\lmb}',\vec{z}')\|_{L^{2^{\ast}}}\leq\dlt$,
then there exists unique permutation $\pi:\setJ\to\setJ$ satisfying
\begin{equation}
d_{\calP}((\vec{\iota},\vec{\lmb},\vec{z}),(\vec{\iota}'\circ\pi,\vec{\lmb}'\circ\pi,\vec{z}'\circ\pi))\leq C\|\calW(\vec{\iota},\vec{\lmb},\vec{z})-\calW(\vec{\iota}',\vec{\lmb}',\vec{z}')\|_{L^{2^{\ast}}}.\label{eq:2.5}
\end{equation}
\end{itemize}
\end{prop}

\begin{proof}
See Appendix~\ref{sec:Proof-of-Proposition-3.1}.
\end{proof}

\subsection{Decomposition lemmas}

The next lemma allows us to decompose a function $u\in\dot{H}^{1}$
near pure multi-bubbles $\calW$ (or, $u\in\calT_{J}(\dlt)$) into
$u=U(\vec{\iota},\vec{\lmb},\vec{z})+g$ with the modified multi-bubble
profile $U$ and suitable orthogonality conditions on $g$. 
\begin{lem}[Static modulation lemma]
\label{lem:static-modulation}Let $J\geq1$. There exist a constant
$C_{J}\geq1$ and a small $\dlt_{J}>0$ with the following property. 
\begin{itemize}
\item For any $u\in\dot{H}^{1}$ and $(\vec{\td{\iota}},\vec{\td{\lmb}},\vec{\td z})\in\calP_{J}(\dlt_{J})$
such that $\|u-\calW(\vec{\td{\iota}},\vec{\td{\lmb}},\vec{\td z})\|_{\dot{H}^{1}}<\dlt_{J}$,
there exists unique $(\vec{\iota},\vec{\lmb},\vec{z})\in B(\vec{\td{\iota}},\vec{\td{\lmb}},\vec{\td z};C_{J}\dlt_{J})$
such that 
\begin{equation}
\lan\calZ_{a\ul{;i}},u-U(\vec{\iota},\vec{\lmb},\vec{z})\ran=0\qquad\forall a\in\{0,1,\dots,N\},\ i\in\setJ.\label{eq:static-orthog}
\end{equation}
Moreover, $\vec{\iota}=\vec{\td{\iota}}$, $(\vec{\iota},\vec{\lmb},\vec{z})\in\calP_{J}(2\dlt_{J})$,
and the map $u\mapsto(\vec{\iota},\vec{\lmb},\vec{z})$ is $C^{1}$
with 
\begin{equation}
\|u-U(\vec{\iota},\vec{\lmb},\vec{z})\|_{\dot{H}^{1}}+d_{\calP}((\vec{\iota},\vec{\lmb},\vec{z}),(\vec{\td{\iota}},\vec{\td{\lmb}},\vec{\td z}))\leq\tfrac{1}{2}C_{J}\|u-U(\vec{\td{\iota}},\vec{\td{\lmb}},\vec{\td z})\|_{\dot{H}^{1}}.\label{eq:static-lipschitz}
\end{equation}
\item For any $\dlt\in(0,\dlt_{J}]$ and $u\in\calT_{J}(\dlt)$, there exists
$(\vec{\iota},\vec{\lmb},\vec{z})\in\calP_{J}(2\dlt)$, unique up
to permutation, such that the orthogonality conditions \eqref{eq:static-orthog}
and $\|u-U(\vec{\iota},\vec{\lmb},\vec{z})\|_{\dot{H}^{1}}<C_{J}\dlt$
hold. 
\end{itemize}
\end{lem}

\begin{rem}
Similar results within symmetry are established in \cite[Lemma B.1]{DuyckaertsKenigMerle2023Acta}
and \cite[Lemma 2.25]{JendrejLawrie2025JAMS}. Further complications
arise here because there is no natural ordering of scales; we could
only state our second item in terms of unique up to permutation.
\end{rem}

\begin{proof}
Let $0<\dlt\ll1$ be a small parameter that can shrink in the course
of the proof. 

\uline{Step 1: Application of the implicit function theorem}. In
this step, we claim a slightly stronger version of the first item,
written in the next paragraph.

There exist $\dlt'\in(0,\frac{1}{10})$ and $C_{J}\geq1$ such that
for any $u\in\dot{H}^{1}$ and $(\vec{\td{\iota}},\vec{\td{\lmb}},\vec{\td z})\in\calP_{J}(\dlt_{J})$
with $\|u-\calW(\vec{\td{\iota}},\vec{\td{\lmb}},\vec{\td z})\|_{\dot{H}^{1}}<\dlt_{J}$,
there exists unique $(\vec{\iota},\vec{\lmb},\vec{z})\in B(\vec{\td{\iota}},\vec{\td{\lmb}},\vec{\td z};\dlt')$
such that \eqref{eq:static-orthog} holds. Moreover, $\vec{\iota}=\vec{\td{\iota}}$
(since $\dlt'<\frac{1}{10}$), $(\vec{\iota},\vec{\lmb},\vec{z})\in\calP_{J}(2\dlt_{J})$,
and the map $u\mapsto(\vec{\iota},\vec{\lmb},\vec{z})$ is $C^{1}$
with the Lipschitz estimate \eqref{eq:static-lipschitz}. 

To show this, define 
\begin{align*}
\bfF & :B_{\dot{H}^{1}}(\calW(\vec{\td{\iota}},\vec{\td{\lmb}},\vec{\td z});1)\times B_{\calP}((\vec{\td{\iota}},\vec{\td{\lmb}},\vec{\td z});1)\to\bbR^{(N+1)J}\\
\bfF & =(F_{a,i})_{a,i},\quad a\in\{0,1,\dots,N\},\ i\in\setJ,
\end{align*}
by 
\[
F_{a,i}(u,\vec{\td{\iota}},\vec{\lmb},\vec{z})=\lan\calZ_{a\ul{;i}},u-U(\vec{\td{\iota}},\vec{\lmb},\vec{z})\ran.
\]
Note that $\bfF(U(\vec{\td{\iota}},\vec{\td{\lmb}},\vec{\td z}),\vec{\td{\iota}},\vec{\td{\lmb}},\vec{\td z})=0$.
Next, using $\|\lmb_{j}\rd_{z_{j}^{b}}U\|_{\dot{H}^{1}}\aleq1$ and
\eqref{eq:tdU-Hdot1-est}, we have 
\begin{align*}
\|u-U(\vec{\td{\iota}},\vec{\lmb},\vec{z})\|_{\dot{H}^{1}} & \leq\|u-U(\vec{\td{\iota}},\vec{\td{\lmb}},\vec{\td z})\|_{\dot{H}^{1}}+\calO\big(d_{\calP}((\vec{\td{\iota}},\vec{\lmb},\vec{z}),(\vec{\td{\iota}},\vec{\td{\lmb}},\vec{\td z}))\big)\\
 & \leq\|u-\calW(\vec{\td{\iota}},\vec{\td{\lmb}},\vec{\td z})\|_{\dot{H}^{1}}+\calO\big(\dlt_{J}^{\frac{N+2}{2}}+d_{\calP}((\vec{\td{\iota}},\vec{\lmb},\vec{z}),(\vec{\td{\iota}},\vec{\td{\lmb}},\vec{\td z}))\big).
\end{align*}
This together with \eqref{eq:calZ-calV-inner-prod} and \eqref{eq:tdU-lmb-deriv-est}
implies that $\bfF$ is $C^{1}$ and 
\begin{align*}
\lmb_{j}\rd_{z_{j}^{b}}F_{a,i}(u,\vec{\lmb},\vec{z}) & =\lan\calZ_{a\ul{;i}},\calV_{b;j}\ran+\lan\calZ_{a\ul{;i}},-\lmb_{j}\rd_{z_{j}^{b}}\td U\ran+\chf_{i=j}\lan\lmb_{i}\rd_{z_{i}^{b}}(\calZ_{a\ul{;i}}),u-U(\vec{\td{\iota}},\vec{\lmb},\vec{z})\ran\\
 & =\{\chf_{(a,i)=(b,j)}+\calO(\tfrac{\lmb_{j}}{\lmb_{i}}R_{ij}^{-(N-4)})\}+\calO(R^{-\frac{N+2}{2}}+\|u-U(\vec{\td{\iota}},\vec{\lmb},\vec{z})\|_{\dot{H}^{1}})\\
 & =\chf_{(a,i)=(b,j)}+\calO\big(\dlt_{J}^{N-6}+\dlt_{J}^{\frac{N+2}{2}}+d_{\calP}((\vec{\td{\iota}},\vec{\lmb},\vec{z}),(\vec{\td{\iota}},\vec{\td{\lmb}},\vec{\td z}))\big).
\end{align*}
Applying the implicit function theorem (more precisely a quantitative
version \cite[Remark 2.26]{JendrejLawrie2025JAMS}) in a neighborhood
of $(\calW(\vec{\td{\iota}},\vec{\td{\lmb}},\vec{\td z}),\vec{\td{\iota}},\vec{\td{\lmb}},\vec{\td z})$,
we get the claim. In particular, the first item of this lemma is proved.

\uline{Step 2: Proof of the second item}. By the definition of
$u\in\calT_{J}(\dlt)$, we can choose $(\vec{\td{\iota}},\vec{\td{\lmb}},\vec{\td z})\in\calP_{J}(\dlt)$
such that $\|u-\calW(\vec{\td{\iota}},\vec{\td{\lmb}},\vec{\td z})\|_{\dot{H}^{1}}<\dlt$.
The existence of $(\vec{\iota},\vec{\lmb},\vec{z})\in\calP_{J}(2\dlt)$
as in the second item follows the first item and $\|u-U(\vec{\td{\iota}},\vec{\td{\lmb}},\vec{\td z})\|_{\dot{H}^{1}}<\dlt+\calO(\dlt^{\frac{N+2}{2}})<2\dlt$.

We turn to the proof of uniqueness up to permutation. Let $(\vec{\iota},\vec{\lmb},\vec{z})\in\calP_{J}(2\dlt_{J})$
be such that $\|u-U(\vec{\iota},\vec{\lmb},\vec{z})\|_{\dot{H}^{1}}<C_{J}\dlt_{J}$
and \eqref{eq:static-orthog} hold. By the uniqueness statement of
Step~1, it suffices to show that $(\vec{\iota},\vec{\lmb},\vec{z})\circ\pi\in B(\vec{\td{\iota}},\vec{\td{\lmb}},\vec{\td z};\dlt')$
for some permutation $\pi:\setJ\to\setJ$. Note that 
\begin{align*}
\|\calW(\vec{\iota},\vec{\lmb},\vec{z})-\calW(\vec{\td{\iota}},\vec{\td{\lmb}},\vec{\td z})\|_{\dot{H}^{1}} & \leq\|(\calW-U)(\vec{\iota},\vec{\lmb},\vec{z})\|_{\dot{H}^{1}}+\|(\calW-U)(\vec{\td{\iota}},\vec{\td{\lmb}},\vec{\td z})\|_{\dot{H}^{1}}\\
 & \quad+\|u-U(\vec{\iota},\vec{\lmb},\vec{z})\|_{\dot{H}^{1}}+\|u-U(\vec{\td{\iota}},\vec{\td{\lmb}},\vec{\td z})\|_{\dot{H}^{1}}=\calO(\dlt_{J}).
\end{align*}
Since $\dlt_{J}>0$ is small, we can apply Proposition~\ref{prop:distance-multi-bubbles}
to find a permutation $\pi:\setJ\to\setJ$ satisfying $d_{\calP}((\vec{\iota}',\vec{\lmb}',\vec{z}')\circ\pi,(\vec{\td{\iota}},\vec{\td{\lmb}},\vec{\td z}))=\calO(\dlt_{J})$.
Further shrinking $\dlt_{J}>0$ if necessary, we get $(\vec{\iota}',\vec{\lmb}',\vec{z}')\circ\pi\in B(\vec{\iota},\vec{\lmb},\vec{z};\dlt')$
as desired. This completes the proof.
\end{proof}
In dynamical applications, we need to decompose a solution $u(t)$,
which is a $C^{1}$-curve in $\calT_{J}(\dlt)$, into $U(\vec{\iota},\vec{\lmb}(t),\vec{z}(t))+g(t)$.
We may apply the static modulation Lemma~\ref{lem:static-modulation}
for each time $t$, but this requires an extra care (a path lifting
argument) because $(\vec{\iota},\vec{\lmb},\vec{z})$ are determined
only up to permutation of indices.
\begin{lem}[Modulation of a curve in $\calT_{J}(\dlt_{J})$]
\label{lem:curve-modulation}Let $J\geq1$. Let the constants $\dlt_{J}$
and $C_{J}$ be as in Lemma~\ref{lem:static-modulation}. Let $u:I\to\calT_{J}(\dlt_{J})$,
where $I$ is an interval, be a continuous (resp., $C^{1}$) curve
and let $t_{0}\in I$ and $(\vec{\iota},\vec{\td{\lmb}},\vec{\td z})\in\calP_{J}(\dlt_{J})$
be such that $\|u(t_{0})-\calW(\vec{\iota},\vec{\td{\lmb}},\vec{\td z})\|_{\dot{H}^{1}}<\dlt_{J}$.
Then, there exists unique continuous (resp., $C^{1}$) curve $(\vec{\iota},\vec{\lmb},\vec{z}):I\to\calP_{J}(2\dlt_{J})$
such that 
\begin{align}
d_{\calP}((\vec{\iota},\vec{\lmb}(t_{0}),\vec{z}(t_{0})),(\vec{\iota},\vec{\td{\lmb}},\vec{\td z})) & <C_{J}\dlt_{J},\label{eq:2.14}\\
\|u(t)-U(\vec{\iota},\vec{\lmb}(t),\vec{z}(t))\|_{\dot{H}^{1}} & <C_{J}\dlt_{J}\qquad\forall t\in I,\label{eq:2.15}
\end{align}
and (here $\calZ_{a\ul{;i}}(t)$ is defined using $\vec{\iota},\vec{\lmb}(t),\vec{z}(t)$)
\begin{equation}
\lan\calZ_{a\ul{;i}}(t),u(t)-U(\vec{\iota},\vec{\lmb}(t),\vec{z}(t))\ran=0\qquad\forall a\in\{0,1,\dots,N\},\ i\in\setJ,\ t\in I.\label{eq:curve-orthog}
\end{equation}
Moreover, if $u(t)\in\calT_{J}(\dlt)$ for some $t\in I$ and $\dlt\in(0,\dlt_{J}]$,
then $(\vec{\iota},\vec{\lmb}(t),\vec{z}(t))\in\calP_{J}(2\dlt)$
and 
\begin{align*}
\|u(t)-U(\vec{\iota},\vec{\lmb}(t),\vec{z}(t))\|_{\dot{H}^{1}} & <C_{J}\dlt.
\end{align*}
If $u(t_{0})\in\calT_{J}(\dlt)$ for some $\dlt\in(0,\dlt_{J}]$,
then 
\begin{equation}
d_{\calP}((\vec{\iota},\vec{\lmb}(t_{0}),\vec{z}(t_{0})),(\vec{\iota},\vec{\td{\lmb}},\vec{\td z}))<C_{J}\dlt.\label{eq:2.17}
\end{equation}
\end{lem}

\begin{proof}
This is a standard consequence of Lemma~\ref{lem:static-modulation}
and a path lifting argument.

\uline{Step 1: Uniqueness}. Assume $(\vec{\iota},\vec{\lmb},\vec{z})$
and $(\vec{\iota},\vec{\lmb}',\vec{z}')$ are two such continuous
curves in $\calP_{J}(2\dlt_{J})$. Let $\wh I\coloneqq\{t\in I:\vec{\lmb}(t)=\vec{\lmb}'(t)$
and $\vec{z}(t)=\vec{z}'(t)\}$; we need to show $I=\wh I$. By connectivity,
it suffices to show that $\wh I$ is non-empty, open, and closed in
$I$. First, $t_{0}\in\wh I$ due to \eqref{eq:2.14} (and similarly
for $\vec{\lmb}'(t_{0}),\vec{z}'(t_{0})$) and the first item of Lemma~\ref{lem:static-modulation}.
Next, $\wh I$ is closed in $I$ by continuity. Finally, we show that
$\wh I$ is open in $I$. Let $\tau\in\wh I$. As $u(\tau)\in\calT_{J}(\dlt_{J})$,
there is $(\vec{\td{\iota}}^{(\tau)},\vec{\td{\lmb}}^{(\tau)},\vec{\td z}^{(\tau)})\in\calP_{J}(\dlt_{J})$
such that $\|u(\tau)-\calW(\vec{\td{\iota}}^{(\tau)},\vec{\td{\lmb}}^{(\tau)},\vec{\td z}^{(\tau)})\|_{\dot{H}^{1}}<\dlt_{J}$.
By Lemma~\ref{lem:static-modulation}, there exists a permutation
$\pi^{(\tau)}$ such that $(\vec{\iota},\vec{\lmb}(\tau),\vec{z}(\tau))\circ\pi^{(\tau)}=(\vec{\iota},\vec{\lmb}'(\tau),\vec{z}'(\tau))\circ\pi^{(\tau)}\in B(\vec{\td{\iota}}^{(\tau)},\vec{\td{\lmb}}^{(\tau)},\vec{\td z}^{(\tau)};C_{J}\dlt_{J})$.
Now, choose an interval $I^{(\tau)}\ni\tau$ open in $I$ such that
$\|u(t)-\calW(\vec{\td{\iota}}^{(\tau)},\vec{\td{\lmb}}^{(\tau)},\vec{\td z}^{(\tau)})\|_{\dot{H}^{1}}<\dlt_{J}$
and $(\vec{\iota},\vec{\lmb}(t),\vec{z}(t))\circ\pi^{(\tau)},(\vec{\iota},\vec{\lmb}'(t),\vec{z}'(t))\circ\pi^{(\tau)}\in B(\vec{\td{\iota}}^{(\tau)},\vec{\td{\lmb}}^{(\tau)},\vec{\td z}^{(\tau)};C_{J}\dlt_{J})$
for all $t\in I^{(\tau)}$. By the uniqueness statement of the first
item of Lemma~\ref{lem:static-modulation}, we have $(\vec{\iota},\vec{\lmb}(t),\vec{z}(t))\circ\pi^{(\tau)}=(\vec{\iota},\vec{\lmb}'(t),\vec{z}'(t))\circ\pi^{(\tau)}$
and hence $(\vec{\iota},\vec{\lmb}(t),\vec{z}(t))=(\vec{\iota},\vec{\lmb}'(t),\vec{z}'(t))$
for $t\in I^{(\tau)}$. This says $I^{(\tau)}\subseteq\wh I$, so
$\wh I$ is open in $I$. This completes the proof of uniqueness.

It remains to show the existence of a curve $(\vec{\iota},\vec{\lmb}(t),\vec{z}(t))$.

\uline{Step 2: Reduction to the case of compact intervals}. We
show that the case of compact intervals implies the general case.
Indeed, for a general interval $I$, choose compact subintervals $t_{0}\in I_{1}\subset I_{2}\subset\dots$
of $I$ such that $I=\bigcup_{n}I_{n}$. On each compact interval
$I_{n}$, we apply the existence result to obtain $(\vec{\iota}^{(I_{n})},\vec{\lmb}^{(I_{n})},\vec{z}^{(I_{n})})$.
By the uniqueness result, we have $(\vec{\iota}^{(I_{n})},\vec{\lmb}^{(I_{n})},\vec{z}^{(I_{n})})|_{I_{m}}=(\vec{\iota}^{(I_{m})},\vec{\lmb}^{(I_{m})},\vec{z}^{(I_{m})})$
whenever $I_{m}\subset I_{n}$; the family $\{(\vec{\iota}^{(I_{n})},\vec{\lmb}^{(I_{n})},\vec{z}^{(I_{n})})\}_{n\in\bbN}$
is compatible. Taking the union defines a desired curve $\{(\vec{\iota},\vec{\lmb},\vec{z})\}_{n\in\bbN}$
on $I$.

\uline{Step 3: Existence of a continuous curve on a compact interval}.
Let $I=[a,b]$ be compact. In this step, we show that there is a continuous
curve $(\vec{\iota},\vec{\lmb},\vec{z}):I\to\calP_{J}(2\dlt_{J})$
satisfying \eqref{eq:2.14}, \eqref{eq:2.15}, and \eqref{eq:curve-orthog}.
The case of $a=b$ easily follows from the first item of Lemma~\ref{lem:static-modulation}.
Now assume $b>a$. Scaling in time, we may assume $I=[0,1]$. For
each $t\in I$, since $u(t)\in\calT_{J}(\dlt_{J})$, we can choose
$(\vec{\iota}^{(t)},\vec{\td{\lmb}}^{(t)},\vec{\td z}^{(t)})\in\calP_{J}(\dlt_{J})$
and an interval $I^{(t)}\ni t$ open in $I$ such that $\|u(\tau)-\calW(\vec{\iota}^{(t)},\vec{\td{\lmb}}^{(t)},\vec{\td z}^{(t)})\|_{\dot{H}^{1}}<\dlt_{J}$
for all $\tau\in I^{(t)}$. When $t=t_{0}$, fix $(\vec{\iota}^{(t_{0})},\vec{\td{\lmb}}^{(t_{0})},\vec{\td z}^{(t_{0})})=(\vec{\iota},\vec{\td{\lmb}},\vec{\td z})$.
By the compactness of $I$ (by the Lebesgue number lemma for example),
there exists a partition $a=\tau_{0}<\tau_{1}<\dots<\tau_{K}=b$ such
that each $I_{n}\coloneqq[\tau_{n-1},\tau_{n}]$, $n=1,\dots,K$,
is contained in some $I^{(t_{n})}$. We may assume $\{t_{1},\dots,t_{K}\}\ni t_{0}$;
otherwise, we may add at most two $\tau$s into the partition $\{\tau_{0},\dots,\tau_{K}\}$
so that $I^{(t_{0})}$ contains an interval in the new partition.
Fix $n_{0}\in\{1,\dots,K\}$ such that $t_{n_{0}}=t_{0}$. On each
$I_{n}$, we apply the first item of Lemma~\ref{lem:static-modulation}
to obtain a continuous curve $(\vec{\iota}^{(t_{n})},\vec{\lmb}^{(t_{n})},\vec{z}^{(t_{n})}):I_{n}\to B(\vec{\iota}^{(t_{n})},\vec{\td{\lmb}}^{(t_{n})},\vec{\td z}^{(t_{n})};C_{J}\dlt_{J})$
satisfying \eqref{eq:2.15} and \eqref{eq:curve-orthog} on each $I_{n}$.
Note that \eqref{eq:2.14} is satisfied for $(\vec{\iota}^{(t_{n_{0}})},\vec{\lmb}^{(t_{n_{0}})},\vec{z}^{(t_{n_{0}})})$.
At this moment, the curves $\{(\vec{\iota}^{(t_{n})},\vec{\lmb}^{(t_{n})},\vec{z}^{(t_{n})})\}_{n}$
are not necessarily compatible at endpoints of each $I_{n}$, but
they are compatible up to permutation by the second part of Lemma~\ref{lem:static-modulation}.
Keeping $(\vec{\iota}^{(t_{n_{0}})},\vec{\lmb}^{(t_{n_{0}})},\vec{z}^{(t_{n_{0}})})$
and gluing permutations of other curves $(\vec{\iota}^{(t_{n})},\vec{\lmb}^{(t_{n})},\vec{z}^{(t_{n})})$,
$n\neq n_{0}$, we obtain a continuous curve $(\vec{\iota},\vec{\lmb},\vec{z}):I\to\calP_{J}(2\dlt_{J})$
satisfying \eqref{eq:2.14}, \eqref{eq:2.15}, and \eqref{eq:curve-orthog}.

\uline{Step 4: Existence of a \mbox{$C^{1}$} curve when \mbox{$u:I\to\calT_{J}(\dlt_{J})$}
is \mbox{$C^{1}$}}. Assume $u:I\to\dot{H}^{1}$ is $C^{1}$ and let
$(\vec{\iota},\vec{\lmb},\vec{z}):I\to\calP_{J}(2\dlt_{J})$ be the
unique continuous curve found in Step~3. Let $t\in I$. Choose $(\vec{\iota}^{(t)},\vec{\td{\lmb}}^{(t)},\vec{\td z}^{(t)})\in\calP_{J}(\dlt_{J})$
and an interval $I^{(t)}\ni t$ open in $I$ as in Step~3. It suffices
to show that $(\vec{\iota},\vec{\lmb},\vec{z})|_{I^{(t)}}$ is $C^{1}$.
Applying the first item of Lemma~\ref{lem:static-modulation}, we
obtain a $C^{1}$-curve $(\vec{\iota},\vec{\lmb}^{(t)},\vec{z}^{(t)}):I^{(t)}\to\calP_{J}(2\dlt_{J})$
around $(\vec{\iota}^{(t)},\vec{\td{\lmb}}^{(t)},\vec{\td z}^{(t)})$.
By the second item of Lemma~\ref{lem:static-modulation}, for each
$\tau\in I^{(t)}$, there is a permutation $\pi(\tau)$ such that
$(\vec{\iota},\vec{\lmb}(\tau),\vec{z}(\tau))=(\vec{\iota},\vec{\lmb}^{(t)}(\tau),\vec{z}^{(t)}(\tau))\circ\pi(\tau)$.
By continuity, $\pi(\tau)$ must be independent of $\tau$, so $(\vec{\iota},\vec{\lmb},\vec{z})|_{I^{(t)}}=(\vec{\iota},\vec{\lmb}^{(t)},\vec{z}^{(t)})\circ\pi$.
Hence $(\vec{\iota},\vec{\lmb},\vec{z})|_{I^{(t)}}$ is $C^{1}$ as
desired.

\uline{Step 5: Proof of the further property}. The last sentence
of this lemma easily follows from the second part of Lemma~\ref{lem:static-modulation}.
For the proof of \eqref{eq:2.17}, use \eqref{eq:static-lipschitz}
as well. This completes the proof.
\end{proof}

\subsection{Coercivity of dissipation}

Having established the decomposition $u=U(\vec{\iota},\vec{\lmb},\vec{z})+g$
with $g$ satisfying the orthogonality conditions \eqref{eq:static-orthog},
we prove a crucial coercivity estimate for the dissipation term $\|\Dlt u+f(u)\|_{L^{2}}^{2}$.
This provides a monotonicity estimate for the remainder $g$ (and
also on $\frkr_{a,i}/\lmb_{i}$), which will allow us to control the
error terms involving $g$ in modulation estimates.
\begin{prop}[Coercivity of dissipation near multi-bubbles]
Let $J\geq1$. There exists a small $\dlt>0$ such that if $u=U(\vec{\iota},\vec{\lmb},\vec{z})+g$
for some $(\vec{\iota},\vec{\lmb},\vec{z})\in\calP_{J}(\dlt)$ and
$g\in\dot{H}^{1}\cap\dot{H}^{2}$ satisfying the orthogonality conditions
\eqref{eq:static-orthog} and smallness $\|g\|_{\dot{H}^{1}}<\dlt$,
then 
\begin{equation}
\|\Dlt u+f(u)\|_{L^{2}}^{2}\aeq\|g\|_{\dot{H}^{2}}^{2}+\sum_{a,i}\frac{\frkr_{a,i}^{2}}{\lmb_{i}^{2}}.\label{eq:coercivity-dissipation}
\end{equation}
\end{prop}

\begin{proof}
Since $\Dlt U+f(U)=\sum_{a,i}\frkr_{a,i}\calV_{a\ul{;i}}$, we have
\[
\Dlt u+f(u)=\Dlt U+\Dlt g+f(U+g)=\tsum{a,i}{}\frkr_{a,i}\calV_{a\ul{;i}}+\Dlt g+f(U+g)-f(U).
\]
Taking the $L^{2}$-norm, we obtain 
\begin{equation}
\begin{aligned}\|\Dlt u+f(u)\|_{L^{2}}^{2} & =\|\tsum{a,i}{}\frkr_{a,i}\calV_{a\ul{;i}}\|_{L^{2}}^{2}+\|\Dlt g+f(U+g)-f(U)\|_{L^{2}}^{2}\\
 & \quad+2\lan\tsum{a,i}{}\frkr_{a,i}\calV_{a\ul{;i}},\Dlt g+f(U+g)-f(U)\ran.
\end{aligned}
\label{eq:2.40}
\end{equation}

For the first term of RHS\eqref{eq:2.40}, we use \eqref{eq:calV-calV-inner-prod}
to get 
\begin{equation}
\|\tsum{a,i}{}\frkr_{a,i}\calV_{a\ul{;i}}\|_{L^{2}}^{2}=\sum_{a,i}\frac{\frkr_{a,i}^{2}}{\lmb_{i}^{2}}\|\calV_{a}\|_{L^{2}}^{2}+\sum_{(a,i)\neq(b,j)}\frac{\frkr_{a,i}\frkr_{b,j}}{\lmb_{i}\lmb_{j}}\calO(R_{ij}^{-(N-4)})\aeq\sum_{a,i}\frac{\frkr_{a,i}^{2}}{\lmb_{i}^{2}}.\label{eq:2.41-1}
\end{equation}
For the second term of RHS\eqref{eq:2.40}, we write 
\begin{align*}
\Dlt g+f(U+g)-f(U) & =\calL_{\calW}g+\{f(U+g)-f(U)-f'(\calW)g\}\\
 & =\calL_{\calW}g+\NL_{U}(g)+\{f'(U)-f'(\calW)\}g\\
 & =\calL_{\calW}g+\calO(|g|^{p}+|\td U|^{p-1}|g|),
\end{align*}
where in the last equality we used \eqref{eq:f(a,b)-3} for $\NL_{U}(g)$
and $\{f'(U)-f'(\calW)\}=\calO(|U-\calW|^{p-1})$. Applying the coercivity
estimate \eqref{eq:calL-H2-coer} for $\calL_{\calW}g$, observe 
\begin{align*}
\|\calL_{\calW}g\|_{L^{2}} & \aeq\|g\|_{\dot{H}^{2}},\\
\||g|^{p}+|\td U|^{p-1}|g|\|_{L^{2}} & \aleq(\|g\|_{L^{2^{\ast}}}^{p-1}+\|\td U\|_{L^{2^{\ast}}}^{p-1})\|g\|_{L^{2^{\ast\ast}}}\aleq(\dlt^{p-1}+R^{-2p})\|g\|_{\dot{H}^{2}}.
\end{align*}
Thus the second term of RHS\eqref{eq:2.40} is of size 
\begin{equation}
\|\Dlt g+f(U+g)-f(U)\|_{L^{2}}^{2}\aeq\|g\|_{\dot{H}^{2}}^{2}.\label{eq:2.41-2}
\end{equation}

In view of \eqref{eq:2.40}, \eqref{eq:2.41-1}, and \eqref{eq:2.41-2},
it remains to show that the last term of RHS\eqref{eq:2.40} is much
smaller than $\sum_{a,i}\lmb_{i}^{-2}|\frkr_{a,i}|^{2}+\|g\|_{\dot{H}^{2}}^{2}$.
Using $\Dlt\calV_{a;i}=-f'(W_{;i})\calV_{a;i}$, it suffices to show
\[
|\lmb_{i}^{-1}\lan\calV_{a;i},f(U+g)-f(U)-f'(W_{;i})g\ran|\aleq\{o_{R\to\infty}(1)+\dlt^{p-1}\}\|g\|_{\dot{H}^{2}}.
\]
Using $|\calV_{a;i}|\aleq|W_{;i}|$ and \eqref{eq:f(a,b)-4}, we have
\begin{align*}
 & \lmb_{i}^{-1}|\calV_{a;i}\{f(U+g)-f(U)-f'(W_{;i})g\}|\\
 & \aleq\lmb_{i}^{-1}\{\tsum{j\neq i}{}f(W_{;i},W_{;j})+f(W_{;i},\td U)+f(W_{;i},g)\}|g|\\
 & \aleq\{\tsum{j\neq i}{}\lmb_{i}^{-1}f(W_{;i},W_{;j})+\lmb_{i}^{-1}|W_{;i}|(|\td U|^{p-1}+|g|^{p-1})\}|g|.
\end{align*}
Using \eqref{eq:f(W1,W2)-H2dual-est} for $f(W_{;i},W_{;j})$ and
\eqref{eq:tdU-Hdot1-est} for $\td U$, we get 
\begin{align*}
 & \lmb_{i}^{-1}|\lan\calV_{a;i},f(U+g)-f(U)-f'(W_{;i})g\ran|\\
 & \aleq\{\tsum{j\neq i}{}\lmb_{i}^{-1}\|f(W_{;i},W_{;j})\|_{L^{(2^{\ast\ast})'}}+\|\lmb_{i}^{-1}W_{;i}\|_{L^{2}}(\|\td U\|_{L^{2^{\ast}}}^{p-1}+\|g\|_{L^{2^{\ast}}}^{p-1})\}\|g\|_{L^{2^{\ast\ast}}}\\
 & \aleq\{o_{R_{i}\to\infty}(1)+(R^{-2p}+\dlt^{p-1})\}\|g\|_{\dot{H}^{2}}
\end{align*}
as desired. This completes the proof.
\end{proof}

\subsection{\label{subsec:finite-time-non-existence}Non-existence of non-colliding
finite-time blow-up}

In this subsection, we sketch the proof of Proposition~\ref{prop:finite-time-non-existence}.
As mentioned in the introduction, one can simply localize the argument
of \cite{CollotMerleRaphael2017CMP} to prove this.

Suppose $u(t)$ as in Proposition~\ref{prop:finite-time-non-existence}
exists. As the assumption suggests, we set $J=1$ and $\iota=+1$.
Introduce small parameters with the following parameter dependence:
\[
0<\dlt_{2}\ll r_{0}\ll\dlt_{0}\ll1.
\]
We introduce the first sequence of times $\Ttbexit_{n}$ defined by
(here, we are implicitly assuming that $n$ is sufficiently large
depending on $\dlt_{2},r_{0},\dlt_{0}$) 
\[
\Ttbexit_{n}\coloneqq\sup\{\tau\in[t_{n},T_{+}):\chi_{2r_{0}}u(t)\in\calT_{1}(\dlt_{0})\}\in(t_{n},T_{+}].
\]
As $\dlt_{0}$ is small, we can uniquely decompose $\chi_{2r_{0}}u(t)=W_{\lmb(t),z(t)}+\wh g(t)$
such that 
\[
\lan[\chi_{R_{\circ}}\calV_{a}]_{;1}(t),\wh g(t)\ran=0\quad\text{and}\quad\|\wh g(t)\|_{\dot{H}^{1}}\aleq\dlt_{0},
\]
where $R_{\circ}>10$ is some fixed large constant (such that the
matrix $\{\lan\chi_{R_{\circ}}\calV_{a},\calV_{b}\ran\}_{ab}$ is
invertible, in particular). Defining $g(t)=\wh g(t)+(1-\chi_{2r_{0}})u(t)$,
we have a decomposition 
\[
u(t)=W_{\lmb(t),z(t)}+g(t).
\]
Next, we introduce another sequence of times $T_{n}$ defined by 
\[
T_{n}\coloneqq\sup\{\tau\in[t_{n},\Ttbexit_{n}):\lmb(t)+|z(t)|<\dlt_{2}\}.
\]
Since $\dlt_{2}$ is much smaller than $r_{0}$ and $[\chi_{R_{\circ}}\calV_{a}]_{;1}$
is supported in the region $|x-z(t)|\leq2R_{\circ}\lmb(t)$, the previous
orthogonality conditions for $\wh g(t)$ transfer to $g(t)$ for all
$t\in[t_{n},T_{n})$: 
\begin{equation}
\lan[\chi_{R_{\circ}}\calV_{a}]_{\ul{;1}},g\ran=0.\label{eq:finite-time-orthog}
\end{equation}
As a consequence of \eqref{eq:finite-time-assumption-3}, we have
\[
\lmb(t_{n})+|z(t_{n})|=o_{n\to\infty}(1)\quad\text{and}\quad\|\chi_{2r_{0}}g(t_{n})\|_{\dot{H}^{1}}=o_{r_{0}\to0}(1).
\]
As a consequence of \eqref{eq:finite-time-assumption-2} (and $\dot{H}^{1}$-boundedness
of $\chi u(t)$), we have 
\[
\sup_{t\in[t_{n},T_{n})}\|(\chi_{4r_{0}}-\chi_{r_{0}/2})g(t)\|_{\dot{H}^{1}}=o_{r_{0}\to0}(1)\quad\text{and}\quad\sup_{t\in[t_{n},T_{n})}\|\chi_{4r_{0}}g(t)\|_{\dot{H}^{1}}=\calO(1).
\]

\begin{lem}[Localized spacetime estimate]
Consider the localized energy 
\[
E_{2r_{0}}[u(t)]\coloneqq\int\Big\{\frac{1}{2}|\nabla u(t)|^{2}-\frac{1}{p+1}|u(t)|^{p+1}\Big\}\chi_{2r_{0}}^{2}.
\]
Then, we have 
\begin{equation}
\sup_{t\in[t_{n},T_{n})}\big|E_{\chi_{2r_{0}}}[u(t)]-E_{\chi_{2r_{0}}}[u(t_{n})]\big|+\int_{t_{n}}^{T_{n}}\|\chi_{2r_{0}}g(t)\|_{\dot{H}^{2}}^{2}dt=o_{n\to\infty}(1).\label{eq:finite-time-spacetime}
\end{equation}
\end{lem}

\begin{proof}
We begin with the identity 
\begin{align*}
\tfrac{d}{dt}E_{2r_{0}}[u] & =-\tint{}{}\chi_{2r_{0}}^{2}(\rd_{t}u)^{2}-\tint{}{}(\nabla u\cdot\nabla(\chi_{2r_{0}}^{2}))(\rd_{t}u)\\
 & =-\tint{}{}\chi_{2r_{0}}^{2}\{\Dlt u+f(u)\}^{2}+\tint{}{}\tsum{a,b}{}\rd_{ab}(\chi_{2r_{0}}^{2})\rd_{a}u\rd_{b}u-\tint{}{}\{\tfrac{1}{2}|\nabla u|^{2}-\tfrac{1}{p+1}|u|^{p+1}\}\Dlt(\chi_{2r_{0}}^{2}).\\
 & =-\|\chi_{2r_{0}}\{\Dlt u+f(u)\}\|_{L^{2}}^{2}+\calO_{r_{0}}(1).
\end{align*}
Using $T_{+}<+\infty$, $\dot{H}^{1}$-boundedness of $\chi u(t)$
as $t\to T_{+}$, and $t_{n}\to T_{+}$, the above display gives 
\begin{equation}
\sup_{t\in[t_{n},T_{n})}\big|E_{\chi_{2r_{0}}}[u(t)]-E_{\chi_{2r_{0}}}[u(t_{n})]\big|+\int_{t_{n}}^{T_{n}}\|\chi_{2r_{0}}\{\Dlt u+f(u)\}\|_{L^{2}}^{2}dt=o_{n\to\infty}(1).\label{eq:finite-time-integrability-o_n(1)}
\end{equation}
To get \eqref{eq:finite-time-spacetime} from the above, we need to
show that $\|\chi_{2r_{0}}\{\Dlt u+f(u)\}\|_{L^{2}}$ controls $\|\chi_{2r_{0}}g(t)\|_{\dot{H}^{2}}$
with an acceptable error. We substitute the decomposition $u(t)=W_{;1}(t)+g(t)$
to get 
\[
\chi_{2r_{0}}\{\Dlt u+f(u)\}=\calL_{W_{;1}}(\chi_{2r_{0}}g)+\chi_{2r_{0}}\NL_{W_{;1}}(g)+[\chi_{2r_{0}},\Dlt]g.
\]
By a modified version of the linear coercivity estimate \eqref{eq:calL-one-bubble-coer}
adapted to the current orthogonality conditions \eqref{eq:finite-time-orthog},
we have 
\begin{align*}
\|\calL_{W_{;1}}(\chi_{2r_{0}}g)\|_{L^{2}} & \aeq\|\chi_{2r_{0}}g\|_{\dot{H}^{2}}.
\end{align*}
The remaining terms are acceptable errors: 
\begin{align*}
\|[\chi_{2r_{0}},\Dlt]g\|_{L^{2}} & \aleq r_{0}^{-1}\|\chf_{2r_{0}\leq|x|\leq4r_{0}}|x|^{-1}g\|_{L^{2}}\aleq\calO_{r_{0}}(1),\\
\|\chi_{2r_{0}}\NL_{W_{;1}}(g)\|_{L^{2}} & \aleq\|\chi_{4r_{0}}g\|_{L^{2^{\ast}}}^{p-1}\|\chi_{2r_{0}}g\|_{L^{2^{\ast\ast}}}\aleq o_{r_{0}\to0}(\|\chi_{2r_{0}}g\|_{\dot{H}^{2}}).
\end{align*}
The previous two displays imply 
\[
\|\chi_{2r_{0}}g\|_{\dot{H}^{2}}^{2}\aleq\|\chi_{2r_{0}}(\Dlt u+f(u))\|_{L^{2}}^{2}+\calO_{r_{0}}(1).
\]
Substituting this into \eqref{eq:finite-time-integrability-o_n(1)}
gives \eqref{eq:finite-time-spacetime}.
\end{proof}
\begin{lem}[Rough modulation estimates]
We have 
\begin{equation}
|\lmb_{t}|+|z_{t}|+\lmb|\lan[\chi_{R_{\circ}}\calY]_{\ul{;1}},g\ran_{t}|\aleq\|\chi_{2r_{0}}g\|_{\dot{H}^{2}}.\label{eq:finite-time-rough-mod}
\end{equation}
\end{lem}

\begin{proof}
We begin with the evolution equation for $g(t)$: 
\[
\rd_{t}g=\sum_{b=0}^{N}\frac{(z^{b})_{t}}{\lmb}\calV_{b;1}+\Dlt g+f(W+g)-f(W).
\]
Differentiating the orthogonality conditions \eqref{eq:finite-time-orthog}
in time, we get 
\[
0=\frac{d}{dt}\lan[\chi_{R_{\circ}}\calV_{a}]_{\ul{;1}},g\ran=\sum_{b}\frac{(z^{b})_{t}}{\lmb}\{\lan\chi_{R_{\circ}}\calV_{a},\calV_{b}\ran+\calO(\|\chi_{2r_{0}}g\|_{\dot{H}^{1}})\}+\calO(\lmb^{-1}\|\chi_{2r_{0}}g\|_{\dot{H}^{2}}).
\]
As $\lan\chi_{R_{\circ}}\calV_{a},\calV_{b}\ran$ is invertible and
$\|\chi_{2r_{0}}g\|_{\dot{H}^{1}}$ is small by the bootstrap hypothesis,
we conclude $|\lmb_{t}|+|z_{t}|\aleq\|\chi_{2r_{0}}g\|_{\dot{H}^{2}}$.
For $\lmb|\lan[\chi_{R_{\circ}}\calY]_{\ul{;1}},g\ran_{t}|$, simply
write the equation for $\lan[\chi_{R_{\circ}}\calY]_{\ul{;1}},g\ran_{t}$
and estimate all terms.
\end{proof}
\begin{lem}[Control of $\lmb$, $z$, and $\|g\|_{\dot{H}^{1}}$]
We have 
\begin{align}
\sup_{t\in[t_{n},T_{n})}\{\lmb(t)+|z(t)|+|\lan[\chi_{R_{\circ}}\calY]_{\ul{;1}},g\ran|\} & =o_{n\to\infty}(1),\label{eq:finite-time-lmb-z-control}\\
\sup_{t\in[t_{n},T_{n})}\|\chi_{2r_{0}}g(t)\|_{\dot{H}^{1}} & =o_{r_{0}\to0}(1).\label{eq:finite-time-control-g}
\end{align}
\end{lem}

\begin{proof}
\uline{Proof of \mbox{\eqref{eq:finite-time-lmb-z-control}}}.
By \eqref{eq:finite-time-rough-mod} and \eqref{eq:finite-time-spacetime},
we get 
\begin{align*}
\sup_{t\in[t_{n},T_{n})}\{\lmb(t)+|z(t)|\} & \leq\{\lmb(t_{n})+|z(t_{n})|\}+\int_{t_{n}}^{T_{n}}\calO(\|\chi_{2r_{0}}g(t)\|_{\dot{H}^{2}})dt\\
 & \leq o_{n\to\infty}(1)+o_{n\to\infty}(1)\cdot\sqrt{T_{n}-t_{n}}\leq o_{n\to\infty}(1).
\end{align*}
Next, we use $\lan[\chi_{R_{\circ}}\calY]_{\ul{;1}},g\ran=\calO(\lmb\|\chi_{2r_{0}}g\|_{\dot{H}^{2}})$,
\eqref{eq:finite-time-rough-mod}, and \eqref{eq:finite-time-spacetime}
to get 
\[
\sup_{t\in[t_{n},T_{n})}|\lan[\chi_{R_{\circ}}\calY]_{\ul{;1}},g\ran^{2}(t)|\leq|\lan[\chi_{R_{\circ}}\calY]_{\ul{;1}},g\ran^{2}(t_{n})|+\int_{t_{n}}^{T_{n}}\calO(\|\chi_{2r_{0}}g(t)\|_{\dot{H}^{2}}^{2})dt=o_{n\to\infty}(1).
\]
\uline{Proof of \mbox{\eqref{eq:finite-time-control-g}}}. We use
\eqref{eq:finite-time-spacetime}; we expand the localized energy
as 
\begin{align*}
E_{2r_{0}}[u(t)] & =\tint{}{}\{\tfrac{1}{2}|\nabla W_{;1}|^{2}-\tfrac{1}{p+1}|W_{;1}|^{p+1}\}\chi_{2r_{0}}^{2}-\tint{}{}\{2\nabla\chi_{2r_{0}}\cdot\nabla W_{;1}\}\chi_{2r_{0}}g\\
 & \quad+\tint{}{}\{\tfrac{1}{2}|\nabla g|^{2}-f'(W_{;1})g^{2}\}\chi_{2r_{0}}^{2}+\tint{}{}\calO(|g|^{p+1}\chi_{2r_{0}}^{2})\\
 & =\{E[W]+\calO_{r_{0}}(\lmb^{2D})\}+\calO_{r_{0}}(\lmb^{D}\|\chi_{2r_{0}}g\|_{\dot{H}^{1}})\\
 & \quad+\{-\tfrac{1}{2}\lan\calL_{W_{;1}}(\chi_{2r_{0}}g),\chi_{2r_{0}}g\ran+\calO(\|\chf_{2r_{0}\leq|x|\leq4r_{0}}|x|^{-1}g\|_{L^{2}}^{2})\}+\calO(\|\chi_{4r_{0}}g\|_{\dot{H}^{1}}^{p-1}\|\chi_{2r_{0}}g\|_{\dot{H}^{1}}^{2}).
\end{align*}
Therefore, we obtain 
\[
-\tfrac{1}{2}\lan\calL_{W_{;1}}(\chi_{2r_{0}}g(t)),\chi_{2r_{0}}g(t)\ran=E_{2r_{0}}[u(t)]-E[W]+o_{r_{0}\to0}(1)\cdot(1+\|\chi_{2r_{0}}g(t)\|_{\dot{H}^{1}}^{2}).
\]
Combining this with \eqref{eq:finite-time-spacetime} and $E_{2r_{0}}[u(t_{n})]-E[W]=o_{r_{0}\to0}(1)$,
we obtain 
\[
-\lan\calL_{W_{;1}}(\chi_{2r_{0}}g(t)),\chi_{2r_{0}}g(t)\ran=o_{r_{0}\to0}(1)\cdot(1+\|\chi_{2r_{0}}g(t)\|_{\dot{H}^{1}}^{2}).
\]
By a modified version of \eqref{eq:calL-one-bubble-coer-H1-quad-form}
adapted to the current orthogonality conditions \eqref{eq:finite-time-orthog}
and $\lan[\chi_{R_{\circ}}\calY]_{\ul{;1}},g\ran=o_{n\to\infty}(1)$,
we conclude 
\[
\|\chi_{2r_{0}}g(t)\|_{\dot{H}^{1}}^{2}=o_{r_{0}\to0}(1)\cdot(1+\|\chi_{2r_{0}}g(t)\|_{\dot{H}^{1}}^{2}).
\]
This gives \eqref{eq:finite-time-control-g}.
\end{proof}
\begin{cor}[Closing bootstrap]
\label{cor:finite-time-closing-bootstrap}For all sufficiently large
$n$, we have $T_{n}=\Ttbexit_{n}=T_{+}$. Moreover, we have 
\begin{align*}
\lmb(t)+|z(t)|+\tint t{T_{+}}\|\chi_{2r_{0}}g(\tau)\|_{\dot{H}^{2}}^{2}d\tau & =o_{t\to T_{+}}(1),\\
\|\chi_{2r_{0}}g(t)\|_{\dot{H}^{1}} & =o_{r_{0}\to0}(1),
\end{align*}
as $t\to T_{+}$.
\end{cor}

\begin{proof}
This follows from \eqref{eq:finite-time-lmb-z-control}, \eqref{eq:finite-time-spacetime},
and \eqref{eq:finite-time-control-g}.
\end{proof}
In order to derive a contradiction, we show that \emph{$\lmb(t)$
cannot go to zero} as $t\to T_{+}<+\infty$, contradicting to $\lmb(t)\to0$
above. For this purpose, we need a refined modulation estimate for
$\lmb(t)$. 
\begin{lem}[Refined modulation estimate]
We have 
\begin{equation}
\Big|\frac{\lmb_{t}}{\lmb}+\frac{d}{dt}o_{r_{0}\to0}(1)\Big|\aleq_{r_{0}}\lmb^{2(D-2)}+\|\chi g\|_{\dot{H}^{2}}^{2}.\label{eq:finite-time-ref-mod}
\end{equation}
\end{lem}

\begin{proof}
We begin with the identity 
\begin{align*}
\frac{d}{dt}\lan\chi_{r_{0}}[\Lmb W]_{\ul{;1}},g\ran & =\frac{\lmb_{t}}{\lmb}\lan\chi_{r_{0}}[\Lmb W]_{\ul{;1}},[\Lmb W]_{;1}\ran-\frac{\lmb_{t}}{\lmb}\lan\chi_{r_{0}}[\Lmb_{-1}\Lmb W]_{\ul{;1}},g\ran-\frac{z_{t}}{\lmb}\cdot\lan\chi_{r_{0}}[\nabla\Lmb W]_{\ul{;1}},g\ran\\
 & \quad+\lan\chi_{r_{0}}[\Lmb W]_{\ul{;1}},\calL_{W_{;1}}g+\NL_{W_{;1}}(g)\ran.
\end{align*}
We treat each term of the right hand side as follows. First, with
\eqref{eq:finite-time-rough-mod}, 
\[
\frac{\lmb_{t}}{\lmb}\lan\chi_{r_{0}}[\Lmb W]_{\ul{;1}},[\Lmb W]_{;1}\ran=\frac{\lmb_{t}}{\lmb}\{\|\Lmb W\|_{L^{2}}^{2}+\calO_{r_{0}}(\lmb^{N-4})\}=\frac{\lmb_{t}}{\lmb}\|\Lmb W\|_{L^{2}}^{2}+\calO_{r_{0}}(\lmb^{2D-3}\|\chi_{2r_{0}}g\|_{\dot{H}^{2}}).
\]
Next, we express $\Lmb_{-1}\Lmb W(y)=-(D-2)\Lmb W(y)+\calO(\lan y\ran^{-N})$
and use \eqref{eq:finite-time-rough-mod} to obtain 
\[
\frac{\lmb_{t}}{\lmb}\lan\chi_{r_{0}}[\Lmb_{-1}\Lmb W]_{\ul{;1}},g\ran=-(D-2)\frac{\lmb_{t}}{\lmb}\lan\chi_{r_{0}}[\Lmb W]_{\ul{;1}},g\ran+\calO(\|\chi_{2r_{0}}g\|_{\dot{H}^{2}}^{2}).
\]
Next, we use $\nabla\Lmb W(y)=\calO(\lan y\ran^{-(N-1)})$, $N>6$,
and \eqref{eq:finite-time-rough-mod} to obtain 
\[
\frac{z_{t}}{\lmb}\cdot\lan\chi_{r_{0}}[\nabla\Lmb W]_{\ul{;1}},g\ran=\calO(\|\chi_{2r_{0}}g\|_{\dot{H}^{2}}^{2}).
\]
Finally, using $\calL_{W}\Lmb W=0$ and \eqref{eq:f(a,b)-3}, we get
\begin{align*}
\lan\chi_{r_{0}}[\Lmb W]_{\ul{;1}},\calL_{W_{;1}}g\ran & =\lan[\Dlt,\chi_{r_{0}}][\Lmb W]_{\ul{;1}},g\ran=\calO_{r_{0}}(\lmb^{D-2}\|\chi_{2r_{0}}g\|_{\dot{H}^{2}}),\\
\lan\chi_{r_{0}}[\Lmb W]_{\ul{;1}},\NL_{W_{;1}}(g)\ran & =\|\calO(|W_{;1}|^{p-1}|\chi_{2r_{0}}g|^{2})\|_{L^{1}}=\calO(\|\chi_{2r_{0}}g\|_{\dot{H}^{2}}^{2}).
\end{align*}
Gathering the previous observations gives 
\[
\frac{d}{dt}\lan\chi[\Lmb W]_{\ul{;1}},g\ran=\frac{\lmb_{t}}{\lmb}\|\Lmb W\|_{L^{2}}^{2}+(D-2)\frac{\lmb_{t}}{\lmb}\lan\chi_{r_{0}}[\Lmb W]_{\ul{;1}},g\ran+\calO_{r_{0}}(\lmb^{D-2}\|\chi_{2r_{0}}g\|_{\dot{H}^{2}}+\|\chi_{2r_{0}}g\|_{\dot{H}^{2}}^{2}).
\]
Dividing both sides by $\|\Lmb W\|_{L^{2}}^{2}+(D-2)\lan\chi_{r_{0}}[\Lmb W]_{\ul{;1}},g\ran$,
we get 
\[
\frac{\lmb_{t}}{\lmb}+\frac{d}{dt}h(\lan\chi_{r_{0}}[\Lmb W]_{\ul{;1}},g\ran)=\calO_{r_{0}}(\lmb^{D-2}\|\chi_{2r_{0}}g\|_{\dot{H}^{2}}+\|\chi_{2r_{0}}g\|_{\dot{H}^{2}}^{2}),
\]
where $h(x)\coloneqq-\int_{0}^{x}(\|\Lmb W\|_{L^{2}}^{2}+(D-2)y)^{-1}dy$.
Since $\lan\chi_{r_{0}}[\Lmb W]_{\ul{;1}},g\ran=\calO(\|\chi_{2r_{0}}g\|_{\dot{H}^{1}})=o_{r_{0}\to0}(1)$,
we conclude \eqref{eq:finite-time-ref-mod}.
\end{proof}
\begin{proof}[Proof of the non-existence of $u(t)$]
By $\lmb\aleq1$ and \eqref{eq:finite-time-control-g}, the right
hand side of \eqref{eq:finite-time-ref-mod} is integrable in time
as $t\to T_{+}<+\infty$. This says that $\log\lmb+o_{r_{0}\to0}(1)$
is convergent as $t\to T_{+}$. In particular, $|\log\lmb(t)|$ cannot
go to infinity, contradicting to $\lmb(t)\to0$ of Corollary~\ref{cor:finite-time-closing-bootstrap}.
Therefore, $u(t)$ satisfying the properties of Proposition~\ref{prop:finite-time-non-existence}
cannot exist.
\end{proof}

\section{Several facts regarding (non-)degenerate configurations}

This section collects several facts regarding the (non-)degeneracy
of configurations. Here, \emph{
\[
z_{1}^{\ast},\dots,z_{J}^{\ast}\in\bbR^{N}\text{ are distinct and }\iota_{1},\dots,\iota_{J}\in\{\pm\}.
\]
}The analysis here is independent of time and the decomposition $u=U+g$.
We only deal with $\vec{\iota}$, $\vec{\lmb}$, and $\vec{z}$. Note
also that this section is only relevant to our second main theorem
(Theorem~\ref{thm:main-lmb-to-zero-classification}). Recall Definition~\ref{def:non-degeneracy-of-configurations}.
\begin{lem}
\label{lem:non-degen}Let $J\geq2$ and $\{(\iota_{1},z_{1}^{\ast}),\dots,(\iota_{J},z_{J}^{\ast})\}$
be non-degenerate. Then, we have 
\begin{equation}
|A^{\ast}\vec{\lmb}^{D}|\aeq\lmb_{\max}^{D},\qquad\forall\vec{\lmb}\in(0,\infty)^{J}.\label{eq:A*lmb^D-is-lmb_max^D}
\end{equation}
\end{lem}

\begin{proof}
First, we claim that there exists a vector $\vec{\beta}\in\bbR^{J}$
satisfying 
\begin{equation}
A^{\ast}\vec{\beta}\in(0,\infty)^{J}.\label{eq:gmm-04}
\end{equation}
This follows from the hyperplane separation theorem. By the non-degeneracy,
\[
\ker(A^{\ast})\cap\bbS_{\geq0}^{J-1}=\emptyset,\qquad\text{where }\bbS_{\geq0}^{J-1}\coloneqq\{\vec{c}\in[0,\infty)^{J}:|\vec{c}|=1\}.
\]
Applying the hyperplane separation theorem to the (closed) linear
subspace $\ker(A^{\ast})$ and the compact set $\bbS_{\geq0}^{J-1}$,
there exists $\vec{\gmm}\in\bbR^{J}$ such that 
\begin{align*}
\lan\vec{\gmm},\vec{c}\ran & =0\qquad\forall\vec{c}\in\ker(A^{\ast}),\\
\lan\vec{\gmm},\vec{c}\ran & \geq1\qquad\forall\vec{c}\in\bbS_{\geq0}^{J-1}.
\end{align*}
The first property implies $\vec{\gmm}\in\ker(A^{\ast})^{\perp}$.
As $A^{\ast}$ is symmetric, we obtain $\vec{\gmm}\in\mathrm{ran}(A^{\ast})$,
so $\vec{\gmm}=A^{\ast}\vec{\beta}$ for some $\vec{\beta}\in\bbR^{J}$.
The second property implies $A^{\ast}\vec{\beta}=\vec{\gmm}\in(0,\infty)^{J}$,
completing the proof of \eqref{eq:gmm-04}.

We are ready to show \eqref{eq:A*lmb^D-is-lmb_max^D}. As $A^{\ast}$
is symmetric and $A^{\ast}\vec{\beta}\in(0,\infty)^{J}$, we have
\[
|A^{\ast}\vec{\lmb}^{D}|\ageq|\lan A^{\ast}\vec{\lmb}^{D},\vec{\beta}\ran|=|\lan\vec{\lmb}^{D},A^{\ast}\vec{\beta}\ran|\aeq\lmb_{\max}^{D}.
\]
The other inequality $|A^{\ast}\vec{\lmb}^{D}|\aleq\lmb_{\max}^{D}$
is obvious. This completes the proof.
\end{proof}
\begin{lem}
\label{lem:distance-from-degen}Let $J\geq2$. For any $\dlt_{1}>0$
and $\vec{\lmb}\in(0,\infty)^{J}$ with 
\begin{equation}
\min_{\emptyset\neq\calI\subseteq\setJ}\bigg(\frac{|A_{\calI}^{\ast}\vec{\lmb}_{\calI}^{D}|}{|\vec{\lmb}_{\calI}^{D}|}+\sum_{i\notin\calI}\frac{\lmb_{i}}{\lmb_{\max}}\bigg)\geq\dlt_{1},\label{eq:dist-from-degen}
\end{equation}
we have 
\begin{align}
\sum_{i=1}^{J}(A^{\ast}\vec{\lmb}^{D})_{i}^{2}\lmb_{i}^{2D-2} & \aeq_{\dlt_{1}}\lmb_{\max}^{4D-2},\label{eq:case1-calD-non-degen}\\
\lmb_{\secmax} & \ageq\dlt_{1}\lmb_{\max}.\label{eq:case1-secmax=00003Dmax}
\end{align}
\end{lem}

\begin{rem}
\eqref{eq:case1-calD-non-degen} implies that the equations of $\lmb_{i,t}$
(see \eqref{eq:intro-lmb-eqn}) do not degenerate. 
\end{rem}

\begin{proof}
\uline{Proof of \mbox{\eqref{eq:case1-secmax=00003Dmax}}}. Let
$i_{0}\in\setJ$ be such that $\lmb_{\max}=\lmb_{i_{0}}$. Taking
$\calI=\{i_{0}\}$ in \eqref{eq:dist-from-degen} gives 
\[
\dlt_{1}\lmb_{\max}\leq\tsum{i\neq i_{0}}{}\lmb_{i}\aleq\lmb_{\secmax}.
\]

\uline{Proof of \mbox{\eqref{eq:case1-calD-non-degen}}}. As the
$(\aleq)$-inequality is obvious, we prove the $(\ageq_{\dlt_{1}})$-inequality.
Suppose not. Then, there exist $\dlt_{1}>0$ and a sequence $\vec{\lmb}_{n}\in(0,\infty)^{J}$
such that 
\[
\min_{\emptyset\neq\calI\subseteq\setJ}\bigg(\frac{|A_{\calI}^{\ast}\vec{\lmb}_{\calI,n}^{D}|}{|\vec{\lmb}_{\calI,n}^{D}|}+\sum_{i\notin\calI}\frac{\lmb_{i,n}}{\lmb_{\max,n}}\bigg)\geq\dlt_{1},\quad\sum_{i=1}^{J}(A^{\ast}\vec{\lmb}_{n}^{D})_{i}^{2}\lmb_{i,n}^{2D-2}\to0,\quad\text{and}\quad\lmb_{\max,n}=1.
\]
Passing to a subsequence, $\vec{\lmb}_{n}\to\exists\vec{\lmb}_{\infty}\in[0,1]^{J}$
with 
\[
\min_{\emptyset\neq\calI\subseteq\setJ}\bigg(\frac{|A_{\calI}^{\ast}\vec{\lmb}_{\calI,\infty}^{D}|}{|\vec{\lmb}_{\calI,\infty}^{D}|}+\sum_{i\notin\calI}\frac{\lmb_{i,\infty}}{\lmb_{\max,\infty}}\bigg)\geq\dlt_{1},\quad\sum_{i=1}^{J}(A^{\ast}\vec{\lmb}_{\infty}^{D})_{i}^{2}\lmb_{i,\infty}^{2D-2}=0,\quad\text{and}\quad\lmb_{\max,\infty}=1.
\]
Take $\calI=\{i\in\setJ:\lmb_{i,\infty}\neq0\}$. Note that $\calI\neq\emptyset$
by $\lmb_{\max,\infty}=1$. By the definition of $\calI$, we have
$\lmb_{i,\infty}=0$ for all $i\notin\calI$ and $A_{\calI}^{\ast}\vec{\lmb}_{\calI,\infty}^{D}=0$.
This implies 
\[
\frac{|A_{\calI}^{\ast}\vec{\lmb}_{\calI,n}^{D}|}{|\vec{\lmb}_{\calI,n}^{D}|}+\sum_{i\notin\calI}\frac{\lmb_{i,n}}{\lmb_{\max,n}}\to0,
\]
which is a contradiction. This completes the proof.
\end{proof}
When the signs are all the same or only one sign is different from
the others, we show that the configuration is always totally non-degenerate.
\begin{lem}
\label{lem:almost-same-sign-implies-totally-non-degen}Let $J\geq2$.
If $\#\{i\in\setJ:\iota_{i}=+1\}\leq1$ or $\#\{i\in\setJ:\iota_{i}=-1\}\leq1$,
then the configuration $\{(\iota_{1},z_{1}^{\ast}),\dots,(\iota_{J},z_{J}^{\ast})\}$
is totally non-degenerate. In particular, any degenerate configuration
has at least two positive signs and two negative signs so that $J\geq4$.
\end{lem}

\begin{proof}
Without loss of generality, it suffices to show that the configuration
is non-degenerate when $\iota_{1}=\dots=\iota_{J-1}=\iota_{J}$ or
$\iota_{1}=\dots=\iota_{J-1}=-\iota_{J}$. Suppose $A^{\ast}\vec{\frkc}=0$
for some $\vec{\frkc}\in[0,\infty)^{J}$. 

Consider first the case $\iota_{1}=\dots=\iota_{J-1}=\iota_{J}$.
Since $J\geq2$ and $A_{ij}^{\ast}>0$ whenever $i\neq j$, we easily
have $\vec{\frkc}=0$. Thus the configuration is non-degenerate.

We turn to the case $\iota_{1}=\dots=\iota_{J-1}=-\iota_{J}$. Observe
$0=(A^{\ast}\vec{\frkc})_{J}=\sum_{j\neq J}A_{Jj}\frkc_{j}$. Since
$\vec{\frkc}\in[0,\infty)^{J}$ and $A_{Jj}>0$ for all $j\neq J$,
we conclude that $\frkc_{j}=0$ for all $j\neq J$. Next, observe
$0=(A\vec{\frkc})_{1}=\sum_{j\neq1}A_{1j}\frkc_{j}=A_{1J}\frkc_{J}$.
Since $A_{1J}\neq0$, we conclude $\frkc_{J}=0$, so $\vec{\frkc}=0$.
Thus the configuration is non-degenerate.
\end{proof}
We turn our attention to minimally degenerate configurations. By degeneracy,
$A^{\ast}$ has a nonzero kernel element in $[0,\infty)^{J}$.
\begin{lem}
\label{lem:min-degen-implies-unique-pos-vec}Let $\{(\iota_{1},z_{1}^{\ast}),\dots,(\iota_{J},z_{J}^{\ast})\}$
be minimally degenerate. Then, we have the following.
\begin{itemize}
\item $\ker A^{\ast}$ is one-dimensional and there exists a unique positive
vector $\vec{\frkc}\in(0,\infty)^{J}$ with $|\vec{\frkc}|=1$ such
that $\ker A^{\ast}=\bbR\vec{\frkc}$. 
\item We have an equivalence of distances: 
\begin{equation}
\Big|\frac{\vec{\ell}}{|\vec{\ell}|}-\vec{\frkc}\Big|\aeq\frac{|A^{\ast}\vec{\ell}|}{|\vec{\ell}|}\qquad\text{for any }0\neq\vec{\ell}\in[0,\infty)^{J}.\label{eq:case3-equiv-dist}
\end{equation}
\item For any $\dlt_{1}>0$ and $\vec{\lmb}\in(0,\infty)^{J}$ such that
$\lmb_{\secmax}\geq\dlt_{1}\lmb_{\max}$, we have 
\begin{equation}
\sum_{i=1}^{J}\lmb_{i}^{2D-2}(A^{\ast}\vec{\lmb}^{D})_{i}^{2}\aeq_{\dlt_{1}}\lmb_{\max}^{2D-2}|A^{\ast}\vec{\lmb}^{D}|^{2}.\label{eq:case3-dist-from-kernel}
\end{equation}
\end{itemize}
\end{lem}

\begin{proof}
\uline{Proof of the first item}. First, we show that any $0\neq\vec{\frkc}\in[0,\infty)^{J}\cap\ker A^{\ast}$
is a positive vector. Otherwise, without loss of generality, $\frkc_{J}=0$.
Then, the reduced vector $0\neq(\frkc_{1},\dots,\frkc_{J-1})\in[0,\infty)^{J-1}$
belongs to the kernel of $(A_{ij}^{\ast})_{1\leq i,j\leq J-1}$. Note
also that $J-1\geq2$ due to Lemma~\ref{lem:almost-same-sign-implies-totally-non-degen}.
Therefore, $\{(\iota_{1},z_{1}^{\ast}),\dots(\iota_{J-1},z_{J-1}^{\ast})\}$
is degenerate, contradicting to the minimal degeneracy of $\{(\iota_{1},z_{1}^{\ast}),\dots,(\iota_{J},z_{J}^{\ast})\}$.

Now, pick any $0\neq\vec{\frkc}\in[0,\infty)^{J}\cap\ker A^{\ast}$
and renormalize so that $|\vec{\frkc}|=1$. By the previous paragraph,
$\vec{\frkc}$ is a positive vector. It suffices to show that $\ker A^{\ast}=\bbR\vec{\frkc}$.
Let $0\neq\vec{\frkc}'\in\ker A^{\ast}$ be arbitrary. Then, there
exists $a\neq0$ such that $\vec{\frkc}-a\vec{\frkc}'$ touches the
boundary of $[0,\infty)^{J}$. Since $\vec{\frkc}-a\vec{\frkc}'\in\ker A^{\ast}$
as well, it must be zero by the first paragraph. Hence, $\vec{\frkc}'=a^{-1}\vec{\frkc}\in\bbR\vec{\frkc}$
as desired. 

\uline{Proof of \mbox{\eqref{eq:case3-equiv-dist}}}. This easily
follows from the one-dimensionality of $\ker A^{\ast}$. 

\uline{Proof of \mbox{\eqref{eq:case3-dist-from-kernel}}}. As
the $(\aleq)$-inequality is obvious, it suffices to show the $(\ageq)$-inequality.
Suppose not. Then, there exists a sequence $\vec{\lmb}_{n}\in(0,\infty)^{J}$
such that $\lmb_{\secmax,n}\aeq\lmb_{\max,n}$ and 
\begin{equation}
\sum_{i\in\setJ}\lmb_{i,n}^{2D-2}(A^{\ast}\vec{\lmb}_{n}^{D})_{i}^{2}=o_{n\to\infty}(1)\cdot\lmb_{\max,n}^{2D-2}|A^{\ast}\vec{\lmb}_{n}^{D}|^{2}.\label{eq:3.9-1}
\end{equation}
Passing to a subsequence, we may assume $\vec{\lmb}_{n}^{D}/|\vec{\lmb}_{n}^{D}|\to\exists\vec{\ell}\in[0,1]^{J}$. 

We claim that $\ell_{i}>0$  for all $i\in\setJ$. Suppose not. Then,
$\calI\coloneqq\{j\in\setJ:\ell_{j}>0\}$ is a proper subset of $\setJ$.
Note that $|\calI|\geq2$ due to $\lmb_{\max,n}\aeq\lmb_{\secmax,n}$.
Since $\{(\iota_{1},z_{1}^{\ast}),\dots,(\iota_{J},z_{J}^{\ast})\}$
is minimally degenerate, $\{(\iota_{i},z_{i}^{\ast}):i\in\calI\}$
is non-degenerate. Hence $|A_{\calI}^{\ast}\vec{\lmb}_{\calI,n}|\aeq|\vec{\lmb}_{\calI,n}|$
by Lemma~\ref{lem:non-degen}. Since $\lmb_{i,n}\aeq\lmb_{\max,n}$
for all $i\in\calI$ by the definition of $\calI$, we obtain 
\[
\sum_{i\in\calI}\lmb_{i,n}^{2D-2}(A_{\calI}^{\ast}\vec{\lmb}_{\calI,n}^{D})_{i}^{2}\aeq\lmb_{\max,n}^{4D-2}.
\]
Now, observe 
\[
\sum_{i\in\setJ}\lmb_{i,n}^{2D-2}(A^{\ast}\vec{\lmb}_{n}^{D})_{i}^{2}=\sum_{i\in\calI}\lmb_{i,n}^{2D-2}(A_{\calI}^{\ast}\vec{\lmb}_{\calI,n}^{D})_{i}^{2}+o_{n\to\infty}(1)\cdot\lmb_{\max,n}^{4D-2}\aeq\lmb_{\max,n}^{4D-2}.
\]
This contradicts to \eqref{eq:3.9-1}. This completes the proof of
the claim.

By the previous claim, $\lmb_{i,n}\aeq\lmb_{\max,n}$ for all $i\in\setJ$.
Thus we get 
\[
\sum_{i\in\setJ}\lmb_{i,n}^{2D-2}(A^{\ast}\vec{\lmb}_{n}^{D})_{i}^{2}\aeq\lmb_{\max,n}^{2D-2}\sum_{i\in\setJ}(A^{\ast}\vec{\lmb}_{n}^{D})_{i}^{2}=\lmb_{\max,n}^{2D-2}|A^{\ast}\vec{\lmb}_{n}^{D}|^{2},
\]
contradicting to \eqref{eq:3.9-1}. This completes the proof of \eqref{eq:case3-dist-from-kernel}.
\end{proof}
\begin{lem}
\label{lem:case3-non-degen-quantity}Let $\{(\iota_{1},z_{1}^{\ast}),\dots,(\iota_{J},z_{J}^{\ast})\}$
be minimally degenerate; let $\vec{\frkc}=(\frkc_{1},\dots,\frkc_{J})\in(0,\infty)^{J}$
be the unique positive vector with $|\vec{\frkc}|=1$ such that $\ker A^{\ast}=\bbR\vec{\frkc}$,
given by Lemma~\ref{lem:min-degen-implies-unique-pos-vec}. Assume
\begin{equation}
\sum_{i=1}^{J}|v_{i}^{\ast}|^{2}>0,\quad\text{where}\quad v_{i}^{\ast}\coloneqq\sum_{j=1}^{J}\frac{A_{ij}^{\ast}\frkc_{i}\frkc_{j}}{|z_{i}^{\ast}-z_{j}^{\ast}|^{2}}(z_{j}^{\ast}-z_{i}^{\ast})\in\bbR^{N}.\label{eq:case3-v_i-not-all-zero}
\end{equation}
Then, for any $K_{0}>0$ and $\dlt_{1}>0$, there exists $\dlt_{2}>0$
such that 
\begin{equation}
\sum_{i=1}^{J}\lmb_{i}^{2D-2}(A[\vec{z}]\vec{\lmb}^{D})_{i}^{2}+\sum_{i=1}^{J}\lmb_{i}^{2D}\Big|\sum_{j=1}^{J}\frac{A[\vec{z}]_{ij}\lmb_{j}^{D}(z_{i}-z_{j})}{|z_{i}-z_{j}|^{2}}\Big|^{2}\aeq_{K_{0},\dlt_{1}}\lmb_{\max}^{2D-2}|A[\vec{z}]\vec{\lmb}^{D}|^{2}+\lmb_{\max}^{4D}\label{eq:case3-non-degen-quantity}
\end{equation}
for any $\vec{\lmb}\in(0,\infty)^{J}$ and $\vec{z}\in(\bbR^{N})^{J}$
satisfying $\lmb_{\max}\leq K_{0}$, $\lmb_{\secmax}\geq\dlt_{1}\lmb_{\max}$,
and $|\vec{z}-\vec{z}^{\ast}|<\dlt_{2}$.
\end{lem}

\begin{rem}
\eqref{eq:case3-non-degen-quantity} says that the equations of $z_{i,t}$
(see \eqref{eq:intro-z-eqn}) do not degenerate whenever the equations
of $\lmb_{i,t}$ degenerate.
\end{rem}

\begin{proof}
As the $(\aleq)$-inequality is obvious, it suffices to show the $(\ageq)$-inequality.
Suppose not. Then, there exist sequences $\vec{\lmb}_{n}\in(0,\infty)^{J}$
and $\vec{z}_{n}\in(\bbR^{N})^{J}$ such that $\lmb_{\max,n}\aleq1$,
$\lmb_{\secmax,n}\aeq\lmb_{\max,n}$, $\vec{z}_{n}\to\vec{z}^{\ast}$,
and 
\begin{equation}
\begin{aligned}\sum_{i}\lmb_{i,n}^{2D-2}(A[\vec{z}_{n}]\vec{\lmb}_{n}^{D})_{i}^{2}+\sum_{i}\lmb_{i,n}^{2D}\Big|\sum_{j}\frac{A[\vec{z}_{n}]_{ij}\lmb_{j,n}^{D}(z_{i,n}-z_{j,n})}{|z_{i,n}-z_{j,n}|^{2}}\Big|^{2}\\
=o_{n\to\infty}(1)\cdot(\lmb_{\max,n}^{2D-2}|A[\vec{z}_{n}]\vec{\lmb}_{n}^{D}|^{2} & +\lmb_{\max,n}^{4D}).
\end{aligned}
\label{eq:7.30-1}
\end{equation}
Passing to a subsequence, we may assume that $|A[\vec{z}_{n}]\vec{\lmb}_{n}^{D}|/|\vec{\lmb}_{n}^{D}|\to\exists c\in[0,\infty)$
and $\vec{\lmb}_{n}^{D}/|\vec{\lmb}_{n}^{D}|\to\exists\vec{\ell}\in[0,1]^{J}$. 

We note that $\ell_{i}>0$ for all $i\in\setJ$; the same argument
in the proof of \eqref{eq:case3-dist-from-kernel} gives $\ell_{i}>0$
for all $i\in\setJ$, so this will not be repeated here. By $\ell_{i}>0$
for all $i\in\setJ$, we have $\lmb_{i,n}\aeq\lmb_{\max,n}$ for all
$i\in\setJ$. This together with \eqref{eq:7.30-1} gives 
\begin{equation}
\lmb_{\max,n}^{2D-2}|A[\vec{z}_{n}]\vec{\lmb}_{n}^{D}|^{2}+\sum_{i}\lmb_{i,n}^{2D}\Big|\sum_{j}\frac{A[\vec{z}_{n}]_{ij}\lmb_{j,n}^{D}(z_{i,n}-z_{j,n})}{|z_{i,n}-z_{j,n}|^{2}}\Big|^{2}=o_{n\to\infty}(1)\cdot\lmb_{\max,n}^{4D}.\label{eq:7.30-3}
\end{equation}
In particular, $|A[\vec{z}_{n}]\vec{\lmb}_{n}^{D}|=o_{n\to\infty}(1)\cdot\lmb_{\max,n}^{D+1}=o_{n\to\infty}(1)\cdot\lmb_{\max,n}^{D}$.
This together with $\vec{z}_{n}\to\vec{z}^{\ast}$ gives $|A^{\ast}\vec{\lmb}_{n}^{D}|=o_{n\to\infty}(1)\cdot\lmb_{\max,n}^{D}$.
By \eqref{eq:case3-equiv-dist}, we conclude 
\[
\frac{\vec{\lmb}_{n}^{D}}{|\vec{\lmb}_{n}^{D}|}=\vec{\frkc}+o_{n\to\infty}(1).
\]
Using this and \eqref{eq:case3-v_i-not-all-zero}, we see that the
second term of LHS\eqref{eq:7.30-3} becomes 
\[
\sum_{i}\lmb_{i,n}^{2D}\Big|\sum_{j}\frac{A[\vec{z}_{n}]_{ij}\lmb_{j,n}^{D}(z_{i,n}-z_{j,n})}{|z_{i,n}-z_{j,n}|^{2}}\Big|^{2}=|\vec{\lmb}_{n}^{D}|^{4}\cdot\Big\{\sum_{i}|v_{i}^{\ast}|^{2}+o_{n\to\infty}(1)\Big\}\aeq\lmb_{\max,n}^{4D},
\]
contradicting to \eqref{eq:7.30-3}. This completes the proof.
\end{proof}
The following lemma says that the assumption \eqref{eq:case3-v_i-not-all-zero}
is always satisfied when $J=4$, allowing us to say that our main
theorems give a complete classification result for $J\leq4$. 
\begin{lem}
\label{lem:J=00003D4-implies-not-too-degenerate}If $J=4$, then every
degenerate configuration is minimally degenerate and satisfies \eqref{eq:case3-v_i-not-all-zero}.
\end{lem}

\begin{proof}
Let $\{(\iota_{i},z_{i}^{\ast})\}_{i=1}^{4}$ be a degenerate configuration.
It is minimally degenerate by Lemma~\ref{lem:almost-same-sign-implies-totally-non-degen}.
To show that \eqref{eq:case3-v_i-not-all-zero} is satisfied, suppose
not, i.e., $v_{1}^{\ast}=v_{2}^{\ast}=v_{3}^{\ast}=v_{4}^{\ast}=0$.
By Lemma~\ref{lem:almost-same-sign-implies-totally-non-degen} again,
we may assume $\iota_{1}=\iota_{2}=-\iota_{3}=-\iota_{4}$. Let $\vec{\frkc}\in(0,\infty)^{4}$
be as in Lemma~\ref{lem:min-degen-implies-unique-pos-vec}.

\uline{Case 1: All \mbox{$z_{i}^{\ast}$} are collinear}. Without
loss of generality, we may assume $z_{i}^{\ast}\in\bbR=\bbR\times\{0\}^{N-1}$
so that there is an ordering of $z_{i}^{\ast}$. It suffices to consider
the following three cases:
\begin{enumerate}
\item[(a)] $z_{1}^{\ast}<z_{2}^{\ast}<z_{3}^{\ast}<z_{4}^{\ast}$,
\item[(b)] $z_{1}^{\ast}<z_{3}^{\ast}<z_{4}^{\ast}<z_{2}^{\ast}$,
\item[(c)] $z_{1}^{\ast}<z_{3}^{\ast}<z_{2}^{\ast}<z_{4}^{\ast}$.
\end{enumerate}
In Case (a), the equation $v_{2}^{\ast}=0$ gives 
\[
\frac{\frkc_{1}}{|z_{1}^{\ast}-z_{2}^{\ast}|^{2D+1}}+\frac{\frkc_{3}}{|z_{2}^{\ast}-z_{3}^{\ast}|^{2D+1}}+\frac{\frkc_{4}}{|z_{3}^{\ast}-z_{4}^{\ast}|^{2D+1}}=0.
\]
Since $\frkc_{1},\frkc_{3},\frkc_{4}>0$, we get a contradiction.

In Case (b), we rewrite the equations of $v_{3}^{\ast}=0$ and $(A^{\ast}\vec{\frkc})_{3}=0$
as 
\begin{align*}
\frac{\frkc_{2}}{|z_{2}^{\ast}-z_{3}^{\ast}|^{2D+1}} & =\frac{\frkc_{1}}{|z_{1}^{\ast}-z_{3}^{\ast}|^{2D+1}}+\frac{\frkc_{4}}{|z_{3}^{\ast}-z_{4}^{\ast}|^{2D+1}},\\
\frac{\frkc_{4}}{|z_{3}^{\ast}-z_{4}^{\ast}|^{2D}} & =\frac{\frkc_{1}}{|z_{1}^{\ast}-z_{3}^{\ast}|^{2D}}+\frac{\frkc_{2}}{|z_{2}^{\ast}-z_{3}^{\ast}|^{2D}}.
\end{align*}
Substituting the former into the latter, we get 
\[
\frac{\frkc_{4}}{|z_{3}^{\ast}-z_{4}^{\ast}|^{2D}}=\Big(\frac{1}{|z_{1}^{\ast}-z_{3}^{\ast}|^{2D}}+\frac{|z_{2}^{\ast}-z_{3}^{\ast}|}{|z_{1}^{\ast}-z_{3}^{\ast}|^{2D+1}}\Big)\frkc_{1}+\frac{|z_{2}^{\ast}-z_{3}^{\ast}|}{|z_{3}^{\ast}-z_{4}^{\ast}|^{2D+1}}\frkc_{4}>\frac{|z_{2}^{\ast}-z_{3}^{\ast}|}{|z_{3}^{\ast}-z_{4}^{\ast}|^{2D+1}}\frkc_{4}.
\]
Since $\frkc_{4}>0$ and $|z_{2}^{\ast}-z_{3}^{\ast}|>|z_{3}^{\ast}-z_{4}^{\ast}|$
in Case (b), we get a contradiction.

We turn to the most delicate Case (c). We rewrite the equations $v_{4}^{\ast}=0$
and $(A^{\ast}\vec{\frkc})_{4}=0$ as 
\begin{align*}
\frac{\frkc_{1}}{|z_{1}^{\ast}-z_{4}^{\ast}|^{2D+1}}+\frac{\frkc_{2}}{|z_{2}^{\ast}-z_{4}^{\ast}|^{2D+1}} & =\frac{\frkc_{3}}{|z_{3}^{\ast}-z_{4}^{\ast}|^{2D+1}},\\
\frac{\frkc_{1}}{|z_{1}^{\ast}-z_{4}^{\ast}|^{2D}}+\frac{\frkc_{2}}{|z_{2}^{\ast}-z_{4}^{\ast}|^{2D}} & =\frac{\frkc_{3}}{|z_{3}^{\ast}-z_{4}^{\ast}|^{2D}}.
\end{align*}
Multiplying the former by $|z_{1}^{\ast}-z_{4}^{\ast}|$, subtracting
the latter, and using $z_{1}^{\ast}<z_{3}^{\ast}<z_{2}^{\ast}<z_{4}^{\ast}$,
we obtain 
\begin{equation}
\frac{|z_{1}^{\ast}-z_{2}^{\ast}|}{|z_{2}^{\ast}-z_{4}^{\ast}|^{2D+1}}\frkc_{2}=\frac{|z_{1}^{\ast}-z_{3}^{\ast}|}{|z_{3}^{\ast}-z_{4}^{\ast}|^{2D+1}}\frkc_{3}.\label{eq:0.04}
\end{equation}
On the other hand, we rewrite the equations $v_{1}^{\ast}=0$ and
$(A^{\ast}\vec{\frkc})_{1}=0$ as 
\begin{align*}
\frac{\frkc_{3}}{|z_{1}^{\ast}-z_{3}^{\ast}|^{2D+1}}+\frac{\frkc_{4}}{|z_{1}^{\ast}-z_{4}^{\ast}|^{2D+1}} & =\frac{\frkc_{2}}{|z_{1}^{\ast}-z_{2}^{\ast}|^{2D+1}},\\
\frac{\frkc_{3}}{|z_{1}^{\ast}-z_{3}^{\ast}|^{2D}}+\frac{\frkc_{4}}{|z_{1}^{\ast}-z_{4}^{\ast}|^{2D}} & =\frac{\frkc_{2}}{|z_{1}^{\ast}-z_{2}^{\ast}|^{2D}}.
\end{align*}
Multiplying the former by $|z_{1}^{\ast}-z_{4}^{\ast}|$ and subtracting
the latter, we similarly obtain 
\begin{equation}
\frac{|z_{3}^{\ast}-z_{4}^{\ast}|}{|z_{1}^{\ast}-z_{3}^{\ast}|^{2D+1}}\frkc_{3}=\frac{|z_{2}^{\ast}-z_{4}^{\ast}|}{|z_{1}^{\ast}-z_{2}^{\ast}|^{2D+1}}\frkc_{2}.\label{eq:0.05}
\end{equation}
Solving \eqref{eq:0.04} and \eqref{eq:0.05}, one concludes the identity
\[
|z_{1}^{\ast}-z_{2}^{\ast}||z_{3}^{\ast}-z_{4}^{\ast}|=|z_{1}^{\ast}-z_{3}^{\ast}||z_{2}^{\ast}-z_{4}^{\ast}|.
\]
This is impossible under $z_{1}^{\ast}<z_{3}^{\ast}<z_{2}^{\ast}<z_{4}^{\ast}$,
so we get a contradiction.

\uline{Case 2: \mbox{$z_{1}^{\ast},z_{2}^{\ast},z_{3}^{\ast},z_{4}^{\ast}$}
are not collinear}. First, we claim 
\begin{equation}
\frac{|z_{1}^{\ast}-z_{3}^{\ast}|}{|z_{1}^{\ast}-z_{2}^{\ast}|}=\frac{|z_{3}^{\ast}-z_{4}^{\ast}|}{|z_{2}^{\ast}-z_{4}^{\ast}|}.\label{eq:0.03}
\end{equation}
To show this, we use the equations $v_{1}^{\ast}=v_{4}^{\ast}=0$
to obtain 
\begin{align*}
\frac{\frkc_{2}(z_{1}^{\ast}-z_{2}^{\ast})}{|z_{1}^{\ast}-z_{2}^{\ast}|^{2D+2}}-\frac{\frkc_{3}(z_{1}^{\ast}-z_{3}^{\ast})}{|z_{1}^{\ast}-z_{3}^{\ast}|^{2D+2}}-\frac{\frkc_{4}(z_{1}^{\ast}-z_{4}^{\ast})}{|z_{1}^{\ast}-z_{4}^{\ast}|^{2D+2}} & =0,\\
-\frac{\frkc_{1}(z_{4}^{\ast}-z_{1}^{\ast})}{|z_{1}^{\ast}-z_{4}^{\ast}|^{2D+2}}-\frac{\frkc_{2}(z_{4}^{\ast}-z_{2}^{\ast})}{|z_{2}^{\ast}-z_{4}^{\ast}|^{2D+2}}+\frac{\frkc_{3}(z_{4}^{\ast}-z_{3}^{\ast})}{|z_{3}^{\ast}-z_{4}^{\ast}|^{2D+2}} & =0.
\end{align*}
Since $z_{1}^{\ast},z_{2}^{\ast},z_{3}^{\ast},z_{4}^{\ast}$ are not
collinear, there exists a vector $n\in\bbR^{N}$ with $(z_{1}^{\ast}-z_{4}^{\ast})\cdot n=0$
such that $(z_{1}^{\ast}-z_{2}^{\ast})\cdot n\neq0$ or $(z_{1}^{\ast}-z_{3}^{\ast})\cdot n\neq0$.
Taking the inner product of each of the above equations with $n$,
we conclude 
\begin{align*}
\frac{\frkc_{2}}{|z_{1}^{\ast}-z_{2}^{\ast}|^{2D+2}}(z_{1}^{\ast}-z_{2}^{\ast})\cdot n-\frac{\frkc_{3}}{|z_{1}^{\ast}-z_{3}^{\ast}|^{2D+2}}(z_{1}^{\ast}-z_{3}^{\ast})\cdot n & =0,\\
-\frac{\frkc_{2}}{|z_{2}^{\ast}-z_{4}^{\ast}|^{2D+2}}(z_{1}^{\ast}-z_{2}^{\ast})\cdot n+\frac{\frkc_{3}}{|z_{3}^{\ast}-z_{4}^{\ast}|^{2D+2}}(z_{1}^{\ast}-z_{3}^{\ast})\cdot n & =0.
\end{align*}
Viewing this as a $2\times2$ linear system admitting a nonzero solution,
we conclude that
\[
\frkc_{2}\frkc_{3}\Big\{\frac{1}{|z_{1}^{\ast}-z_{2}^{\ast}|^{2D+2}|z_{3}^{\ast}-z_{4}^{\ast}|^{2D+2}}-\frac{1}{|z_{1}^{\ast}-z_{3}^{\ast}|^{2D+2}|z_{2}^{\ast}-z_{4}^{\ast}|^{2D+2}}\Big\}=0.
\]
As $\frkc_{2},\frkc_{3}\neq0$, we get \eqref{eq:0.03}. 

Having established the claim \eqref{eq:0.03}, we are ready to derive
a contradiction. On one hand, the equation $(A^{\ast}\vec{\frkc})_{2}=0$
gives 
\begin{equation}
\frac{\frkc_{1}}{|z_{2}^{\ast}-z_{1}^{\ast}|^{2D}}-\frac{\frkc_{4}}{|z_{2}^{\ast}-z_{4}^{\ast}|^{2D}}=\frac{\frkc_{3}}{|z_{2}^{\ast}-z_{3}^{\ast}|^{2D}}>0.\label{eq:0.01}
\end{equation}
On the other hand, the equation $(A^{\ast}\vec{\frkc})_{3}=0$ gives
\[
\frac{\frkc_{1}}{|z_{3}^{\ast}-z_{1}^{\ast}|^{2D}}-\frac{\frkc_{4}}{|z_{3}^{\ast}-z_{4}^{\ast}|^{2D}}=-\frac{\frkc_{2}}{|z_{3}^{\ast}-z_{2}^{\ast}|^{2D}}.
\]
Multiplying the above by $\frac{|z_{1}^{\ast}-z_{3}^{\ast}|^{2D}}{|z_{1}^{\ast}-z_{2}^{\ast}|^{2D}}$
and using \eqref{eq:0.03}, we get 
\begin{equation}
\frac{\frkc_{1}}{|z_{1}^{\ast}-z_{2}^{\ast}|^{2D}}-\frac{\frkc_{4}}{|z_{2}^{\ast}-z_{4}^{\ast}|^{2D}}=-\frac{\frkc_{2}}{|z_{2}^{\ast}-z_{3}^{\ast}|^{2D}}\frac{|z_{1}^{\ast}-z_{3}^{\ast}|^{2D}}{|z_{1}^{\ast}-z_{2}^{\ast}|^{2D}}<0.\label{eq:0.02}
\end{equation}
Equations \eqref{eq:0.01} and \eqref{eq:0.02} contradict each other. 
\end{proof}
The following two lemmas are not directly used in the proof of our
main results, but these are included due to independent interests.
Lemma~\ref{lem:large-family-degen} says that there is a large family
of degenerate configurations (as a codimension-one manifold of all
configurations, at least when $J=4$). Lemma~\ref{lem:large-dim-kernel}
says that the nonnegative kernel of $A^{\ast}$ (the number of linearly
independent nonnegative kernel elements) can be large.
\begin{lem}[Families of $J=4$ degenerate configurations]
\label{lem:large-family-degen}Let $J=4$ and $\iota_{1}=\iota_{2}=-\iota_{3}=-\iota_{4}$. 
\begin{itemize}
\item (A family including lozenge) For any $d_{3}\in\bbR$, there exists
$d_{4}<d_{3}$ such that the configuration $z_{1}=(-1,0,\dots,0)$,
$z_{2}=(1,0,\dots,0)$, $z_{3}=(0,d_{3},0,\dots,0)$, $z_{4}=(0,d_{4},0,\dots,0)$
becomes degenerate.
\item (Codimension-one existence of degenerate configurations near lozenge)
Let $d>0$ be the unique number such that $d^{2}+1=2^{2+\frac{1}{D}}d$.
Let $z_{1}^{\ast}=(-1,0,\dots,0)$, $z_{2}^{\ast}=(1,0,\dots,0)$,
$z_{3}^{\ast}=(0,d,0,\dots,0)$, and $z_{4}^{\ast}=(0,-d,0,\dots,0)$
form a degenerate configuration. For any $z_{1},z_{2},z_{3},(z_{4}^{1},z_{4}^{3},z_{4}^{4},\dots,z_{4}^{N})$
close to $z_{1}^{\ast},z_{2}^{\ast},z_{3}^{\ast},0$, there exists
$z_{4}^{2}$ close to $-d$ such that $\{z_{1},z_{2},z_{3},z_{4}\}$
becomes degenerate.
\end{itemize}
\end{lem}

\begin{proof}
\uline{Proof of the first item}. With a vector of the form $\vec{c}=(1,1,c_{3},c_{4})$,
consider the equation $A\vec{c}=0$. The equations $(A\vec{c})_{1}=0$,
$(A\vec{c})_{3}=0$, $(A\vec{c})_{4}=0$ give 
\begin{align*}
2^{-2D}-(d_{3}^{2}+1)^{-D}c_{3}-(d_{4}^{2}+1)^{-D}c_{4} & =0,\\
(d_{3}^{2}+1)^{-D}\cdot2-(d_{3}-d_{4})^{-2D}c_{4} & =0,\\
(d_{4}^{2}+1)^{-D}\cdot2-(d_{3}-d_{4})^{-2D}c_{3} & =0,
\end{align*}
respectively. The latter two equations give 
\begin{equation}
\begin{aligned}c_{3} & =2(d_{3}-d_{4})^{2D}(d_{4}^{2}+1)^{-D},\\
c_{4} & =2(d_{3}-d_{4})^{2D}(d_{3}^{2}+1)^{-D}.
\end{aligned}
\label{eq:4.18}
\end{equation}
Substituting these into the first equation gives the relation 
\[
F(d_{3},d_{4})\coloneqq2^{-2D}-4(d_{3}-d_{4})^{2D}(d_{3}^{2}+1)^{-D}(d_{4}^{2}+1)^{-D}=0.
\]
Observe that $\lim_{d_{4}\nearrow d_{3}}F(d_{3},d_{4})=2^{-2D}>0$
and $\lim_{d_{4}\to-\infty}F(d_{3},d_{4})=2^{-2D}-4<0$. Thus there
exists $d_{4}=d_{4}(d_{3})\in(-\infty,d_{3})$ such that $F(d_{3},d_{4})=0$.
With this $d_{4}$, the matrix $A[\vec{z}]$ has a nonnegative kernel
element $\vec{c}=(1,1,c_{3},c_{4})$ with $c_{3},c_{4}$ defined by
\eqref{eq:4.18}, so $\{z_{1},z_{2},z_{3},z_{4}\}$ is degenerate.

\uline{Proof of the second item}. Consider a function $\bfF=(F_{1},F_{2},F_{3},F_{4})\in\bbR^{4}$,
where each $F_{i}$ is defined by 
\[
F_{i}(\vec{z},\vec{c})=\sum_{j\neq i}\iota_{i}\iota_{j}|z_{i}-z_{j}|^{-2D}c_{j}.
\]
Note that $\bfF=0$ if and only if $A[\vec{z}]\vec{c}=0$. Let $\vec{c}^{\ast}=(1,1,d^{D},d^{D})$
so that $A[\vec{z}^{\ast}]\vec{c}^{\ast}=0$. By the implicit function
theorem, it suffices to check the non-singularity of the matrix 
\[
\frac{\rd\bfF}{\rd(c_{1},c_{2},c_{3},z_{4}^{2})}\bigg|_{(\vec{z},\vec{c})=(\vec{z}^{\ast},\vec{c}^{\ast})}.
\]
This matrix is (using $d^{2}+1=2^{2+\frac{1}{D}}d$) 
\[
\begin{pmatrix}0 & 2^{-2D} & -2^{-2D-1}d^{-D} & -2D\cdot2^{-(D+1)(2+\frac{1}{D})}d^{-D}\\
2^{-2D} & 0 & -2^{-2D-1}d^{-D} & -2D\cdot2^{-(D+1)(2+\frac{1}{D})}d^{-D}\\
-2^{-2D-1}d^{-D} & -2^{-2D-1}d^{-D} & 0 & 2D\cdot2^{-2D-1}d^{-D-1}\\
-2^{-2D-1}d^{-D} & -2^{-2D-1}d^{-D} & 2^{-2D}d^{-2D} & -4D\cdot2^{-(D+1)(2+\frac{1}{D})}d^{-D}+2D\cdot2^{-2D-1}d^{-D-1}
\end{pmatrix}.
\]
It has a nonzero determinant: 
\begin{align*}
 & 2D\cdot2^{-8D-4}d^{-5D-1}\det\begin{pmatrix}0 & 2d^{D} & -d^{D} & -2^{-(2+\frac{1}{D})}d\\
2d^{D} & 0 & -d^{D} & -2^{-(2+\frac{1}{D})}d\\
-1 & -1 & 0 & 1\\
-1 & -1 & 2 & -2^{-(1+\frac{1}{D})}d+1
\end{pmatrix}\\
 & =2D\cdot2^{-8D-4}d^{-5D-1}\cdot(-2^{1-\frac{1}{D}}d^{2D+1}-2^{1-\frac{1}{D}}d^{D+1}+8d^{2D})\\
 & =2D\cdot2^{-8D-3-\frac{1}{D}}d^{-4D}\cdot(-d^{D}-1+2^{2+\frac{1}{D}}d^{D-1})\\
 & =2D\cdot2^{-8D-3-\frac{1}{D}}d^{-4D}\cdot(d^{D-2}-1)\neq0.\qedhere
\end{align*}
\end{proof}

\begin{lem}[Multi-dimensional kernel]
\label{lem:large-dim-kernel}Let $1\leq L\leq N-1$ be an integer.
Consider the following $J=4L$-point configuration 
\[
\left|\begin{aligned}\iota_{4k-3} & =\iota_{4k-2}=+1, & \iota_{4k-1} & =\iota_{4k}=-1,\\
z_{4k-3}^{\ast} & =(\tfrac{1}{2},0,\dots,\tfrac{q_{0}}{2},\dots,0), & z_{4k-1}^{\ast} & =(\tfrac{1}{2},0,\dots,-\tfrac{q_{0}}{2},\dots,0),\\
z_{4k-2}^{\ast} & =(-\tfrac{1}{2},0,\dots,\tfrac{q_{0}}{2},\dots,0), & z_{4k}^{\ast} & =(-\tfrac{1}{2},0,\dots,-\tfrac{q_{0}}{2},\dots,0),
\end{aligned}
\right.
\]
where $k\in\llbracket L\rrbracket$, $q_{0}\in(0,\infty)$ is the
unique number satisfying $q_{0}^{-2D}+(q_{0}^{2}+1)^{-D}=1$, and
$\frac{q_{0}}{2}$ is placed at the $k+1$-th coordinate. Then, this
configuration is degenerate and the matrix $A^{\ast}$ has at least
$L$ linearly independent nonnegative kernel elements.
\end{lem}

\begin{proof}
It suffices to show that $(0,\dots,0,1,1,1,1,0,\dots,0)\in\bbR^{4L}$,
where $1$s are placed at $4k-3,\dots,4k$-th coordinates ($1\leq k\leq L$),
is a kernel element of $A^{\ast}$. This is equivalent to checking
that 
\[
\frac{\chf_{i\neq4k-3}}{|z_{i}^{\ast}-z_{4k-3}^{\ast}|^{2D}}+\frac{\chf_{i\neq4k-2}}{|z_{i}^{\ast}-z_{4k-2}^{\ast}|^{2D}}-\frac{\chf_{i\neq4k-1}}{|z_{i}^{\ast}-z_{4k-1}^{\ast}|^{2D}}-\frac{\chf_{i\neq4k}}{|z_{i}^{\ast}-z_{4k}^{\ast}|^{2D}}=0,\qquad\forall i\in\setJ,\ \forall k\in\llbracket L\rrbracket.
\]
By symmetry, we may assume $i=4k_{0}-3$ for some $k_{0}\in\llbracket L\rrbracket$.
By symmetry of indices, we may assume $k_{0}=1$ and $k\in\{1,2\}$.
If $k=1$, this amounts to checking that $1-q_{0}^{-2D}-(q_{0}^{2}+1)^{-D}=0$,
which is the definition of $q_{0}$. If $k=2$, then this follows
from $|z_{1}^{\ast}-z_{5}^{\ast}|=|z_{1}^{\ast}-z_{7}^{\ast}|$ and
$|z_{1}^{\ast}-z_{6}^{\ast}|=|z_{1}^{\ast}-z_{8}^{\ast}|$.
\end{proof}

\section{\label{sec:Dynamical-case-separation}Dynamical case separation}

From now on, we consider a global $\dot{H}^{1}$-solution $u(t)$
as in Assumption~\ref{assumption:sequential-on-param}. Recall the
notation therein. The goal of this section is twofold: (i) to prove
basic properties (including modulation estimates) for $u(t)$ and
(ii) to separate dynamical scenarios of $u(t)$ depending on some
quantities.

First, we introduce two important \emph{time sequences} satisfying
\[
t_{n}<T_{n}(K_{0},\dlt_{2},\alp)\leq\Ttbexit_{n}(\alp)\leq+\infty
\]
for all large $n$. We fix $K_{0}\geq1$ such that $\lmb_{\max,n}\leq\frac{1}{2}K_{0}$
for all $n$. Let $\alp\in(0,\alptube)$ and $\dlt_{2}>0$, where
$\alptube>0$ denotes the constant $\dlt_{J}$ of Lemma~\ref{lem:curve-modulation}.
First, since $\{u(t_{n}),\vec{\iota},\vec{\lmb}_{n},\vec{z}_{n}\}$
is a $W$-bubbling sequence, we have $u(t_{n})\in\calT_{J}(o_{n\to\infty}(1))$.
In particular, for all large $n$, we have a well-defined time 
\begin{equation}
\Ttbexit_{n}(\alp)\coloneqq\{\tau\in[t_{n},+\infty):u(t)\in\calT_{J}(\alp)\}.\label{eq:def-Ttbexit_n}
\end{equation}
Next, since $\alp<\alptube$, we can find a unique curve $(\vec{\iota},\vec{\lmb},\vec{z}):\br{[t_{n},\Ttbexit_{n}(\alp))}\to\br{\calP_{J}(2\alp)}$
according to Lemma~\ref{lem:curve-modulation} (applied with $t_{0}=t_{n}$,
$\vec{\td{\lmb}}=\vec{\lmb}_{n}$, $\vec{\td z}=\vec{z}_{n}$). By
$u(t_{n})\in\calT_{J}(o_{n\to\infty}(1))$ and \eqref{eq:2.17}, we
have $d_{\calP}((\vec{\lmb}(t_{n}),\vec{z}(t_{n})),(\vec{\lmb}_{n},\vec{z}_{n}))=o_{n\to\infty}(1)$.
This combined with our initial assumptions on $\{u(t_{n}),\vec{\iota},\vec{\lmb}_{n},\vec{z}_{n}\}$
implies that 
\begin{align}
 & \{u(t_{n}),\vec{\iota},\vec{\lmb}(t_{n}),\vec{z}(t_{n})\}\text{ is again a \ensuremath{W}-bubbling sequence,}\label{eq:5.2}\\
 & \lmb_{\max}(t_{n})=(1+o_{n\to\infty}(1))\lmb_{\max,n}\leq(\tfrac{1}{2}+o_{n\to\infty}(1))K_{0},\quad\text{and}\label{eq:5.3}\\
 & \vec{z}(t_{n})\to\vec{z}^{\ast}.\label{eq:5.4}
\end{align}
Therefore, for all large $n$, we have a well-defined time 
\begin{equation}
T_{n}(K_{0},\dlt_{2},\alp)\coloneqq\sup\{\tau\in[t_{n},\Ttbexit_{n}(\alp)):\lmb_{\max}(t)<K_{0}\text{ and }|\vec{z}(t)-\vec{z}^{\ast}|<\dlt_{2}\ \forall t\in[t_{n},\tau]\}.\label{eq:def-T_n(K,delta,alpha)}
\end{equation}

\subsection{\label{subsec:Basic-estimates-for-R,U,frkr}Basic estimates for $R$,
$\protect\td U$, and $\protect\frkr_{a,i}$}

Here, we collect more quantitative estimates in the regime $\lmb_{\max}(t)<K_{0}$
and $|\vec{z}(t)-\vec{z}^{\ast}|<\dlt_{2}$, where $\dlt_{2}$ is
sufficiently small. As $K_{0}$ is fixed, implicit constants depending
on $K_{0}$ will be ignored. 
\begin{lem}[$R_{ij}$ in terms of $\vec{\lmb}$]
\label{lem:non-coll-R_ij}We have 
\begin{align}
R^{-1} & \aeq\sqrt{\lmb_{\max}\lmb_{\secmax}},\label{eq:non-coll-R}\\
R_{i}^{-1} & \aeq\chf_{\lmb_{i}=\lmb_{\max}}\sqrt{\lmb_{\secmax}\lmb_{\max}}+\chf_{\lmb_{i}<\lmb_{\max}}\sqrt{\lmb_{i}\lmb_{\max}},\label{eq:non-coll-R_i}\\
R_{ij}^{-1} & \aeq\sqrt{\lmb_{i}\lmb_{j}}\quad\text{if }i\neq j.\label{eq:non-coll-R_ij}
\end{align}
\end{lem}

\begin{proof}
As $z_{1}^{\ast},\dots,z_{J}^{\ast}$ are distinct and $\dlt_{2}$
is small, $|z_{i}-z_{j}|\aeq1$. This together with $\lmb_{\max}<K_{0}\aleq1$
gives $R_{ij}^{-1}\aeq\sqrt{\lmb_{i}\lmb_{j}}$, which is \eqref{eq:non-coll-R_ij}.
The remaining estimates are simple consequences of \eqref{eq:non-coll-R_ij}. 
\end{proof}
\begin{lem}[Estimates for $\td U$ and $\frkr_{a,i}$]
\label{lem:non-coll-tdU-est}We have 
\begin{gather}
\tsum{j\neq i}{}\|f(W_{;i},W_{;j})\|_{L^{(2^{\ast\ast})'}}\aleq\lmb_{i}^{2}\lmb_{\secmax}^{\frac{N-5}{4}}\lmb_{\max}^{\frac{N+1}{4}}\qquad\forall i\in\setJ,\label{eq:non-coll-f(W_i,W_j)-H2dual}\\
\|\td U\|_{\dot{H}^{1}}\aleq\lmb_{\max}^{\frac{N+2}{4}}\lmb_{\secmax}^{\frac{N+2}{4}},\quad\|\td U\|_{\dot{H}^{2}}\aleq\lmb_{\secmax}^{\frac{N-2}{4}}\lmb_{\max}^{\frac{N+2}{4}},\quad\|\lmb_{i}\rd_{z_{i}^{a}}\td U\|_{\dot{H}^{1}}\aleq\lmb_{i}\lmb_{\secmax}^{\frac{N-2}{4}}\lmb_{\max}^{\frac{N+2}{4}},\label{eq:non-coll-tdU-est}\\
\max_{a,i}\frac{|\frkr_{a,i}|}{\lmb_{i}}\aleq\lmb_{\secmax}^{\frac{N-4}{2}}\lmb_{\max}^{\frac{N-2}{2}},\label{eq:non-coll-frkr_a,i-est-0}
\end{gather}
and 
\begin{equation}
\frkr_{a,i}=\sum_{j\neq i}\frac{\lan\calV_{a;i},f'(W_{;i})W_{;j}\ran}{\|\calV_{a}\|_{L^{2}}^{2}}+\calO(\chf_{\lmb_{i}=\lmb_{\max}}\lmb_{i}^{2}\lmb_{\secmax}^{\frac{N-2}{2}}\lmb_{\max}^{\frac{N-2}{2}}+\chf_{\lmb_{i}<\lmb_{\max}}\lmb_{i}^{2}\lmb_{\secmax}^{\frac{N-4}{2}}\lmb_{\max}^{\frac{N}{2}}).\label{eq:non-coll-frkr_a,i-est}
\end{equation}
\end{lem}

\begin{proof}
\uline{Proof of \mbox{\eqref{eq:non-coll-f(W_i,W_j)-H2dual}}}.
To simplify the notation, denote $\lmb_{ij}^{+}\coloneqq\max\{\lmb_{i},\lmb_{j}\}$,
$\lmb_{ij}^{-}\coloneqq\min\{\lmb_{i},\lmb_{j}\}$ for $i\neq j$.
Recall $R_{ij}\aeq\sqrt{\lmb_{i}\lmb_{j}}$ from \eqref{eq:non-coll-R_ij}.
Using \eqref{eq:f(W1,W2)-H2dual-est}, we obtain 
\begin{align*}
\chf_{N=7}\|f(W_{;i},W_{;j})\|_{L^{(2^{\ast\ast})'}} & \aleq(\lmb_{ij}^{-})^{3/4}(\lmb_{ij}^{+})^{1/4}R_{ij}^{-7/2}\aeq(\lmb_{ij}^{-})^{5/2}(\lmb_{ij}^{+})^{2},\\
\chf_{N=8}\|f(W_{;i},W_{;j})\|_{L^{(2^{\ast\ast})'}} & \aleq_{\eps}(\lmb_{ij}^{-})^{1-\eps}(\lmb_{ij}^{+})^{\eps}R_{ij}^{-4}\aeq(\lmb_{ij}^{-})^{3-\eps}(\lmb_{ij}^{+})^{2+\eps},\\
\chf_{N=9}\|f(W_{;i},W_{;j})\|_{L^{(2^{\ast\ast})'}} & \aleq\lmb_{ij}^{-}R_{ij}^{-N/2}\aeq(\lmb_{ij}^{-})^{(N+4)/4}(\lmb_{ij}^{+})^{N/4},
\end{align*}
where $\eps>0$ can be chosen arbitrarily small when $N=8$. It is
easy to check that these are bounded by \eqref{eq:non-coll-f(W_i,W_j)-H2dual}.

\uline{Proof of \mbox{\eqref{eq:non-coll-tdU-est}} for \mbox{$\|\td U\|_{\dot{H}^{1}}$}
and \mbox{$\|\td U\|_{\dot{H}^{2}}$}}. This easily follows from \eqref{eq:tdU-Hdot1-est},
\eqref{eq:tdU-Hdot2-est}, and $\frkp(R_{i})\aleq R_{i}^{-\frac{N+2}{2}}$
(due to this, this $\dot{H}^{2}$-estimate even has a room). 

\uline{Proof of \mbox{\eqref{eq:non-coll-frkr_a,i-est-0}}}. This
follows from \eqref{eq:frkr/lmb-est}: 
\[
\max_{a,i}\frac{|\frkr_{a,i}|}{\lmb_{i}}\aleq\tsum k{}\lmb_{k}^{-1}R_{k}^{-\frac{N+2}{2}}R^{-\frac{N-6}{2}}\aleq\lmb_{\secmax}^{\frac{N-4}{2}}\lmb_{\max}^{\frac{N-2}{2}}.
\]

\uline{Proof of \mbox{\eqref{eq:non-coll-tdU-est}} for \mbox{$\|\lmb_{i}\rd_{z_{i}^{a}}\td U\|_{\dot{H}^{1}}$}}.
This follows from \eqref{eq:tdU-lmb-deriv-est}, \eqref{eq:non-coll-tdU-est}
for $\|\td U\|_{\dot{H}^{2}}$, and \eqref{eq:non-coll-frkr_a,i-est-0}.

\uline{Proof of \mbox{\eqref{eq:non-coll-frkr_a,i-est}}}. By \eqref{eq:2.54-1}
and \eqref{eq:calV-calV-inner-prod}, we have 
\[
\|\calV_{a}\|_{L^{2}}^{2}\frkr_{a,i}=\lan\calV_{a;i},\Psi\ran+\lan\calV_{a;i},f(\calW+\td U)-f(\calW)-f'(W_{;i})\td U\ran+\tsum{b,j}{}\calO(\chf_{i\neq j}R_{ij}^{-(N-4)}\frac{\lmb_{i}}{\lmb_{j}})\frkr_{b,j}.
\]
The first term of the right hand side can be estimated using \eqref{eq:calV-Psi-leading}
with $\eps=N-6>0$: 
\begin{align*}
\lan\calV_{a;i},\Psi\ran & =\tsum{j\neq i}{}\lan\calV_{a;i},f'(W_{;i})W_{;j}\ran+\calO(R^{-(N-4)}R_{i}^{-4})\\
 & =\tsum{j\neq i}{}\lan\calV_{a;i},f'(W_{;i})W_{;j}\ran+\calO(\chf_{\lmb_{i}=\lmb_{\max}}\lmb_{i}^{2}\lmb_{\secmax}^{\frac{N}{2}}\lmb_{\max}^{\frac{N-4}{2}}+\chf_{\lmb_{i}<\lmb_{\max}}\lmb_{i}^{2}\lmb_{\secmax}^{\frac{N-4}{2}}\lmb_{\max}^{\frac{N}{2}}).
\end{align*}
The second term can be estimated using \eqref{eq:2.54-2}, \eqref{eq:tdU-Hdot1-est},
\eqref{eq:tdU-Hdot2-est}: 
\[
|\lan\calV_{a;i},f(\calW+\td U)-f(\calW)-f'(W_{;i})\td U\ran|\aleq R_{i}^{-\frac{N+2}{2}}R^{-\frac{N+2}{2}}+\lmb_{i}^{2}\|\td U\|_{\dot{H}^{2}}^{2}\aleq\lmb_{i}^{2}\lmb_{\secmax}^{\frac{N-2}{2}}\lmb_{\max}^{\frac{N+2}{2}}.
\]
The last term can be estimated using \eqref{eq:non-coll-frkr_a,i-est-0}:
\begin{align*}
\tsum{b,j}{}\calO(\chf_{i\neq j}R_{ij}^{-(N-4)}\frac{\lmb_{i}}{\lmb_{j}})\frkr_{b,j} & =\calO(\lmb_{i}R_{i}^{-(N-4)})\cdot\calO(\max_{b,j}\frac{|\frkr_{b,j}|}{\lmb_{j}})\\
 & =\calO(\chf_{\lmb_{i}=\lmb_{\max}}\lmb_{\secmax}^{\frac{N-4}{2}}\lmb_{i}^{\frac{N-2}{2}}+\chf_{\lmb_{i}<\lmb_{\max}}\lmb_{i}^{\frac{N-2}{2}}\lmb_{\max}^{\frac{N-4}{2}})\cdot\calO(\lmb_{\secmax}^{\frac{N-4}{2}}\lmb_{\max}^{\frac{N-2}{2}})\\
 & =\calO(\chf_{\lmb_{i}=\lmb_{\max}}\lmb_{i}^{2}\lmb_{\secmax}^{N-4}\lmb_{\max}^{N-4}+\chf_{\lmb_{i}<\lmb_{\max}}\lmb_{i}^{2}\lmb_{\secmax}^{N-5}\lmb_{\max}^{N-3}).
\end{align*}
Gathering the above three estimates concludes the proof of \eqref{eq:non-coll-frkr_a,i-est}.
\end{proof}

\subsection{\label{subsec:Basic-properties-for-non-coll-sol}Basic properties
for solutions on $[t_{n},T_{n}(K_{0},\protect\dlt_{2},\protect\alp))$}

In this subsection, let $u(t)$ be a global $\dot{H}^{1}$-solution
as in Assumption~\ref{assumption:sequential-on-param}. Recall the
time sequence $T_{n}(K_{0},\dlt_{2},\alp)$. Here, $K_{0}$ is fixed
and $\dlt_{2}$ and $\alp$ are assumed to be small. Implicit constants
depending on $K_{0}$ will be ignored.

We consider the time variations of $\lmb_{i}$ and $z_{i}$, which
are obtained by differentiating in time the orthogonality conditions
\eqref{eq:modulation-identity-psi}. We begin with rough estimates
of these time variations, which we call rough \emph{modulation estimates}. 
\begin{lem}[Rough modulation estimates]
\label{lem:rough-mod}We have an algebraic identity 
\begin{equation}
\begin{aligned}\rd_{t}\lan\psi_{\ul{;i}},g\ran & =\sum_{b,j}\Big\{\lan\psi_{\ul{;i}},\calV_{b;j}\ran-\lan\psi_{\ul{;i}},\lmb_{j}\rd_{z_{j}^{b}}\td U\ran+\chf_{j=i}\lan\lmb_{i}\rd_{z_{i}^{b}}(\psi_{\ul{;i}}),g\ran\Big\}\frac{(z_{j}^{b})_{t}}{\lmb_{j}}\\
 & \quad+\sum_{b,j}\lan\psi_{\ul{;i}},\calV_{b;j}\ran\frac{\frkr_{b,j}}{\lmb_{j}^{2}}+\lan\psi_{\ul{;i}},\Dlt g+f(U+g)-f(U)\ran
\end{aligned}
\label{eq:modulation-identity-psi}
\end{equation}
for any $\psi\in\dot{H}^{1}$ with $\Lmb\psi,\nabla\psi\in\dot{H}^{1}$.
We have rough modulation estimates 
\begin{equation}
|\lmb_{i,t}|+|z_{i,t}|+\lmb_{i}|\lan\calY_{\ul{;i}},g\ran_{t}|\aleq\|g\|_{\dot{H}^{2}}+\max_{b,j}\frac{|\frkr_{b,j}|}{\lmb_{j}}\aleq\|g\|_{\dot{H}^{2}}+\lmb_{\secmax}^{D-1}\lmb_{\max}^{D}.\label{eq:rough-mod-est}
\end{equation}
\end{lem}

\begin{proof}
\uline{Proof of \mbox{\eqref{eq:modulation-identity-psi}}}. Substituting
$u=U+g$ and the equation \eqref{eq:U-eqn} of $U$ into \eqref{eq:NLH},
we get 
\[
\rd_{t}g=-\rd_{t}U+\sum_{b,j}\frac{\frkr_{b,j}}{\lmb_{j}^{2}}\calV_{b;j}+\Dlt g+f(U+g)-f(U).
\]
Testing against $\psi_{\ul{;i}}$, we get 
\[
\rd_{t}\lan\psi_{\ul{;i}},g\ran=\lan\psi_{\ul{;i}},-\rd_{t}U\ran+\sum_{b,j}\lan\psi_{\ul{;i}},\calV_{b;j}\ran\frac{\frkr_{b,j}}{\lmb_{j}^{2}}+\lan\rd_{t}\psi_{\ul{;i}},g\ran+\lan\psi_{\ul{;i}},\Dlt g+f(U+g)-f(U)\ran.
\]
Substituting into the above the following
\begin{align*}
-\rd_{t}U & =-\sum_{b,j}\frac{(z_{j}^{b})_{t}}{\lmb_{j}}\cdot\lmb_{j}\rd_{z_{j}^{b}}(\calW+\td U)=\sum_{b,j}\frac{(z_{j}^{b})_{t}}{\lmb_{j}}\cdot(\calV_{b;j}-\lmb_{j}\rd_{z_{j}^{b}}\td U),\\
\rd_{t}\psi_{\ul{;i}} & =\sum_{b}\frac{(z_{i}^{b})_{t}}{\lmb_{i}}\cdot\lmb_{i}\rd_{z_{i}^{b}}(\psi_{\ul{;i}}),
\end{align*}
we obtain \eqref{eq:modulation-identity-psi}.

\uline{Proof of \mbox{\eqref{eq:rough-mod-est}} for \mbox{$\lmb_{i,t}$}
and \mbox{$z_{i,t}$}.} Substituting $\psi=\calZ_{a}$ into \eqref{eq:modulation-identity-psi}
and using the orthogonality conditions $\lan\calZ_{a\ul{;i}},g\ran=0$
\eqref{eq:curve-orthog}, we get 
\begin{equation}
\begin{aligned}0 & =\sum_{b,j}\Big\{\lan\calZ_{a\ul{;i}},\calV_{b;j}\ran-\lan\calZ_{a\ul{;i}},\lmb_{j}\rd_{z_{j}^{b}}\td U\ran+\chf_{j=i}\lan\lmb_{i}\rd_{z_{i}^{b}}(\calZ_{a\ul{;i}}),g\ran\Big\}\frac{(z_{j}^{b})_{t}}{\lmb_{j}}\\
 & \quad+\sum_{b,j}\lan\calZ_{a\ul{;i}},\calV_{b;j}\ran\frac{\frkr_{b,j}}{\lmb_{j}^{2}}+\lan\calZ_{a\ul{;i}},\Dlt g+f(U+g)-f(U)\ran.
\end{aligned}
\label{eq:modulation-identity-calZ}
\end{equation}
Dividing by the diagonal $d_{ai}^{(\calZ)}\coloneqq1-\lan\calZ_{a\ul{;i}},\lmb_{i}\rd_{z_{i}^{a}}\td U\ran+\lan\lmb_{i}\rd_{z_{i}^{a}}(\calZ_{a\ul{;i}}),g\ran$,
this can be rewritten as 
\begin{equation}
\frac{(z_{i}^{a})_{t}}{\lmb_{i}}+\sum_{b,j}\calM_{aibj}^{(\calZ)}\frac{(z_{j}^{b})_{t}}{\lmb_{j}}=\wh h_{ai}^{(\calZ)},\label{eq:M^Z-wh-h-identity}
\end{equation}
where 
\begin{align}
\calM_{aibj}^{(\calZ)} & \coloneqq\chf_{(a,i)\neq(b,j)}\frac{1}{d_{ai}^{(\calZ)}}\Big\{\lan\calZ_{a\ul{;i}},\calV_{b;j}\ran-\lan\calZ_{a\ul{;i}},\lmb_{j}\rd_{z_{j}^{b}}\td U\ran+\chf_{j=i}\lan\lmb_{i}\rd_{z_{i}^{b}}(\calZ_{a\ul{;i}}),g\ran\Big\},\label{eq:M^Z-def}\\
\wh h_{ai}^{(\calZ)} & \coloneqq-\frac{1}{d_{ai}^{(\calZ)}}\Big\{\sum_{b,j}\lan\calZ_{a\ul{;i}},\calV_{b;j}\ran\frac{\frkr_{b,j}}{\lmb_{j}^{2}}+\lan\calZ_{a\ul{;i}},\Dlt g+f(U+g)-f(U)\ran\Big\}.\label{eq:wh-h^Z-def}
\end{align}
Note that the division by $d_{ai}^{(\calZ)}$ is justified because
(see \eqref{eq:tdU-lmb-deriv-est} for $\|\lmb_{i}\rd_{z_{i}^{a}}\td U\|_{\dot{H}^{1}}$)
\[
d_{ai}^{(\calZ)}=1+\calO(\|\lmb_{i}\rd_{z_{i}^{a}}\td U\|_{\dot{H}^{1}}+\|g\|_{\dot{H}^{1}})=1+\calO(\alp).
\]

We estimate $\calM_{aibj}^{(\calZ)}$ and $\wh h_{ai}^{(\calZ)}$.
We estimate $\wh h_{ai}^{(\calZ)}$ crudely as 
\[
|\wh h_{ai}^{(\calZ)}|\aleq\frac{1}{\lmb_{i}}\Big(\max_{b,j}\frac{|\frkr_{b,j}|}{\lmb_{j}}+\|g\|_{\dot{H}^{2}}\Big).
\]
Next, we observe that $\calM_{aibj}^{(\calZ)}$ is small; we use \eqref{eq:calZ-calV-inner-prod}
and \eqref{eq:non-coll-tdU-est} to have 
\begin{align*}
|\calM_{aibj}^{(\calZ)}| & \aleq\chf_{(a,i)\neq(b,j)}|\lan\calZ_{a\ul{;i}},\calV_{b;j}\ran|+\|\lmb_{j}\rd_{z_{j}^{b}}\td U\|_{\dot{H}^{1}}+\chf_{j=i}\|g\|_{\dot{H}^{1}},\\
 & \aleq R^{-(N-4)}\lmb_{i}^{-1}\lmb_{j}+\calO(\lmb_{j}R^{-\frac{N-2}{2}})+\chf_{j=i}\calO(\alp)\aleq o_{\alp\to0}(1)\cdot\lmb_{i}^{-1}\lmb_{j}.
\end{align*}
This implies the mapping property of $\calM_{aibj}^{(\calZ)}$: 
\[
\tsum{b,j}{}|\calM_{aibj}^{(\calZ)}|\lmb_{j}^{-1}\aleq o_{\alp\to0}(1)\cdot\lmb_{i}^{-1}.
\]
Therefore, we can solve the system \eqref{eq:M^Z-wh-h-identity} with
the estimate 
\[
\frac{|(z_{i}^{a})_{t}|}{\lmb_{i}}\aleq\frac{1}{\lmb_{i}}(\max_{b,j}\frac{|\frkr_{b,j}|}{\lmb_{j}}+\|g\|_{\dot{H}^{2}}).
\]
Finally, substituting \eqref{eq:non-coll-frkr_a,i-est-0} into the
above gives \eqref{eq:rough-mod-est} for $\lmb_{i,t}$ and $z_{i,t}$.

\uline{Proof of \mbox{\eqref{eq:rough-mod-est}} for \mbox{$\lan\calY_{\ul{;i}},g\ran_{t}$}}.
It suffices to show that the right hand side of \eqref{eq:modulation-identity-psi}
when $\psi=\calY$ is of size $\calO(\lmb_{\max}^{2D-1}+\|g\|_{\dot{H}^{2}})$.
The terms with $(b,j)=(a,i)$ are of size 
\[
\calO(1)\cdot\calO\Big(\Big|\frac{(z_{i}^{a})_{t}}{\lmb_{i}}\Big|+\Big|\frac{\frkr_{a,i}}{\lmb_{i}^{2}}\Big|\Big)=\frac{1}{\lmb_{i}}\cdot\calO(\max_{b,j}\frac{|\frkr_{b,j}|}{\lmb_{j}}+\|g\|_{\dot{H}^{2}}).
\]
The terms with $(b,j)\neq(a,i)$ were essentially estimated in the
previous paragraph with a crude bound 
\[
\sum_{b,j}o_{\alp\to0}(1)\frac{\lmb_{j}}{\lmb_{i}}\cdot\calO\Big(\Big|\frac{(z_{i}^{a})_{t}}{\lmb_{i}}\Big|+\Big|\frac{\frkr_{a,i}}{\lmb_{i}^{2}}\Big|\Big)=\frac{o_{\alp\to0}(1)}{\lmb_{i}}\cdot\calO(\max_{b,j}\frac{|\frkr_{b,j}|}{\lmb_{j}}+\|g\|_{\dot{H}^{2}}).
\]
The last term $\lan\calY_{\ul{;i}},\Dlt g+f(U+g)-f(U)\ran$ was also
essentially estimated in the previous paragraph with a crude bound
\[
\lan\calY_{\ul{;i}},\Dlt g+f(U+g)-f(U)\ran=\calO(\|\calY_{\ul{;i}}\|_{L^{2}}\|g\|_{\dot{H}^{2}})=\calO(\lmb_{i}^{-1}\|g\|_{\dot{H}^{2}}).
\]
This completes the proof of \eqref{eq:rough-mod-est} for $\lan\calY_{\ul{;i}},g\ran_{t}$.
\end{proof}
The following is a useful corollary of the rough modulation estimates.
\begin{lem}
We have 
\begin{equation}
E[u(t)]\to JE[W]\qquad\text{as }t\to+\infty,\label{eq:E(u(t))-go-to-JE(W)}
\end{equation}
and (denoting $T_{n}=T_{n}(K_{0},\dlt_{2},\alp)$) 
\begin{align}
\int_{t_{n}}^{T_{n}}\Big(\|g\|_{\dot{H}^{2}}^{2}+\sum_{a,i}\frac{\frkr_{a,i}^{2}}{\lmb_{i}^{2}}\Big)dt & =o_{n\to\infty}(1),\label{eq:spacetime-control}\\
\sup_{t\in[t_{n},T_{n})}|\lan\calY_{;i}(t),g(t)\ran| & =o_{n\to\infty}(1),\qquad\forall i\in\setJ.\label{eq:a_i-is-o(1)}
\end{align}
\end{lem}

\begin{proof}
\uline{Proof of \mbox{\eqref{eq:E(u(t))-go-to-JE(W)}}}. This follows
from $E[u(t_{n})]\to JE[W]$, $t_{n}\to+\infty$, and monotone decreasing
property of $t\mapsto E[u(t)]$. 

\uline{Proof of \mbox{\eqref{eq:spacetime-control}}}. By the nonlinear
energy identity \eqref{eq:energy-identity}, the boundedness of $\|u(t_{n})\|_{\dot{H}^{1}}$,
and $t_{n}\to+\infty$, we get $\int_{0}^{\infty}\|\Dlt u+f(u)\|_{L^{2}}^{2}dt<+\infty$.
In particular, 
\[
\int_{t_{n}}^{\infty}\|\Dlt u+f(u)\|_{L^{2}}^{2}dt=o_{n\to\infty}(1).
\]
Applying \eqref{eq:coercivity-dissipation} on $[t_{n},T_{n})$ gives
\eqref{eq:spacetime-control}.

\uline{Proof of \mbox{\eqref{eq:a_i-is-o(1)}}}. By \eqref{eq:rough-mod-est}
and $\lmb_{i}^{-1}|\lan\calY_{;i},g\ran|\aleq\|g\|_{\dot{H}^{2}}$,
we have 
\[
|(\lan\calY_{;i},g\ran^{2})_{t}|\aleq\frac{|\lan\calY_{;i},g\ran|}{\lmb_{i}}(\|g\|_{\dot{H}^{2}}+\max_{b,j}\frac{|\frkr_{b,j}|}{\lmb_{j}})\aleq\|g\|_{\dot{H}^{2}}^{2}+\max_{b,j}\frac{\frkr_{b,j}^{2}}{\lmb_{j}^{2}}.
\]
Integrating this and applying $\lan\calY_{;i},g\ran(t_{n})=\calO(\|g(t_{n})\|_{\dot{H}^{1}})=o_{n\to\infty}(1)$
and \eqref{eq:case1-spacetime-est}, we conclude 
\[
\sup_{t\in[t_{n},T_{n})}|\lan\calY_{;i},g\ran(t)|^{2}\leq|\lan\calY_{;i},g\ran(t_{n})|^{2}+\int_{t_{n}}^{T_{n}}\calO\Big(\|g\|_{\dot{H}^{2}}^{2}+\max_{b,j}\frac{\frkr_{b,j}^{2}}{\lmb_{j}^{2}}\Big)dt=o_{n\to\infty}(1).\qedhere
\]
\end{proof}
For classification purposes, rough modulation estimates are not enough.
First, we have not identified any leading term of time variations
of $\lmb_{i}$ and $z_{i}$. A more serious problem is that the error
$\|g\|_{\dot{H}^{2}}/\lmb_{i}$ (of $\lmb_{i,t}/\lmb_{i}$) is far
from being time-integrable, which is necessary to justify that $g$
does not enter the modulation dynamics. We need the following refined
modulation estimates.
\begin{lem}[Refined modulation estimates]
\label{lem:non-coll-ref-mod-est}We have 
\begin{equation}
\begin{aligned}\Big|\frac{\lmb_{i,t}}{\lmb_{i}}+\frac{\frkr_{0,i}}{\lmb_{i}^{2}}-\frac{\rd_{t}\lan[\Lmb W]_{\ul{;i}},g\ran}{\|\Lmb W\|_{L^{2}}^{2}}\Big|+\chf_{a\neq0}\Big|\frac{(z_{i}^{a})_{t}}{\lmb_{i}}+\frac{\frkr_{a,i}}{\lmb_{i}^{2}}\Big|+\Big|\lan\calY_{\ul{;i}},g\ran_{t}-\frac{e_{0}}{\lmb_{i}^{2}}\lan\calY_{\ul{;i}},g\ran\Big|\quad\\
\aleq\|g\|_{\dot{H}^{2}}^{2}+\lmb_{\secmax}^{D-2}\lmb_{\max}^{D-1}\|g\|_{\dot{H}^{2}}+\lmb_{\secmax}^{D-\frac{3}{2}}\lmb_{\max}^{D+\frac{3}{2}}
\end{aligned}
\label{eq:non-coll-refined-mod-est}
\end{equation}
with the estimate 
\begin{equation}
\frac{\lan[\Lmb W]_{\ul{;i}},g\ran}{\|\Lmb W\|_{L^{2}}^{2}}=\calO\Big(\min\{\|g\|_{\dot{H}^{1}},\|g\|_{\dot{H}^{2}}+\lmb_{i}\|g\|_{\dot{H}^{1}}\}\Big).\label{eq:non-coll-refined-mod-est-corr-bound}
\end{equation}
\end{lem}

\begin{proof}
\uline{Proof of \mbox{\eqref{eq:non-coll-refined-mod-est}}}. First,
we claim for $\psi\in\{\calV_{0},\calV_{1},\dots,\calV_{N},\calY\}$
that 
\begin{equation}
\begin{aligned}\rd_{t}\lan\psi_{\ul{;i}},g\ran & =\sum_{b}\lan\psi,\calV_{b}\ran\Big(\frac{(z_{i}^{b})_{t}}{\lmb_{i}}+\frac{\frkr_{b,i}}{\lmb_{i}^{2}}\Big)+\chf_{\psi=\calY}\frac{e_{0}}{\lmb_{i}^{2}}\lan\calY_{\ul{;i}},g\ran\\
 & \quad+\calO(\|g\|_{\dot{H}^{2}}^{2}+\lmb_{\secmax}^{D-2}\lmb_{\max}^{D-1}\|g\|_{\dot{H}^{2}}+\lmb_{\secmax}^{D-\frac{3}{2}}\lmb_{\max}^{D+\frac{3}{2}}).
\end{aligned}
\label{eq:5.24}
\end{equation}
Recall the identity \eqref{eq:modulation-identity-psi}. First, we
use \eqref{eq:non-coll-tdU-est} to have 
\begin{align*}
\lan\psi_{\ul{;i}},\lmb_{j}\rd_{z_{j}^{b}}\td U\ran & =\calO(\|\lmb_{j}\rd_{z_{j}^{b}}\td U\|_{\dot{H}^{1}})=\calO(\lmb_{j}\lmb_{\secmax}^{D/2}\lmb_{\max}^{1+D/2}),\\
\lan\lmb_{i}\rd_{z_{i}^{b}}(\psi_{\ul{;i}}),g\ran & =\calO(\lmb_{i}\|g\|_{\dot{H}^{2}}),
\end{align*}
which together with \eqref{eq:rough-mod-est} and $\max_{b,j}\frac{|\frkr_{b,j}|}{\lmb_{j}}=\calO(\lmb_{\secmax}^{D-1}\lmb_{\max}^{D})$
give 
\begin{equation}
\begin{aligned} & \sum_{b,j}\Big\{-\lan\psi_{\ul{;i}},\lmb_{j}\rd_{z_{j}^{b}}\td U\ran+\chf_{j=i}\lan\lmb_{i}\rd_{z_{i}^{b}}(\psi_{\ul{;i}}),g\ran\Big\}\frac{(z_{j}^{b})_{t}}{\lmb_{j}}\\
 & =\calO(\lmb_{\secmax}^{D/2}\lmb_{\max}^{1+D/2}+\|g\|_{\dot{H}^{2}})\cdot\calO(\|g\|_{\dot{H}^{2}}+\lmb_{\secmax}^{D-1}\lmb_{\max}^{D})=\calO(\|g\|_{\dot{H}^{2}}^{2}+\lmb_{\secmax}^{D}\lmb_{\max}^{D+2}).
\end{aligned}
\label{eq:5.25}
\end{equation}
Next, since $\calL_{W_{;i}}\calV_{a;i}=0$ and $\calL_{W_{;i}}\calY_{;i}=\frac{e_{0}}{\lmb_{i}^{2}}\calY_{;i}$,
we have 
\[
\lan\psi_{\ul{;i}},\Dlt g+f(U+g)-f(U)\ran=\lan\psi_{\ul{;i}},f(U+g)-f(U)-f'(W_{;i})g\ran+\chf_{\psi=\calY}\frac{e_{0}}{\lmb_{i}^{2}}\lan\calY_{\ul{;i}},g\ran.
\]
The first term of the right hand side can be estimated using $|\psi|\aleq W$
and \eqref{eq:f(a,b)-4}: 
\begin{align*}
 & |\lan\psi_{\ul{;i}},f(U+g)-f(U)-f'(W_{;i})g\ran|\\
 & \aleq\lmb_{i}^{-2}\big\|\{\tsum{j\neq i}{}f(W_{;i},W_{;j})+f(W_{;i},\td U)+f(W_{;i},g)\}g\big\|_{L^{1}}\\
 & \aleq\lmb_{i}^{-2}\tsum{j\neq i}{}\|f(W_{;i},W_{;j})\|_{L^{(2^{\ast\ast})'}}\|g\|_{L^{2^{\ast\ast}}}+\|\td U\|_{L^{2^{\ast\ast}}}\|g\|_{L^{2^{\ast\ast}}}+\|g\|_{L^{2^{\ast\ast}}}^{2},
\end{align*}
which combined with \eqref{eq:non-coll-f(W_i,W_j)-H2dual} and \eqref{eq:non-coll-tdU-est}
gives 
\begin{equation}
|\lan\psi_{a\ul{;i}},\Dlt g+f(U+g)-f(U)\ran|\aleq\lmb_{\secmax}^{\frac{N-5}{4}}\lmb_{\max}^{\frac{N+1}{4}}\|g\|_{\dot{H}^{2}}+\|g\|_{\dot{H}^{2}}^{2}\aleq\|g\|_{\dot{H}^{2}}^{2}+\lmb_{\secmax}^{D-\frac{3}{2}}\lmb_{\max}^{D+\frac{3}{2}}.\label{eq:5.26}
\end{equation}
Finally, we use $\chf_{j\neq i}\lan\psi_{\ul{;i}},\calV_{b;j}\ran=\calO(\chf_{j\neq i}\lmb_{i}^{D-2}\lmb_{j}^{D})$,
\eqref{eq:rough-mod-est}, and \eqref{eq:non-coll-frkr_a,i-est-0}
to get 
\begin{equation}
\begin{aligned}\Big|\sum_{b,j}\chf_{j\neq i}\lan\psi_{\ul{;i}},\calV_{b;j}\ran\Big(\frac{(z_{j}^{b})_{t}}{\lmb_{j}}+\frac{\frkr_{b,j}}{\lmb_{j}^{2}}\Big)\Big| & \aleq\tsum{j\neq i}{}\lmb_{i}^{D-2}\lmb_{j}^{D-1}\cdot(\|g\|_{\dot{H}^{2}}+\lmb_{\secmax}^{D-1}\lmb_{\max}^{D})\\
 & \aleq\lmb_{\secmax}^{D-2}\lmb_{\max}^{D-1}\|g\|_{\dot{H}^{2}}+\lmb_{\secmax}^{2D-3}\lmb_{\max}^{2D-1}.
\end{aligned}
\label{eq:5.27}
\end{equation}
Substituting \eqref{eq:5.25}, \eqref{eq:5.26}, and \eqref{eq:5.27}
into \eqref{eq:modulation-identity-psi} concludes the proof of \eqref{eq:5.24}.

Having established the claim \eqref{eq:5.24}, we are now ready to
finish the proof. Substituting $\psi=\calV_{a}=\|\rd_{1}W\|_{L^{2}}^{2}\calZ_{a}$
for $a=1,\dots,N$ and the orthogonality conditions $\lan\calZ_{a\ul{;i}},g\ran=0$
into \eqref{eq:5.24} gives \eqref{eq:non-coll-refined-mod-est} for
$(z_{i}^{a})_{t}$. Substituting $\psi=\calV_{0}=\Lmb W$ into \eqref{eq:5.24}
gives \eqref{eq:non-coll-refined-mod-est} for $\lmb_{i,t}$. Substituting
$\psi=\calY$ into \eqref{eq:5.24} and using $\lan\calY,\calV_{b}\ran=0$
give \eqref{eq:non-coll-refined-mod-est} for $\lan\calY_{\ul{;i}},g\ran_{t}$.
This completes the proof of \eqref{eq:non-coll-refined-mod-est}.

\uline{Proof of \mbox{\eqref{eq:non-coll-refined-mod-est-corr-bound}}}.
For any radius $r\geq1$, observe (the bound saturates when $N=7$)
\begin{align*}
\|[\Lmb W]_{\ul{;i}}g\|_{L^{1}} & \aleq\lmb_{i}\|\chf_{|y|\leq r}|y|^{2}\Lmb W\|_{L_{y}^{2}}\||x|^{-2}g\|_{L^{2}}+\|\chf_{|y|>r}|y|\Lmb W\|_{L_{y}^{2}}\||x|^{-1}g\|_{L^{2}},\\
 & \aleq\lmb_{i}r^{1/2}\|g\|_{\dot{H}^{2}}+r^{-1/2}\|g\|_{\dot{H}^{1}}.
\end{align*}
Taking $r^{1/2}=\lmb_{i}^{-1}$ gives \eqref{eq:non-coll-refined-mod-est-corr-bound}.
\end{proof}

\subsection{\label{subsec:Case-separation}Case separation proposition}

This subsection is independent of the previous subsections. The goal
of this subsection is to separate the dynamical scenarios of our solution
$u(t)$ into three cases, depending on the dynamical behavior of the
quantities 
\[
\frac{\lmb_{\secmax}(t)}{\lmb_{\max}(t)}\qquad\text{and}\qquad\min_{\emptyset\neq\calI\subseteq\setJ}\bigg(\frac{|A_{\calI}^{\ast}\vec{\lmb}_{\calI}^{D}(t)|}{|\vec{\lmb}_{\calI}^{D}(t)|}+\sum_{j\notin\calI}\frac{\lmb_{j}(t)}{\lmb_{\max}(t)}\bigg).
\]
If the first quantity goes to zero for a sequence of times, we will
see the scenario of \emph{$J-1$ bubbles concentrating at distinct
points on top of one common bubble} (Theorem~\ref{thm:main-one-bubble-tower-classification}).
The second quantity is important in view of Lemma~\ref{lem:distance-from-degen};
this quantity being away from zero for all time will lead to the \emph{non-degenerate
regime} (Proposition~\ref{prop:case1-main}). The remaining case
will be collected into a \emph{degenerate regime}, which will be considered
in Proposition~\ref{prop:case3-main} under an additional assumption
on the given configuration. This case separation is stated more precisely
in the following proposition.
\begin{prop}[Case separation]
\label{prop:case-separation}For any $u(t)$ satisfying Assumption~\ref{assumption:sequential-on-param},
one of the following scenarios holds. 
\begin{itemize}
\item \textbf{Case 1}: After passing to a subsequence of $t_{n}$ (still
denoted by $t_{n}$), there exist a sequence $T_{n}'\in(t_{n},+\infty]$,
constants $\dlt_{1}>0$, $\dlt_{2}^{\ast}>0$, $\alp^{\ast}\in(0,\alptube)$,
and a continuous curve $(\vec{\lmb},\vec{z}):\br{[t_{n},T_{n}')}\to\br{\calP_{J}(2\alp^{\ast})}$
such that:
\begin{itemize}
\item (Initial conditions) $\{u(t_{n}),\vec{\iota},\vec{\lmb}(t_{n}),\vec{z}(t_{n})\}$
is a $W$-bubbling sequence, $\vec{z}(t_{n})\to\vec{z}^{\ast}$, and
$\lmb_{\max}(t_{n})\to0$.
\item (Controls on $[t_{n},T_{n}')$) For all $t\in[t_{n},T_{n}')$, we
have $u(t)\in\calT_{J}(\alp^{\ast})$, $|\vec{z}(t)-\vec{z}^{\ast}|<\dlt_{2}^{\ast}$,
$\lmb_{\max}(t)<1$, \eqref{eq:curve-orthog}, and 
\begin{equation}
\min_{\emptyset\neq\calI\subseteq\setJ}\bigg(\frac{|A_{\calI}^{\ast}\vec{\lmb}_{\calI}^{D}(t)|}{|\vec{\lmb}_{\calI}^{D}(t)|}+\sum_{j\notin\calI}\frac{\lmb_{j}(t)}{\lmb_{\max}(t)}\bigg)\geq\dlt_{1}.\label{eq:5.28}
\end{equation}
\item (Exit conditions) If $T_{n}'<+\infty$, then either $u(T_{n}')\notin\calT_{J}(\alp^{\ast})$,
$|\vec{z}(T_{n}')-\vec{z}^{\ast}|\geq\dlt_{2}^{\ast}$, or $\lmb_{\max}(T_{n}')\geq1$.
\end{itemize}
\item \textbf{Case 2}: There exist $j_{0}\in\setJ$ and sequences $t_{n}'\to+\infty$,
$\vec{\lmb}_{n}'\in(0,\infty)^{J}$, and $\vec{z}_{n}'\in(\bbR^{N})^{J}$
such that $\{u(t_{n}'),\vec{\iota},\vec{\lmb}_{n}',\vec{z}_{n}'\}$
is a $W$-bubbling sequence and 
\begin{align*}
\vec{z}_{n} & \to\vec{z}^{\ast},\\
1 & \ageq\lmb_{i_{0},n}\gg\lmb_{i,n}\quad\text{for all }i\neq i_{0}.
\end{align*}
\item \textbf{Case 3}: There exist two time sequences $t_{n}'<T_{n}'\leq+\infty$
with $t_{n}'\to+\infty$, constants $\dlt_{1}>0$, $\dlt_{2}^{\ast}>0$,
$\alp^{\ast}\in(0,\alptube)$, and a continuous curve $(\vec{\lmb},\vec{z}):\br{[t_{n}',T_{n}')}\to\br{\calP_{J}(2\alp^{\ast})}$
such that:
\begin{itemize}
\item (Initial conditions) $\{u(t_{n}'),\vec{\iota},\vec{\lmb}(t_{n}'),\vec{z}(t_{n}')\}$
is a $W$-bubbling sequence, $\vec{z}(t_{n}')\to\vec{z}^{\ast}$,
$\lmb_{\max}(t_{n}')\to0$, and there exists a subset $\calI\subseteq\setJ$
with $|\calI|\geq2$ such that 
\begin{align*}
\lmb_{j}(t_{n}')\aeq\lmb_{\max}(t_{n}') & \quad\text{for all }j\in\calI,\\
\lmb_{j}(t_{n}')\ll\lmb_{\max}(t_{n}') & \quad\text{for all }j\notin\calI,\\
|A_{\calI}^{\ast}\vec{\lmb}_{\calI}^{D}(t_{n}')| & \ll\lmb_{\max}^{D}(t_{n}'),
\end{align*}
and the configuration $\{(\iota_{j},z_{j}^{\ast}):j\in\calI\}$ is
degenerate.
\item (Controls on $[t_{n}',T_{n}')$) For all $t\in[t_{n}',T_{n}')$, we
have $u(t)\in\calT_{J}(\alp^{\ast})$, $|\vec{z}(t)-\vec{z}^{\ast}|<\dlt_{2}^{\ast}$,
$\lmb_{\max}(t)<1$, \eqref{eq:curve-orthog}, and 
\begin{equation}
\frac{\lmb_{\secmax}(t)}{\lmb_{\max}(t)}\geq\dlt_{1}.\label{eq:5.29}
\end{equation}
\item (Exit conditions) If $T_{n}'<+\infty$, then either $u(T_{n}')\notin\calT_{J}(\alp^{\ast})$,
$|\vec{z}(T_{n}')-\vec{z}^{\ast}|\geq\dlt_{2}^{\ast}$, or $\lmb_{\max}(T_{n}')\geq1$.
\end{itemize}
\end{itemize}
\end{prop}

\begin{proof}
\uline{Step 1: Statement for case separation}. Consider the following
sentence 
\begin{equation}
\begin{aligned} & \exists\dlt_{1}>0,\ \exists\dlt_{2}^{\ast}>0,\ \exists\alp^{\ast}\in(0,\alptube),\ \exists\text{a subsequence of }t_{n}\text{ such that}\\
 & \inf_{t\in[t_{n},T_{n}(K_{0},\dlt_{2}^{\ast},\alp^{\ast}))}\min_{\emptyset\neq\calI\subseteq\setJ}\bigg(\frac{|A_{\calI}^{\ast}\vec{\lmb}_{\calI}^{D}(t)|}{|\vec{\lmb}_{\calI}^{D}(t)|}+\sum_{j\notin\calI}\frac{\lmb_{j}(t)}{\lmb_{\max}(t)}\bigg)\geq\dlt_{1}.
\end{aligned}
\label{eq:statement}
\end{equation}
In this step, we show that \eqref{eq:statement} implies \textbf{Case~1}.
By considering $\calI=\{j\}$ for all $j\in\setJ$ in the last line
of \eqref{eq:statement}, we obtain 
\[
\inf_{t\in[t_{n},T_{n}(K_{0},\dlt_{2},\alp))}\frac{\lmb_{\secmax}(t)}{\lmb_{\max}(t)}\ageq\dlt_{1}.
\]
This together with $(\vec{\iota},\vec{\lmb}(t_{n}),\vec{z}(t_{n}))\in\calP_{J}(o_{n\to\infty}(1))$
implies $\lmb_{\max}(t_{n})=o_{n\to\infty}(1)$. Now, we choose $T_{n}'\coloneqq T_{n}(1,\dlt_{2}^{\ast},\alp^{\ast})\in(t_{n},T_{n}(K_{0},\dlt_{2}^{\ast},\alp^{\ast})]$,
which is well-defined for all large $n$. By the definition of $T_{n}'$,
the first three controls on $[t_{n},T_{n}')$ and the exit conditions
at $t=T_{n}'$ when $T_{n}'<+\infty$ are immediate. Finally, \eqref{eq:5.28}
follows from \eqref{eq:statement}.

\uline{Step 2: Negation of \mbox{\eqref{eq:statement}}, choice
of \mbox{$\calI$}, and another time sequence \mbox{$t_{n}'$}}. Taking
the negation of \eqref{eq:statement} and using the fact that $u(t_{n})\in\calT_{J}(o_{n\to\infty}(1))$,
$\vec{z}(t_{n})\to\vec{z}^{\ast}$, and $\lmb_{\max}(t_{n})<K_{0}$,
we get 
\begin{equation}
\begin{aligned} & \forall\dlt_{1}>0,\ \forall\dlt_{2}>0,\ \forall\alp\in(0,\alptube),\ \exists n_{0}\in\bbN\ \forall n\geq n_{0},\\
 & u(t_{n})\in\calT_{J}(\alp),\ |\vec{z}(t_{n})-\vec{z}^{\ast}|<\dlt_{2},\ \lmb_{\max}(t_{n})<K_{0},\ \text{and}\\
 & \inf_{t\in[t_{n},T_{n}(K_{0},\dlt_{2},\alp))}\min_{\emptyset\neq\calI\subseteq\setJ}\bigg(\frac{|A_{\calI}^{\ast}\vec{\lmb}_{\calI}^{D}(t)|}{|\vec{\lmb}_{\calI}^{D}(t)|}+\sum_{j\notin\calI}\frac{\lmb_{j}(t)}{\lmb_{\max}(t)}\bigg)<\dlt_{1}.
\end{aligned}
\label{eq:negation-case1}
\end{equation}
Under \eqref{eq:negation-case1}, we ultimately need to prove the
dichotomy: \textbf{Case~2} or \textbf{Case~3}.

Choose any sequences $\dlt_{1,n}\searrow0$, $\dlt_{2,n}\searrow0$,
and $\alp_{n}\searrow0$. By \eqref{eq:negation-case1}, we can find
a subsequence of $t_{n}$ such that $u(t_{n})\in\calT_{J}(\alp_{n})$,
$|\vec{z}(t_{n})-\vec{z}^{\ast}|<\dlt_{2,n}$, $\lmb_{\max}(t_{n})<K_{0}$,
and 
\[
\inf_{t\in[t_{n},T_{n}(K_{0},\dlt_{2,n},\alp_{n}))}\min_{\emptyset\neq\calI\subseteq\setJ}\bigg(\frac{|A_{\calI}^{\ast}\vec{\lmb}_{\calI}^{D}(t)|}{|\vec{\lmb}_{\calI}^{D}(t)|}+\sum_{j\notin\calI}\frac{\lmb_{j}(t)}{\lmb_{\max}(t)}\bigg)<\dlt_{1,n}.
\]
Passing to a further subsequence of $t_{n}$, there exist a subset
$\emptyset\neq\calI\subseteq\setJ$ and another sequence $t_{n}'\in[t_{n},T_{n}(K_{0},\dlt_{2,n},\alp_{n}))$
such that 
\[
\frac{|A_{\calI}^{\ast}\vec{\lmb}_{\calI}^{D}(t_{n}')|}{|\vec{\lmb}_{\calI}^{D}(t_{n}')|}+\sum_{j\notin\calI}\frac{\lmb_{j}(t_{n}')}{\lmb_{\max}(t_{n}')}<\dlt_{1,n}.
\]

Therefore, we can choose a \emph{minimal subset $\emptyset\neq\calI\subseteq\setJ$}
with the following property: there exist a subsequence of $t_{n}$
and sequences $\dlt_{2,n}\searrow0$, and $\alp_{n}\searrow0$ such
that $u(t_{n})\in\calT_{J}(\alp_{n})$, $|\vec{z}(t_{n})-\vec{z}^{\ast}|<\dlt_{2,n}$,
$\lmb_{\max}(t_{n})<K_{0}$, and there is another sequence $t_{n}'\in[t_{n},T_{n}(K_{0},\dlt_{2,n},\alp_{n}))$
satisfying 
\begin{equation}
\frac{|A_{\calI}^{\ast}\vec{\lmb}_{\calI}^{D}(t_{n}')|}{|\vec{\lmb}_{\calI}^{D}(t_{n}')|}+\sum_{j\notin\calI}\frac{\lmb_{j}(t_{n}')}{\lmb_{\max}(t_{n}')}\to0.\label{eq:after-diagonal-lmb}
\end{equation}

First, observe that $\{u(t_{n}'),\vec{\iota},\vec{\lmb}(t_{n}'),\vec{z}(t_{n}')\}$
is a $W$-bubbling sequence due to $u(t_{n}')\in\calT_{J}(\alp_{n})$,
$(\vec{\iota},\vec{\lmb}(t_{n}'),\vec{z}(t_{n}'))\in\calP_{J}(2\alp_{n})$,
and $\alp_{n}\searrow0$. Next, $\lmb_{\max}(t_{n}')<K_{0}$. Finally,
by $|\vec{z}_{n}'-\vec{z}^{\ast}|<\dlt_{2,n}\searrow0$, we have $\vec{z}_{n}'\to\vec{z}^{\ast}$.

Passing to a further subsequence of $t_{n}$, for all $j\in\setJ$,
the ratios $\ell_{j,n}\coloneqq\lmb_{j}(t_{n}')/\lmb_{\max}(t_{n}')$
converge to some $\ell_{j}^{\ast}\in[0,1]$. We claim that 
\begin{equation}
A_{\calI}^{\ast}\vec{\ell}_{\calI}^{\ast D}=0\qquad\text{and}\qquad\ell_{j}^{\ast}>0\impmi j\in\calI.\label{eq:4.6}
\end{equation}
Indeed, \eqref{eq:after-diagonal-lmb} gives $A_{\calI}^{\ast}\vec{\ell}_{\calI}^{\ast D}=0$
and $\ell_{j}^{\ast}=0$ for all $j\notin\calI$. To show $\ell_{j}^{\ast}>0$
for $j\in\calI$, suppose not. Then, the set $\wh{\calI}\coloneqq\{j\in\calI:\ell_{j}^{\ast}>0\}$
is a proper subset of $\calI$. Note that $\ell_{j}^{\ast}=0$ for
$j\notin\wh{\calI}$. This together with $A_{\calI}^{\ast}\vec{\ell}_{\calI}^{\ast D}=0$
gives $A_{\wh{\calI}}^{\ast}\vec{\ell}_{\wh{\calI}}^{\ast D}=0$ also.
This contradicts to the minimality of $\calI$. Hence, the claim \eqref{eq:4.6}
is proved.

\uline{Step 3: Dichotomy under \mbox{\eqref{eq:negation-case1}}}.
We separate the cases to obtain a dichotomy (i.e., \textbf{Case~2}
or \textbf{Case~3} of this proposition): (\textbf{Case~A}) $|\calI|=1$
or (\textbf{Case~B}) $|\calI|\geq2$. By Step~2, \textbf{Case~A}
implies \textbf{Case~2} by setting $\lmb_{n}'=\vec{\lmb}(t_{n}')$
and $z_{n}'=\vec{z}(t_{n}')$.

We turn to show that \textbf{Case~B} belongs to \textbf{Case~3}.
First, we check the initial conditions of \textbf{Case~3}. By Step~2,
it suffices to show that $\lmb_{\max}(t_{n}')=o_{n\to\infty}(1)$
and the configuration $\{(\iota_{j},z_{j}^{\ast}):j\in\calI\}$ is
degenerate. The latter fact easily follows from \eqref{eq:after-diagonal-lmb}
and Lemma~\ref{lem:non-degen}. The former follows from $\lmb_{j}(t_{n}')\aeq\lmb_{\max}(t_{n}')$
for $j\in\calI$, $|\calI|\geq2$, and $(\vec{\iota},\vec{\lmb}(t_{n}'),\vec{z}(t_{n}'))\in\calP_{J}(o_{n\to\infty}(1))$.
Thus the initial conditions of \textbf{Case~3} are satisfied. By
the initial conditions at $t=t_{n}'$, for any $\dlt_{2}>0$ and $\alp\in(0,\alptube)$,
for all large $n$, we have a well-defined time 
\[
T_{n}'(\dlt_{2},\alp)\coloneqq\sup\{\tau\in[t_{n}',\Ttbexit_{n}(\alp)):\lmb_{\max}(t)<1\text{ and }|\vec{z}(t)-\vec{z}^{\ast}|<\dlt_{2}\text{ for all }t\in[t_{n}',\tau]\}.
\]

We claim that there exist $\dlt_{1}>0$, $\dlt_{2}^{\ast}>0$, $\alp^{\ast}\in(0,\alptube)$,
and a subsequence of $t_{n}'$ such that $\lmb_{\secmax}(t)\geq\dlt_{1}\lmb_{\max}(t)$
for all $t\in[t_{n}',T_{n}'(\dlt_{2}^{\ast},\alp^{\ast}))$; we show
this using $|\calI|\geq2$ and the minimality of $\calI$. Suppose
not. Then, for any $\dlt_{1}>0$, $\dlt_{2}>0$, and $\alp\in(0,\alptube)$,
$\inf_{t\in[t_{n}',T_{n}'(\dlt_{2},\alp))}\lmb_{\secmax}(t)/\lmb_{\max}(t)<\dlt_{1}$
for all large $n$. Then, we can find $\dlt_{2,n}''\searrow0$, $\alp_{n}''\searrow0$,
and $t_{n}''\in[t_{n},T_{n}'(\dlt_{2,n}'',\alp_{n}''))$ such that
$\lmb_{\secmax}(t_{n}'')/\lmb_{\max}(t_{n}'')\to0$. Passing to a
further subsequence, we may assume $\lmb_{j_{0}}(t_{n}'')\aeq\lmb_{\max}(t_{n}'')$
for some $j_{0}\in\setJ$, so \eqref{eq:after-diagonal-lmb} with
$\calI=\{j_{0}\}$ is satisfied at times $t_{n}''$. Since $t_{n}''\in[t_{n},T_{n}'(\dlt_{2,n}'',\alp_{n}''))$,
we have $u(t_{n}'')\in\calT_{J}(\alp_{n}'')$, $|\vec{z}(t_{n}'')-\vec{z}^{\ast}|<\dlt_{2,n}''$,
and $\lmb_{\max}(t_{n}'')<1$, so this contradicts to the minimality
of $\calI$ and $|\calI|\geq2$. This proves the claim.

Choose $\dlt_{1},\dlt_{2}^{\ast},\alp^{\ast}$ as in the previous
claim and define $T_{n}'\coloneqq T_{n}'(\dlt_{2}^{\ast},\alp^{\ast})$.
The definition $T_{n}'=T_{n}'(\dlt_{2}^{\ast},\alp^{\ast})$ gives
the first three controls on $[t_{n}',T_{n}')$ and the exit conditions
at $t=T_{n}'$ when $T_{n}'<+\infty$. The previous claim gives the
last control $\lmb_{\secmax}(t)\geq\dlt_{1}\lmb_{\max}(t)$ for all
$t\in[t_{n}',T_{n}')$. This completes the proof.
\end{proof}

\section{\label{sec:Case2}Proof of Theorem~\ref{thm:main-one-bubble-tower-classification}}

This section is devoted to the proof of Theorem~\ref{thm:main-one-bubble-tower-classification}.
This in particular treats \textbf{Case~2} of Proposition~\ref{prop:case-separation}.
By \eqref{eq:main1-lmb_max>>lmb_secmax}, passing to a subsequence,
we may fix $i_{0}\in\setJ$ such that $\lmb_{i_{0},n}=\lmb_{\max,n}$
and 
\[
\frac{\sum_{i\neq i_{0}}\lmb_{i,n}}{\lmb_{i_{0},n}}=o_{n\to\infty}(1).
\]
Define the sets 
\[
\calI_{+}\coloneqq\{i\in\setJ:\iota_{i}=\iota_{i_{0}}\}\ni i_{0}\quad\text{and}\quad\calI_{-}\coloneqq\{i\in\setJ:\iota_{i}=-\iota_{i_{0}}\}.
\]
We will ultimately prove $\calI_{+}=\{i_{0}\}$, but this is not known
at this point.

Fix $K_{0}>0$ such that $\lmb_{i_{0},n}\leq\tfrac{1}{2}K_{0}$ for
all $n$. Since $K_{0}$ is fixed, any implicit constants depending
on $K_{0}$ will be ignored. Let $0<\dlt_{0}\ll1$ be a small parameter
that can shrink in the course of the proof. For all large $n$ depending
on $\dlt_{0}$, we define three sequences of times
\[
t_{n}<T_{n}'\leq T_{n}\leq\Ttbexit_{n}\leq+\infty
\]
as follows. We define $\Ttbexit_{n}\coloneqq\Ttbexit_{n}(\dlt_{0})$
(see \eqref{eq:def-Ttbexit_n}), decompose $u(t)$ for $t\in[t_{n},\Ttbexit_{n})$
as in Lemma~\ref{lem:curve-modulation} (cf.~beginning of Section~\ref{sec:Dynamical-case-separation}),
define $T_{n}\coloneqq T_{n}(K_{0},\dlt_{0})$ by
\[
T_{n}\coloneqq\sup\{\tau\in[t_{n},\Ttbexit_{n}):\lmb_{\max}(t)<K_{0},\ |\vec{z}(t)-\vec{z}^{\ast}|<\dlt_{0},\text{ and }\frac{\sum_{i\in\calI_{-}}\lmb_{i}(t)}{\sum_{i\in\calI_{+}}\lmb_{i}(t)}<\dlt_{0}\ \forall t\in[t_{n},\tau]\},
\]
and finally define $T_{n}'$ by 
\[
T_{n}'\coloneqq\sup\{\tau\in[t_{n},T_{n}):\frac{\sum_{i\neq i_{0}}\lmb_{i}(t)}{\lmb_{i_{0}}(t)}<(\dlt_{0})^{\frac{1}{100}}\text{ for all }t\in[t_{n},\tau]\}.
\]
This $T_{n}'$ is an intermediate time up to which $\lmb_{i_{0}}(t)=\lmb_{\max}(t)$
is ensured. If $\calI_{+}=\{i_{0}\}$, then $T_{n}'=T_{n}$ and $T_{n}'$
is not necessary. However, we introduced $T_{n}'$ because $\calI_{+}$
is not assumed to be $\{i_{0}\}$ at the beginning.

\subsection{Basic observations on $[t_{n},T_{n})$}

We begin with some basic observations on solutions $u(t)$ on $[t_{n},T_{n})$.
The estimates proved in Section~\ref{subsec:Basic-estimates-for-R,U,frkr}
and Section~\ref{subsec:Basic-properties-for-non-coll-sol} are crucial.
\begin{lem}
We have 
\begin{equation}
\max_{a,i}\frac{|\frkr_{a,i}|}{\lmb_{i}}\aeq\lmb_{\secmax}^{D-1}\lmb_{\max}^{D}.\label{eq:case2-max-frkr_a,i}
\end{equation}
In particular, we have a rough modulation estimate 
\begin{equation}
|\lmb_{i,t}|+|z_{i,t}|+\lmb_{i}|a_{i,t}|\aleq\|g\|_{\dot{H}^{2}}+\lmb_{\secmax}^{D-1}\lmb_{\max}^{D},\label{eq:case2-rough-mod-est}
\end{equation}
and a spacetime estimate 
\begin{equation}
\int_{t_{n}}^{T_{n}}(\|g\|_{\dot{H}^{2}}^{2}+\lmb_{\secmax}^{2D-2}\lmb_{\max}^{2D})dt=o_{n\to\infty}(1).\label{eq:case2-spacetime-est}
\end{equation}
\end{lem}

\begin{proof}
\uline{Proof of \mbox{\eqref{eq:case2-max-frkr_a,i}}}. As the
$(\aleq)$-inequality is proved in \eqref{eq:non-coll-frkr_a,i-est-0},
we prove the $(\ageq)$-inequality. 

For a small constant $\gmm>0$ to be chosen, assume $\lmb_{\secmax}<\gmm\lmb_{\max}$.
Choose distinct $i,j_{0}\in\setJ$ such that $\lmb_{i}=\lmb_{\secmax}$
and $\lmb_{j_{0}}=\lmb_{\max}$. (It is not necessary that $j_{0}=i_{0}$.)
By \eqref{eq:2.48}, we have 
\begin{align*}
|\lan[\Lmb W]_{;i},f'(W_{;i})W_{;j_{0}}\ran| & \aeq\lmb_{i}^{D}\lmb_{j_{0}}^{D}\aeq\lmb_{\secmax}^{D}\lmb_{\max}^{D},\\
\chf_{j\neq i,j_{0}}|\lan[\Lmb W]_{;i},f'(W_{;i})W_{;j}\ran| & \aleq\chf_{j\neq i,j_{0}}\lmb_{i}^{D}\lmb_{j}^{D}\aleq\lmb_{\secmax}^{2D}\aleq\gmm^{D}\lmb_{\secmax}^{D}\lmb_{\max}^{D}.
\end{align*}
Choosing $\gmm$ small, \eqref{eq:non-coll-frkr_a,i-est} gives $\lmb_{i}^{-1}|\frkr_{0,i}|\aeq\lmb_{\secmax}^{D-1}\lmb_{\max}^{D}$.
This concludes \eqref{eq:case2-max-frkr_a,i} when $\lmb_{\secmax}<\gmm\lmb_{\max}$.

Now assume $\lmb_{\secmax}\geq\gmm\lmb_{\max}$. As $\gmm$ was fixed
in the previous paragraph, we ignore any dependence on $\gmm$ from
now on. Choose $i\in\setJ$ such that $\lmb_{i}=\lmb_{\max}$. Observe
also that $(\vec{\iota},\vec{\lmb},\vec{z})\in\calP_{J}(2\dlt_{0})$
and $\lmb_{\secmax}\geq\gmm\lmb_{\max}$ imply the smallness of $\lmb_{\max}$:
\[
\lmb_{\max}\aleq\sqrt{\lmb_{\max}\lmb_{\secmax}}\aeq R^{-1}\aleq\dlt_{0}.
\]
Now, using \eqref{eq:2.48} with $\Lmb W(y)=-\frac{N-2}{2}\kpp_{\infty}|y|^{-2D}+\calO(|y|^{-2D-2})$,
we have for $j\neq i$ 
\begin{align*}
 & \lan[\Lmb W]_{;i},f'(W_{;i})W_{;j}\ran\\
 & =\{-\iota_{i}\iota_{j}\kpp_{\infty}\tfrac{N-2}{2}\tint{}{}f(W)|z_{j}-z_{i}|^{-2D}\lmb_{i}^{D}\lmb_{j}^{D}+\calO(\lmb_{i}^{D+2}\lmb_{j}^{D})\}+\calO(R_{ij}^{-N}(1+\log R_{ij}))\\
 & =-\iota_{i}\iota_{j}\kpp_{\infty}\tfrac{N-2}{2}\tint{}{}f(W)|z_{j}-z_{i}|^{-2D}\lmb_{i}^{D}\lmb_{j}^{D}+\calO(\dlt_{0}^{2}|\log\dlt_{0}|\cdot\lmb_{i}^{D}\lmb_{j}^{D}).
\end{align*}
Now, observe (for $j\in\calI_{-}$, we use $\lmb_{j}\aleq\dlt_{0}\lmb_{\max}$
by the bootstrap hypothesis on $[t_{n},T_{n})$) 
\begin{align*}
\chf_{j\in\calI_{-}}|\lan[\Lmb W]_{;i},f'(W_{;i})W_{;j}\ran| & \aleq\lmb_{i}^{D}\lmb_{j}^{D}\aleq\dlt_{0}^{D}\lmb_{\max}^{2D},\\
-\chf_{j\in\calI_{+}\setminus\{i\}}\lan[\Lmb W]_{;i},f'(W_{;i})W_{;j}\ran & \geq\{\kpp_{\infty}\tfrac{N-2}{2}\tint{}{}f(W)|z_{j}-z_{i}|^{-2D}-\calO(\dlt_{0}^{2}|\log\dlt_{0}|)\}\lmb_{i}^{D}\lmb_{j}^{D}.
\end{align*}
Note that $\lmb_{\secmax}\geq\gmm\lmb_{\max}$ and $\chf_{j\in\calI_{-}}\lmb_{j}\aleq\dlt_{0}\lmb_{\max}$
(recall $\dlt_{0}$ is a small parameter) implies that $\lmb_{\secmax}=\lmb_{k}$
for some $k\in\calI_{+}\setminus\{i\}$. Thus the previous display
implies 
\begin{align*}
 & -\tsum{j\neq i}{}\lan[\Lmb W]_{;i},f'(W_{;i})W_{;j}\ran\\
 & \geq\{\kpp_{\infty}\tfrac{N-2}{2}\tint{}{}f(W)|z_{k}-z_{i}|^{-2D}-\calO(\dlt_{0}^{2}|\log\dlt_{0}|)\}\lmb_{i}^{D}\lmb_{k}^{D}-\calO(\dlt_{0}^{D}\lmb_{\max}^{2D})\aeq\lmb_{\secmax}^{D}\lmb_{\max}^{D}.
\end{align*}
This combined with \eqref{eq:non-coll-frkr_a,i-est} and $\lmb_{\secmax}\aeq\lmb_{\max}$
gives $\lmb_{i}^{-1}|\frkr_{0,i}|\aeq\lmb_{\secmax}^{D-1}\lmb_{\max}^{D}$.
This concludes \eqref{eq:case2-max-frkr_a,i} when $\lmb_{\secmax}\geq\gmm\lmb_{\max}$.
This completes the proof of \eqref{eq:case2-max-frkr_a,i}.

\uline{Proof of \mbox{\eqref{eq:case2-rough-mod-est}}}. This is
a restatement of \eqref{eq:rough-mod-est}.

\uline{Proof of \mbox{\eqref{eq:case2-spacetime-est}}}. This follows
from \eqref{eq:spacetime-control} and \eqref{eq:case2-max-frkr_a,i}.
\end{proof}

\subsection{Modulation analysis and closing bootstrap on $[t_{n},T_{n}')$}

In this section, we analyze $u(t)$ on the interval $[t_{n},T_{n}')$.
The definition of $T_{n}'$ ensures $\lmb_{i_{0}}=\lmb_{\max}$ and
$\lmb_{\secmax}<(\dlt_{0})^{\frac{1}{100}}\lmb_{\max}$ on $[t_{n},T_{n}')$.
The goal of this subsection is to close the bootstrap hypotheses used
in the definition of $T_{n}$ on $[t_{n},T_{n}')$. (We will prove
$T_{n}=T_{n}'$ in the next subsection.)
\begin{lem}[Smallness of $\|g(t)\|_{\dot{H}^{1}}+R^{-1}(t)$]
We have 
\begin{align}
\sup_{t\in[t_{n},T_{n}')}R^{-1}(t) & =o_{n\to\infty}(1),\label{eq:case2-Rinv-is-o(1)}\\
\sup_{t\in[t_{n},T_{n}')}\|g(t)\|_{\dot{H}^{1}} & =o_{n\to\infty}(1).\label{eq:case2-gHdot1-is-o(1)}
\end{align}
\end{lem}

\begin{proof}
\uline{Proof of \mbox{\eqref{eq:case2-Rinv-is-o(1)}}}. Since $R^{-1}\aleq\sqrt{\lmb_{\secmax}\lmb_{\max}}\aleq\sqrt{\lmb_{\secmax}}$,
it suffices to show $\lmb_{\secmax}=o_{n\to\infty}(1)$. For any $i\neq i_{0}$,
we use \eqref{eq:case2-rough-mod-est} to have 
\[
|(\lmb_{i}^{2D})_{t}|\aleq\lmb_{i}^{2D-1}|\lmb_{i,t}|\aleq\lmb_{i}^{2D-1}(\|g\|_{\dot{H}^{2}}+\lmb_{\secmax}^{D-1}\lmb_{\max}^{D})\aleq\|g\|_{\dot{H}^{2}}^{2}+\lmb_{\secmax}^{2D-2}\lmb_{\max}^{2D}.
\]
Integrating this with $\lmb_{\secmax}(t_{n})=o_{n\to\infty}(1)$ and
\eqref{eq:case2-spacetime-est}, we get 
\[
\sup_{t\in[t_{n},T_{n}')}\lmb_{i}^{2D}(t)\aleq\lmb_{i}^{2D}(t_{n})+\int_{t_{n}}^{T_{n}'}\calO(\|g\|_{\dot{H}^{2}}^{2}+\lmb_{\secmax}^{2D-2}\lmb_{\max}^{2D})dt=o_{n\to\infty}(1).
\]
As $i\neq i_{0}$ was arbitrary, this completes the proof of $\lmb_{\secmax}=o_{n\to\infty}(1)$
and hence the proof of \eqref{eq:case2-Rinv-is-o(1)}.

\uline{Proof of \mbox{\eqref{eq:case2-gHdot1-is-o(1)}}}. This
is a consequence of \eqref{eq:E(u(t))-go-to-JE(W)}, \eqref{eq:a_i-is-o(1)},
\eqref{eq:calL-H1-coer-quad-form}, and \eqref{eq:case2-Rinv-is-o(1)}.
We expand 
\begin{align*}
E[u] & =E[U]-\lan\Dlt U+f(U),g\ran-\tfrac{1}{2}\lan\calL_{U}g,g\ran+\tint{}{}\{\tfrac{|U+g|^{p+1}}{p+1}-\tfrac{|U|^{p+1}}{p+1}-f(U)g-f'(U)g^{2}\}\\
 & =E[U]-\tsum{b,j}{}\frkr_{b,j}\lan\calV_{b\ul{;j}},g\ran-\tfrac{1}{2}\lan\calL_{U}g,g\ran+\tint{}{}\calO(|g|^{p+1})\\
 & =(JE[W]+o_{R\to\infty}(1))+o_{R\to\infty}(1)\|g\|_{\dot{H}^{1}}-\{\lan\calL_{\calW}g,g\ran+o_{R\to\infty}(\|g\|_{\dot{H}^{1}}^{2})\}+\calO(\|g\|_{\dot{H}^{1}}^{p+1}).
\end{align*}
By \eqref{eq:case2-Rinv-is-o(1)} and $\|g\|_{\dot{H}^{1}}\aleq\dlt_{0}$,
we get 
\[
E[u]=JE[W]-\tfrac{1}{2}\lan\calL_{\calW}g,g\ran+\dlt_{0}^{p-1}\calO(\|g\|_{\dot{H}^{1}}^{2})+o_{n\to\infty}(1).
\]
Applying the coercivity estimate \eqref{eq:calL-H1-coer-quad-form}
to $-\lan\calL_{\calW}g,g\ran$ and using the orthogonality conditions
$\lan\calZ_{a\ul{;i}},g\ran=0$ and smallness \eqref{eq:a_i-is-o(1)},
we get 
\[
E[u]\geq JE[W]+(c-\calO(\dlt_{0}^{p-1}))\|g\|_{\dot{H}^{1}}^{2}-o_{n\to\infty}(1)
\]
for some universal constant $c>0$. Since $E[u(t)]\to JE[W]$ and
$\dlt_{0}$ is small, we conclude \eqref{eq:case2-gHdot1-is-o(1)}.
\end{proof}
So far, we have only dealt with rough modulation estimates. To control
the dynamics more precisely, we need refined modulation estimates
\eqref{eq:non-coll-refined-mod-est} and hence the leading term of
each $\frkr_{a,i}$.
\begin{lem}[Leading term of $\frkr_{a,i}$]
\label{lem:case2.1-leading-frkr_a,i}We have 
\begin{align}
\frkr_{0,i} & =-\kpp_{0}\iota_{i}\iota_{i_{0}}W_{\lmb_{i_{0}}}(z_{i_{0}}-z_{i})\lmb_{i}^{D}+\calO(\lmb_{\secmax}^{D}\lmb_{i}^{D}+\lmb_{i}^{2}\lmb_{\secmax}^{D-1}\lmb_{\max}^{D})\quad\text{if }i\neq i_{0},\label{eq:case2.1-frkr_0,i}\\
|\frkr_{0,i_{0}}| & \aleq\lmb_{\secmax}^{D}\lmb_{i_{0}}^{D},\label{eq:case2.1-frkr_0,i_0}\\
|\frkr_{a,i}| & \aleq\lmb_{i}^{2}\lmb_{\secmax}^{D-1}\lmb_{\max}^{D}\quad\text{if }a\neq0.\label{eq:case2.1-frkr_a,i}
\end{align}
\end{lem}

\begin{proof}
\uline{Proof of \mbox{\eqref{eq:case2.1-frkr_0,i}}}. Let $i\neq i_{0}$.
By \eqref{eq:non-coll-frkr_a,i-est}, we have 
\[
\frkr_{0,i}=\frac{\lan\calV_{a;i},f'(W_{;i})W_{;i_{0}}\ran}{\|\Lmb W\|_{L^{2}}^{2}}+\sum_{j\neq i,i_{0}}\frac{\lan\calV_{a;i},f'(W_{;i})W_{;j}\ran}{\|\Lmb W\|_{L^{2}}^{2}}+\calO(\lmb_{i}^{2}\lmb_{\secmax}^{D-1}\lmb_{\max}^{D+1}).
\]
Applying \eqref{eq:2.48} and the definition of $\kpp_{0}$ \eqref{eq:def-kpp},
we have 
\begin{align*}
\frkr_{0,i} & =-\kpp_{0}\iota_{i}\iota_{j}W_{\lmb_{i_{0}}}(z_{i_{0}}-z_{i})\lmb_{i}^{D}+\calO\big(\lmb_{i}^{D+2}\lmb_{i_{0}}^{D}(1+|\log\lmb_{i}|)\big)+\calO(\lmb_{i}^{D}\lmb_{\secmax}^{D})+\calO(\lmb_{i}^{2}\lmb_{\secmax}^{D-1}\lmb_{\max}^{D+1})\\
 & =-\kpp_{0}\iota_{i}\iota_{i_{0}}W_{\lmb_{i_{0}}}(z_{i_{0}}-z_{i})\lmb_{i}^{D}+\calO(\lmb_{\secmax}^{D}\lmb_{i}^{D})+\calO(\lmb_{i}^{2}\lmb_{\secmax}^{D-1}\lmb_{\max}^{D+1}|\log\lmb_{\max}|).
\end{align*}

\uline{Proof of \mbox{\eqref{eq:case2.1-frkr_0,i_0}}}. By \eqref{eq:non-coll-frkr_a,i-est}
and \eqref{eq:2.48}, one has $|\frkr_{0,i_{0}}|\aleq\lmb_{\secmax}^{D}\lmb_{i_{0}}^{D}$.

\uline{Proof of \mbox{\eqref{eq:case2.1-frkr_a,i}}}. Let $a\neq0$.
By \eqref{eq:non-coll-frkr_a,i-est} and \eqref{eq:2.49}, one has
$|\frkr_{a,i_{0}}|\aleq\lmb_{\secmax}^{D}\lmb_{i_{0}}^{D+1}$ and
$|\frkr_{a,i}|\aleq\lmb_{i}^{2}\lmb_{\secmax}^{D-1}\lmb_{\max}^{D}$.
These are acceptable in \eqref{eq:case2.1-frkr_a,i}.
\end{proof}
\begin{prop}[Refined modulation estimates]
\label{prop:case2-modulation-est}We have 
\begin{equation}
\begin{aligned}\chf_{i\neq i_{0}}\Big|\frac{\lmb_{i,t}}{\lmb_{i}}+\frac{d}{dt}o_{n\to\infty}(1)-\kpp_{0}\iota_{i}\iota_{i_{0}}W_{\lmb_{i_{0}}}(z_{i_{0}}-z_{i})\lmb_{i}^{D-2}\Big|+\Big|\frac{\lmb_{i_{0},t}}{\lmb_{i_{0}}}+\frac{d}{dt}o_{n\to\infty}(1)\Big|\qquad\\
\aleq\|g\|_{\dot{H}^{2}}^{2}+\lmb_{\secmax}^{D}\lmb_{i}^{D-2}+\lmb_{\secmax}^{D-\frac{3}{2}}\lmb_{\max}^{D+\frac{1}{2}},
\end{aligned}
\label{eq:case2.1-lmb_i-ref-mod}
\end{equation}
and 
\begin{align}
\Big|\frac{z_{i,t}}{\lmb_{i}}\Big| & \aleq\|g\|_{\dot{H}^{2}}^{2}+\lmb_{\secmax}^{D-\frac{3}{2}}\lmb_{\max}^{D+\frac{1}{2}}.\label{eq:case2.1-z_i-ref-mod}
\end{align}
\end{prop}

\begin{proof}
This follows from substituting the estimates of Lemma~\ref{lem:case2.1-leading-frkr_a,i}
into \eqref{eq:non-coll-refined-mod-est}, where we also use $\lmb_{\secmax}^{D-2}\lmb_{\max}^{D-1}\|g\|_{\dot{H}^{2}}\aleq\|g\|_{\dot{H}^{2}}^{2}+\lmb_{\secmax}^{2D-4}\lmb_{\max}^{2D-2}\aleq\|g\|_{\dot{H}^{2}}^{2}+\lmb_{\secmax}^{D-\frac{3}{2}}\lmb_{\max}^{D+\frac{1}{2}}$
by $D\geq\frac{5}{2}$.
\end{proof}
In order to integrate refined modulation estimates \eqref{eq:case2.1-lmb_i-ref-mod}
and \eqref{eq:case2.1-z_i-ref-mod}, it is convenient to introduce
corresponding refined modulation parameters $\vec{\zeta}(t)\approx\vec{\lmb}(t)$
and $\vec{\xi}(t)\approx\vec{z}(t)$ as in the following corollary.
The point is that we get rid of $\|g\|_{\dot{H}^{1}}^{2}$ in the
error, thanks to the spacetime estimate \eqref{eq:case2-spacetime-est}.
\begin{cor}[Modulation estimates in simple form integrated from $t_{n}$]
\label{cor:case2.1-mod-est-simple-form}There exist functions $\zeta_{i}(t)$
and $\xi_{i}(t)$ on $[t_{n},T_{n}')$ such that 
\begin{align}
\zeta_{i}(t) & =(1+o_{n\to\infty}(1))\lmb_{i}(t),\label{eq:case2.1-zeta_i-ptwise}\\
\xi_{i}(t) & =z_{i}(t)+o_{n\to\infty}(1),\label{eq:case2.1-xi_i-ptwise}
\end{align}
and 
\begin{align}
\Big|\frac{\zeta_{i,t}}{\zeta_{i}}-\iota_{i}\iota_{i_{0}}\kpp_{0}W_{\zeta_{i_{0}}}(z_{i_{0}}-z_{i})\zeta_{i}^{D-2}\Big| & \aleq(\dlt_{0})^{\frac{D}{100}}\zeta_{i}^{D-2}\zeta_{i_{0}}^{D}+\zeta_{\secmax}^{D-\frac{3}{2}}\zeta_{i_{0}}^{D}\qquad\text{if }i\neq i_{0},\label{eq:case2.1-zeta_i-diff-ineq}\\
|\zeta_{i_{0},t}|+|\vec{\xi}_{t}| & \aleq\zeta_{\secmax}^{D-\frac{3}{2}}\zeta_{i_{0}}^{D+\frac{1}{2}}.\label{eq:case2.1-zeta_i_0-diff-ineq}
\end{align}
\end{cor}

\begin{proof}
\uline{Proof of \mbox{\eqref{eq:case2.1-zeta_i-ptwise}}}. Take
$\zeta_{i}(t)=\lmb_{i}(t)\exp(o_{n\to\infty}(1)+\int_{t_{n}}^{t}\calO(\|g(\tau)\|_{\dot{H}^{2}}^{2})d\tau)$,
where $o_{n\to\infty}(1)$ is the one in the left hand side of \eqref{eq:case2.1-lmb_i-ref-mod},
and $\calO(\|g\|_{\dot{H}^{2}}^{2})$ is from the right hand side
of \eqref{eq:case2.1-lmb_i-ref-mod}. By \eqref{eq:case2-spacetime-est},
we get \eqref{eq:case2.1-zeta_i-ptwise}.

\uline{Proof of \mbox{\eqref{eq:case2.1-xi_i-ptwise}}}. Write
\eqref{eq:case2.1-z_i-ref-mod} as $z_{i,t}=\calO(\|g\|_{\dot{H}^{2}}^{2})+\calO(\lmb_{\secmax}^{D-\frac{3}{2}}\lmb_{\max}^{D+\frac{3}{2}})$
and define $\xi_{i}(t)=z_{i}(t)+\int_{t_{n}}^{t}\calO(\|g(\tau)\|_{\dot{H}^{2}}^{2})d\tau$.
By \eqref{eq:case2-spacetime-est}, we get \eqref{eq:case2.1-xi_i-ptwise}.

\uline{Proof of \mbox{\eqref{eq:case2.1-zeta_i-diff-ineq}}}. Let
$i\neq i_{0}$. By the definition of $\zeta_{i}$, $\zeta_{i,t}/\zeta_{i}$
satisfies \eqref{eq:case2.1-lmb_i-ref-mod} without $\frac{d}{dt}o_{n\to\infty}(1)$
and $\calO(\|g\|_{\dot{H}^{2}}^{2})$. Then, we replace $\lmb_{i}$
by $\zeta_{i}$ using \eqref{eq:case2.1-zeta_i-ptwise} and apply
$\lmb_{\secmax}^{D}\lmb_{i}^{D-2}\aleq(\dlt_{0})^{\frac{D}{100}}\lmb_{i}^{D-2}\lmb_{i_{0}}^{D}$.
This gives \eqref{eq:case2.1-zeta_i-diff-ineq}.

\uline{Proof of \mbox{\eqref{eq:case2.1-zeta_i_0-diff-ineq}}}.
Proceeding similarly as in the proof of \eqref{eq:case2.1-zeta_i-diff-ineq}
gives \eqref{eq:case2.1-zeta_i_0-diff-ineq}.
\end{proof}
With the above refined modulation estimates at hand, we turn to close
the hypotheses in the definition of $T_{n}$ on the time interval
$[t_{n},T_{n}')$. The following inequality exploits non-degeneracy
in the equations of $\lmb_{i,t}$ and it is the key to control the
modulation parameters on $[t_{n},T_{n}')$.
\begin{lem}[Key inequality]
We have 
\begin{align}
\frac{d}{dt}\Big(\sum_{i\neq i_{0}}\iota_{i}\iota_{i_{0}}\sqrt{\zeta_{i}}\Big) & \aeq\lmb_{\secmax}^{D-\frac{3}{2}}\lmb_{\max}^{D}.\label{eq:case2-key-diff-ineq}
\end{align}
In particular, we have 
\begin{equation}
\int_{\tau_{1}}^{\tau_{2}}\lmb_{\secmax}^{D-\frac{3}{2}}\lmb_{\max}^{D}\aleq\sqrt{\lmb_{\secmax}}(\tau_{1})+\sqrt{\lmb_{\secmax}}(\tau_{2})\text{ for any }\tau_{1},\tau_{2}\in[t_{n},T_{n}')\text{ with }\tau_{1}<\tau_{2}.\label{eq:case2-key-int-ineq}
\end{equation}
\end{lem}

\begin{proof}
\eqref{eq:case2-key-diff-ineq} follows from \eqref{eq:case2.1-zeta_i-diff-ineq}
and $W_{\zeta_{i_{0}}}(z_{i_{0}}-z_{i})\aeq\zeta_{i_{0}}^{D}$. Integrating
\eqref{eq:case2-key-diff-ineq} gives \eqref{eq:case2-key-int-ineq}.
\end{proof}
\begin{lem}[Control of $\zeta_{i_{0}}$ and $\vec{\xi}$]
We have 
\begin{align}
\sup_{t\in[t_{n},T_{n}')}\Big|\frac{\lmb_{i_{0}}(t)}{\lmb_{i_{0}}(t_{n})}-1\Big| & \aleq\dlt_{0}^{1/200},\label{eq:case2.1-lmb_i_0-control}\\
\sup_{t\in[t_{n},T_{n}')}|\vec{z}(t)-\vec{z}^{\ast}| & =o_{n\to\infty}(1),\label{eq:case2.1-z_i-control}\\
\sup_{t\in[t_{n},T_{n}')}\frac{\sum_{j\in\calI_{-}}\lmb_{j}(t)}{\lmb_{i_{0}}(t)} & =o_{n\to\infty}(1).\label{eq:case2.1-minus-ratio-control}
\end{align}
\end{lem}

\begin{proof}
By \eqref{eq:case2.1-zeta_i-ptwise} and \eqref{eq:case2.1-xi_i-ptwise},
we may freely replace $\lmb_{i}$, $z_{i}$ by $\zeta_{i}$, $\xi_{i}$,
respectively.

\uline{Proof of \mbox{\eqref{eq:case2.1-lmb_i_0-control}}}. We
have $|(\sqrt{\zeta_{i_{0}}})_{t}|\aleq\zeta_{\secmax}^{D-\frac{3}{2}}\zeta_{i_{0}}^{D}$
by \eqref{eq:case2.1-zeta_i_0-diff-ineq}. Integrating this using
\eqref{eq:case2-key-int-ineq} and $\lmb_{\secmax}<(\dlt_{0})^{\frac{1}{100}}\lmb_{i_{0}}$,
we get 
\begin{align*}
\sqrt{\zeta_{i_{0}}}(t)=\sqrt{\zeta_{i_{0}}}(t_{n})+\tint{t_{n}}t\calO(\zeta_{\secmax}^{D-\frac{3}{2}}\zeta_{\max}^{D})d\tau=\sqrt{\zeta_{i_{0}}}(t_{n})+\calO(\sqrt{\zeta_{\secmax}}(t_{n})+\sqrt{\zeta_{\secmax}}(t))\quad\\
=\sqrt{\zeta_{i_{0}}}(t_{n})\cdot(1+\calO((\dlt_{0})^{\frac{1}{200}}))+(\dlt_{0})^{\frac{1}{200}}\cdot\calO(\sqrt{\zeta_{i_{0}}}(t)).
\end{align*}
This gives \eqref{eq:case2.1-lmb_i_0-control}. 

\uline{Proof of \mbox{\eqref{eq:case2.1-z_i-control}}}. Integrating
\eqref{eq:case2.1-zeta_i_0-diff-ineq} with \eqref{eq:case2-key-int-ineq},
we get 
\begin{align*}
\sup_{t\in[t_{n},T_{n}')}|\vec{\xi}(t)-\vec{z}^{\ast}| & \leq|\vec{\xi}(t_{n})-\vec{z}^{\ast}|+\tint{t_{n}}{T_{n}'}\calO(\zeta_{\secmax}^{D-\frac{3}{2}}\zeta_{\max}^{D})dt\\
 & \aleq|\vec{\xi}(t_{n})-\vec{z}^{\ast}|+\sqrt{\zeta_{\secmax}}(t_{n})+\limsup_{t\to T_{n}}\sqrt{\zeta_{\secmax}}(t).
\end{align*}
Applying $|\vec{\xi}(t_{n})-\vec{z}^{\ast}|=o_{n\to\infty}(1)$ and
$\lmb_{\secmax}(t)=o_{n\to\infty}(1)$ (by \eqref{eq:case2-Rinv-is-o(1)}),
we get \eqref{eq:case2.1-z_i-control}. 

\uline{Proof of \mbox{\eqref{eq:case2.1-minus-ratio-control}}}.
For any $i\in\calI_{-}$, we have $\iota_{i}\iota_{i_{0}}=-1$ so
$\zeta_{i,t}\leq\zeta_{i}\calO(\zeta_{\secmax}^{D-\frac{3}{2}}\zeta_{\max}^{D})$
by \eqref{eq:case2.1-zeta_i-diff-ineq}. Integrating this using \eqref{eq:case2-key-int-ineq}
and $\lmb_{\secmax}=o_{n\to\infty}(1)$, we conclude $\zeta_{i}(t)\leq\zeta_{i}(t_{n})\cdot(1+o_{n\to\infty}(1))$
for $t\in[t_{n},T_{n}')$. Combining this with \eqref{eq:case2.1-lmb_i_0-control}
and $\lmb_{i}(t_{n})/\lmb_{i_{0}}(t_{n})=o_{n\to\infty}(1)$, we get
\[
\sup_{t\in[t_{n},T_{n}')}\frac{\zeta_{i}(t)}{\zeta_{i_{0}}(t)}\aleq\frac{\zeta_{i}(t_{n})}{\zeta_{i_{0}}(t_{n})}=o_{n\to\infty}(1)
\]
as desired.
\end{proof}
\begin{cor}[Closing bootstrap provided $T_{n}'=T_{n}$]
\label{cor:case2-closing-bootstrap}Suppose $T_{n}'=T_{n}$ for all
large $n$. Then, $T_{n}=\Ttbexit_{n}=+\infty$, $\calI_{+}=\{i_{0}\}$,
and 
\begin{equation}
R^{-1}(t)+\|g(t)\|_{\dot{H}^{1}}+\frac{\lmb_{\secmax}(t)}{\lmb_{i_{0}}(t)}\to0\qquad\text{as }t\to+\infty.\label{eq:case2-closing-bootstrap}
\end{equation}
\end{cor}

\begin{proof}
Assume $T_{n}'=T_{n}$. First, we prove $T_{n}=\Ttbexit_{n}=+\infty$.
We have $T_{n}=\Ttbexit_{n}$ by $T_{n}'=T_{n}$, \eqref{eq:case2.1-lmb_i_0-control},
\eqref{eq:case2.1-z_i-control}, and \eqref{eq:case2.1-minus-ratio-control}.
Next, we have $\Ttbexit_{n}=+\infty$ by $T_{n}'=\Ttbexit$, \eqref{eq:case2-Rinv-is-o(1)},
and \eqref{eq:case2-gHdot1-is-o(1)}. Hence, $T_{n}=\Ttbexit_{n}=+\infty$.
These controls also give 
\begin{equation}
R^{-1}(t)+\|g(t)\|_{\dot{H}^{1}}+\frac{\sum_{j\in\calI_{-}}\lmb_{i}(t)}{\lmb_{i_{0}}(t)}\to0\qquad\text{as }t\to+\infty.\label{eq:7.25-1}
\end{equation}

Next, we prove $\calI_{+}=\{i_{0}\}$. Suppose not. Choose $i\in\calI_{+}\setminus\{i_{0}\}$.
By \eqref{eq:case2.1-zeta_i-diff-ineq}, $(\zeta_{i})_{t}\geq\zeta_{i}\calO(\zeta_{\secmax}^{D-\frac{3}{2}}\zeta_{\max}^{D})$.
Integrating this using \eqref{eq:case2-key-int-ineq} and $\lmb_{\secmax}(t)=o_{n\to\infty}(1)$,
we conclude $\zeta_{i}(t)\geq\zeta_{i}(t_{n})\cdot(1-o_{n\to\infty}(1))$.
Fix $n=n_{0}$ sufficiently large so that $T_{n}'=+\infty$; we have
$\lmb_{i}(t)\ageq\lmb_{i}(t_{n_{0}})$ for all $t\in[t_{n_{0}},+\infty)$.
On the other hand, $\lmb_{i_{0}}(t)\ageq\lmb_{i_{0}}(t_{n_{0}})$
for all $t\in[t_{n_{0}},+\infty)$ by \eqref{eq:case2.1-lmb_i_0-control}.
Hence $R_{ii_{0}}^{-1}(t)$ cannot converge to $0$ as $t\to+\infty$,
contradicting to \eqref{eq:case2-Rinv-is-o(1)}.

Finally, \eqref{eq:case2-closing-bootstrap} follows from $\calI_{+}=\{i_{0}\}$
and \eqref{eq:7.25-1}.
\end{proof}

\subsection{Modulation analysis on $[T_{n}',T_{n}'')$ and proof of $T_{n}'=T_{n}$}

The goal of this subsection is to show $T_{n}'=T_{n}$ for all large
$n$. We prove this by contradiction. Throughout this subsection,
suppose 
\[
T_{n}'<T_{n}
\]
for a subsequence of $n$, which we still denote by a sequence indexed
by $n$. 

As mentioned at the beginning of Section~\ref{sec:Case2}, $T_{n}'<T_{n}$
implies 
\[
\calI_{+}\supsetneq\{i_{0}\}.
\]
By \eqref{eq:case2.1-lmb_i_0-control}, we have $\lmb_{i_{0}}(T_{n}')\leq(\frac{1}{2}+o_{n\to\infty}(1))K_{0}$.
By \eqref{eq:case2.1-z_i-control} and \eqref{eq:case2.1-minus-ratio-control},
we have 
\[
|\vec{z}(T_{n}')-\vec{z}^{\ast}|+\frac{\sum_{j\in\calI_{-}}\lmb_{j}(T_{n}')}{\lmb_{i_{0}}(T_{n}')}=o_{n\to\infty}(1).
\]
The fact that $T_{n}'<T_{n}$ implies $\tsum{i\neq i_{0}}{}\lmb_{i}(T_{n}')=(\dlt_{0})^{\frac{1}{100}}\lmb_{i_{0}}(T_{n}')$.
In particular, $\lmb_{\secmax}(T_{n}')\aeq(\dlt_{0})^{\frac{1}{100}}\lmb_{\max}(T_{n}')$
and $\lmb_{\max}(T_{n}')=o_{n\to\infty}(1)$ (by \eqref{eq:case2-Rinv-is-o(1)}).
Now, the following time is well-defined for all large $n$: 
\[
T_{n}''=\sup\{\tau\in[T_{n}',T_{n}):\frac{\lmb_{\secmax}(t)}{\lmb_{\max}(t)}>(\dlt_{0})^{\frac{1}{50}}\text{ for all }t\in[T_{n}',\tau]\}\in(T_{n}',T_{n}].
\]
We will bootstrap the hypothesis $\frac{\lmb_{\secmax}(t)}{\lmb_{\max}(t)}>(\dlt_{0})^{\frac{1}{50}}$
so that $T_{n}''=T_{n}$. Moreover, on $[T_{n}',T_{n}'')$, $\lmb_{\secmax}$
and $\lmb_{\max}$ arise from indices in $\calI_{+}$, so both $\lmb_{\max}$
and $\lmb_{\secmax}$ are increasing in time. This will then contradict
to the asymptotic orthogonality of parameters.
\begin{lem}[Smallness of $\|g(t)\|_{\dot{H}^{1}}+R^{-1}(t)$]
We have 
\begin{align}
\sup_{t\in[T_{n}',T_{n}'')}\{R^{-1}(t)+\lmb_{\max}(t)\} & =o_{n\to\infty}(1),\label{eq:case2.2-Rinv-is-o(1)-1}\\
\sup_{t\in[T_{n}',T_{n}'')}\|g(t)\|_{\dot{H}^{1}} & =o_{n\to\infty}(1).\label{eq:case2.2-gHdot1-is-o(1)-1}
\end{align}
\end{lem}

\begin{proof}
\uline{Proof of \mbox{\eqref{eq:case2.2-Rinv-is-o(1)-1}}}. Since
$R^{-1}(t)\aleq\lmb_{\max}(t)$, it suffices to show that $\lmb_{\max}(t)=o_{n\to\infty}(1)$.
For any $i\in\setJ$, \eqref{eq:case2-rough-mod-est} gives 
\[
|(\lmb_{i}^{2D})_{t}|\aleq\lmb_{i}^{2D-1}(\|g\|_{\dot{H}^{2}}+\lmb_{\max}^{2D-1})\aleq\|g\|_{\dot{H}^{2}}^{2}+\lmb_{\max}^{4D-2}.
\]
By \eqref{eq:case2-max-frkr_a,i} and $\lmb_{\secmax}>(\dlt_{0})^{\frac{1}{50}}\lmb_{\max}$,
we have $\int_{T_{n}'}^{T_{n}''}(\|g\|_{\dot{H}^{2}}^{2}+\lmb_{\max}^{4D-2})dt=o_{n\to\infty}(1)$.
Integrating the above display with $\lmb_{\max}(T_{n}')=o_{n\to\infty}(1)$
gives $\lmb_{i}^{2D}(t)=o_{n\to\infty}(1)$. As $i\in\setJ$ was arbitrary,
we get $\lmb_{\max}(t)=o_{n\to\infty}(1)$ as desired.

\uline{Proof of \mbox{\eqref{eq:case2.2-gHdot1-is-o(1)-1}}}. The
proof of \eqref{eq:case2-gHdot1-is-o(1)} works the same after replacing
\eqref{eq:case2-Rinv-is-o(1)} by \eqref{eq:case2.2-Rinv-is-o(1)-1}.
\end{proof}
As before, we need refined modulation estimates and hence the leading
term of each $\frkr_{a,i}$.
\begin{lem}[Leading term of $\frkr_{a,i}$]
\label{lem:case2.2-leading-frkr_a,i}We have 
\begin{align}
\frkr_{0,i} & =-\kpp_{0}\kpp_{\infty}\sum_{j\in\calI_{+}\setminus\{i\}}\frac{\lmb_{i}^{D}\lmb_{j}^{D}}{|z_{j}-z_{i}|^{2D}}+(\dlt_{0})^{\frac{49}{50}D}\cdot\calO(\lmb_{i}^{2}\lmb_{\secmax}^{D}\lmb_{\max}^{D-2}),\quad\text{if }i\in\calI_{+},\label{eq:case2.2-frkr_0,i-plus}\\
|\frkr_{0,i}| & \aleq(\dlt_{0})^{\frac{49}{50}(D-2)-\frac{2}{50}}\cdot\lmb_{i}^{2}\lmb_{\secmax}^{D}\lmb_{\max}^{D-2}\quad\text{if }i\in\calI_{-},\label{eq:case2.2-frkr_0,i-minus}\\
|\frkr_{a,i}| & \aleq o_{n\to\infty}(1)\cdot\lmb_{i}\lmb_{\secmax}^{D}\lmb_{\max}^{D-1}\quad\text{if }a\neq0.\label{eq:case2.2-frkr_a,i}
\end{align}
\end{lem}

\begin{proof}
\uline{Proof of \mbox{\eqref{eq:case2.2-frkr_0,i-plus}}}. Let
$i\in\calI_{+}$. By \eqref{eq:non-coll-frkr_a,i-est}, we have 
\[
\frkr_{0,i}=\sum_{j\in\calI_{+}\setminus\{i\}}\frac{\lan\calV_{0;i},f'(W_{;i})W_{;j}\ran}{\|\Lmb W\|_{L^{2}}^{2}}+\sum_{j\in\calI_{-}}\frac{\lan\calV_{0;i},f'(W_{;i})W_{;j}\ran}{\|\Lmb W\|_{L^{2}}^{2}}+\calO(\lmb_{i}^{2}\lmb_{\secmax}^{D-1}\lmb_{\max}^{D+1}).
\]
Applying \eqref{eq:2.48}, $\lmb_{\max}=o_{n\to\infty}(1)$, and $\lmb_{\max}\aeq_{\dlt_{0}}\lmb_{\secmax}$,
one has 
\[
\frkr_{0,i}=-\kpp_{0}\kpp_{\infty}\sum_{j\in\calI_{+}\setminus\{i\}}\frac{\lmb_{i}^{D}\lmb_{j}^{D}}{|z_{j}-z_{i}|^{2D}}+\calO(\tsum{j\in\calI_{-}}{}\lmb_{j}^{D}\lmb_{i}^{D})+o_{n\to\infty}(1)\cdot\lmb_{i}^{2}\lmb_{\secmax}^{D}\lmb_{\max}^{D-2}.
\]
Applying $\max_{j\in\calI_{-}}\lmb_{j}\aleq\dlt_{0}\lmb_{\max}\aleq(\dlt_{0})^{\frac{49}{50}}\lmb_{\secmax}$,
we get \eqref{eq:case2.2-frkr_0,i-plus}. 

\uline{Proof of \mbox{\eqref{eq:case2.2-frkr_0,i-minus}}}. Let
$i\in\calI_{-}$. By \eqref{eq:non-coll-frkr_a,i-est} and \eqref{eq:2.48},
one has $|\frkr_{0,i}|\aleq\lmb_{i}^{D}\lmb_{\max}^{D}+\lmb_{i}^{2}\lmb_{\secmax}^{D-1}\lmb_{\max}^{D+1}$.
Applying $\lmb_{i}\aleq(\dlt_{0})^{\frac{49}{50}}\lmb_{\secmax}$,
$\lmb_{\max}\aleq(\dlt_{0})^{-\frac{1}{50}}\lmb_{\secmax}$, and $\lmb_{\max}=o_{n\to\infty}(1)$,
we conclude \eqref{eq:case2.2-frkr_0,i-minus}. 

\uline{Proof of \mbox{\eqref{eq:case2.2-frkr_a,i}}}. Let $a\neq0$.
As in the proof of \eqref{eq:case2.1-frkr_a,i}, we have the same
bound $|\frkr_{a,i}|\aleq\lmb_{i}^{2}\lmb_{\secmax}^{D-1}\lmb_{\max}^{D}$.
Applying $\lmb_{\max}\aleq(\dlt_{0})^{-\frac{1}{50}}\lmb_{\secmax}$
and $\lmb_{\max}=o_{n\to\infty}(1)$ gives \eqref{eq:case2.2-frkr_a,i}. 
\end{proof}
\begin{prop}[Refined modulation estimates]
\label{prop:case2.2-ref-mod}We have 
\begin{equation}
\begin{aligned}\Big|\frac{\lmb_{i,t}}{\lmb_{i}}+\frac{d}{dt}o_{n\to\infty}(1)-\kpp_{0}\kpp_{\infty} & \sum_{j\in\calI_{+}\setminus\{i\}}\frac{\lmb_{i}^{D-2}\lmb_{j}^{D}}{|z_{j}-z_{i}|^{2D}}\Big|\\
 & \aleq\|g\|_{\dot{H}^{2}}^{2}+(\dlt_{0})^{\frac{49}{50}D}\lmb_{\secmax}^{D}\lmb_{\max}^{D-2}\quad\text{if }i\in\calI_{+},
\end{aligned}
\label{eq:case2.2-lmb_i-plus-ref-mod}
\end{equation}
and 
\begin{align}
\Big|\frac{\lmb_{i,t}}{\lmb_{i}}+\frac{d}{dt}o_{n\to\infty}(1)\Big| & \aleq\|g\|_{\dot{H}^{2}}^{2}+(\dlt_{0})^{\frac{49}{50}(D-2)-\frac{2}{50}}\lmb_{\secmax}^{D}\lmb_{\max}^{D-2},\quad\text{if }i\in\calI_{-},\label{eq:case2.2-lmb_i-minus-ref-mod}\\
|z_{i,t}| & \aleq\|g\|_{\dot{H}^{2}}^{2}+o_{n\to\infty}(1)\cdot\lmb_{\secmax}^{D}\lmb_{\max}^{D-1}.\label{eq:case2.2-z_i-ref-mod}
\end{align}
\end{prop}

\begin{proof}
This easily follows from substituting Lemma~\ref{lem:case2.2-leading-frkr_a,i}
and $\lmb_{\max}=o_{n\to\infty}(1)$ into \eqref{eq:non-coll-refined-mod-est}.
\end{proof}
\begin{cor}[Modulation estimates in simple form integrated from $T_{n}'$]
There exist functions $\zeta_{i}(t)$ and $\xi_{i}(t)$ on $[T_{n}',T_{n}'')$
such that 
\begin{align}
\zeta_{i}(t) & =\lmb_{i}(t)\cdot(1+o_{n\to\infty}(1)),\label{eq:case2.2-zeta-ptwise}\\
\xi_{i}(t) & =z_{i}(t)+o_{n\to\infty}(1),\label{eq:case2.2-xi-ptwise}
\end{align}
and 
\begin{align}
\frac{\zeta_{i,t}}{\zeta_{i}} & =\sum_{j\in\calI_{+}\setminus\{i\}}\kpp_{0}\kpp_{\infty}\frac{\zeta_{i}^{D-2}\zeta_{j}^{D}}{|z_{j}-z_{i}|^{2D}}+\calO((\dlt_{0})^{\frac{49}{50}D}\zeta_{\secmax}^{D}\zeta_{\max}^{D-2})\quad\text{if }i\in\calI_{+},\label{eq:case2.2-zeta-positive-diff}\\
\Big|\frac{\zeta_{i,t}}{\zeta_{i}}\Big| & \aleq(\dlt_{0})^{\frac{49}{50}(D-2)-\frac{2}{50}}\zeta_{\secmax}^{D}\zeta_{\max}^{D-2}\quad\text{if }i\in\calI_{-},\label{eq:case2.2-zeta-negative-diff}\\
|\vec{\xi}_{t}| & \aleq o_{n\to\infty}(1)\cdot\zeta_{\secmax}^{D}\zeta_{\max}^{D-1}.\label{eq:case2.2-xi-diff}
\end{align}
\end{cor}

\begin{proof}
The proof of Corollary~\ref{cor:case2.1-mod-est-simple-form} works
the same; one uses Proposition~\ref{prop:case2.2-ref-mod} instead
of Proposition~\ref{prop:case2-modulation-est}.
\end{proof}
Having established refined modulation estimates, we now close all
the bootstrap hypotheses used in the definition of $T_{n}$ and $T_{n}''$.
The key to control the dynamics of parameters is the following.
\begin{lem}[Key inequalities]
We have 
\begin{align}
\rd_{t}(\tsum{j\in\calI_{+}}{}\zeta_{j}) & \aeq\zeta_{\secmax}^{D-1}\zeta_{\max}^{D},\label{eq:case2.2-key-diff-ineq-1}\\
\rd_{t}(\tsum{j\in\calI_{+}\setminus\{i\}}{}\zeta_{j}^{2}) & \aeq\zeta_{\secmax}^{D}\zeta_{\max}^{D}\aeq\rd_{t}(\tsum{j\in\calI_{+}}{}\zeta_{j}^{2}),\qquad\forall i\in\calI_{+}.\label{eq:case2.2-key-diff-ineq-2}
\end{align}
\end{lem}

\begin{proof}
This follows from \eqref{eq:case2.2-zeta-positive-diff}. For the
proof of \eqref{eq:case2.2-key-diff-ineq-2}, notice that $|\calI_{+}|\geq2$
implies 
\[
\tsum{j\in\calI_{+}\setminus\{i\}}{}\tsum{k\in\calI_{+}\setminus\{j\}}{}\zeta_{j}^{D}\zeta_{k}^{D}\aeq\zeta_{\secmax}^{D}\zeta_{\max}^{D}
\]
with the implicit constant independent of $\dlt_{0}$.
\end{proof}
\begin{lem}
We have 
\begin{align}
\sup_{t\in[T_{n}',T_{n}'')}|\vec{z}(t)-\vec{z}^{\ast}| & =o_{n\to\infty}(1),\label{eq:case2.2-z-control}\\
\inf_{t\in[T_{n}',T_{n}'')}\lmb_{\max}(t) & \ageq\lmb_{\max}(T_{n}'),\label{eq:case2.2-max-control}\\
\sup_{t\in[T_{n}',T_{n}'')}\frac{\lmb_{\secmax}(t)}{\lmb_{\max}(t)} & \ageq(\dlt_{0})^{\frac{1}{100}},\label{eq:case2.2-max2-ratio-control}\\
\sup_{t\in[T_{n}',T_{n}'')}\frac{\sum_{j\in\calI_{-}}\lmb_{j}(t)}{\lmb_{\max}(t)} & =o_{n\to\infty}(1).\label{eq:case2.2-minus-ratio-control}
\end{align}
\end{lem}

\begin{proof}
\uline{Proof of \mbox{\eqref{eq:case2.2-z-control}}}. This follows
from \eqref{eq:case2.2-xi-diff} and \eqref{eq:case2.2-key-diff-ineq-1}:
\[
|\vec{\xi}(t)-\vec{z}^{\ast}|\leq|\vec{\xi}(T_{n}')-\vec{z}^{\ast}|+o_{n\to\infty}(1)\cdot\tint{T_{n}'}{T_{n}''}\zeta_{\secmax}^{D-1}\zeta_{\max}^{D}dt=o_{n\to\infty}(1).
\]

\uline{Proof of \mbox{\eqref{eq:case2.2-max-control}}}. $\tsum{j\in\calI_{+}}{}\zeta_{j}(t)$
is increasing by \eqref{eq:case2.2-key-diff-ineq-1}. As $\tsum{j\in\calI_{+}}{}\zeta_{j}\aeq\lmb_{\max}$,
we get \eqref{eq:case2.2-max-control}.

\uline{Proof of \mbox{\eqref{eq:case2.2-max2-ratio-control}}}.
For any $i\in\calI_{+}$ and $t\in[T_{n}',T_{n}'')$, we use \eqref{eq:case2.2-key-diff-ineq-2}
and $\tsum{i\neq i_{0}}{}\lmb_{i}^{2}(T_{n}')\aeq(\dlt_{0})^{\frac{1}{50}}\lmb_{i_{0}}^{2}(T_{n}')$
to have 
\begin{align*}
\tsum{j\in\calI_{+}\setminus\{i\}}{}\zeta_{j}^{2}(t) & =\tsum{j\in\calI_{+}\setminus\{i\}}{}\zeta_{j}^{2}(T_{n}')+\tint{T_{n}'}t\rd_{t}(\tsum{j\in\calI_{+}\setminus\{i\}}{}\zeta_{j}^{2})d\tau\\
 & \ageq(\dlt_{0})^{\frac{1}{50}}\{\tsum{j\in\calI_{+}}{}\zeta_{j}^{2}(T_{n}')+\tint{T_{n}'}t\rd_{t}(\tsum{j\in\calI_{+}}{}\zeta_{j}^{2})d\tau\}\aeq(\dlt_{0})^{\frac{1}{50}}\tsum{j\in\calI_{+}}{}\zeta_{j}^{2}(t).
\end{align*}
Choosing $i$ (depending on $t$) such that $\tsum{j\in\calI_{+}\setminus\{i\}}{}\zeta_{j}^{2}(t)\aeq\lmb_{\secmax}^{2}(t)$,
we get 
\[
\lmb_{\secmax}^{2}(t)\ageq(\dlt_{0})^{\frac{1}{50}}\lmb_{\max}^{2}(t).
\]
Taking the square root gives \eqref{eq:case2.2-max2-ratio-control}.

\uline{Proof of \mbox{\eqref{eq:case2.2-minus-ratio-control}}}.
By \eqref{eq:case2.2-zeta-negative-diff} and \eqref{eq:case2.2-key-diff-ineq-1},
we have 
\begin{align*}
-\frac{d}{dt}\Big(\frac{\sum_{j\in\calI_{-}}\zeta_{j}(t)}{\sum_{j\in\calI_{+}}\zeta_{j}(t)}\Big) & =\frac{\sum_{j\in\calI_{-}}\zeta_{j}(t)}{\sum_{j\in\calI_{+}}\zeta_{j}(t)}\Big\{\frac{\rd_{t}(\tsum{j\in\calI_{+}}{}\zeta_{j})}{\sum_{j\in\calI_{+}}\zeta_{j}(t)}-\calO((\dlt_{0})^{\frac{49}{50}(D-2)-\frac{2}{50}}\zeta_{\secmax}^{D}\zeta_{\max}^{D-2})\Big\}\\
 & \ageq\frac{\sum_{j\in\calI_{-}}\zeta_{j}(t)}{\sum_{j\in\calI_{+}}\zeta_{j}(t)}\cdot\zeta_{\secmax}^{D-1}\zeta_{\max}^{D-1}.
\end{align*}
Thus this ratio is decreasing in time and we conclude
\[
\sup_{t\in[T_{n}',T_{n}'')}\frac{\sum_{j\in\calI_{-}}\zeta_{j}(t)}{\sum_{j\in\calI_{+}}\zeta_{j}(t)}\aleq\frac{\sum_{j\in\calI_{-}}\lmb_{j}(T_{n}')}{\lmb_{\max}(T_{n}')}=o_{n\to\infty}(1).\qedhere
\]
\end{proof}
\begin{cor}
\label{cor:case2-T_n'-is-T_n}We get a contradiction. Therefore, $T_{n}'=T_{n}$
for all large $n$.
\end{cor}

\begin{proof}
First, we show $T_{n}''=T_{n}=\Ttbexit_{n}=+\infty$ for all large
$n$. We have $T_{n}''=T_{n}$ by \eqref{eq:case2.2-max2-ratio-control}.
We have $T_{n}=\Ttbexit_{n}$ by $\lmb_{\max}(t)=o_{n\to\infty}(1)$,
\eqref{eq:case2.2-z-control}, and \eqref{eq:case2.2-minus-ratio-control}.
We have $\Ttbexit_{n}=+\infty$ by \eqref{eq:case2-Rinv-is-o(1)}
and \eqref{eq:case2-gHdot1-is-o(1)}. 

Moreover, we have $\lmb_{\max}(t)\to0$ as $t\to+\infty$ by $R^{-1}(t)\to0$
and $\lmb_{\secmax}\ageq_{\dlt_{0}}\lmb_{\max}$. Now, fix a large
$n_{0}$ such that $T_{n_{0}}''=+\infty$. By \eqref{eq:case2.2-max-control},
$\lmb_{\max}(t)\ageq\lmb_{\max}(T_{n_{0}}')$ on $[T_{n_{0}}',+\infty)$,
which contradicts to $\lmb_{\max}(t)\to0$.
\end{proof}

\subsection{Dynamics: proof of Theorem~\ref{thm:main-one-bubble-tower-classification}}

By Corollary~\ref{cor:case2-T_n'-is-T_n} and Corollary~\ref{cor:case2-closing-bootstrap},
we have $T_{n}'=+\infty$ for all large $n$, $\calI_{+}=\{i_{0}\}$,
and $u(t)$ satisfies a continuous-in-time resolution with $\vec{z}(t)\to\vec{z}^{\ast}$.
It remains to show \eqref{eq:main1-lmb_i_0} and \eqref{eq:main1-lmb_i}.
We begin with the following refined modulation estimates, integrated
backwards from $+\infty$.
\begin{cor}[Modulation estimates in simple form integrated from $+\infty$]
There exist functions $\zeta_{i}$ defined for all large $t$ such
that 
\begin{equation}
\zeta_{i}(t)=\lmb_{i}(t)\cdot(1+o(1))\label{eq:case2-zeta_i-ptwise-final}
\end{equation}
and 
\begin{align}
\frac{\zeta_{i,t}}{\zeta_{i}} & =-\kpp_{0}W_{\zeta_{i_{0}}}(z_{i_{0}}^{\ast}-z_{i}^{\ast})\zeta_{i}^{D-2}+o(1)\cdot\zeta_{i}^{D-2}\zeta_{i_{0}}^{D},\qquad\forall i\neq i_{0},\label{eq:case2-zeta_i-diff-final}\\
|\zeta_{i_{0},t}| & \aleq\zeta_{\secmax}^{D-\frac{3}{2}}\zeta_{i_{0}}^{D+\frac{1}{2}}.\label{eq:case2-zeta_i_0-diff-final}
\end{align}
\end{cor}

\begin{proof}
We only prove \eqref{eq:case2-zeta_i-ptwise-final} for $i\neq i_{0}$
and \eqref{eq:case2-zeta_i-diff-final}. By $T_{n}'=+\infty$, $\calI_{+}=\{i_{0}\}$,
and \eqref{eq:case2-closing-bootstrap}, we can rewrite \eqref{eq:case2.1-lmb_i-ref-mod}
as 
\[
\frac{\lmb_{i,t}}{\lmb_{i}}+\frac{d}{dt}o(1)+\kpp_{0}W_{\lmb_{i_{0}}}(z_{i_{0}}^{\ast}-z_{i}^{\ast})\lmb_{i}^{D-2}=\calO(\|g\|_{\dot{H}^{2}}^{2}+\lmb_{\secmax}^{D-\frac{3}{2}}\lmb_{\max}^{D+\frac{1}{2}})+o(1)\cdot\lmb_{i}^{D-2}\lmb_{i_{0}}^{D}.
\]
Define $\zeta_{i}(t)=\lmb_{i}(t)\exp\{o(1)-\int_{t}^{\infty}\calO(\|g\|_{\dot{H}^{2}}^{2}+\lmb_{\secmax}^{D-\frac{3}{2}}\lmb_{\max}^{D+\frac{1}{2}})d\tau\}$.
The integral from $t$ to $\infty$ is of size $o(1)$ by \eqref{eq:case2-spacetime-est}
and \eqref{eq:case2-key-int-ineq} with $\lmb_{\secmax}=o(1)$. This
concludes the proof of \eqref{eq:case2-zeta_i-ptwise-final} and \eqref{eq:case2-zeta_i-diff-final}
for $i\neq i_{0}$.

The proof of \eqref{eq:case2-zeta_i-ptwise-final} for $i=i_{0}$
and \eqref{eq:case2-zeta_i_0-diff-final} is easier and omitted.
\end{proof}
We are ready to prove \eqref{eq:main1-lmb_i_0} and \eqref{eq:main1-lmb_i}.
\begin{proof}
\uline{Proof of \mbox{\eqref{eq:main1-lmb_i_0}}}. It suffices
to show this for $\zeta_{i_{0}}$ by \eqref{eq:case2-zeta_i-ptwise-final}.
Revisiting the proof of \eqref{eq:case2.1-lmb_i_0-control}, but using
$\lmb_{\secmax}/\lmb_{i_{0}}=o(1)$ this time, gives 
\[
\sqrt{\zeta_{i_{0}}}(t)=\sqrt{\zeta_{i_{0}}}(t_{n})\cdot(1+o_{n\to\infty}(1))
\]
for all $t\geq t_{n}$. This completes the proof of $\lmb_{i_{0}}\to\exists\lmb_{i_{0}}^{\ast}\in(0,\infty)$.

\uline{Proof of \mbox{\eqref{eq:main1-lmb_i}}}. Let $i\neq i_{0}$.
Since $\zeta_{i_{0}}\to\lmb_{i_{0}}^{\ast}\in(0,\infty)$, we can
rewrite \eqref{eq:case2-zeta_i-diff-final} as 
\[
\frac{\zeta_{i,t}}{\zeta_{i}^{D-1}}=-\kpp_{0}W_{\lmb_{i_{0}}^{\ast}}(z_{i_{0}}^{\ast}-z_{i}^{\ast})+o(1).
\]
Integrating this forwards in time, we get 
\[
\zeta_{i}(t)=\Big\{(D-2)\kpp_{0}W_{\lmb_{i_{0}}^{\ast}}(z_{i_{0}}^{\ast}-z_{i}^{\ast})\Big\}^{-\frac{1}{D-2}}t^{-\frac{1}{D-2}}(1+o(1)).
\]
Applying $\lmb_{i}(t)=\zeta_{i}(t)\cdot(1+o(1))$ and $D-2=\frac{N-6}{2}$,
we conclude 
\begin{equation}
\lmb_{i}(t)=\Big\{\frac{N-6}{2}\kpp_{0}W_{\lmb_{i_{0}}^{\ast}}(z_{i_{0}}^{\ast}-z_{i}^{\ast})\Big\}^{-\frac{2}{N-6}}t^{-\frac{2}{N-6}}(1+o(1)).\qedhere\label{eq:case2-lmb-ratio}
\end{equation}
\end{proof}

\section{\label{sec:Case1}Modulation analysis in Case 1}

This section is devoted to the dynamics in \textbf{Case~1} of Proposition~\ref{prop:case-separation}.
We also prove the first item of Theorem~\ref{thm:main-lmb-to-zero-classification}
at the end of this section. The main result of this section is the
following.
\begin{prop}
\label{prop:case1-main}Consider the solution $u(t)$ belonging to
\textbf{Case~1} of Proposition~\ref{prop:case-separation}. Recall
the notation therein. Then, the following hold.
\begin{itemize}
\item (Continuous-in-time resolution) $T_{n}'=+\infty$ for all large $n$
and the resolution \eqref{eq:main1-conti-time-res-with-z-conv} holds.
\item (Non-degeneracy) The configuration $\{(\iota_{j},z_{j}^{\ast}):j\in\setJ\}$
is non-degenerate.
\item (Asymptotics of $\lmb_{i}(t)$) The scales $\lmb_{i}(t)$ satisfy
\eqref{eq:main-thm-nondegen-lmb-rate} and the ratio vector $t^{1/(N-4)}\vec{\lmb}(t)$
converges to a connected component of the set \eqref{eq:main-thm-nondegen-ratio-set}.
\end{itemize}
\end{prop}

Let $\dlt_{2}\in(0,\dlt_{2}^{\ast}]$ and $\alp\in(0,\alp^{\ast}]$
be two \emph{small parameters} that can shrink in the course of the
proof satisfying the parameter dependence 
\[
\alp,\dlt_{2}\ll1.
\]
Consider the time (recall the definition \eqref{eq:def-T_n(K,delta,alpha)})
\[
T_{n}\coloneqq T_{n}(1,\dlt_{2},\alp)\in(t_{n},T_{n}'].
\]
Recall \eqref{eq:5.28}; there exists $\dlt_{1}>0$ (independent of
$n$) such that 
\begin{equation}
\min_{\emptyset\neq\calI\subseteq\setJ}\bigg(\frac{|A_{\calI}^{\ast}\vec{\lmb}_{\calI}^{D}(t)|}{|\vec{\lmb}_{\calI}^{D}(t)|}+\sum_{j\notin\calI}\frac{\lmb_{j}(t)}{\lmb_{\max}(t)}\bigg)\geq\dlt_{1}\quad\text{for all }t\in[t_{n},T_{n}).\label{eq:5.1}
\end{equation}
The estimates in the following assume $t\in[t_{n},T_{n})$. As $\dlt_{1}$
is fixed, any dependence on $\dlt_{1}$ will be ignored.

\subsection{Basic observations}
\begin{rem}
\label{rem:7.2}Lemma~\ref{lem:case1-R_ij}, Lemma~\ref{lem:case1-tdU-est},
and Lemma~\ref{lem:case1-leading-frkr_a,i} remain valid under $\lmb_{\max}(t)\aeq\lmb_{\secmax}(t)$.
In consequence, these lemmas will be used again in Section~\ref{sec:Case3}.
\end{rem}

\begin{lem}[$R_{ij}$ in terms of $\vec{\lmb}$]
\label{lem:case1-R_ij}We have 
\begin{align}
R^{-1} & \aeq\lmb_{\max}\aeq\lmb_{\secmax},\label{eq:case1-R}\\
R_{i}^{-1} & \aeq\sqrt{\lmb_{i}\lmb_{\max}},\label{eq:case1-R_i}\\
R_{ij}^{-1} & \aeq\sqrt{\lmb_{i}\lmb_{j}}\quad\text{if }i\neq j.\label{eq:case1-R_ij}
\end{align}
\end{lem}

\begin{proof}
We have $\lmb_{\max}\aeq\lmb_{\secmax}$ by \eqref{eq:case1-secmax=00003Dmax}.
The proof then follows from Lemma~\ref{lem:non-coll-R_ij}.
\end{proof}
\begin{lem}[Estimates for $\td U$]
\label{lem:case1-tdU-est}For each $i\in\setJ$, we have the following
estimates: 
\begin{gather}
\tsum{j\neq i}{}\|f(W_{;i},W_{;j})\|_{L^{(2^{\ast\ast})'}}\aleq\lmb_{i}^{2}\lmb_{\max}^{\frac{N-2}{2}},\label{eq:case1-f(W_i,W_j)-H2dual}\\
\|\td U\|_{\dot{H}^{1}}\aleq\lmb_{\max}^{\frac{N+2}{2}},\quad\|\td U\|_{\dot{H}^{2}}\aleq\lmb_{\max}^{\frac{N}{2}},\quad\|\lmb_{i}\rd_{z_{i}^{a}}\td U\|_{\dot{H}^{1}}\aleq\lmb_{i}\lmb_{\max}^{\frac{N}{2}}.\label{eq:case1-tdU-est}
\end{gather}
\end{lem}

\begin{proof}
The proof follows from Lemma~\ref{lem:non-coll-tdU-est} and $\lmb_{\max}\aeq\lmb_{\secmax}$.
\end{proof}
\begin{lem}[Leading terms of $\frkr_{a,i}$]
\label{lem:case1-leading-frkr_a,i}We have 
\begin{align}
\frkr_{0,i} & =-\sum_{j\neq i}A_{ij}[\vec{z}]\lmb_{i}^{D}\lmb_{j}^{D}+\calO(\lmb_{i}^{2}\lmb_{\max}^{2D}|\log\lmb_{\max}|),\label{eq:case1-frkr_0,i}\\
\frkr_{a,i} & =\frac{\kpp_{1}}{\kpp_{0}}\sum_{j\neq i}A_{ij}[\vec{z}]\frac{\lmb_{i}^{D+1}\lmb_{j}^{D}}{|z_{i}-z_{j}|^{2}}(z_{i}^{a}-z_{j}^{a})+\calO(\lmb_{i}^{2}\lmb_{\max}^{2D}|\log\lmb_{\max}|)\quad\text{if }a\neq0.\label{eq:case1-frkr_a,i}
\end{align}
\end{lem}

\begin{proof}
We begin with \eqref{eq:non-coll-frkr_a,i-est}, which we recall under
the assumption $\lmb_{\max}\aeq\lmb_{\secmax}$ as 
\begin{equation}
\frkr_{a,i}=\sum_{j\neq i}\frac{\lan\calV_{a;i},f'(W_{;i})W_{;j}\ran}{\|\calV_{a}\|_{L^{2}}^{2}}+\calO(\lmb_{i}^{2}\lmb_{\max}^{2D}).\label{eq:6.15}
\end{equation}
Therefore, it suffices to show the following claim: 
\begin{align}
\frac{\lan\calV_{0;i},f'(W_{;i})W_{;j}\ran}{\|\calV_{0}\|_{L^{2}}^{2}} & =-A_{ij}[\vec{z}]\lmb_{i}^{D}\lmb_{j}^{D}+\calO(\lmb_{i}^{2}\lmb_{\max}^{2D}|\log\lmb_{\max}|),\label{eq:6.16}\\
\frac{\lan\calV_{a;i},f'(W_{;i})W_{;j}\ran}{\|\calV_{a}\|_{L^{2}}^{2}} & =\frac{\kpp_{1}}{\kpp_{0}}A_{ij}[\vec{z}]\frac{\lmb_{i}^{D+1}\lmb_{j}^{D}}{|z_{i}-z_{j}|^{2}}(z_{i}-z_{j})+\calO(\lmb_{i}^{2}\lmb_{\max}^{2D+1}|\log\lmb_{\max}|),\label{eq:6.17}
\end{align}
where $i\neq j$ and $a\in\{1,\dots,N\}$.

\uline{Proof of \mbox{\eqref{eq:6.16}}}. Recall $\calV_{0}=\Lmb W$
and \eqref{eq:def-kpp_infty} so that 
\[
|W(y)-\kpp_{\infty}|y|^{-(N-2)}|+|\Lmb W(y)+\tfrac{N-2}{2}\kpp_{\infty}|y|^{-(N-2)}|\aleq|y|^{-N}\qquad\forall|y|\ageq1.
\]
Using this together with \eqref{eq:2.48} and \eqref{eq:case1-R_ij},
we get 
\[
\lan[\Lmb W]_{;i},f'(W_{;i})W_{;j}\ran=-\frac{N-2}{2}\int f(W)\cdot\kpp_{\infty}\iota_{i}\iota_{j}\frac{\lmb_{i}^{\frac{N-2}{2}}\lmb_{j}^{\frac{N-2}{2}}}{|z_{i}-z_{j}|^{N-2}}\{1+\calO(\lmb_{\max}^{2}|\log\lmb_{\max}|)\}.
\]
Using \eqref{eq:def-kpp}, $D>2$, and simplifying the error, we get
\[
\frac{\lan\calV_{0;i},f'(W_{;i})W_{;j}\ran}{\|\calV_{0}\|_{L^{2}}^{2}}=-\kpp_{\infty}\kpp_{0}\iota_{i}\iota_{j}\frac{\lmb_{i}^{D}\lmb_{j}^{D}}{|z_{i}-z_{j}|^{2D}}+\calO(\lmb_{i}^{2}\lmb_{\max}^{2D}|\log\lmb_{\max}|).
\]
Recalling the definition $A_{ij}[\vec{z}]$ from \eqref{eq:def-A_jk},
we get \eqref{eq:6.16}.

\uline{Proof of \mbox{\eqref{eq:6.17}}}. When $a\in\{1,\dots,N\}$,
we recall $\calV_{a}=\rd_{a}W$ and 
\[
\nabla W(y)=-\kpp_{\infty}(N-2)y|y|^{-N}+\calO(|y|^{-(N+1)})\qquad\forall|y|\ageq1.
\]
Using this together with \eqref{eq:2.49} and \eqref{eq:case1-R_ij},
we get 
\[
\lan[\nabla W]_{;i},f'(W_{;i})W_{;j}\ran=(N-2)\kpp_{\infty}\int f(W)\cdot\iota_{i}\iota_{j}\frac{\lmb_{i}^{\frac{N}{2}}\lmb_{j}^{\frac{N-2}{2}}}{|z_{i}-z_{j}|^{N}}(z_{i}-z_{j})+\calO(\lmb_{i}^{\frac{N}{2}}\lmb_{j}^{\frac{N-2}{2}}\lmb_{\max}^{2}|\log\lmb_{\max}|).
\]
Using \eqref{eq:def-kpp_1} and simplifying the error, we get 
\[
\frac{\lan\calV_{a;i},f'(W_{;i})W_{;j}\ran}{\|\calV_{a}\|_{L^{2}}^{2}}=\kpp_{1}\kpp_{\infty}\iota_{i}\iota_{j}\frac{\lmb_{i}^{\frac{N}{2}}\lmb_{j}^{\frac{N-2}{2}}}{|z_{i}-z_{j}|^{N}}(z_{i}-z_{j})+\calO(\lmb_{i}^{2}\lmb_{\max}^{N-1}|\log\lmb_{\max}|).
\]
Recalling the definition $A_{ij}$ from \eqref{eq:def-A_jk}, we get
\eqref{eq:6.17}.
\end{proof}
\begin{lem}
\label{lem:7.6}We have 
\begin{equation}
\max_{a,i}\frac{|\frkr_{a,i}|}{\lmb_{i}}\aeq\lmb_{\max}^{2D-1}.\label{eq:case1-max-frkr_a,i}
\end{equation}
In particular, we have a rough modulation estimate 
\begin{equation}
|\lmb_{i,t}|+|z_{i,t}|+\lmb_{i}|a_{i,t}|\aleq\|g\|_{\dot{H}^{2}}+\lmb_{\max}^{2D-1},\label{eq:case1-rough-mod-est}
\end{equation}
and smallness of following quantities: 
\begin{align}
\int_{t_{n}}^{T_{n}}(\lmb_{\max}^{4D-2}+\|g\|_{\dot{H}^{2}}^{2})dt & =o_{n\to\infty}(1),\label{eq:case1-spacetime-est}\\
\sup_{t\in[t_{n},T_{n})}R^{-1}(t)\aeq\sup_{t\in[t_{n},T_{n})}\lmb_{\max}(t) & =o_{n\to\infty}(1),\label{eq:case1-lmb_max-is-o(1)}\\
\sup_{t\in[t_{n},T_{n})}\|g(t)\|_{\dot{H}^{1}} & =o_{n\to\infty}(1).\label{eq:case1-g(t)-Hdot1-is-o(1)}
\end{align}
\end{lem}

\begin{proof}
\uline{Proof of \mbox{\eqref{eq:case1-max-frkr_a,i}}}. Note that
$\chf_{a\neq0}\lmb_{i}^{-1}|\frkr_{a,i}|\aleq\lmb_{\max}^{2D}$ by
\eqref{eq:case1-frkr_a,i}. Now, by \eqref{eq:case1-frkr_0,i}, we
get 
\[
\sum_{i}\frac{|\frkr_{0,i}|^{2}}{\lmb_{i}^{2}}=\sum_{i}(A[\vec{z}]\vec{\lmb}^{D})_{i}^{2}\lmb_{i}^{2D-2}+\calO(\lmb_{\max}^{4D+1}|\log\lmb_{\max}|).
\]
By $|\vec{z}-\vec{z}^{\ast}|<\dlt_{2}$ (small), $\lmb_{\max}\aeq R^{-1}\aleq\alp$
(small), \eqref{eq:5.1}, and Lemma~\ref{lem:distance-from-degen},
we conclude \eqref{eq:case1-max-frkr_a,i}.

\uline{Proof of \mbox{\eqref{eq:case1-rough-mod-est}}}. This follows
from \eqref{eq:rough-mod-est}.

\uline{Proof of \mbox{\eqref{eq:case1-spacetime-est}}}. This follows
from \eqref{eq:spacetime-control} and \eqref{eq:case1-max-frkr_a,i}.

\uline{Proof of \mbox{\eqref{eq:case1-lmb_max-is-o(1)}}}. The
proof of \eqref{eq:case2.2-Rinv-is-o(1)-1} works the same; we omit
the proof.

\uline{Proof of \mbox{\eqref{eq:case1-g(t)-Hdot1-is-o(1)}}}. The
proof of \eqref{eq:case2-gHdot1-is-o(1)} works the same; we omit
the proof.
\end{proof}

\subsection{Refined modulation estimates}

\begin{prop}[Refined modulation estimates in Case~1]
\label{prop:case1-modulation-est}We have refined modulation estimates
\begin{align}
\Big|\frac{\lmb_{i,t}}{\lmb_{i}}+\frac{d}{dt}f_{i}-\sum_{j=1}^{J}A_{ij}^{\ast}\lmb_{i}^{D-2}\lmb_{j}^{D}\Big| & \aleq\|g\|_{\dot{H}^{2}}^{2}+(\lmb_{\max}+|\vec{z}-\vec{z}^{\ast}|)\lmb_{\max}^{2D-2},\label{eq:case1-refined-mod-est-lmb}\\
\Big|\frac{z_{i,t}}{\lmb_{i}}\Big| & \aleq\|g\|_{\dot{H}^{2}}^{2}+\lmb_{\max}^{2D-1},\label{eq:case1-refined-mod-est-z}
\end{align}
where $f_{i}=f_{i}(t)$ satisfies 
\begin{equation}
|f_{i}|\aleq\min\{\|g\|_{\dot{H}^{1}},\|g\|_{\dot{H}^{2}}+\lmb_{i}\|g\|_{\dot{H}^{1}}\}.\label{eq:case1-f_i-bound}
\end{equation}
\end{prop}

\begin{proof}
\eqref{eq:case1-refined-mod-est-lmb} and \eqref{eq:case1-refined-mod-est-z}
follow from substituting Lemma~\ref{lem:case1-leading-frkr_a,i}
and $A_{ij}[\vec{z}]=A_{ij}^{\ast}+\calO(|\vec{z}-\vec{z}^{\ast}|)$
into \eqref{eq:non-coll-refined-mod-est}. \eqref{eq:case1-f_i-bound}
is a restatement of \eqref{eq:non-coll-refined-mod-est-corr-bound}. 
\end{proof}
As in the previous section, we introduce corresponding refined modulation
parameters $\vec{\zeta}(t)\approx\vec{\lmb}(t)$ and $\vec{\xi}(t)\approx\vec{z}(t)$
as follows.
\begin{cor}[Modulation estimates in simple form integrated from $t_{n}$]
\label{cor:case1-simple-mod-est-integ-from-tn}There exist functions
$\zeta_{i}(t)$ and $\xi_{i}(t)$ on $[t_{n},T_{n})$ such that 
\begin{align}
\zeta_{i}(t) & =(1+o_{n\to\infty}(1))\lmb_{i}(t),\label{eq:case1-lmb-to-wh-lmb}\\
\xi_{i}(t) & =z_{i}(t)+o_{n\to\infty}(1),\label{eq:case1-z-to-wh-z}
\end{align}
and 
\begin{align}
\Big|\zeta_{i,t}-\sum_{j=1}^{J}A_{ij}^{\ast}\zeta_{i}^{D-1}\zeta_{j}^{D}\Big| & \aleq(|\vec{z}-\vec{z}^{\ast}|+o_{n\to\infty}(1))\zeta^{2D-1},\label{eq:case1-lmb-mod-est}\\
|\xi_{i,t}| & \aleq o_{n\to\infty}(1)\cdot\zeta^{2D-1},\label{eq:case1-z-mod-est}
\end{align}
where we denoted $\zeta\coloneqq\sqrt{\zeta_{1}^{2}+\dots+\zeta_{J}^{2}}\aeq\lmb_{\max}$.
In particular, we have 
\begin{equation}
|\rd_{t}\log\zeta|\aleq\zeta^{2D-2}.\label{eq:case1-log-lmb-est}
\end{equation}
\end{cor}

\begin{proof}
\uline{Proof of \mbox{\eqref{eq:case1-lmb-to-wh-lmb}} and \mbox{\eqref{eq:case1-lmb-mod-est}}}.
By \eqref{eq:case1-refined-mod-est-lmb}, \eqref{eq:case1-lmb_max-is-o(1)},
and \eqref{eq:case1-g(t)-Hdot1-is-o(1)}, we have 
\[
\frac{d}{dt}\Big\{\log\lmb_{i}+o_{n\to\infty}(1)\Big\}-\frac{1}{\lmb_{i}^{2}}\sum_{j=1}^{J}A_{ij}^{\ast}\lmb_{i}^{D}\lmb_{j}^{D}=\calO(\|g\|_{\dot{H}^{2}}^{2})+\calO((|\vec{z}-\vec{z}^{\ast}|+o_{n\to\infty}(1))\lmb_{\max}^{2D-2}).
\]
Integrating the term $\calO(\|g\|_{\dot{H}^{2}}^{2})$ using the spacetime
control \eqref{eq:case1-spacetime-est}, we get \eqref{eq:case1-lmb-to-wh-lmb}
and \eqref{eq:case1-lmb-mod-est}.

\uline{Proof of \mbox{\eqref{eq:case1-z-to-wh-z}} and \mbox{\eqref{eq:case1-z-mod-est}}}.
Express \eqref{eq:case1-refined-mod-est-z} in the form 
\[
|z_{i,t}|=\calO(\lmb_{i}\|g\|_{\dot{H}^{2}}^{2}+\lmb_{i}\lmb_{\max}^{2D-1})=\calO(\|g\|_{\dot{H}^{2}}^{2})+o_{n\to\infty}(1)\cdot\lmb_{\max}^{2D}.
\]
Proceeding as in the previous paragraph, we get \eqref{eq:case1-z-to-wh-z}
and \eqref{eq:case1-z-mod-est}.

\uline{Proof of \mbox{\eqref{eq:case1-log-lmb-est}}}. This is
clear from \eqref{eq:case1-lmb-mod-est}.
\end{proof}

\subsection{Control of $\vec{z}(t)$ and proof of $T_{n}=+\infty$}

In order to close the bootstrap (i.e., to show $T_{n}=+\infty$),
by \eqref{eq:case1-lmb_max-is-o(1)} and \eqref{eq:case1-g(t)-Hdot1-is-o(1)},
it remains to prove 
\[
\sup_{t\in[t_{n},T_{n})}|\vec{z}(t)-\vec{z}^{\ast}|=o_{n\to\infty}(1).
\]
The proof of this $\vec{z}$-control is a little tricky. We begin
with an inequality that exploits non-degeneracy in the equations of
$\lmb_{i,t}$.
\begin{lem}[Key inequality]
We have 
\begin{equation}
\rd_{t}\lan\vec{\zeta}^{D},A^{\ast}\vec{\zeta}^{D}\ran\aeq\zeta^{4D-2}.\label{eq:case1-A-lmb^D-lmb^D-diff-ineq}
\end{equation}
In particular, we have 
\begin{align}
\int_{\tau_{1}}^{\tau_{2}}\zeta^{4D-2}dt & \aleq\zeta^{2D}(\tau_{1})+\zeta^{2D}(\tau_{2})\text{ for any }\tau_{1},\tau_{2}\in[t_{n},T_{n})\text{ with }\tau_{1}<\tau_{2}.\label{eq:case1-int-ineq-lmb_max^4D-2}
\end{align}
\end{lem}

\begin{proof}
\uline{Proof of \mbox{\eqref{eq:case1-A-lmb^D-lmb^D-diff-ineq}}}.
By the symmetry of $A^{\ast}$, \eqref{eq:case1-lmb-mod-est}, and
$|\vec{z}-\vec{z}^{\ast}|<\dlt_{2}$, we get 
\begin{align*}
\rd_{t}\lan\vec{\zeta}^{D},A^{\ast}\vec{\zeta}^{D}\ran=2\lan A^{\ast}\vec{\zeta}^{D},\rd_{t}(\vec{\zeta}^{D})\ran & =2D\tsum i{}(A^{\ast}\vec{\zeta}^{D})_{i}\cdot\{\zeta_{i}^{2D-2}(A^{\ast}\vec{\zeta}^{D})_{i}+\dlt_{2}\cdot\calO(\zeta^{3D-2})\}\\
 & =2D\tsum i{}(A^{\ast}\vec{\zeta}^{D})_{i}^{2}\zeta_{i}^{2D-2}+\dlt_{2}\cdot\calO(\zeta^{4D-2}).
\end{align*}
Applying \eqref{eq:case1-calD-non-degen} to the first term, we get
\eqref{eq:case1-A-lmb^D-lmb^D-diff-ineq}. 

\uline{Proof of \mbox{\eqref{eq:case1-int-ineq-lmb_max^4D-2}}}.
This follows from integrating \eqref{eq:case1-A-lmb^D-lmb^D-diff-ineq}
and $\lan\vec{\zeta}^{D},A^{\ast}\vec{\zeta}^{D}\ran=\calO(\zeta^{2D})$.
\end{proof}
\begin{lem}[Control of $\vec{z}(t)$]
\label{lem:case1-control-z}We have 
\begin{align}
\int_{\tau_{1}}^{\tau_{2}}\zeta^{2D-1}dt & \aleq\zeta(\tau_{1})+\zeta(\tau_{2})\text{ for any }\tau_{1},\tau_{2}\in[t_{n},T_{n})\text{ with }\tau_{1}<\tau_{2}.\label{eq:case1-int-ineq-lmb_max^2D-1}
\end{align}
In particular, we have 
\begin{equation}
\sup_{t\in[t_{n},T_{n})}|\vec{z}(t)-\vec{z}^{\ast}|+\int_{t_{n}}^{T_{n}}\zeta^{2D-1}dt=o_{n\to\infty}(1).\label{eq:case1-z-control}
\end{equation}
\end{lem}

\begin{proof}
\uline{Proof of \mbox{\eqref{eq:case1-int-ineq-lmb_max^2D-1}}}.
Choose $\tau'\in[\tau_{1},\tau_{2}]$ such that $\zeta(\tau')=\inf_{\tau\in[\tau_{1},\tau_{2}]}\zeta(\tau)$.
On the interval $[\tau_{1},\tau']$, consider the function $\vphi(t)=\zeta^{2D}(\tau')+\int_{t}^{\tau'}\zeta^{4D-2}(t')dt'>0$.
By \eqref{eq:case1-int-ineq-lmb_max^4D-2}, $\vphi(t)$ satisfies
\[
\vphi(t)\aleq\zeta^{2D}(t)\qquad\text{and}\qquad-\vphi_{t}(t)=\zeta^{4D-2}(t).
\]
This implies 
\begin{equation}
\int_{\tau_{1}}^{\tau'}\zeta^{2D-1}dt\aleq\int_{\tau_{1}}^{\tau'}\frac{-\vphi_{t}}{\vphi^{(2D-1)/2D}}dt\aeq\int_{\tau_{1}}^{\tau'}(-\vphi^{1/2D})_{t}dt\aleq\vphi^{1/2D}(\tau_{1})\aleq\zeta(\tau_{1}).\label{eq:case1-5.37}
\end{equation}
On the interval $[\tau',\tau_{2}]$, similarly consider the function
$\vphi(t)=\zeta^{2D}(\tau')+\int_{\tau'}^{t}\zeta^{4D-2}(t')dt'>0$
and observe 
\[
\vphi(t)\aleq\zeta^{2D}(t)\qquad\text{and}\qquad\vphi_{t}(t)=\zeta^{4D-2}(t).
\]
This implies 
\begin{equation}
\int_{\tau'}^{\tau_{2}}\zeta^{2D-1}dt\aleq\int_{\tau'}^{\tau_{2}}\frac{\vphi_{t}}{\vphi^{(2D-1)/2D}}dt\aeq\int_{\tau'}^{\tau_{2}}(\vphi^{1/2D})_{t}dt\aleq\vphi^{1/2D}(\tau_{2})\aleq\zeta(\tau_{2}).\label{eq:case1-5.38}
\end{equation}
\eqref{eq:case1-int-ineq-lmb_max^2D-1} now follows from \eqref{eq:case1-5.37}
and \eqref{eq:case1-5.38}.

\uline{Proof of \mbox{\eqref{eq:case1-z-control}}}. By \eqref{eq:case1-z-mod-est}
and \eqref{eq:case1-int-ineq-lmb_max^2D-1}, we get 
\[
\sup_{t\in[t_{n},T_{n})}|\vec{\xi}(t)-\vec{z}^{\ast}|+\int_{t_{n}}^{T_{n}}\zeta^{2D-1}dt\aleq|\vec{\xi}(t_{n})-\vec{z}^{\ast}|+\zeta(t_{n})+\limsup_{t\to T_{n}}\zeta(t).
\]
Applying \eqref{eq:case1-z-to-wh-z}, $|\vec{z}(t_{n})-\vec{z}^{\ast}|=o_{n\to\infty}(1)$,
and \eqref{eq:case1-lmb_max-is-o(1)} to the above, we get \eqref{eq:case1-z-control}. 
\end{proof}
\begin{cor}[Closing bootstrap]
\label{cor:case1-closing-bootstrap}We have $T_{n}=+\infty$ for
all large $n$. Moreover, 
\begin{equation}
\|g(t)\|_{\dot{H}^{1}}+\lmb_{\max}(t)+|\vec{z}(t)-\vec{z}^{\ast}|\to0\qquad\text{ as }t\to+\infty.\label{eq:case1-all-go-to-zero}
\end{equation}
\end{cor}

\begin{proof}
For the proof of $T_{n}=+\infty$, recall $T_{n}=T_{n}(1,\dlt_{2},\alp)\leq\Ttbexit_{n}(\alp)\leq T_{n}'$
and the exit condition for $T_{n}'$ (see Proposition~\ref{prop:case-separation}).
By \eqref{eq:case1-lmb_max-is-o(1)} and \eqref{eq:case1-z-control},
$T_{n}=\Ttbexit_{n}(\alp)$. By \eqref{eq:case1-lmb_max-is-o(1)}
and \eqref{eq:case1-g(t)-Hdot1-is-o(1)}, $\Ttbexit_{n}(\alp)=T_{n}'$.
By all these controls \eqref{eq:case1-lmb_max-is-o(1)}, \eqref{eq:case1-g(t)-Hdot1-is-o(1)},
\eqref{eq:case1-z-control}, and the exit condition for $T_{n}'$,
we get $T_{n}'=+\infty$. These controls also imply \eqref{eq:case1-all-go-to-zero}.
\end{proof}
\begin{cor}[Modulation estimates in simple form integrated from $+\infty$]
There exist functions $\zeta_{i}(t)$ and $\xi_{i}(t)$ defined for
all large $t$ such that 
\begin{align}
\zeta_{i}(t) & =(1+o(1))\lmb_{i}(t),\label{eq:case1-lmb-to-wh-lmb-1}\\
\xi_{i}(t) & =z_{i}(t)+o(1)\cdot\sup_{\tau\in[t,+\infty)}\zeta_{\max}(\tau),\label{eq:case1-z-to-wh-z-1}
\end{align}
and 
\begin{align}
\Big|\frac{\zeta_{i,t}}{\zeta_{i}}-\sum_{j=1}^{J}(A_{ij}^{\ast}+o(1))\zeta_{i}^{D-2}\zeta_{j}^{D}\Big| & \aleq(\zeta+|\vec{z}-\vec{z}^{\ast}|)\zeta^{2D-2},\label{eq:case1-lmb-mod-est-1}\\
|\xi_{i,t}| & \aleq\zeta^{2D},\label{eq:case1-z-mod-est-1}
\end{align}
where we denoted $\zeta\coloneqq\sqrt{\zeta_{1}^{2}+\dots+\zeta_{J}^{2}}\aeq\lmb_{\max}$.
Moreover, we have \eqref{eq:case1-log-lmb-est}, \eqref{eq:case1-A-lmb^D-lmb^D-diff-ineq},
and \eqref{eq:case1-int-ineq-lmb_max^2D-1} valid as well.
\end{cor}

\begin{proof}
One basically follows the proof of Corollary~\ref{cor:case1-simple-mod-est-integ-from-tn}.
This time, however, we use $T_{n}=+\infty$ and keep the error terms
of the refined modulation estimates \eqref{eq:case1-refined-mod-est-lmb}
and \eqref{eq:case1-refined-mod-est-z}.

\uline{Proof of \mbox{\eqref{eq:case1-z-to-wh-z-1}} and \mbox{\eqref{eq:case1-z-mod-est-1}}}.
Express \eqref{eq:case1-refined-mod-est-z} in the form 
\[
|z_{i,t}|=\calO(\lmb_{i}\|g\|_{\dot{H}^{2}}^{2}+\lmb_{i}\lmb_{\max}^{2D-1}).
\]
This time, we integrate the first term $\calO(\lmb_{i}\|g\|_{\dot{H}^{2}}^{2})$
by adding 
\[
\int_{t}^{\infty}\calO(\lmb_{i}(\tau)\|g(\tau)\|_{\dot{H}^{2}}^{2})d\tau=\sup_{\tau\ge t}\lmb_{\max}(\tau)\cdot\calO(\int_{t}^{\infty}\|g(\tau)\|_{\dot{H}^{2}}^{2}d\tau)=o(1)\cdot\sup_{\tau\ge t}\lmb_{\max}(\tau)
\]
to $z_{i}$. This gives \eqref{eq:case1-z-to-wh-z-1} and \eqref{eq:case1-z-mod-est-1}.

\uline{Proof of \mbox{\eqref{eq:case1-lmb-to-wh-lmb-1}} and \mbox{\eqref{eq:case1-lmb-mod-est-1}}}.
By \eqref{eq:case1-refined-mod-est-lmb} and \eqref{eq:case1-g(t)-Hdot1-is-o(1)},
we have 
\[
\frac{d}{dt}\Big\{\log\lmb_{i}+o_{n\to\infty}(1)\Big\}-\frac{1}{\lmb_{i}^{2}}\sum_{j=1}^{J}A_{ij}^{\ast}\lmb_{i}^{D}\lmb_{j}^{D}=\calO(\|g\|_{\dot{H}^{2}}^{2})+\calO((\lmb_{\max}+|\vec{z}-\vec{z}^{\ast}|)\lmb_{\max}^{2D-2}).
\]
This time, we integrate the term $\calO(\|g\|_{\dot{H}^{2}}^{2})$
by adding $\int_{t}^{\infty}\calO(\|g(\tau)\|_{\dot{H}^{2}}^{2})d\tau=o(1)$
(using the spacetime control \eqref{eq:case1-spacetime-est}) to $\log\lmb_{i}$.
Replacing $\lmb_{i}$ by $\zeta_{i}$ as well, we get \eqref{eq:case1-lmb-to-wh-lmb-1}
and \eqref{eq:case1-lmb-mod-est-1}.
\end{proof}

\subsection{Dynamics of Case~1: proof of Proposition~\ref{prop:case1-main}}

The continuous-in-time resolution with $\vec{z}(t)\to\vec{z}^{\ast}$
was already established by Corollary~\ref{cor:case1-closing-bootstrap}.
It remains to prove the second and third items of Proposition~\ref{prop:case1-main}. 
\begin{proof}[\uline{Proof of \mbox{\eqref{eq:main-thm-nondegen-lmb-rate}}}]
Here, we prove $\lmb_{i}(t)\aeq t^{-1/(2D-2)}$ for all $i\in\setJ$.

\uline{Step 1: Lower bound for \mbox{$\lmb_{\max}$}}. In this
step, we show 
\begin{equation}
\lmb_{\max}(t)\ageq t^{-1/(2D-2)}.\label{eq:case1-lmb_max-lower-bd}
\end{equation}
Since $|\rd_{t}\log\zeta|\aleq\zeta^{2D-2}$, we have $|\rd_{t}(\zeta^{-(2D-2)})|\aleq1$.
Integrating this gives $\zeta^{-(2D-2)}\aleq t$ as $t\to\infty$.
As $\zeta\aeq\lmb_{\max}$, we get \eqref{eq:case1-lmb_max-lower-bd}. 

\uline{Step 2: Sequential upper bound for \mbox{$\lmb_{\max}$}}.
In this step, we claim 
\begin{equation}
\lmb_{\max}(\tau_{n})\aleq\tau_{n}^{-1/(2D-2)}\quad\text{for some }\tau_{n}\to+\infty.\label{eq:case1-lmb_max-seq-upper-bd}
\end{equation}
By \eqref{eq:case1-int-ineq-lmb_max^2D-1} and $\zeta(\tau)\to0$
as $\tau\to+\infty$, we obtain $\int_{t}^{\infty}\zeta^{2D-1}(t')dt'\aleq\zeta(t)$
for all large $t$. Letting $\vphi(t)\coloneqq\int_{t}^{\infty}\zeta^{2D-1}(t')dt'$,
this implies $-\vphi_{t}=\zeta^{2D-1}\ageq\vphi^{2D-1}$, whose integration
gives $\vphi(t)\aleq t^{-1/(2D-2)}$. Expanding the definition of
$\vphi(t)$, we get 
\[
\int_{t}^{\infty}\zeta^{2D-1}(t')dt'\aleq t^{-1/(2D-2)}.
\]
This gives the sequential bound \eqref{eq:case1-lmb_max-seq-upper-bd}.

\uline{Step 3: Upper bound for \mbox{$\lmb_{\max}(t)$}}. In this
step, we show 
\begin{equation}
\lmb_{\max}(t)\aeq t^{-1/(2D-2)}.\label{eq:case1-lmb_max-rate}
\end{equation}
By \eqref{eq:case1-lmb_max-lower-bd}, it remains to show $\lmb_{\max}(t)\aleq t^{-1/(2D-2)}$. 

Suppose not. Denote $\nu(t)\coloneqq t^{1/(2D-2)}\zeta(t)\ageq1$.
By the sequential upper bound \eqref{eq:case1-lmb_max-seq-upper-bd},
for any large constants $K_{0}$ and $K$ with $K>K_{0}$, there exist
four times $t_{K_{0}}<t_{K}<t_{2K}<t_{K_{0}}'$ with 
\[
\nu(t_{K_{0}})=\nu(t_{K_{0}}')=K_{0},\quad\nu(t_{K})=K,\quad\nu(t_{2K})=2K,
\]
and 
\begin{align*}
\nu(t) & \geq K_{0}\quad\text{for all }t\in[t_{K_{0}},t_{K_{0}}'],\\
K\leq\nu(t) & \leq2K\quad\text{for all }t\in[t_{K},t_{2K}].
\end{align*}
Let $t\in[t_{K_{0}},t_{K_{0}}'${]}. By $\frac{d}{dt}\calO(\zeta^{2D})\aeq\zeta^{4D-2}$
(see \eqref{eq:case1-A-lmb^D-lmb^D-diff-ineq}), we have $(t\frac{d}{dt}-\frac{D}{D-1})\calO(\nu^{2D})\aeq\nu^{4D-2}$.
Choosing $K_{0}$ sufficiently large, $\nu^{4D-2}\geq K_{0}^{2D-2}\nu^{2D}\geq\frac{D}{D-1}\calO(\nu^{2D})$,
so $t\frac{d}{dt}\calO(\nu^{2D})\aeq\nu^{4D-2}$. Integrating this
on $[t_{K_{0}},t_{K_{0}}']$, we have 
\begin{equation}
\int_{t_{K_{0}}}^{t_{K_{0}}'}\nu^{4D-2}(t)\frac{dt}{t}\aleq\nu^{2D}(t_{K_{0}})+\nu^{2D}(t_{K_{0}}')\aleq K_{0}^{2D}.\label{eq:7.40}
\end{equation}
On the other hand, let $\in[t_{K},t_{2K}]$. Since $|\rd_{t}\log\zeta|\aleq\zeta^{2D-2}$
and $\nu\ageq1$, we have $t\rd_{t}\log\nu=\calO(\nu^{2D-2}+1)=\calO(\nu^{2D-2})$.
Hence, $t\rd_{t}\nu^{2D}=\calO(\nu^{4D-2})$. Integrating this on
$[t_{K},t_{2K}]$, we have 
\begin{equation}
K^{2D}\aeq\nu^{2D}(t_{2K})-\nu^{2D}(t_{K})\aleq\int_{t_{K}}^{t_{2K}}\nu^{4D-2}(t)\frac{dt}{t}.\label{eq:7.41}
\end{equation}
Since $K$ is allowed to be arbitrarily large, \eqref{eq:7.41} contradicts
to \eqref{eq:7.40}. This completes the proof of \eqref{eq:case1-lmb_max-rate}.

\uline{Step 4: Upper bound for \mbox{$|\vec{z}(t)-\vec{z}^{\ast}|$}}.
In this step, we claim that 
\begin{equation}
|\vec{z}(t)-\vec{z}^{\ast}|=o(1)\cdot\lmb_{\max}(t).\label{eq:7.42}
\end{equation}
Indeed, we use \eqref{eq:case1-z-to-wh-z-1} and \eqref{eq:case1-z-mod-est-1}
to get 
\[
|\vec{z}(t)-\vec{z}^{\ast}|\aleq o(1)\cdot\Big(\sup_{\tau\geq t}\lmb_{\max}(\tau)+\int_{t}^{\infty}\zeta^{2D-1}(\tau)d\tau\Big).
\]
Applying \eqref{eq:case1-lmb_max-rate}, we conclude \eqref{eq:7.42}.

\uline{Step 5: Proof of \mbox{\eqref{eq:main-thm-nondegen-lmb-rate}}}.
In this last step, we show $\lmb_{1}(t)\aeq\dots\aeq\lmb_{J}(t)\aeq t^{-1/(2D-2)}$.
This step requires the refined error bounds as written in \eqref{eq:case1-lmb-mod-est-1}
and \eqref{eq:case1-z-mod-est-1}. By \eqref{eq:case1-lmb_max-rate},
it suffices to show the lower bound $\lmb_{i}(t)\ageq t^{-1/(2D-2)}$
for all $i\in\setJ$.

We use \eqref{eq:case1-lmb-mod-est-1}, \eqref{eq:7.42}, and \eqref{eq:case1-lmb_max-rate}
to get 
\[
\Big|\frac{\zeta_{i,t}}{\zeta_{i}}-\sum_{j=1}^{J}(A_{ij}^{\ast}+o(1))\zeta_{i}^{D-2}\zeta_{j}^{D}\Big|\aleq t^{-(2D-1)/(2D-2)}.
\]
Notice that the right hand side is integrable as $t\to\infty$. Thus
we can introduce another function $\wh{\zeta}_{i}(t)=(1+o(1))\cdot\zeta_{i}(t)$
as $t\to\infty$ such that 
\begin{equation}
\frac{\wh{\zeta}_{i,t}}{\wh{\zeta}_{i}}=\sum_{j=1}^{J}(A_{ij}^{\ast}+o(1))\wh{\zeta}_{i}^{D-2}\wh{\zeta}_{j}^{D}.\label{eq:case1-wh-zeta-diff}
\end{equation}
Dividing both sides by $\wh{\zeta}_{i}^{D-2}$ and applying \eqref{eq:case1-lmb_max-rate},
we obtain 
\[
|\rd_{t}(\wh{\zeta}_{i}^{-(D-2)})|\aleq\lmb_{\max}^{D}\aleq t^{-D/(2D-2)}.
\]
Integrating this gives the lower bound $\wh{\zeta}_{i}(t)\ageq t^{-1/(2D-2)}$
as desired. This completes the proof of \eqref{eq:main-thm-nondegen-lmb-rate}.
\end{proof}
\begin{proof}[\uline{Proof of the non-degeneracy of \mbox{$\{(\iota_{1},z_{1}^{\ast}),\dots,(\iota_{J},z_{J}^{\ast})\}$}}]
Suppose $\{(\iota_{1},z_{1}^{\ast}),\dots,(\iota_{J},z_{J}^{\ast})\}$
were to be degenerate. Then, there exists $0\neq\vec{\frkc}\in\ker(A^{\ast})\cap[0,\infty)^{J}$.
The idea is to study the quantity $\lan\vec{\frkc},\vec{\wh{\zeta}}^{-(D-2)}\ran$,
where $\wh{\zeta}_{i}$ was defined in Step~5 of the proof of \eqref{eq:main-thm-nondegen-lmb-rate}.
Since all $\lmb_{i}(t)$ are equivalent by \eqref{eq:main-thm-nondegen-lmb-rate},
we have 
\[
\lan\vec{\frkc},\vec{\wh{\zeta}}^{-(D-2)}\ran\aeq\lmb_{\max}^{-(D-2)}\aeq t^{(D-2)/(2D-2)}.
\]
On the other hand, since $A^{\ast}\vec{\frkc}=0$ and $A^{\ast}$
is symmetric, \eqref{eq:case1-wh-zeta-diff} implies 
\[
\rd_{t}\lan\vec{\frkc},\vec{\wh{\zeta}}^{-(D-2)}\ran=\lan\vec{\frkc},A^{\ast}\vec{\wh{\zeta}}^{D}\ran+o(\lmb_{\max}^{D})=o(\lmb_{\max}^{D})=o(t^{-D/(2D-2)}).
\]
Integrating this gives $\lan\vec{\frkc},\vec{\wh{\zeta}}^{-(D-2)}\ran=o(t^{(D-2)/(2D-2)})$,
which contradicts to \eqref{eq:main-thm-nondegen-lmb-rate}. Therefore,
$\{(\iota_{1},z_{1}^{\ast}),\dots,(\iota_{J},z_{J}^{\ast})\}$ is
non-degenerate.
\end{proof}
\begin{proof}[\uline{Proof of convergence of the ratio vector to a set}]
Here, we prove that the ratio vector $t^{1/(2D-2)}\vec{\lmb}(t)$
converges to a connected component of the set \eqref{eq:main-thm-nondegen-ratio-set}.
The modulation estimate \eqref{eq:case1-wh-zeta-diff} seem insufficient
for this purpose; we will derive a stronger version \eqref{eq:case1-td-zeta-diff}
below. 

Recall \eqref{eq:case1-refined-mod-est-lmb} combined with \eqref{eq:7.42}:
\[
\Big|\frac{\lmb_{i,t}}{\lmb_{i}}+\frac{d}{dt}f_{i}-\sum_{j=1}^{J}A_{ij}^{\ast}\lmb_{i}^{D-2}\lmb_{j}^{D}\Big|\aleq\|g\|_{\dot{H}^{2}}^{2}+\lmb_{\max}^{2D-1}.
\]
Using $\|g\|_{\dot{H}^{1}}=o(1)$ and \eqref{eq:case1-f_i-bound}
for $f_{i}$, we see that 
\begin{equation}
\td{\zeta}_{i}(t)\coloneqq\lmb_{i}(t)e^{f_{i}(t)}=\lmb_{i}(t)\cdot\big(1+\calO(\min\{o(1),\|g\|_{\dot{H}^{2}}+o(\lmb_{\max})\})\big)\label{eq:td-zeta-ptwise}
\end{equation}
satisfies 
\[
\Big|\frac{\td{\zeta}_{i,t}}{\td{\zeta}_{i}}-\sum_{j=1}^{J}A_{ij}^{\ast}\td{\zeta}_{i}^{D-2}\td{\zeta}_{j}^{D}\Big|\aleq\|g\|_{\dot{H}^{2}}^{2}+\lmb_{\max}^{2D-1}+(\|g\|_{\dot{H}^{2}}+o(\lmb_{\max}))\cdot\lmb_{\max}^{2D-2}\aleq\|g\|_{\dot{H}^{2}}^{2}+\lmb_{\max}^{2D-1}.
\]
The right hand side is integrable by \eqref{eq:case1-spacetime-est}
and $\lmb_{\max}\aeq t^{-1/(2D-2)}$. Thus we have proved
\begin{equation}
\td{\zeta}_{i,t}=\sum_{j=1}^{J}A_{ij}^{\ast}\td{\zeta}_{i}^{D-1}\td{\zeta}_{j}^{D}+\td{\zeta}_{i}\cdot L_{1},\label{eq:case1-td-zeta-diff}
\end{equation}
where here and below $L_{1}=L_{1}(t)$ denotes an integrable function
in time. 

Now, introduce renormalized variables 
\[
\mu_{i}(t)\coloneqq t^{1/(2D-2)}\td{\zeta}_{i}(t).
\]
We have 
\begin{align}
\mu_{i} & =t^{1/(2D-2)}\lmb_{i}\cdot(1+o(1))\aeq1,\label{eq:case1-mu_i-ptwise}\\
\mu_{i,t} & =\frac{1}{t}\Big\{\sum_{j=1}^{J}A_{ij}^{\ast}\mu_{i}^{D-1}\mu_{j}^{D}+\frac{\mu_{i}}{(2D-2)}\Big\}+L_{1}.\label{eq:case1-mu_i-evol-eq-prelim}
\end{align}
Introducing a functional (cf.~\cite[p.295]{CortazarDelPinoMusso2020JEMS})
\[
I(\vec{\mu})\coloneqq-\frac{1}{2D}\sum_{i,j=1}^{J}A_{ij}^{\ast}\mu_{i}^{D}\mu_{j}^{D}-\frac{1}{4D-4}\sum_{i=1}^{J}\mu_{i}^{2},\qquad\vec{\mu}\in(0,\infty)^{J},
\]
we can rewrite \eqref{eq:case1-mu_i-evol-eq-prelim} as 
\begin{equation}
\mu_{i,t}=-t^{-1}\rd_{\mu_{i}}I(\vec{\mu})+L_{1}.\label{eq:case1-mu_i-evol-eq}
\end{equation}
This system can be viewed as an almost gradient flow associated to
$I(\vec{\mu})$ after introducing a renormalized time. We use this
almost gradient structure to prove the convergence of ratios.

We show $\nabla I(\vec{\mu}(t))\to0$ as $t\to+\infty$. Observe 
\[
\frac{d}{dt}I(\vec{\mu})=\lan\nabla I,-t^{-1}\nabla I+L_{1}\ran=-t^{-1}|\nabla I|^{2}+L_{1}.
\]
Integrating this gives $\int_{t_{0}}^{\infty}|\nabla I(\vec{\mu}(t))|^{2}\frac{dt}{t}<+\infty$.
In particular, $\nabla I(\vec{\mu}(\tau_{n}))\to0$ for some $\tau_{n}\to+\infty$.
Now, observe 
\[
\frac{d}{dt}|\nabla I|^{2}=2\lan\nabla I,(\nabla^{2}I)(-t^{-1}\nabla I+L_{1})\ran=\calO(t^{-1}|\nabla I|^{2})+L_{1}.
\]
Since the right hand side is integrable, we see that $|\nabla I(\vec{\mu}(t))|^{2}$
converges as $t\to+\infty$. As we know $\nabla I(\vec{\mu}(\tau_{n}))\to0$,
we conclude $\nabla I(\vec{\mu}(t))\to0$ as $t\to+\infty$.

Since $\mu_{1}\aeq\dots\aeq\mu_{J}\aeq1$, $t\mapsto\vec{\mu}(t)$
is a continuous curve whose trajectory has compact closure in $(0,\infty)^{J}$.
This fact together with $\nabla I(\vec{\mu}(t))\to0$ implies that
$\vec{\mu}(t)$ converges a connected component of the set $\{\vec{\mu}\in(0,\infty)^{J}:\nabla I(\vec{\mu})=0\}$,
which is the same as \eqref{eq:main-thm-nondegen-ratio-set}. 
\end{proof}

\subsection{Proof of the first item of Theorem~\ref{thm:main-lmb-to-zero-classification}}
\begin{proof}
Let $u(t)$ be as in Assumption~\ref{assumption:sequential-on-param}
and suppose Theorem~\ref{thm:main-one-bubble-tower-classification}
does not apply. Assume the configuration $\{(\iota_{1},z_{1}^{\ast}),\dots,(\iota_{J},z_{j}^{\ast})\}$
is totally non-degenerate. Apply the case separation Proposition~\ref{prop:case-separation}
to $u(t)$. Note that Case~2 is excluded because we assumed Theorem~\ref{thm:main-one-bubble-tower-classification}
does not apply. Case~3 is also excluded by the definition of total
non-degeneracy. Therefore, only Case~1 is possible. Applying Proposition~\ref{prop:case1-main}
gives the desired conclusion.
\end{proof}

\section{\label{sec:Case3}Modulation analysis in Case 3 under minimal degeneracy}

In this section, we consider the dynamics in \textbf{Case~3} of Proposition~\ref{prop:case-separation}
\emph{under the additional assumption} that $\{(\iota_{1},z_{1}^{\ast}),\dots,(\iota_{J},z_{J}^{\ast})\}$
is minimally degenerate satisfying \eqref{eq:case3-v_i-not-all-zero}.
We also prove the second item of Theorem~\ref{thm:main-lmb-to-zero-classification}
at the end of this section. The goal of this section is the following. 
\begin{prop}[Main proposition in Case 3 under minimal degeneracy]
\label{prop:case3-main}Consider the solution $u(t)$ belonging to
\textbf{Case~3} of Proposition~\ref{prop:case-separation}. Recall
the notation therein. Assume further that $\{(\iota_{1},z_{1}^{\ast}),\dots,(\iota_{J},z_{J}^{\ast})\}$
is minimally degenerate satisfying \eqref{eq:case3-v_i-not-all-zero}.
Then, the following hold.
\begin{itemize}
\item (Continuous-in-time resolution) $T_{n}'=+\infty$ for all large $n$
and the resolution \eqref{eq:main1-conti-time-res-with-z-conv} holds. 
\item (Asymptotics of $\lmb_{i}(t)$) As $t\to+\infty$, we have 
\begin{equation}
\lmb_{i}(t)=\Big\{\Big(\frac{\kpp_{0}}{\kpp_{1}}\frac{\sum_{j=1}^{J}\frkc_{j}^{4/(N-2)}}{(N-3)^{2}(2N-4)\sum_{j=1}^{J}|v_{j}^{\ast}|^{2}}\Big)^{\frac{1}{2N-6}}\frkc_{i}^{2/(N-2)}+o(1)\Big\}\cdot t^{-1/(N-3)}.\label{eq:case3-lmb-rate}
\end{equation}
\end{itemize}
\end{prop}

Let us denote $t_{n}=t_{n}'$, where $t_{n}'$ is as in \textbf{Case~3}
of Proposition~\ref{prop:case-separation}. Let $\dlt_{1}\in(0,1)$,
$\dlt_{2}\in(0,\dlt_{2}^{\ast}]$, and $\alp\in(0,\alp^{\ast}]$ be
three \emph{small parameters} that can shrink in the course of the
proof satisfying the parameter dependence 
\[
\dlt_{2}\ll\dlt_{1}\ll1\quad\text{and}\quad\alp\ll1.
\]
As in Section~\ref{sec:Case1}, introduce 
\[
T_{n}\coloneqq T_{n}(1,\dlt_{2},\alp)\in(t_{n},\Ttbexit_{n}(\alp)].
\]
Recall \eqref{eq:5.29}; we have (implicit constant independent of
$n$)
\begin{equation}
\lmb_{\max}\aeq\lmb_{\secmax}\quad\text{on }[t_{n},T_{n}).\label{eq:8.2}
\end{equation}
We denote $\lmb(t)\coloneqq(\lmb_{1}^{2D}(t)+\dots+\lmb_{J}^{2D}(t))^{\frac{1}{2D}}$. 

We also need to improve the sequential convergence of the ratio 
\begin{equation}
\frac{|A^{\ast}\vec{\lmb}^{D}|}{\lmb^{D}}(t_{n})=o_{n\to\infty}(1)\label{eq:7.12-2}
\end{equation}
to the continuous-in-time convergence. For this purpose, we introduce
an intermediate time 
\begin{equation}
T_{n}''\coloneqq\sup\{\tau\in[t_{n},T_{n}):\frac{|A^{\ast}\vec{\lmb}^{D}(t)|}{\lmb^{D}(t)}<\dlt_{1}\text{ for all }t\in[t_{n},\tau]\}\in(t_{n},T_{n}].\label{eq:case3-def-T_n''}
\end{equation}
We will show later in Section~\ref{subsec:case3-ratio} that $T_{n}''=T_{n}$
for all large $n$. The proof of $T_{n}''=T_{n}$ is delicate because
the bootstrap hypothesis on the ratio in \eqref{eq:case3-def-T_n''}
does not seem to propagate. We instead derive a contradiction if $T_{n}''<T_{n}$,
which relies on the classification Proposition~\ref{prop:case1-main}. 

This section is more involved than Section~\ref{sec:Case1} because
the equations of $\lmb_{i,t}$ can degenerate. We need better estimates
on the errors of the $\lmb_{i,t}$ equations compared to Section~\ref{sec:Case1}.
A non-degeneracy in this section stems from the condition \eqref{eq:case3-v_i-not-all-zero},
which roughly says that the equations of $z_{i,t}$ do not degenerate
whenever the equations of $\lmb_{i,t}$ degenerate. 

\subsection{Basic observations on $[t_{n},T_{n})$}

By \eqref{eq:8.2} and Remark~\ref{rem:7.2}, we can use Lemma~\ref{lem:case1-R_ij},
Lemma~\ref{lem:case1-tdU-est}, and Lemma~\ref{lem:case1-leading-frkr_a,i}
in this section. Lemma~\ref{lem:7.6} becomes weaker in this degenerate
regime.
\begin{lem}[Spacetime bound]
\label{lem:case3-spacetime}We have 
\begin{equation}
\max_{a,i}\frac{|\frkr_{a,i}|}{\lmb_{i}}\aeq\lmb^{D-1}|A[\vec{z}]\vec{\lmb}^{D}|+\lmb^{2D}.\label{eq:case3-max-frkr_a,i}
\end{equation}
In particular, we have a rough modulation estimate
\begin{equation}
|\lmb_{i,t}|+|z_{i,t}|+\lmb_{i}|a_{i,t}|\aleq\|g\|_{\dot{H}^{2}}+\lmb^{D-1}|A[\vec{z}]\vec{\lmb}^{D}|+\lmb^{2D}\label{eq:case3-rough-mod-est}
\end{equation}
and a spacetime estimate 
\begin{align}
\int_{t_{n}}^{T_{n}}(\lmb^{2D-2}|A[\vec{z}]\vec{\lmb}^{D}|^{2}+\lmb^{4D}+\|g\|_{\dot{H}^{2}}^{2})dt & =o_{n\to\infty}(1),\label{eq:case3-spacetime-est}\\
\sup_{t\in[t_{n},T_{n})}R^{-1}(t)\aeq\sup_{t\in[t_{n},T_{n})}\lmb(t) & =o_{n\to\infty}(1),\label{eq:case3-lmb_max=00003Do(1)}\\
\sup_{t\in[t_{n},T_{n})}\|g(t)\|_{\dot{H}^{1}} & =o_{n\to\infty}(1).\label{eq:case3-g(t)-Hdot1=00003Do(1)}
\end{align}
\end{lem}

\begin{proof}
\uline{Proof of \mbox{\eqref{eq:case3-max-frkr_a,i}}}. By \eqref{eq:case1-frkr_0,i}
and \eqref{eq:case1-frkr_a,i}, we get 
\begin{align*}
\sum_{a,i}\frac{\frkr_{a,i}^{2}}{\lmb_{i}^{2}} & =\sum_{i}\lmb_{i}^{2D-2}(A[\vec{z}]\vec{\lmb}^{D})_{i}^{2}+\sum_{i}\lmb_{i}^{2D}\Big|\frac{\kpp_{1}}{\kpp_{0}}\sum_{j}\frac{A_{ij}[\vec{z}]\lmb_{j}^{D}(z_{i}-z_{j})}{|z_{i}-z_{j}|^{2}}\Big|^{2}\\
 & \quad+\calO((\lmb^{D-1}|A[\vec{z}]\vec{\lmb}^{D}|+\lmb^{2D})\cdot\lmb^{2D+1}|\log\lmb|).
\end{align*}
By \eqref{eq:8.2}, $\lmb\aeq R^{-1}<\alp$ (small), and $|\vec{z}-\vec{z}^{\ast}|<\dlt_{2}$
(small), applying Lemma~\ref{lem:case3-non-degen-quantity} gives
\eqref{eq:case3-max-frkr_a,i}.

\uline{Proof of \mbox{\eqref{eq:case3-rough-mod-est}}}. This follows
from \eqref{eq:rough-mod-est} and \eqref{eq:case3-max-frkr_a,i}.

\uline{Proof of \mbox{\eqref{eq:case3-spacetime-est}}}. This follows
from \eqref{eq:spacetime-control} and \eqref{eq:case3-max-frkr_a,i}.

\uline{Proof of \mbox{\eqref{eq:case3-lmb_max=00003Do(1)}}}. For
any $i\in\setJ$, we use \eqref{eq:case3-rough-mod-est} to have 
\[
|(\lmb_{i}^{2D+2})_{t}|\aleq\lmb_{i}^{2D+1}|\lmb_{i,t}|\aleq\lmb_{i}^{2D+1}(\|g\|_{\dot{H}^{2}}+\lmb^{2D-1})\aleq\|g\|_{\dot{H}^{2}}^{2}+\lmb^{4D}.
\]
Integrating this on $[t_{n},T_{n})$ with $\lmb(t_{n})=o_{n\to\infty}(1)$
and \eqref{eq:case3-spacetime-est}, we conclude \eqref{eq:case3-lmb_max=00003Do(1)}.

\uline{Proof of \mbox{\eqref{eq:case3-g(t)-Hdot1=00003Do(1)}}}.
The proof of \eqref{eq:case2-gHdot1-is-o(1)} works the same; we omit
the proof.
\end{proof}
As a consequence of Lemma~\ref{lem:case1-leading-frkr_a,i}, we have
the following refined modulation estimates.
\begin{prop}[Refined modulation estimates]
\label{prop:case3-modulation-est}We have refined modulation estimates
\begin{align}
\Big|\frac{\lmb_{i,t}}{\lmb_{i}}+\frac{d}{dt}f_{i}-\sum_{j=1}^{J}A_{ij}[\vec{z}]\lmb_{i}^{D-2}\lmb_{j}^{D}\Big| & \aleq\|g\|_{\dot{H}^{2}}^{2}+\lmb^{2D-3}\|g\|_{\dot{H}^{2}}+\lmb^{2D}|\log\lmb|,\label{eq:case3-ref-mod-lmb}\\
\Big|z_{i,t}-\frac{\kpp_{1}}{\kpp_{0}}\sum_{j=1}^{J}\frac{A_{ij}[\vec{z}]\lmb_{i}^{D}\lmb_{j}^{D}(z_{j}-z_{i})}{|z_{i}-z_{j}|^{2}}\Big| & \aleq\lmb\{\|g\|_{\dot{H}^{2}}^{2}+\lmb^{2D-3}\|g\|_{\dot{H}^{2}}+\lmb^{2D}|\log\lmb|\},\label{eq:case3-ref-mod-z}
\end{align}
where $f_{i}=f_{i}(t)$ satisfies 
\begin{equation}
|f_{i}|\aleq\min\{\|g\|_{\dot{H}^{1}},\|g\|_{\dot{H}^{2}}+\lmb_{i}\|g\|_{\dot{H}^{1}}\}.\label{eq:case3-est-f_i}
\end{equation}
\end{prop}

\begin{proof}
This follows from \eqref{eq:non-coll-refined-mod-est} with $\lmb_{\max}\aeq\lmb_{\secmax}$,
Lemma~\ref{lem:case1-leading-frkr_a,i}, and \eqref{eq:non-coll-refined-mod-est-corr-bound}.
Note that there are not so much we can simplify further.
\end{proof}

\subsection{\label{subsec:case3-control-z}Control of $\vec{z}$ on $[t_{n},T_{n}'')$}

The goal of this subsection is to control $\vec{z}(t)$ on $[t_{n},T_{n}'')$
by showing that $\vec{z}(t)=\vec{z}^{\ast}+o_{n\to\infty}(1)$ for
$t\in[t_{n},T_{n}'')$. Notice that we control on $[t_{n},T_{n}'')$,
not on $[t_{n},T_{n})$. We will prove $T_{n}''=T_{n}$ in the next
subsection. 

The estimates in this section assume $t\in[t_{n},T_{n}'')$. By $\frac{|A^{\ast}\vec{\lmb}^{D}(t)|}{\lmb^{D}(t)}<\dlt_{1}$,
\eqref{eq:case3-equiv-dist}, and the fact that $\vec{\frkc}\in(0,\infty)^{J}$
is a positive vector, we have 
\begin{equation}
\Big|\frac{\vec{\lmb}^{D}(t)}{\lmb^{D}(t)}-\vec{\frkc}\Big|\aleq\dlt_{1}\quad\text{and}\quad\lmb_{1}(t)\aeq\dots\aeq\lmb_{J}(t).\label{eq:8.13}
\end{equation}
We begin with more simplified modulation estimates obtained from Proposition~\ref{prop:case3-modulation-est}
under \eqref{eq:8.13}.
\begin{cor}[Modulation estimates in simple form integrated from $t_{n}$]
\label{cor:8.4}There exist functions $\zeta_{i}(t)$ on $[t_{n},T_{n}'')$
such that 
\begin{align}
\zeta_{i}(t) & =(1+o_{n\to\infty}(1))\lmb_{i}(t),\label{eq:case3-lmb-to-wh-lmb}
\end{align}
and 
\begin{align}
\Big|\frac{\zeta_{i,t}}{\zeta_{i}}-\sum_{j=1}^{J}A_{ij}[\vec{z}]\zeta_{i}^{D-2}\zeta_{j}^{D}\Big| & \aleq L_{1,n}+o_{n\to\infty}(1)\cdot\zeta^{2D-1},\label{eq:case3-lmb_j-mod-est}\\
\Big|z_{i,t}-\frac{\kpp_{1}}{\kpp_{0}}v_{i}^{\ast}\zeta^{2D}\Big| & \aleq L_{1,n}\cdot\zeta+\calO(\dlt_{1})\cdot\zeta^{2D},\label{eq:case3-z_j-mod-est}
\end{align}
where $L_{1,n}$ denotes some function $h_{n}(t)$ on $[t_{n},T_{n}'')$
such that $\int_{t_{n}}^{T_{n}''}|h_{n}(t)|dt=o_{n\to\infty}(1)$
and $\zeta\coloneqq\big(\zeta_{1}^{2D}+\cdots+\zeta_{J}^{2D}\big)^{\frac{1}{2D}}\aeq\lmb_{\max}$.
\end{cor}

\begin{rem}
\eqref{eq:case3-z_j-mod-est} combined with \eqref{eq:case3-v_i-not-all-zero}
implies non-degeneracy in the equations of $z_{i,t}$. Notice also
that we have a bound $o_{n\to\infty}(1)\cdot\zeta^{2D-1}$, not $o_{n\to\infty}(1)\cdot\zeta^{2D-2}$
as in \eqref{eq:case1-lmb-mod-est}. This better error estimate is
necessary due to the degeneracy in the equations \eqref{eq:case3-lmb_j-mod-est}
of $\lmb_{i,t}$.
\end{rem}

\begin{proof}
\uline{Proof of \mbox{\eqref{eq:case3-lmb-to-wh-lmb}} and \mbox{\eqref{eq:case3-lmb_j-mod-est}}}.
Define $\zeta_{i}\coloneqq\lmb_{i}e^{f_{i}}$, where $f_{i}$ is as
in \eqref{eq:case3-ref-mod-lmb} and \eqref{eq:case3-est-f_i}. Since
$f_{i}=\calO(\|g\|_{\dot{H}^{1}})=o_{n\to\infty}(1)$ by \eqref{eq:case3-g(t)-Hdot1=00003Do(1)},
we get \eqref{eq:case3-lmb-to-wh-lmb}. Next, using the full control
of \eqref{eq:case3-est-f_i} and $\|g(t)\|_{\dot{H}^{1}}+\lmb(t)=o_{n\to\infty}(1)$,
we have 
\begin{align*}
\lmb_{i}^{D-2}\lmb_{j}^{D} & =\zeta_{i}^{D-2}\zeta_{j}^{D}\cdot\{1+\calO(\|g\|_{\dot{H}^{2}})+o_{n\to\infty}(1)\cdot\lmb\}\\
 & =\zeta_{i}^{D-2}\zeta_{j}^{D}+\calO(\|g\|_{\dot{H}^{2}}^{2})+o_{n\to\infty}(1)\cdot\lmb^{2D-1}.
\end{align*}
Substituting this into the refined modulation estimate \eqref{eq:case3-ref-mod-lmb}
gives 
\[
\Big|\frac{\zeta_{i,t}}{\zeta_{i}}-\sum_{j=1}^{J}A_{ij}[\vec{z}]\zeta_{i}^{D-2}\zeta_{j}^{D}\Big|\aleq\|g\|_{\dot{H}^{2}}^{2}+\lmb^{2D-3}\|g\|_{\dot{H}^{2}}+o_{n\to\infty}(1)\cdot\lmb^{2D-1}.
\]
Note that $\|g\|_{\dot{H}^{2}}^{2}=L_{1,n}$ by the spacetime estimate
\eqref{eq:case3-spacetime-est}. This fact with $D\geq\frac{5}{2}$
implies 
\begin{equation}
\lmb^{2D-3}\|g\|_{\dot{H}^{2}}\leq L_{1,n}+o_{n\to\infty}(1)\cdot\lmb^{4D-6}\leq L_{1,n}+o_{n\to\infty}(1)\cdot\lmb^{2D-1}.\label{eq:8.19}
\end{equation}
This completes the proof of \eqref{eq:case3-lmb_j-mod-est}.

\uline{Proof of \mbox{\eqref{eq:case3-z_j-mod-est}}}. By \eqref{eq:case3-ref-mod-z}
and \eqref{eq:8.19}, we obtain 
\[
\Big|z_{i,t}-\frac{\kpp_{1}}{\kpp_{0}}\sum_{j=1}^{J}\frac{A_{ij}[\vec{z}]\lmb_{i}^{D}\lmb_{j}^{D}(z_{j}-z_{i})}{|z_{i}-z_{j}|^{2}}\Big|\aleq L_{1,n}\cdot\lmb+o_{n\to\infty}(1)\cdot\lmb^{2D}.
\]
Applying $\lmb_{i}^{D}=\lmb^{D}(\frkc_{i}+\calO(\dlt_{1}))$, \eqref{eq:case3-lmb-to-wh-lmb},
and $z_{i}=z_{i}^{\ast}+\calO(\dlt_{2})$, we obtain 
\[
\Big|z_{i,t}-\frac{\kpp_{1}}{\kpp_{0}}\sum_{j=1}^{J}\frac{A_{ij}^{\ast}\frkc_{i}\frkc_{j}(z_{j}^{\ast}-z_{i}^{\ast})}{|z_{i}^{\ast}-z_{j}^{\ast}|^{2}}\zeta^{2D}\Big|\aleq L_{1,n}\cdot\zeta+\calO(\dlt_{1})\cdot\zeta^{2D}.
\]
This is \eqref{eq:case3-z_j-mod-est} by the definition \eqref{eq:case3-v_i-not-all-zero}
of $v_{i}^{\ast}$.
\end{proof}
The following two differential inequalities are the keys to control
$\vec{z}(t)$ on $[t_{n},T_{n}'')$. 
\begin{lem}[Key inequalities]
\label{lem:case3-key-ineq}Let 
\begin{equation}
\alp(t)\coloneqq-\frac{\lan\vec{\frkc},A[\vec{z}(t)]\vec{\frkc}\ran}{4D\sum_{i=1}^{J}|v_{i}^{\ast}|^{2}}\qquad\text{and}\qquad\nu(t)\coloneqq\Big(\frac{\lan\vec{\frkc},\vec{\zeta}^{-D+2}(t)\ran}{\sum_{i=1}^{J}\frkc_{i}^{2/D}}\Big)^{-\frac{1}{D-2}}.\label{eq:case3-def-alp-nu}
\end{equation}
Then, we have 
\begin{align}
|\alp| & \aleq|\vec{z}-\vec{z}^{\ast}|,\label{eq:case3-alp-est}\\
\nu & =(1+\calO(\dlt_{1}))\zeta,\label{eq:case3-nu-est}
\end{align}
and 
\begin{align}
\frac{d}{dt}\alp & =-\frac{\kpp_{1}}{\kpp_{0}}\zeta^{2D}+L_{1,n}\cdot\zeta+\dlt_{1}\cdot\calO(\zeta^{2D}),\label{eq:case3-rd-alp}\\
\frac{d}{dt}\nu & =-\gmm\alp\zeta^{2D-1}+L_{1,n}\cdot\zeta+\dlt_{1}\cdot\calO(|\vec{z}-\vec{z}^{\ast}|\zeta^{2D-1}+\zeta^{2D}),\label{eq:case3-rd-nu}
\end{align}
where $\gmm$ is a constant defined by 
\begin{equation}
\gmm\coloneqq\frac{4D\sum_{i=1}^{J}|v_{i}^{\ast}|^{2}}{\sum_{i=1}^{J}\frkc_{i}^{2/D}}>0.\label{eq:case3-def-gmm}
\end{equation}

\end{lem}

\begin{proof}
\uline{Proof of \mbox{\eqref{eq:case3-alp-est}}}. By $A^{\ast}\vec{\frkc}=0$,
we get 
\[
\lan\vec{\frkc},A[\vec{z}(t)]\vec{\frkc}\ran=\lan\vec{\frkc},(A[\vec{z}(t)]-A^{\ast})\vec{\frkc}\ran=\calO(|A[\vec{z}]-A^{\ast}|)=\calO(|\vec{z}-\vec{z}^{\ast}|).
\]

\uline{Proof of \mbox{\eqref{eq:case3-nu-est}}}. By \eqref{eq:8.13},
we have 
\[
\lan\vec{\frkc},\vec{\zeta}^{-D+2}\ran=\zeta^{-D+2}\{\lan\vec{\frkc},\vec{\frkc}^{-1+\frac{2}{D}}\ran+\calO(\dlt_{1})\}=\{\tsum i{}\frkc_{i}^{2/D}+\calO(\dlt_{1})\}\zeta^{-D+2}.
\]
Dividing the above by $\sum_{i=1}^{J}\frkc_{i}^{2/D}$ and taking
the power of $-\frac{1}{D-2}$ give \eqref{eq:case3-nu-est}.

\uline{Proof of \mbox{\eqref{eq:case3-rd-alp}}}. Applying \eqref{eq:case3-z_j-mod-est}
and $|\vec{z}-\vec{z}^{\ast}|<\dlt_{2}\ll\dlt_{1}$, we get 
\begin{align*}
\rd_{t}A_{ij}[\vec{z}] & =-2D\frac{A_{ij}[\vec{z}]}{|z_{i}-z_{j}|^{2}}(z_{i}-z_{j})\cdot(z_{i}-z_{j})_{t}\\
 & =-2D\frac{\kpp_{1}}{\kpp_{0}}\frac{A_{ij}^{\ast}}{|z_{i}^{\ast}-z_{j}^{\ast}|^{2}}(z_{i}^{\ast}-z_{j}^{\ast})\cdot(v_{i}^{\ast}-v_{j}^{\ast})\zeta^{2D}+L_{1,n}\cdot\zeta+\dlt_{1}\cdot\calO(\zeta^{2D}).
\end{align*}
This implies 
\begin{align*}
\rd_{t}(-\lan\vec{\frkc},A[\vec{z}]\vec{\frkc}\ran) & =-\tsum{i,j}{}(\rd_{t}A_{ij}[\vec{z}])\frkc_{i}\frkc_{j}\\
 & =2D\frac{\kpp_{1}}{\kpp_{0}}\sum_{i,j}\frac{A_{ij}^{\ast}\frkc_{i}\frkc_{j}}{|z_{i}^{\ast}-z_{j}^{\ast}|^{2}}(z_{i}^{\ast}-z_{j}^{\ast})\cdot(v_{i}^{\ast}-v_{j}^{\ast})\zeta^{2D}+L_{1,n}\cdot\zeta+\dlt_{1}\cdot\calO(\zeta^{2D})\\
 & =-4D\frac{\kpp_{1}}{\kpp_{0}}\Big(\sum_{i}|v_{i}^{\ast}|^{2}\Big)\zeta^{2D}+L_{1,n}\cdot\zeta+\dlt_{1}\cdot\calO(\zeta^{2D}),
\end{align*}
where in the last equality we used \eqref{eq:case3-v_i-not-all-zero}.
Dividing this by $4D\sum_{i=1}^{J}|v_{i}^{\ast}|^{2}$ gives \eqref{eq:case3-rd-alp}.

\uline{Proof of \mbox{\eqref{eq:case3-rd-nu}}}. By \eqref{eq:case3-lmb_j-mod-est}
and $\lmb_{\min}\aeq\lmb_{\max}$, we have 
\[
\rd_{t}(\nu^{-D+2})=\frac{\lan\vec{\frkc},\rd_{t}(\vec{\zeta}^{-D+2})\ran}{\sum_{i=1}^{J}\frkc_{i}^{2/D}}=-\frac{D-2}{\sum_{i=1}^{J}\frkc_{i}^{2/D}}\lan\vec{\frkc},A[\vec{z}]\vec{\zeta}^{D}\ran+L_{1,n}\cdot\calO(\zeta^{-D+2})+o_{n\to\infty}(1)\cdot\zeta^{D+1}.
\]
We manipulate the first term using $A^{\ast}\vec{\frkc}=0$, symmetry
of $A^{\ast}$, \eqref{eq:8.13}, and \eqref{eq:case3-def-alp-nu}:
\begin{align*}
\lan\vec{\frkc},A[\vec{z}]\vec{\zeta}^{D}\ran=\lan\vec{\frkc},(A[\vec{z}]-A^{\ast})\vec{\zeta}^{D}\ran & =\lan\vec{\frkc},(A[\vec{z}]-A^{\ast})(\vec{\frkc}+\calO(\dlt_{1}))\ran\zeta^{D}\\
 & =-(4D\tsum i{}|v_{i}^{\ast}|^{2})\alp\zeta^{D}+\dlt_{1}\cdot\calO(|\vec{z}-\vec{z}^{\ast}|\zeta^{D}).
\end{align*}
Substituting this into the previous display and using \eqref{eq:case3-def-gmm},
we get 
\[
\rd_{t}(\nu^{-D+2})=(D-2)\gmm\alp\zeta^{D}+L_{1,n}\cdot\calO(\zeta^{-D+2})+\dlt_{1}\cdot\calO(|\vec{z}-\vec{z}^{\ast}|\zeta^{D})+o_{n\to\infty}(1)\cdot\zeta^{D+1}.
\]
Using $\rd_{t}\nu=-\frac{1}{D-2}\nu^{D-1}\rd_{t}(\nu^{-D+2})$, \eqref{eq:case3-nu-est},
and \eqref{eq:case3-alp-est}, we get \eqref{eq:case3-rd-nu}.
\end{proof}
We use these two differential inequalities \eqref{eq:case3-rd-alp}
and \eqref{eq:case3-rd-nu} to show that $\vec{z}(t)=\vec{z}^{\ast}+o_{n\to\infty}(1)$.
\begin{lem}[Control of $\vec{z}(t)$]
We have 
\begin{equation}
\sup_{t\in[t_{n},T_{n}'')}|\vec{z}(t)-\vec{z}^{\ast}|+\int_{t_{n}}^{T_{n}''}\lmb^{2D}(t)dt=o_{n\to\infty}(1).\label{eq:8.26}
\end{equation}
\end{lem}

\begin{proof}
First, integrating \eqref{eq:case3-z_j-mod-est} with $|\vec{z}(t_{n})-\vec{z}^{\ast}|=o_{n\to\infty}(1)$
and $\lmb=o_{n\to\infty}(1)$ gives 
\begin{equation}
\sup_{t\in[t_{n},T_{n}'')}|\vec{z}(t)-\vec{z}^{\ast}|\aleq\int_{t_{n}}^{T_{n}''}\zeta^{2D}(t)dt+o_{n\to\infty}(1).\label{eq:7.24-1}
\end{equation}
Next, integrating \eqref{eq:case3-rd-alp} with $\alp(t_{n})=\calO(|\vec{z}(t_{n})-\vec{z}^{\ast}|)=o_{n\to\infty}(1)$
and $\lmb=o_{n\to\infty}(1)$ gives 
\begin{equation}
\int_{t_{n}}^{T_{n}''}\zeta^{2D}(t)dt\aleq\sup_{t\in[t_{n},T_{n}'')}|\alp(t)|+o_{n\to\infty}(1).\label{eq:7.25-2}
\end{equation}
Next, we use \eqref{eq:case3-rd-nu}, \eqref{eq:case3-nu-est}, \eqref{eq:case3-alp-est},
and \eqref{eq:case3-rd-alp} to obtain 
\begin{align*}
(\nu^{2})_{t} & =2\nu\{-\gmm\alp\zeta^{2D-1}+L_{1,n}\cdot\zeta+\dlt_{1}\cdot\calO(|\vec{z}-\vec{z}^{\ast}|\zeta^{2D-1}+\zeta^{2D})\}\\
 & =-2\gmm\alp\zeta^{2D}+L_{1,n}\cdot\zeta^{2}+\dlt_{1}\cdot\calO(|\vec{z}-\vec{z}^{\ast}|\zeta^{2D}+\zeta^{2D+1})\\
 & =\gmm\frac{\kpp_{0}}{\kpp_{1}}(\alp^{2})_{t}+L_{1,n}\cdot(|\vec{z}-\vec{z}^{\ast}|\zeta+\zeta^{2})+\dlt_{1}\cdot\calO(|\vec{z}-\vec{z}^{\ast}|\zeta^{2D}+\zeta^{2D+1}).
\end{align*}
Integrating this with $\alp(t_{n})=o_{n\to\infty}(1)$ and $\nu(t)=o_{n\to\infty}(1)$,
we get 
\[
\sup_{t\in[t_{n},T_{n}'')}\alp^{2}(t)\aleq o_{n\to\infty}(1)+\Big(\dlt_{1}\sup_{t\in[t_{n},T_{n}'')}|\vec{z}-\vec{z}^{\ast}|+o_{n\to\infty}(1)\Big)\int_{t_{n}}^{T_{n}''}\zeta^{2D}(t)dt.
\]
Applying \eqref{eq:7.24-1} and \eqref{eq:7.25-2}, we get 
\[
\sup_{t\in[t_{n},T_{n}'')}\alp^{2}(t)\aleq o_{n\to\infty}(1)+\dlt_{1}\Big(\sup_{t\in[t_{n},T_{n}'')}|\alp(t)|\Big)^{2}\aleq o_{n\to\infty}(1)+\dlt_{1}\sup_{t\in[t_{n},T_{n}'')}\alp^{2}(t).
\]
Since $\dlt_{1}$ is small, we conclude that $\sup_{t\in[t_{n},T_{n}'')}|\alp(t)|=o_{n\to\infty}(1)$.
Substituting this into \eqref{eq:7.25-2} and \eqref{eq:7.24-1} completes
the proof of \eqref{eq:8.26}.
\end{proof}

\subsection{\label{subsec:case3-ratio}Proof of $T_{n}''=T_{n}$ and control
of ratio}

The goal of this subsection is twofold: (i) to prove $T_{n}''=T_{n}$
for all large $n$ and (ii) to prove $|A^{\ast}\vec{\lmb}^{D}(t)|\ll\lmb^{D}(t)$
for all $t\in[t_{n},T_{n})$. 

The proof of $T_{n}''=T_{n}$ roughly proceeds as follows. If not,
i.e., $T_{n}''<T_{n}$, then 
\begin{equation}
\frac{|A^{\ast}\vec{\lmb}^{D}(T_{n}'')|}{\lmb^{D}(T_{n}'')}=\dlt_{1}.\label{eq:case3-ratio-at-T_n''}
\end{equation}
We will show in \eqref{eq:8.34} below that $|A^{\ast}\vec{\lmb}^{D}(t)|\ageq\dlt_{1}\lmb^{D}(t)$
for all $t\in[T_{n}'',T_{n})$. Note that, at $t=T_{n}''$, \eqref{eq:case3-lmb_max=00003Do(1)},
\eqref{eq:case3-g(t)-Hdot1=00003Do(1)}, and \eqref{eq:8.26} imply
that $u(T_{n}'')$ can be viewed as a $W$-bubbling sequence with
$\vec{z}(T_{n}'')\to\vec{z}^{\ast}$. As a surprising consequence
of the lower bound on the ratio and the minimal degeneracy of $\{(\iota_{1},z_{1}^{\ast}),\dots,(\iota_{J},z_{J}^{\ast})\}$,
we see that the left hand side of \eqref{eq:5.28} is strictly away
from zero for all $t\in[T_{n}'',T_{n})$. Thus $u(t)$ can be viewed
as a solution belonging to \textbf{Case~1} of Proposition~\ref{prop:case-separation},
which turns out to be impossible by the second item of Proposition~\ref{prop:case1-main}.
\begin{lem}[Key inequality]
\label{lem:case3-ratio-key-ineq}There exist functions $\wh{\zeta}_{i}(t)$
defined on $[t_{n},T_{n})$ such that 
\begin{align}
\wh{\zeta}_{i}(t) & =(1+o_{n\to\infty}(1))\cdot\lmb_{i}(t),\label{eq:case3-wh-zeta}\\
\frac{\wh{\zeta}_{i,t}}{\wh{\zeta}_{i}} & =\wh{\zeta}_{i}^{D-2}(A^{\ast}\vec{\wh{\zeta}}^{D})_{i}+\dlt_{2}\cdot\calO(\wh{\zeta}^{2D-2}).\label{eq:case3-wh-zeta-mod-est}
\end{align}
From this, there exists a function $\vphi(t)$ defined on $[t_{n},T_{n})$
such that 
\begin{align}
|\vphi(t)| & \aleq\wh{\zeta}_{\min}^{-2D}|A^{\ast}\vec{\wh{\zeta}}^{D}|^{2},\label{eq:8.30}\\
\frac{d}{dt}\vphi(t) & \geq\wh{\zeta}_{\min}^{-2D}\{c\wh{\zeta}^{2D-2}|A^{\ast}\wh{\zeta}^{D}|^{2}-\dlt_{2}\cdot\calO(\wh{\zeta}^{4D-2})\}\label{eq:8.31}
\end{align}
for some constant $c>0$. Here, we denoted $\wh{\zeta}\coloneqq(\wh{\zeta}_{1}^{2D}+\dots+\wh{\zeta}_{J}^{2D})^{1/(2D)}$
and $\wh{\zeta}_{\min}=\min_{i\in\setJ}\wh{\zeta}_{i}$. 
\end{lem}

\begin{rem}
Notice that the error size in \eqref{eq:case3-wh-zeta-mod-est} is
only $\dlt_{2}\cdot\calO(\lmb^{2D-2})$, not $o_{n\to\infty}(1)\cdot\lmb^{2D-1}$
as in \eqref{eq:case3-lmb_j-mod-est}. However, since we are in the
degenerate regime, it is crucial to keep the error $o_{n\to\infty}(1)\cdot\lmb^{2D-1}$
in other places in order to justify the leading terms of in the modulation
ODEs. 
\end{rem}

\begin{proof}
\uline{Proof of \mbox{\eqref{eq:case3-wh-zeta}} and \mbox{\eqref{eq:case3-wh-zeta-mod-est}}}.
This easily follows from \eqref{eq:case3-ref-mod-lmb}, $f_{i}=o_{n\to\infty}(1)$
(from \eqref{eq:case3-g(t)-Hdot1=00003Do(1)}), $|\vec{z}-\vec{z}^{\ast}|<\dlt_{2}$
(small), and the spacetime estimate \eqref{eq:case3-spacetime-est}.

\uline{Proof of \mbox{\eqref{eq:8.30}} and \mbox{\eqref{eq:8.31}}}.
It suffices to show that 
\begin{align}
\vphi\coloneqq\lan A^{\ast}\vec{\wh{\zeta}}^{D},\vec{\wh{\zeta}}^{D}\ran\cdot\lan\vec{\frkc},\vec{\wh{\zeta}}^{-D+2}\ran^{\frac{2D}{D-2}} & =\calO(\wh{\zeta}_{\min}^{-2D}|A^{\ast}\vec{\wh{\zeta}}^{D}|^{2}),\label{eq:8.32}\\
\frac{d}{dt}\Big\{\lan A^{\ast}\vec{\wh{\zeta}}^{D},\vec{\wh{\zeta}}^{D}\ran\cdot\lan\vec{\frkc},\vec{\wh{\zeta}}^{-D+2}\ran^{\frac{2D}{D-2}}\Big\} & \geq c\wh{\zeta}^{2D-2}\frac{|A^{\ast}\wh{\zeta}^{D}|^{2}}{\wh{\zeta}_{\min}^{2D}}-\calO(\dlt_{2})\cdot\frac{\wh{\zeta}^{4D-2}}{\wh{\zeta}_{\min}^{2D}}.\label{eq:8.33}
\end{align}

First, we prove \eqref{eq:8.32}. Since $A^{\ast}$ is a fixed symmetric
matrix (not depending on time), we have $\lan A^{\ast}\vec{\zeta},\vec{\zeta}\ran=\calO(|A^{\ast}\vec{\wh{\zeta}}^{D}|^{2})$.
As $\vec{\frkc}\in(0,\infty)^{J}$ is a positive vector and $-D+2=-\frac{N-6}{2}<0$,
we get 
\[
\lan\vec{\frkc},\vec{\wh{\zeta}}^{-D+2}\ran\aeq\wh{\zeta}_{\min}^{-D+2}.
\]
Combining these observations, we get \eqref{eq:8.32}.

We turn to the proof of \eqref{eq:8.33}. We use \eqref{eq:case3-wh-zeta-mod-est}
to compute 
\begin{align*}
\rd_{t}\lan A^{\ast}\vec{\wh{\zeta}}^{D},\vec{\wh{\zeta}}^{D}\ran & =2\lan A^{\ast}\vec{\wh{\zeta}}^{D},\rd_{t}(\vec{\wh{\zeta}}^{D})\ran=2D\tsum i{}\wh{\zeta}_{i}^{2D-2}(A^{\ast}\vec{\wh{\zeta}}^{D})_{i}^{2}+\dlt_{2}\cdot\calO(\wh{\zeta}^{4D-2}),\\
\rd_{t}\lan\vec{\frkc},\vec{\wh{\zeta}}^{-D+2}\ran & =-(D-2)\lan\vec{\frkc},A^{\ast}\vec{\wh{\zeta}}^{D}\ran+\dlt_{2}\cdot\calO(\tsum i{}\wh{\zeta}_{i}^{-D+2}\wh{\zeta}^{2D-2})=\dlt_{2}\cdot\calO(\wh{\zeta}_{\min}^{-D+2}\wh{\zeta}^{2D-2}),
\end{align*}
where in the last equality we used $A^{\ast}\vec{\frkc}=0$. The above
two identities imply 
\begin{align*}
\frac{d}{dt}\Big\{\lan A^{\ast}\vec{\wh{\zeta}}^{D},\vec{\wh{\zeta}}^{D}\ran\cdot\lan\vec{\frkc},\vec{\wh{\zeta}}^{-D+2}\ran^{\frac{2D}{D-2}}\Big\} & =2D\tsum i{}\wh{\zeta}_{i}^{2D-2}(A^{\ast}\vec{\wh{\zeta}}^{D})_{i}^{2}\cdot\lan\vec{\frkc},\vec{\wh{\zeta}}^{-D+2}\ran^{\frac{2D}{D-2}}+\dlt_{2}\cdot\calO(\wh{\zeta}_{\min}^{-2D}\wh{\zeta}^{4D-2}).
\end{align*}
Applying \eqref{eq:case3-dist-from-kernel} and $\lan\vec{\frkc},\vec{\wh{\zeta}}^{-D+2}\ran^{\frac{2D}{D-2}}\aeq\wh{\zeta}_{\min}^{-2D}$,
we conclude \eqref{eq:8.33}.
\end{proof}
\begin{lem}
\label{lem:case3-ratio-geq-dlt_1}Under the assumption $T_{n}''<T_{n}$,
we have 
\begin{equation}
\frac{|A^{\ast}\vec{\lmb}^{D}(t)|}{\lmb^{D}(t)}\ageq\dlt_{1}\qquad\text{for all }t\in[T_{n}'',T_{n}).\label{eq:8.34}
\end{equation}
\end{lem}

\begin{proof}
Suppose not. Let $K>10$ be a large universal constant to be chosen.
By \eqref{eq:case3-wh-zeta}, \eqref{eq:7.12-2}, \eqref{eq:case3-ratio-at-T_n''},
and the negation of \eqref{eq:8.34}, there exists $[a_{n},b_{n}]\subseteq(t_{n},T_{n})$
and two sequences $c_{n},d_{n}\in(a_{n},b_{n})$ with $c_{n}<d_{n}$
such that 
\begin{align}
 & \wh{\zeta}^{-D}|A^{\ast}\vec{\wh{\zeta}}^{D}|(a_{n})=K^{-1}\dlt_{1},\quad\wh{\zeta}^{-D}|A^{\ast}\vec{\wh{\zeta}}^{D}|(b_{n})=K^{-1}\dlt_{1},\label{eq:8.35}\\
 & \wh{\zeta}^{-D}|A^{\ast}\vec{\wh{\zeta}}^{D}|\geq K^{-1}\dlt_{1}\quad\text{on }[a_{n},b_{n}],\label{eq:8.36}\\
 & \wh{\zeta}^{-D}|A^{\ast}\vec{\wh{\zeta}}^{D}|(c_{n})=\tfrac{1}{3}\dlt_{1},\quad\wh{\zeta}^{-D}|A^{\ast}\vec{\wh{\zeta}}^{D}|(d_{n})=\tfrac{2}{3}\dlt_{1}.\label{eq:8.37}
\end{align}
At $t\in\{a_{n},b_{n}\}$, by \eqref{eq:8.30}, \eqref{eq:8.13},
and \eqref{eq:8.35}, we have 
\[
|\vphi(a_{n})|+|\vphi(b_{n})|\aleq K^{-2}\dlt_{1}^{2}.
\]
On the interval $[a_{n},b_{n}]$, by \eqref{eq:8.31}, \eqref{eq:8.36},
and $\dlt_{2}\ll K^{-1}\dlt_{1}$, we have 
\[
\rd_{t}\vphi\ageq\wh{\zeta}_{\min}^{-2D}\wh{\zeta}^{2D-2}|A^{\ast}\vec{\wh{\zeta}}^{D}|^{2}\ageq\wh{\zeta}^{-2}|A^{\ast}\vec{\wh{\zeta}}^{D}|^{2}.
\]
Integrating this on $[a_{n},b_{n}]$ gives 
\begin{equation}
\int_{a_{n}}^{b_{n}}\frac{|A^{\ast}\vec{\wh{\zeta}}^{D}|^{2}}{\wh{\zeta}^{2}}dt\aleq|\vphi(a_{n})|+|\vphi(b_{n})|\aleq K^{-2}\dlt_{1}^{2}.\label{eq:8.38}
\end{equation}
On the other hand, on $[c_{n},d_{n}]$, we have $\rd_{t}(\wh{\zeta}^{-2D}|A^{\ast}\vec{\wh{\zeta}}^{D}|^{2})=\calO(\wh{\zeta}^{-2}|A^{\ast}\vec{\wh{\zeta}}^{D}|^{2})$
by \eqref{eq:case3-wh-zeta-mod-est}, \eqref{eq:8.36}, and $\dlt_{2}\ll K^{-1}\dlt_{1}$.
Integrating this on $[c_{n},d_{n}]$ with \eqref{eq:8.37} and \eqref{eq:8.38}
gives 
\[
\tfrac{1}{3}\dlt_{1}^{2}=\Big|\frac{|A^{\ast}\vec{\wh{\zeta}}^{D}|^{2}}{\wh{\zeta}^{2D}}(d_{n})-\frac{|A^{\ast}\vec{\wh{\zeta}}^{D}|^{2}}{\wh{\zeta}^{2D}}(c_{n})\Big|=\Big|\int_{c_{n}}^{d_{n}}\calO\Big(\frac{|A^{\ast}\vec{\wh{\zeta}}^{D}|^{2}}{\wh{\zeta}^{2}}\Big)dt\Big|\aleq\int_{a_{n}}^{b_{n}}\frac{|A^{\ast}\vec{\wh{\zeta}}^{D}|^{2}}{\wh{\zeta}^{2}}dt\aleq K^{-2}\dlt_{1}^{2}.
\]
Choosing $K$ sufficiently large gives a contradiction. This completes
the proof of \eqref{eq:8.34}.
\end{proof}
\begin{cor}
\label{cor:case3-T_n''-is-T_n}We have $T_{n}''=T_{n}=+\infty$ for
all large $n$. Moreover, we have 
\begin{equation}
\|g(t)\|_{\dot{H}^{1}}+\lmb_{\max}(t)+|\vec{z}(t)-\vec{z}^{\ast}|\to0\qquad\text{as }t\to+\infty.\label{eq:case3-all-go-to-zero}
\end{equation}
\end{cor}

\begin{proof}
\uline{Step 1: Proof of \mbox{$T_{n}''=T_{n}$}}. Fix the constants
$\alp$, $\dlt_{1}$, and $\dlt_{2}$ so that the arguments in this
section (until this point) hold. Now, suppose $T_{n}''<T_{n}$. 

First, by \eqref{eq:case3-lmb_max=00003Do(1)}, \eqref{eq:case3-g(t)-Hdot1=00003Do(1)},
and \eqref{eq:8.26}, we see that $\{u(T_{n}''),\vec{\iota},\vec{\lmb}(T_{n}''),\vec{z}(T_{n}'')\}$
is a $W$-bubbling sequence, $\vec{z}(T_{n}'')\to\vec{z}^{\ast}$,
and $\lmb_{\max}(T_{n}'')\to0$. 

Next, we claim that there exists $\dlt_{1}^{\ast}>0$ (independent
of $n$) such that 
\begin{equation}
\min_{\emptyset\neq\calI\subseteq\setJ}\Big(\frac{|A_{\calI}^{\ast}\vec{\lmb}_{\calI}^{D}(t)|}{|\vec{\lmb}_{\calI}^{D}(t)|}+\sum_{j\notin\calI}\frac{\lmb_{j}(t)}{\lmb_{\max}(t)}\Big)\geq\dlt_{1}^{\ast}\qquad\text{on }[T_{n}'',T_{n}).\label{eq:8.40}
\end{equation}
By \eqref{eq:8.2}, this is true for $|\calI|=1$. By \eqref{eq:8.34},
this is true for $\calI=\setJ$. By the minimal degeneracy of $\{(\iota_{1},\vec{z}_{1}^{\ast}),\dots,(\iota_{J},\vec{z}_{J}^{\ast})\}$
and Lemma~\ref{lem:non-degen}, this is also true for $2\leq|\calI|\leq J-1$.
Therefore, the claim \eqref{eq:8.40} holds. 

Finally, if $T_{n}<+\infty$, separately considering the cases $T_{n}=T_{n}'<+\infty$
(recall the exit condition for $T_{n}'$ in \textbf{Case~3} of Proposition~\ref{prop:case-separation})
and $T_{n}<T_{n}'$ (recall the definition of $T_{n}$), we see that
either $u(T_{n})\notin\calT_{J}(\alp)$, $|\vec{z}(T_{n})-\vec{z}^{\ast}|\geq\dlt_{2}$,
or $\lmb_{\max}(T_{n})\geq1$. 

Therefore, our solution $u(t)$ belongs to \textbf{Case~1} of Proposition~\ref{prop:case-separation},
where $T_{n}''$, $T_{n}$, $\dlt_{1}^{\ast}$, $\dlt_{2}$, $\alp$
in this proof play the role of $t_{n}$, $T_{n}'$, $\dlt_{1}$, $\dlt_{2}^{\ast}$,
$\alp^{\ast}$ in the statement of \textbf{Case~1}, respectively.
Applying Proposition~\ref{prop:case1-main}, we conclude in particular
that $\{(\iota_{1},z_{1}^{\ast}),\dots,(\iota_{J},z_{J}^{\ast})\}$
is non-degenerate, which is a contradiction.

\uline{Step 2: Proof of the remaining statements}. This is standard.
Thanks to $T_{n}''=T_{n}$, the control \eqref{eq:8.26} is now valid
in $[t_{n},T_{n})$. Now, $T_{n}=\Ttbexit_{n}=T_{n}'=+\infty$ follows
from \eqref{eq:case3-lmb_max=00003Do(1)}, \eqref{eq:case3-g(t)-Hdot1=00003Do(1)},
\eqref{eq:8.26}, and the exit condition of $T_{n}'$. These three
controls imply \eqref{eq:case3-all-go-to-zero} as well.
\end{proof}
Having established $T_{n}''=T_{n}=+\infty$, we can repeat the argument
of this subsection to show that $\vec{\lmb}^{D}/\lmb^{D}$ converges
to $\vec{\frkc}$ as $t\to+\infty$ as well.
\begin{cor}[Improved control of the ratio]
We have 
\begin{equation}
\frac{\vec{\lmb}^{D}}{\lmb^{D}}=\vec{\frkc}+o(1)\quad\text{and}\quad\lmb_{1}\aeq\dots\aeq\lmb_{J}\aeq\lmb.\label{eq:case3-ratio-control}
\end{equation}
\end{cor}

\begin{proof}
\uline{Step 1: Analogue of Lemma~\mbox{\ref{lem:case3-ratio-key-ineq}}}.
Revisiting the proof of Lemma~\ref{lem:case3-ratio-key-ineq}, now
with $T_{n}''=T_{n}=+\infty$ and \eqref{eq:case3-all-go-to-zero},
we can replace all $o_{n\to\infty}(1)$ and $\dlt_{2}$ by $o(1)$.
As a result, there exist functions $\wh{\zeta}_{i}(t)$, now defined
for all large $t$, such that 
\begin{align}
\wh{\zeta}_{i}(t) & =(1+o(1))\cdot\lmb_{i}(t),\label{eq:case3-wh-zeta-1}\\
\frac{\wh{\zeta}_{i,t}}{\wh{\zeta}_{i}} & =\wh{\zeta}_{i}^{D-2}(A^{\ast}\vec{\wh{\zeta}}^{D})_{i}+o(\wh{\zeta}^{2D-2}).\label{eq:case3-wh-zeta-mod-est-1}
\end{align}
From this, there exists a function $\vphi(t)$ defined for all large
$t$ such that 
\begin{align}
\vphi(t) & =\calO(\wh{\zeta}_{\min}^{-2D}|A^{\ast}\vec{\wh{\zeta}}^{D}|^{2}),\label{eq:8.44}\\
\frac{d}{dt}\vphi(t) & \geq\wh{\zeta}_{\min}^{-2D}\{c\wh{\zeta}^{2D-2}|A^{\ast}\wh{\zeta}^{D}|^{2}-o(\wh{\zeta}^{4D-2})\}\label{eq:8.45}
\end{align}
for some constant $c>0$.

\uline{Step 2: Proof of \mbox{\eqref{eq:case3-ratio-control}}}.
By \eqref{eq:case3-equiv-dist}, \eqref{eq:case3-wh-zeta-1}, and
the fact that $\vec{\frkc}\in(0,\infty)^{J}$ is a positive vector,
it suffices to show 
\begin{equation}
\frac{|A^{\ast}\vec{\wh{\zeta}}^{D}(t)|}{\wh{\zeta}^{D}(t)}=o(1).\label{eq:case3-ratio-2}
\end{equation}
Suppose not. Let $\dlt>0$ be a small constant and let $K>10$ be
a large constant to be chosen. By \eqref{eq:7.12-2} and the negation
of \eqref{eq:case3-ratio-2}, there exist a sequence of time intervals
$[a_{n},b_{n}]$ with $a_{n}\to+\infty$ and two sequences $c_{n},d_{n}\in(a_{n},b_{n})$
with $c_{n}<d_{n}$ such that 
\begin{align}
 & \wh{\zeta}^{-D}|A^{\ast}\vec{\wh{\zeta}}^{D}|(a_{n})=K^{-1}\dlt,\quad\wh{\zeta}^{-D}|A^{\ast}\vec{\wh{\zeta}}^{D}|(b_{n})=K^{-1}\dlt,\label{eq:7.13-1}\\
 & \wh{\zeta}^{-D}|A^{\ast}\vec{\wh{\zeta}}^{D}|\geq K^{-1}\dlt\quad\text{on }[a_{n},b_{n}],\label{eq:7.13-2}\\
 & \wh{\zeta}^{-D}|A^{\ast}\vec{\wh{\zeta}}^{D}|(c_{n})=\tfrac{1}{2}\dlt,\quad\wh{\zeta}^{-D}|A^{\ast}\vec{\wh{\zeta}}^{D}|(d_{n})=\dlt.\label{eq:7.13-3}
\end{align}
The argument to derive a contradiction is essentially the same as
in the proof of Lemma~\ref{lem:case3-ratio-geq-dlt_1}. Proceeding
similarly as the proof there, we have $|\vphi(a_{n})|+|\vphi(b_{n})|\aleq K^{-2}\dlt^{2}$
and $\rd_{t}\vphi\ageq\wh{\zeta}^{-2}|A^{\ast}\vec{\wh{\zeta}}^{D}|^{2}$
on $[a_{n},b_{n}]$, so by integration we get $\int_{a_{n}}^{b_{n}}\wh{\zeta}^{-2}|A^{\ast}\vec{\wh{\zeta}}^{D}|^{2}dt\aleq K^{-2}\dlt^{2}$.
On the other hand, we have $\rd_{t}(\wh{\zeta}^{-2D}|A^{\ast}\vec{\wh{\zeta}}^{D}|^{2})=\calO(\wh{\zeta}^{-2}|A^{\ast}\vec{\wh{\zeta}}^{D}|^{2})$
on $[c_{n},d_{n}]$, whose integration gives 
\[
\tfrac{3}{4}\dlt^{2}=\Big|\frac{|A^{\ast}\vec{\wh{\zeta}}^{D}|^{2}}{\wh{\zeta}^{2D}}(d_{n})-\frac{|A^{\ast}\vec{\wh{\zeta}}^{D}|^{2}}{\wh{\zeta}^{2D}}(c_{n})\Big|=\Big|\int_{c_{n}}^{d_{n}}\calO\Big(\frac{|A^{\ast}\vec{\wh{\zeta}}^{D}|^{2}}{\wh{\zeta}^{2}}\Big)dt\Big|\aleq\int_{a_{n}}^{b_{n}}\frac{|A^{\ast}\vec{\wh{\zeta}}^{D}|^{2}}{\wh{\zeta}^{2}}dt\aleq K^{-2}\dlt^{2}.
\]
This is a contradiction by choosing $K$ sufficiently large (the smallness
of $\dlt$ was used to justify this argument). This completes the
proof.
\end{proof}

\subsection{\label{subsec:case3-proof-of-prop}Proof of Proposition~\ref{prop:case3-main}}

Note that the continuous-in-time resolution with $\vec{z}(t)\to\vec{z}^{\ast}$
is shown in \eqref{eq:case3-all-go-to-zero}. It remains to prove
the asymptotics of the rates $\lmb_{i}(t)$, i.e., \eqref{eq:case3-lmb-rate}.
We begin with an analogue of Lemma~\ref{lem:case3-key-ineq} with
all $L_{1,n}$ and $\dlt_{1}$ in the error terms are replaced by
$L_{1}$ and $o(1)$ (thanks to \eqref{eq:case3-ratio-control} and
\eqref{eq:case3-all-go-to-zero}), respectively. Then, we integrate
the ODEs of $\alp$ and $\nu$ to derive sharp asymptotics as $t\to+\infty$.
\begin{proof}[\uline{Proof of \mbox{\eqref{eq:case3-lmb-rate}}}]
\uline{Step 1: Analogue of Corollary~\mbox{\ref{cor:8.4}}}.
Revisiting the proof of Corollary~\ref{cor:8.4} using \eqref{eq:case3-ratio-control}
and \eqref{eq:case3-all-go-to-zero}, we can replace all $L_{1,n}$
and $\dlt_{1}$ by $L_{1}$ and $o(1)$, respectively. Thus there
exist functions $\zeta_{i}(t)$ defined for all large $t$ such that
\begin{equation}
\zeta_{i}(t)=(1+o(1))\lmb_{i}(t)\label{eq:8.50}
\end{equation}
and 
\begin{align}
\Big|\frac{\zeta_{i,t}}{\zeta_{i}}-\zeta_{i}^{D-2}(A[\vec{z}]\vec{\zeta}^{D})_{i}\Big| & \aleq L_{1}+o(1)\cdot\zeta^{2D-1},\label{eq:8.51}\\
\Big|z_{i,t}-\frac{\kpp_{1}}{\kpp_{0}}v_{i}^{\ast}\zeta^{2D}\Big| & \aleq L_{1}\cdot\zeta+o(1)\cdot\zeta^{2D},\label{eq:8.52}
\end{align}
where $L_{1}$ denotes some time-integrable function $h(t)$ defined
for all large $t$ and $\zeta\coloneqq(\zeta_{1}^{2D}+\dots+\zeta_{J}^{2D})^{1/(2D)}\aeq\lmb_{\max}$. 

\uline{Step 2: Analogue of Lemma~\mbox{\ref{lem:case3-key-ineq}}}.
Define $\alp(t)$ and $\nu(t)$ by the same formulas \eqref{eq:case3-def-alp-nu}
with $\zeta_{i}$ and $z_{i}$ as in the previous step. Revisiting
the proof of Lemma~\ref{lem:case3-key-ineq}, but using the estimates
in the previous step this time, we obtain similar estimates for $\alp$
and $\nu$ with all $\dlt_{1}$ and $L_{1,n}$ replaced by $o(1)$
and $L_{1}$, respectively. Thus we have 
\begin{align}
\alp & =\calO(|\vec{z}-\vec{z}^{\ast}|),\label{eq:case3-alp-est-1}\\
\nu & =(1+o(1))\zeta,\label{eq:case3-nu-est-1}
\end{align}
and 
\begin{align}
\frac{d}{dt}\alp & =-\frac{\kpp_{1}}{\kpp_{0}}\zeta^{2D}+L_{1}\cdot\zeta+o(\zeta^{2D}),\label{eq:case3-rd-alp-1}\\
\frac{d}{dt}\nu & =-\gmm\alp\zeta^{2D-1}+L_{1}\cdot\zeta+o(|\vec{z}-\vec{z}^{\ast}|\zeta^{2D-1}+\zeta^{2D}).\label{eq:case3-rd-nu-1}
\end{align}

\uline{Step 3: Proof of \mbox{\eqref{eq:case3-lmb-rate}}}. We
begin with the argument used in the proof of \eqref{eq:8.26}. Using
\eqref{eq:8.52} and $\vec{z}(t)\to\vec{z}^{\ast}$, we have 
\begin{equation}
\sup_{\tau\in[t,+\infty)}|\vec{z}(\tau)-\vec{z}^{\ast}|\aleq\int_{t}^{+\infty}\zeta^{2D}(\tau)d\tau+o(1)\cdot\sup_{\tau\in[t,+\infty)}\zeta(\tau).\label{eq:7.24}
\end{equation}
Next, by \eqref{eq:case3-rd-alp-1}, we have 
\begin{equation}
\int_{t}^{+\infty}\zeta^{2D}(\tau)d\tau\aleq|\alp(t)|+o(1)\cdot\sup_{\tau\in[t,+\infty)}\zeta(\tau).\label{eq:7.25}
\end{equation}
Next, as in the proof of \eqref{eq:8.26}, we obtain 
\begin{align}
(\nu^{2})_{t} & =2\nu\{-\gmm\alp\zeta^{2D-1}+L_{1}\cdot\zeta+o(|\vec{z}-\vec{z}^{\ast}|\zeta^{2D-1}+\zeta^{2D})\}\nonumber \\
 & =-2\gmm\alp\zeta^{2D}+L_{1}\cdot\zeta^{2}+o(|\vec{z}-\vec{z}^{\ast}|\zeta^{2D}+\zeta^{2D+1})\nonumber \\
 & =\gmm\frac{\kpp_{0}}{\kpp_{1}}(\alp^{2})_{t}+L_{1}\cdot(|\vec{z}-\vec{z}^{\ast}|\zeta+\zeta^{2})+o(|\vec{z}-\vec{z}^{\ast}|\zeta^{2D}+\zeta^{2D+1}).\label{eq:8.61}
\end{align}
Integrating this from $t=+\infty$ together with $\nu\aeq\zeta$,
$|\alp|\aleq|\vec{z}-\vec{z}^{\ast}|$, and \eqref{eq:7.25} implies
\begin{equation}
\sup_{\tau\in[t,+\infty)}\zeta^{2}(\tau)\aleq\sup_{\tau\in[t,+\infty)}|\vec{z}(\tau)-\vec{z}^{\ast}|^{2}.\label{eq:8.62}
\end{equation}
This substituted into \eqref{eq:7.24} and \eqref{eq:7.25} gives
\[
\sup_{\tau\in[t,+\infty)}\zeta(\tau)\aleq\sup_{\tau\in[t,+\infty)}|\vec{z}(\tau)-\vec{z}^{\ast}|\aeq\sup_{\tau\in[t,+\infty)}|\alp(\tau)|\aeq\int_{t}^{+\infty}\zeta^{2D}(\tau)d\tau.
\]
Substituting this into \eqref{eq:7.25} again and using $|\vec{z}(t)-\vec{z}^{\ast}|\ageq|\alp(t)|$,
we get 
\begin{equation}
|\alp(t)|\aeq|\vec{z}(t)-\vec{z}^{\ast}|\aeq\int_{t}^{+\infty}\zeta^{2D}(\tau)d\tau.\label{eq:7.28}
\end{equation}
(In particular, $|\alp(t)|$ and $|\vec{z}(t)-\vec{z}^{\ast}|$ are
almost decreasing.) Now, we integrate \eqref{eq:8.61} again with
the controls \eqref{eq:case3-nu-est-1} and \eqref{eq:7.28} to have
\begin{equation}
\zeta^{2}(t)=(\gmm\frac{\kpp_{0}}{\kpp_{1}}+o(1))\alp^{2}(t).\label{eq:7.28-1-1}
\end{equation}
On the other hand, we integrate \eqref{eq:case3-rd-alp-1} to obtain
\begin{align}
\alp(t) & =\Big(\frac{\kpp_{1}}{\kpp_{0}}+o(1)\Big)\int_{t}^{\infty}\zeta^{2D}(\tau)d\tau.\label{eq:7.28-0}
\end{align}
Substituting \eqref{eq:7.28-0} into \eqref{eq:7.28-1-1}, we conclude
\begin{equation}
\zeta(t)=\Big(\sqrt{\gmm\frac{\kpp_{1}}{\kpp_{0}}}+o(1)\Big)\cdot\int_{t}^{\infty}\zeta^{2D}(\tau)d\tau.\label{eq:7.54}
\end{equation}

It remains to solve the integral equation \eqref{eq:7.54}. Denoting
$\vphi(t)=\int_{t}^{\infty}\zeta^{2D}(\tau)d\tau$, this integral
equation reads $(-\rd_{t}\vphi)^{1/(2D)}=(\sqrt{\gmm\kpp_{1}/\kpp_{0}}+o(1))\vphi$.
Hence, 
\[
-\rd_{t}\vphi=\Big(\Big(\gmm\frac{\kpp_{1}}{\kpp_{0}}\Big)^{D}+o(1)\Big)\vphi^{2D}.
\]
Integrating this, we get 
\[
\vphi=\Big\{(2D-1)\Big(\gmm\frac{\kpp_{1}}{\kpp_{0}}\Big)^{D}+o(1)\Big\}^{-\frac{1}{2D-1}}t^{-\frac{1}{2D-1}}.
\]
Recalling $\int_{t}^{\infty}\zeta^{2D}(\tau)d\tau=\vphi(t)$ and substituting
this into \eqref{eq:7.54}, we get 
\[
\zeta(t)=\Big\{\Big(\frac{\kpp_{0}}{(2D-1)^{2}\gmm\kpp_{1}}\Big)^{\frac{1}{4D-2}}+o(1)\Big\}\cdot t^{-1/(2D-1)}.
\]
Substituting the definition \eqref{eq:case3-def-gmm} of $\gmm$ and
$\zeta=(1+o(1))\lmb$ into the above, we conclude 
\[
\lmb(t)=\Big\{\Big(\frac{\kpp_{0}}{\kpp_{1}}\frac{\sum_{i=1}^{J}\frkc_{i}^{2/D}}{(2D-1)^{2}4D\sum_{i=1}^{J}|v_{i}^{\ast}|^{2}}\Big)^{\frac{1}{4D-2}}+o(1)\Big\}\cdot t^{-1/(2D-1)}.
\]
Substituting this and $D=\frac{N-2}{2}$ into $\lmb_{i}^{D}=(\frkc_{i}+o(1))\lmb^{D}$,
we finally get 
\[
\lmb_{i}(t)=\Big\{\Big(\frac{\kpp_{0}}{\kpp_{1}}\frac{\sum_{i=1}^{J}\frkc_{i}^{4/(N-2)}}{(N-3)^{2}(2N-4)\sum_{i=1}^{J}|v_{i}^{\ast}|^{2}}\Big)^{\frac{1}{2N-6}}\frkc_{i}^{2/(N-2)}+o(1)\Big\}\cdot t^{-1/(N-3)}.
\]
This completes the proof of \eqref{eq:case3-lmb-rate}.
\end{proof}

\subsection{Proof of the second item of Theorem~\ref{thm:main-lmb-to-zero-classification}}
\begin{proof}
Let $u(t)$ be as in Assumption~\ref{assumption:sequential-on-param}
and suppose Theorem~\ref{thm:main-one-bubble-tower-classification}
does not apply. Assume the configuration $\{(\iota_{1},z_{1}^{\ast}),\dots,(\iota_{J},z_{j}^{\ast})\}$
is minimally degenerate and satisfies \eqref{eq:case3-v_i-not-all-zero}.
Apply the case separation Proposition~\ref{prop:case-separation}
to $u(t)$. Note that Case~2 is excluded because we assumed Theorem~\ref{thm:main-one-bubble-tower-classification}
does not apply. Case~1 is excluded by Proposition~\ref{prop:case1-main}
and the degeneracy of our configuration. Therefore, only Case~3 is
possible. Applying Proposition~\ref{prop:case3-main} gives the desired
conclusion.
\end{proof}

\section{\label{sec:Construction-of-minimal-examples}Construction of minimal
examples}

In this section, we construct a \emph{minimal example} that demonstrates
the \emph{minimally degenerate case} of Theorem~\ref{thm:main-lmb-to-zero-classification}.
The goal of this section is Proposition~\ref{prop:Sec9-main} below.
In Remark~\ref{rem:other-examples} at the end of this section, we
comment on constructing minimal examples of other scenarios.
\begin{prop}
\label{prop:Sec9-main}Let $J=4$. Consider the minimally degenerate
configuration 
\begin{equation}
\left|\begin{aligned}\iota_{1} & =\iota_{2}=+1, & \iota_{3} & =\iota_{4}=-1,\\
z_{1}^{\ast} & =(\tfrac{1}{2},\tfrac{q_{0}}{2},0,\dots,0), & z_{3}^{\ast} & =(\tfrac{1}{2},-\tfrac{q_{0}}{2},0,\dots,0),\\
z_{2}^{\ast} & =(-\tfrac{1}{2},\tfrac{q_{0}}{2},0,\dots,0), & z_{4}^{\ast} & =(-\tfrac{1}{2},-\tfrac{q_{0}}{2},0,\dots,0),
\end{aligned}
\right.\label{eq:Sec9-def-config}
\end{equation}
where $q_{0}\in(1,\infty)$ is the unique number satisfying 
\[
1-\frac{1}{q_{0}^{2D}}-\frac{1}{(1+q_{0}^{2})^{D}}=0.
\]
Note that $\vec{\frkc}=\frac{1}{2}(1,1,1,1)\in\ker A^{\ast}$. Then,
there exists a solution $u(t)$ exhibiting the degenerate scenario
of Theorem~\ref{thm:main-lmb-to-zero-classification}.
\end{prop}

\subsection{\label{subsec:Linearization-formal-ODE}Linearization of formal ODE
system}

From the symmetry of the configuration \eqref{eq:Sec9-def-config},
we may also assume the following symmetry on $\vec{\lmb}(t)$ and
$\vec{z}(t)$: 
\begin{equation}
\lmb_{1}=\lmb_{2}=\lmb_{3}=\lmb_{4}\eqqcolon\lmb\label{eq:Sec9-symm-lmb}
\end{equation}
and 
\begin{equation}
\left|\begin{aligned}z_{1} & =(\tfrac{d}{2},\tfrac{q}{2},0,\dots,0), & z_{3} & =(\tfrac{d}{2},-\tfrac{q}{2},0,\dots,0),\\
z_{2} & =(-\tfrac{d}{2},\tfrac{q}{2},0,\dots,0), & z_{4} & =(-\tfrac{d}{2},-\tfrac{q}{2},0,\dots,0).
\end{aligned}
\right.\label{eq:Sec9-symm-z}
\end{equation}
(We chose a different but more convenient definition of $\lmb$ compared
to Section~\ref{sec:Case3}.)

The formal ODE system is then described by three parameters $\lmb,d,q$
and is given by 
\begin{equation}
\left|\begin{aligned}\lmb_{t} & =\kpp_{0}\kpp_{\infty}(d^{-2D}-q^{-2D}-(d^{2}+q^{2})^{-D})\lmb^{2D-1},\\
d_{t} & =2\kpp_{1}\kpp_{\infty}(-d^{-2D-1}+d(d^{2}+q^{2})^{-D-1})\lmb^{2D},\\
q_{t} & =2\kpp_{1}\kpp_{\infty}(q^{-2D-1}+q(d^{2}+q^{2})^{-D-1})\lmb^{2D}.
\end{aligned}
\right.\label{eq:Sec9-formal-ODE}
\end{equation}
For $(d,q)$ near $(1,q_{0})$, due to degeneracy, it is more appropriate
to Taylor expand the right hand side of $\lmb_{t}$ using 
\begin{equation}
\begin{aligned} & d^{-2D}-q^{-2D}-(d^{2}+q^{2})^{-D}\\
 & =2D(-1+(1+q_{0}^{2})^{-D-1})(d-1)+2D(q_{0}^{-2D-1}+q_{0}(1+q_{0}^{2})^{-D-1})(q-q_{0})\\
 & \quad+\calO((d-1)^{2}+(q-q_{0})^{2}).
\end{aligned}
\label{eq:Sec9-taylor-expn}
\end{equation}
Substituting \eqref{eq:Sec9-taylor-expn} into the ODE system \eqref{eq:Sec9-formal-ODE}
and taking only the leading order terms, we obtain a more simplified
formal ODE system 
\begin{equation}
\left|\begin{aligned}\lmb_{t} & =2D\kpp_{0}\kpp_{\infty}\big\{(-1+(1+q_{0}^{2})^{-D-1})(d-1)+(q_{0}^{-2D-1}+q_{0}(1+q_{0}^{2})^{-D-1})(q-q_{0})\big\}\lmb^{2D-1},\\
d_{t} & =2\kpp_{1}\kpp_{\infty}(-1+(1+q_{0})^{-D-1})\lmb^{2D},\\
q_{t} & =2\kpp_{1}\kpp_{\infty}(q_{0}^{-2D-1}+q_{0}(1+q_{0}^{2})^{-D-1})\lmb^{2D}.
\end{aligned}
\right.\label{eq:Sec9-formal-ODE-2}
\end{equation}
This ODE system \eqref{eq:Sec9-formal-ODE-2} has a solution 
\begin{equation}
(\lmb_{\mathrm{ap}}(t),d_{\mathrm{ap}}(t),q_{\mathrm{ap}}(t))=(c_{\lmb}t^{-\frac{1}{2D-1}},1+c_{d}t^{-\frac{1}{2D-1}},q_{0}+c_{q}t^{-\frac{1}{2D-1}})\qquad\text{as }t\to+\infty,\label{eq:Sec9-lmb,d,q-aprx-sol}
\end{equation}
where the constants $c_{\lmb},c_{d}>0$ and $c_{q}<0$ satisfy 
\begin{align*}
-\frac{c_{\lmb}}{2D-1} & =2D\kpp_{0}\kpp_{\infty}\big\{(-1+(1+q_{0}^{2})^{-D-1})c_{d}+(q_{0}^{-2D-1}+q_{0}(1+q_{0}^{2})^{-D-1})c_{q}\big\} c_{\lmb}^{2D-1},\\
-\frac{c_{d}}{2D-1} & =2\kpp_{1}\kpp_{\infty}(-1+(1+q_{0}^{2})^{-D-1})c_{\lmb}^{2D},\\
-\frac{c_{q}}{2D-1} & =2\kpp_{1}\kpp_{\infty}(q_{0}^{-2D-1}+q_{0}(1+q_{0}^{2})^{-D-1})c_{\lmb}^{2D}.
\end{align*}
In particular, we have 
\[
\frac{c_{d}}{c_{q}}=-q_{0}\quad\text{and}\quad\frac{c_{\lmb}}{c_{q}}=-\Big\{\frac{D\kpp_{0}}{\kpp_{1}}(1+q_{0}^{2})\Big\}^{1/2}.
\]
The precise formulas of $c_{\lmb},c_{d},c_{q}$ can be found by solving
the relations above.

Linearizing \eqref{eq:Sec9-formal-ODE-2} around \eqref{eq:Sec9-lmb,d,q-aprx-sol},
one formally obtains 
\begin{align*}
(\lmb-\lmb_{\mathrm{ap}})_{t} & \approx\frac{1}{t}\Big(-(\lmb-\lmb_{\mathrm{ap}})-\frac{D\kpp_{0}/\kpp_{1}}{2D-1}\frac{c_{d}}{c_{\lmb}}(d-d_{\mathrm{ap}})-\frac{D\kpp_{0}/\kpp_{1}}{2D-1}\frac{c_{q}}{c_{\lmb}}(q-q_{\mathrm{ap}})\Big),\\
(d-d_{\mathrm{ap}})_{t} & \approx\frac{1}{t}\Big(-\frac{2D}{2D-1}\frac{c_{d}}{c_{\lmb}}(\lmb-\lmb_{\mathrm{ap}})\Big),\\
(q-q_{\mathrm{ap}})_{t} & \approx\frac{1}{t}\Big(-\frac{2D}{2D-1}\frac{c_{q}}{c_{\lmb}}(\lmb-\lmb_{\mathrm{ap}})\Big).
\end{align*}
This motivates the decomposition 
\begin{equation}
\left|\begin{aligned}\lmb(t) & =\lmb_{\mathrm{ap}}(t)+t^{-\frac{1}{2D-1}}h_{1}(t),\\
d(t) & =d_{\mathrm{ap}}(t)+t^{-\frac{1}{2D-1}}(\tfrac{2\kpp_{1}}{\kpp_{0}})^{1/2}h_{2}(t),\\
q(t) & =q_{\mathrm{ap}}(t)+t^{-\frac{1}{2D-1}}(\tfrac{2\kpp_{1}}{\kpp_{0}})^{1/2}h_{3}(t),
\end{aligned}
\right.\qquad h(t)\coloneqq\begin{pmatrix}h_{1}(t)\\
h_{2}(t)\\
h_{3}(t)
\end{pmatrix}\in\bbR^{3},\label{eq:Sec9-def-h}
\end{equation}
so that $h(t)$ solves 
\[
\frac{d}{dt}h\approx\Big(M+\tfrac{1}{2D-1}I\Big)\frac{h}{t}
\]
with 
\begin{equation}
M=\begin{pmatrix}-1 & -\frac{\sqrt{2D}q_{0}}{(2D-1)\sqrt{1+q_{0}^{2}}} & \frac{\sqrt{2D}}{(2D-1)\sqrt{1+q_{0}^{2}}}\\
-\frac{\sqrt{2D}q_{0}}{(2D-1)\sqrt{1+q_{0}^{2}}} & 0 & 0\\
\frac{\sqrt{2D}}{(2D-1)\sqrt{1+q_{0}^{2}}} & 0 & 0
\end{pmatrix}.\label{eq:Sec9-def-M}
\end{equation}
The factor $(\tfrac{2\kpp_{1}}{\kpp_{0}})^{1/2}$ in \eqref{eq:Sec9-def-h}
makes the matrix $M$ symmetric. 

By direct computation, $M$ has eigenvalues $0$, $\frac{1}{2D-1}$,
and $-\frac{2D}{2D-1}$. The eigenvalue $0$ corresponds to scaling
invariance of the underlying ODE system. The matrix $M+\tfrac{1}{2D-1}I$
has eigenvalues $\frac{1}{2D-1}$, $\frac{2}{2D-1}$, and $-1$. Now,
we define the stable (resp., unstable) subspace $V_{s}$ (resp., $V_{u}$)
in $\bbR^{3}$ by the span of eigenvectors corresponding to negative
(resp., positive) eigenvalues of $M+\frac{1}{2D-1}I$. The associated
orthogonal projector is denoted by $P_{s}$ (resp., $P_{u}$).

\subsection{\label{subsec:sec9-setup}Setup for rigorous construction}

From now on, we focus on the construction of a \emph{nonlinear solution}
$u(t)$. From the symmetry of the configuration \eqref{eq:Sec9-def-config},
we assume $u(t)$ satisfies the following symmetry: 
\begin{align*}
u(t,x_{1},\dots,-x_{i},\dots,x_{N}) & =u(t,x_{1},\dots,x_{i},\dots,x_{N})\qquad\text{if }i\neq2,\\
u(t,x_{1},-x_{2},x_{3},\dots,x_{N}) & =-u(t,x_{1},x_{2},x_{3},\dots,x_{N}).
\end{align*}
The associated modulation parameters $\vec{\lmb}(t)$ and $\vec{z}(t)$
will also satisfy symmetry conditions \eqref{eq:Sec9-symm-lmb} and
\eqref{eq:Sec9-symm-z}. Due to uniqueness, the modified profiles
$U$ constructed in Proposition~\ref{prop:Modified-Profiles} are
symmetric accordingly. 

\uline{Small parameters}. We will need three small parameters 
\[
1\gg\eps_{1}\gg\eps_{2}\gg t_{0}^{-1}.
\]
Here, $t_{0}$ serves as the initial time of $u(t)$, so $u(t)$ will
be defined only for large times $t\geq t_{0}$. 

\uline{Initial data}. We consider the initial data at time $t_{0}$
of the form 
\begin{equation}
u(t_{0})=\sum_{i=1}^{4}\iota_{i}(W+a_{0}\calY)_{\lmb_{0},(z_{i})_{0}},\label{eq:Sec9-init-form}
\end{equation}
where $(z_{1})_{0},\dots,(z_{4})_{0}\in\bbR^{N}$ satisfy the symmetry
condition \eqref{eq:Sec9-symm-z}, $(\lmb_{0},d_{0},q_{0})$ are in
a small neighborhood $(c_{\lmb}t_{0}^{-\frac{1}{2D-1}},1+c_{d}t_{0}^{-\frac{1}{2D-1}},q_{0}+c_{q}t_{0}^{-\frac{1}{2D-1}})$,
and $a_{0}$ is small. As the degenerate scenario is dynamically unstable,
we consider a three-parameter family of initial data $u(t_{0})$ and
employ a Brouwer argument as in \cite{MerleZaag1997Duke,CoteMartelMerle2011RMI}
to find a special initial data exhibiting the desired asymptotic behavior. 

More precise choice of our family of initial data is 
\begin{equation}
\Big\{ u(t_{0})=\sum_{i=1}^{4}\iota_{i}(W+a_{0}\calY)_{\lmb_{0},(z_{i})_{0}}:(h_{0},a_{0})\in\br{\calU}\Big\},\label{eq:Sec9-def-calO}
\end{equation}
where $\lmb_{0}$ and $(z_{i})_{0}$ are determined by $h_{0}$ through
\eqref{eq:Sec9-def-h} at $t=t_{0}$ and the symmetry condition \eqref{eq:Sec9-symm-z},
and 
\begin{equation}
\br{\calU}\coloneqq\Theta_{u}^{-1}(\br{B_{V_{u}}(0;\eps_{1})}\times\br{B_{\bbR}(0;\eps_{2})})\label{eq:def-init-data-set}
\end{equation}
with an almost identity map $\Theta_{u}^{-1}$ defined in Lemma~\ref{lem:sec9-init-data-map}
below. In Lemma~\ref{lem:sec9-init-data-map}, we start from $(h_{0},a_{0})$,
form $u(t_{0})$ via \eqref{eq:Sec9-init-form}, decompose it according
to Lemma~\ref{lem:static-modulation} with parameters $(\vec{\iota},\vec{\lmb}(t_{0}),\vec{z}(t_{0}))$
(that should satisfy symmetry conditions \eqref{eq:Sec9-symm-lmb}
and \eqref{eq:Sec9-symm-z}), and define $h(t_{0}),a(t_{0})$ via
\eqref{eq:Sec9-def-h} at $t=t_{0}$ and $a(t_{0})\coloneqq\lan\calY_{\ul{;1}},g(t_{0})\ran$,
respectively.
\begin{lem}[Choice of the family of initial data]
\label{lem:sec9-init-data-map}\textup{The map 
\[
\Theta_{u}:(h_{0},a_{0})\in B_{V_{u}}(0;2\eps_{1})\times B_{\bbR}(0;2\eps_{1})\mapsto(P_{u}h(t_{0}),a(t_{0}))\in B_{V_{u}}(0;3\eps_{1})\times B_{\bbR}(0;3\eps_{1})
\]
is a $C^{1}$-diffeomorphism onto its image and is an almost identity
map in the sense that 
\begin{equation}
|P_{s}h(t_{0})|+|P_{u}h(t_{0})-h_{0}|+|a(t_{0})-a_{0}|=o_{t_{0}\to\infty}(1).\label{eq:Sec9-stable-initially-small}
\end{equation}
}
\end{lem}

\begin{proof}
First, observe that it suffices to show that the map 
\[
\Theta:(h_{0},a_{0})\in B_{\bbR^{3}}(0;2\eps_{1})\times B_{\bbR}(0;2\eps_{1})\mapsto(h(t_{0}),a(t_{0}))\in B_{\bbR^{3}}(0;3\eps_{1})\times B_{\bbR}(0;3\eps_{1})
\]
is $C^{1}$ and satisfies 
\begin{align*}
|h(t_{0})-h_{0}|+|a(t_{0})-a_{0}| & =o_{t_{0}\to\infty}(1),\\
\frac{\rd(h(t_{0}),a(t_{0}))}{\rd(h_{0},a_{0})} & =\mathrm{id}_{4\times4}+\calO(\eps_{1}),
\end{align*}
because $\Theta_{u}$ is then obtained by restricting $\Theta$ onto
the (codimension one) subspace $V_{u}\times\bbR$ and taking $P_{u}$
to the $h$-component. Henceforth, we show the above claim for $\Theta$.
As one can proceed similarly as in the proof of Lemma~\ref{lem:static-modulation},
we only sketch the proof. Consider the map $\bfF=(F_{0},F_{1},F_{2},F_{3})$:
\begin{align*}
F_{a}(\lmb_{0},d_{0},q_{0},a_{0};\lmb,d,q,a) & =\lan\calZ_{a\ul{;1}},u_{0}(\lmb_{0},d_{0},q_{0},a_{0})-U(\lmb,d,q)\ran\quad\text{if }a\in\{0,1,2\},\\
F_{3}(\lmb_{0},d_{0},q_{0},a_{0};\lmb,d,q,a) & =a-\lan\calY_{\ul{;1}},u_{0}(\lmb_{0},d_{0},q_{0},a_{0})-U(\lmb,d,q)\ran,
\end{align*}
where we denoted $u_{0}(\lmb_{0},d_{0},q_{0},a_{0})=\sum_{i=1}^{4}\iota_{i}(W+a_{0}\calY)_{\lmb_{0},(z_{i})_{0}}$,
$U(\lmb,d,q)=U(\vec{\iota},\vec{\lmb},\vec{z})$ with $\vec{\lmb},\vec{z}$
defined by symmetry conditions \eqref{eq:Sec9-symm-lmb} and \eqref{eq:Sec9-symm-z},
and $\calZ_{a\ul{;1}},\calY_{\ul{;1}}$ are defined through $(\vec{\iota},\vec{\lmb},\vec{z})$.
Observe 
\begin{align*}
F_{a}(\lmb_{0},d_{0},q_{0},a_{0};\lmb_{0},d_{0},q_{0},a_{0}) & =\lan\calZ_{a\ul{;1}},\td U(\lmb_{0},d_{0},q_{0})+\tsum{i=1}4a_{0}\calY_{;i}\ran=o_{t_{0}\to\infty}(1),\\
F_{3}(\lmb_{0},d_{0},q_{0},a_{0};\lmb_{0},d_{0},q_{0},a_{0}) & =a_{0}-\lan\calY_{\ul{;1}},\td U(\lmb_{0},d_{0},q_{0})+\tsum{i=1}4a_{0}\calY_{;i}\ran=o_{t_{0}\to\infty}(1).
\end{align*}
Observe also 
\begin{align*}
\begin{bmatrix}\lmb\rd_{\lmb}\bfF & 2\lmb\rd_{d}\bfF & 2\lmb\rd_{q}\bfF & \rd_{a}\bfF\end{bmatrix} & =\mathrm{id}_{4\times4}+\calO(\eps_{1}),\\
\begin{bmatrix}\lmb_{0}\rd_{\lmb_{0}}\bfF & 2\lmb_{0}\rd_{d_{0}}\bfF & 2\lmb_{0}\rd_{q_{0}}\bfF & \rd_{a_{0}}\bfF\end{bmatrix} & =-\mathrm{id}_{4\times4}+\calO(\eps_{1}).
\end{align*}
By the implicit function theorem, $(\lmb,d,q,a)$ becomes a function
of $\lmb_{0},d_{0},q_{0},a_{0}$ such that 
\begin{align*}
\frac{|\lmb-\lmb_{0}|+|d-d_{0}|+|q-q_{0}|}{\lmb_{0}}+|a-a_{0}| & =o_{t_{0}\to\infty}(1),\\
\frac{\rd(\lmb,d,q,a)}{\rd(\lmb_{0},d_{0},q_{0},a_{0})} & =\mathrm{id}_{4\times4}+\calO(\eps_{1}).
\end{align*}
Writing this in terms of $(h_{0},a_{0})$, we get the claim. This
completes the sketch of the proof.
\end{proof}
\uline{Bootstrap hypothesis}. Consider an initial data $u(t_{0})$
in \eqref{eq:Sec9-def-calO} at time $t_{0}$; let $u$ be its forward-in-time
evolution with maximal lifespan $[t_{0},T_{+}(u))$. First, define
\[
T'\coloneqq\sup\{\tau\in[t_{0},T_{+}(u)):u(t)\in\calT_{4}(\eps_{1})\text{ and }|h(t)|<10\eps_{1}\}\in(t_{0},T_{+}(u)]
\]
so that $u(t)$ can be decomposed as $u(t)=U(\vec{\iota},\vec{\lmb}(t),\vec{z}(t))+g(t)$
according to Lemma~\ref{lem:curve-modulation} and $\lmb(t)\aeq|d(t)-1|\aeq|q(t)-q_{0}|\aeq t^{-\frac{1}{2D-1}}$
on $[t_{0},T')$. We denote $a(t)\coloneqq\lan\calY_{\ul{;1}}(t),g(t)\ran$. 

To complete the bootstrap setup, we introduce a more precise bootstrap
hypothesis for \emph{refined modulation parameters}. With the current
$(\lmb,d,q,a)$, their time derivatives are controlled only in time
averaged sense, but the Brouwer argument requires pointwise controls.
In order to obtain pointwise controls of time derivatives, we introduce
refined modulation parameters as in the proof of classification theorems.
Let us rewrite Proposition~\ref{prop:case3-modulation-est} in terms
of $(\lmb,d,q)$ as follows. (See Lemma~\ref{lem:non-coll-ref-mod-est}
for $a$.)
\begin{lem}[Refined modulation estimates]
\label{lem:sec9-refined-mod-est}Assume $\lmb(t)\aeq|d(t)-1|\aeq|q(t)-q_{0}|\aeq t^{-\frac{1}{2D-1}}$.
Then, we have 
\begin{align*}
\lmb_{t}+\lmb\frac{d}{dt}f & =\kpp_{0}\kpp_{\infty}(d^{-2D}-q^{-2D}-(d^{2}+q^{2})^{-D})\lmb^{2D-1}+\lmb\cdot\{F_{\lmb}+\calO(\eps_{1}^{2}t^{-1})\},\\
d_{t} & =2\kpp_{1}\kpp_{\infty}(-d^{-2D-1}+d(d^{2}+q^{2})^{-D-1})\lmb^{2D}+\lmb\cdot\{F_{d}+\calO(\eps_{1}^{2}t^{-1})\},\\
q_{t} & =2\kpp_{1}\kpp_{\infty}(q^{-2D-1}+q(d^{2}+q^{2})^{-D-1})\lmb^{2D}+\lmb\cdot\{F_{q}+\calO(\eps_{1}^{2}t^{-1})\},\\
a_{t} & =\frac{e_{0}}{\lmb^{2}}a+F_{a}+\calO(t^{-1})
\end{align*}
for some functions $f,F_{\lmb},F_{d},F_{q},F_{a}$ that depend continuously
with respect to the $\dot{H}^{1}$-topology of $u(t)$ and satisfy
the estimates 
\begin{align*}
|f| & \aleq\|g\|_{\dot{H}^{1}},\\
|F_{\lmb}|+|F_{d}|+|F_{q}| & \aleq\eps_{1}^{-2}\|g\|_{\dot{H}^{2}}^{2},\\
|F_{a}| & \aleq\|g\|_{\dot{H}^{2}}^{2}.
\end{align*}
\end{lem}

\begin{proof}
The bounds follow from Proposition~\ref{prop:case3-modulation-est}
(and Lemma~\ref{lem:non-coll-ref-mod-est} for $a(t)$) together
with $D\geq\frac{5}{2}$: 
\[
\|g\|_{\dot{H}^{2}}^{2}+\lmb^{2D-3}\|g\|_{\dot{H}^{2}}+\lmb^{2D}|\log\lmb|\aleq\|g\|_{\dot{H}^{2}}^{2}+t^{-\frac{1}{2}}\|g\|_{\dot{H}^{1}}+t^{-1}\aleq\eps_{1}^{-2}\|g\|_{\dot{H}^{2}}^{2}+\eps_{1}^{2}t^{-1}.
\]
Regarding the continuity, observe that all the inner products arising
in the modulation estimates (e.g., in \eqref{eq:modulation-identity-psi})
are continuous with respect to $\vec{\lmb}(t)$, $\vec{z}(t)$, and
$g(t)$. Since these depend continuously on $u(t)$, the continuous
dependence for functions $f,F_{\lmb},F_{d},F_{q},F_{a}$ follows.
\end{proof}
With the above lemma at hand, we can define the refined modulation
parameters
\begin{align*}
\wh{\lmb}(t) & \coloneqq\lmb(t)e^{\int_{t_{0}}^{t}F_{\lmb}(\tau)d\tau+f(t)-f(t_{0})},\\
\wh d(t)-1 & \coloneqq(d(t)-1)e^{\int_{t_{0}}^{t}\frac{\lmb F_{d}}{d-1}(\tau)d\tau},\\
\wh q(t)-q_{0} & \coloneqq(q(t)-q_{0})e^{\int_{t_{0}}^{t}\frac{\lmb F_{q}}{q-q_{0}}(\tau)d\tau},\\
\wh a(t) & \coloneqq a(t)+\tint{t_{0}}tF_{a}(\tau)d\tau.
\end{align*}
Define $\wh h$ as in \eqref{eq:Sec9-def-h} by replacing $\lmb,d,q$
by $\wh{\lmb},\wh d,\wh q$, respectively. Notice that $\wh h(t_{0})=h(t_{0})$
and $\wh a(t_{0})=a(t_{0})$. We implicitly used the fact that $c_{d}$
and $c_{q}$ are nonzero, but this is not essential. Define the bootstrap
time 
\[
T\coloneqq\sup\{\tau\in[t_{0},T'):\text{\eqref{eq:Sec9-def-bootstrap} holds for all }t\in[t_{0},\tau]\}\in(t_{0},T'],
\]
where 
\begin{equation}
\left\{ \begin{aligned}|P_{s}\wh h(t)| & \leq\eps_{1},\\
|P_{u}\wh h(t)| & \leq\eps_{1},\\
|\wh a(t)| & \leq\eps_{2}.
\end{aligned}
\right.\label{eq:Sec9-def-bootstrap}
\end{equation}
Note that these hypotheses are satisfied at $t_{0}$ by \eqref{eq:def-init-data-set},
$\wh h(t_{0})=h(t_{0})$, and $\wh a(t_{0})=a(t_{0})$.

\subsection{\label{subsec:sec9-closing-bootstrap}Closing bootstrap}
\begin{lem}[Control of $g(t)$ and its consequences]
We have 
\begin{equation}
\sup_{t\in[t_{0},T)}\|g(t)\|_{\dot{H}^{1}}^{2}+\int_{t_{0}}^{T}\|g(t)\|_{\dot{H}^{2}}^{2}dt\aleq\eps_{2}^{2}.\label{eq:Sec9-control-g(t)}
\end{equation}
From this, we have 
\begin{align*}
\wh{\lmb}(t) & =\lmb(t)\cdot\{1+\calO(\eps_{1}^{-2}\eps_{2}^{2})\},\\
\wh d(t)-1 & =(d(t)-1)\cdot\{1+\calO(\eps_{1}^{-2}\eps_{2}^{2})\},\\
\wh q(t)-q_{0} & =(q(t)-1)\cdot\{1+\calO(\eps_{1}^{-2}\eps_{2}^{2})\},\\
\wh a(t) & =a(t)+\calO(\eps_{2}^{2}).
\end{align*}
\end{lem}

\begin{proof}
We use a simple energy argument. Since $\lmb(t)\aeq t^{-\frac{1}{2D-1}}\aleq o_{t\to\infty}(1)$
and $\|g(t)\|_{\dot{H}^{1}}\aleq\eps_{1}$ (due to $u(t)\in\calT_{4}(\eps_{1})$),
we have 
\begin{align*}
E[u(t)] & =E[U]-\tint{}{}(\Dlt U+f(U))g-\tfrac{1}{2}\lan\calL_{U}g,g\ran+\tint{}{}\calO(|g|^{p+1})\\
 & =(JE[W]+o_{t\to\infty}(1))+o_{t\to\infty}(1)\cdot\|g\|_{\dot{H}^{1}}-(\tfrac{1}{2}\lan\calL_{\calW}g,g\ran+o_{t\to\infty}(1)\cdot\|g\|_{\dot{H}^{1}}^{2})+\calO(\|g\|_{\dot{H}^{1}}^{p+1})\\
 & =JE[W]+o_{t\to\infty}(1)-\tfrac{1}{2}\lan\calL_{\calW}g,g\ran+\calO(\eps_{1}^{p-1})\cdot\|g\|_{\dot{H}^{1}}^{2}.
\end{align*}
Taking $t=t_{0}$ and applying $\|g(t_{0})\|_{\dot{H}^{1}}=\calO(\eps_{2})$
(this follows from \eqref{eq:Sec9-stable-initially-small} and $|a(t_{0})|\aleq\eps_{2}$)
give 
\begin{equation}
E[u(t_{0})]-JE[W]=o_{t_{0}\to\infty}(1)+\calO(\|g(t_{0})\|_{\dot{H}^{1}}^{2})=\calO(\eps_{2}^{2}).\label{eq:9.14}
\end{equation}
Applying the linear coercivity \eqref{eq:calL-H1-coer-quad-form},
there exists a universal constant $c>0$ such that 
\begin{equation}
c\|g(t)\|_{\dot{H}^{1}}^{2}\leq E[u(t)]-JE[W]+o_{t\to\infty}(1)+\calO(a^{2}(t)).\label{eq:Sec9-energy-ineq}
\end{equation}
On the other hand, by \eqref{eq:energy-identity} and \eqref{eq:coercivity-dissipation},
we have (after shrinking $c>0$ if necessary) 
\[
c\int_{t_{0}}^{t}\|g(\tau)\|_{\dot{H}^{2}}^{2}d\tau\leq E[u(t_{0})]-E[u(t)].
\]
Summing the previous two displays and applying \eqref{eq:9.14} gives
\begin{equation}
\|g(t)\|_{\dot{H}^{1}}^{2}+\int_{t_{0}}^{t}\|g(\tau)\|_{\dot{H}^{2}}^{2}d\tau\aleq\eps_{2}^{2}+a^{2}(t).\label{eq:9.17}
\end{equation}
Note that, by \eqref{eq:Sec9-def-bootstrap} and the above display,
$a(t)=\wh a(t)+\int_{t_{0}}^{t}\calO(\|g(\tau)\|_{\dot{H}^{2}}^{2})d\tau=\calO(\eps_{2})+\calO(a^{2}(t))$.
Since $a(t)=\calO(\|g(t)\|_{\dot{H}^{1}})=\calO(\eps_{1})$ (smaller
than $1$), we conclude $a(t)=\calO(\eps_{2})$. Substituting this
into \eqref{eq:9.17} completes the proof of \eqref{eq:Sec9-control-g(t)}.
The estimates for $\wh{\lmb}$, $\wh d$, $\wh q$, and $\wh a$ follow
from substituting the control \eqref{eq:Sec9-control-g(t)} into their
definitions.
\end{proof}
Next, we turn to control modulation parameters. 
\begin{lem}[Modulation estimates in terms of $\wh h$ and $\wh a$]
We have 
\begin{align}
\wh h_{t} & =\Big(M+\tfrac{1}{2D-1}I+\calO(\eps_{1})\Big)\frac{\wh h}{t}+\calO(\frac{\eps_{1}^{-2}\eps_{2}^{2}}{t}),\label{eq:sec9-wh-h-mod-est}\\
\wh a_{t} & =\frac{e_{0}}{\lmb^{2}}\wh a+\calO(\frac{\eps_{2}^{2}}{\lmb^{2}}+\frac{1}{t}).\label{eq:sec9-wh-a-mod-est}
\end{align}
\end{lem}

\begin{proof}
In view of \eqref{eq:Sec9-taylor-expn}, we have 
\begin{align*}
 & d^{-2D}-q^{-2D}-(d^{2}+q^{2})^{-D}\\
 & =2D(-1+(1+q_{0}^{2})^{-D-1})(d-1)+2D(q_{0}^{-2D-1}+q_{0}(1+q_{0}^{2})^{-D-1})(q-q_{0})+\calO(t^{-\frac{2}{2D-1}})\\
 & =2D(-1+(1+q_{0}^{2})^{-D-1})(\wh d-1)+2D(q_{0}^{-2D-1}+q_{0}(1+q_{0}^{2})^{-D-1})(\wh q-q_{0})+\calO(\eps_{1}^{-2}\eps_{2}^{2}t^{-\frac{1}{2D-1}}).
\end{align*}
Combining this with Lemma~\ref{lem:sec9-refined-mod-est}, we obtain
(cf.~\eqref{eq:Sec9-formal-ODE-2})
\begin{align*}
\wh{\lmb}_{t} & =2D\kpp_{0}\kpp_{\infty}\big\{(-1+(1+q_{0}^{2})^{-D-1})(\wh d-1)+(q_{0}^{-2D-1}+q_{0}(1+q_{0}^{2})^{-D-1})(\wh q-q_{0})\big\}\wh{\lmb}^{2D-1}\\
 & \quad+\calO(\eps_{1}^{-2}\eps_{2}^{2}t^{-\frac{2D}{2D-1}}),\\
\wh d_{t} & =2\kpp_{1}\kpp_{\infty}(-1+(1+q_{0}^{2})^{-D-1})\wh{\lmb}^{2D}+\calO(\eps_{1}^{-2}\eps_{2}^{2}t^{-\frac{2D}{2D-1}}),\\
\wh q_{t} & =2\kpp_{1}\kpp_{\infty}(q_{0}^{-2D-1}+q_{0}(1+q_{0}^{2})^{-D-1})\wh{\lmb}^{2D}+\calO(\eps_{1}^{-2}\eps_{2}^{2}t^{-\frac{2D}{2D-1}}).
\end{align*}
Combining this with the definition \eqref{eq:Sec9-def-h} of $\wh h$
and recalling the derivation of \eqref{eq:Sec9-def-M}, we obtain
\eqref{eq:sec9-wh-h-mod-est}. The proof of \eqref{eq:sec9-wh-a-mod-est}
easily follows from Lemma~\ref{lem:sec9-refined-mod-est} and $\wh a(t)=a(t)+\calO(\eps_{2}^{2})$.
\end{proof}
\begin{lem}[Control of stable directions]
We have 
\begin{equation}
|P_{s}\wh h|(t)\aleq\eps_{1}^{3/2}.\label{eq:Sec9-control-stable}
\end{equation}
\end{lem}

\begin{proof}
By \eqref{eq:sec9-wh-h-mod-est} and \eqref{eq:Sec9-def-bootstrap},
we have 
\[
\frac{d}{dt}\frac{1}{2}|P_{s}\wh h|^{2}=\frac{1}{t}\lan P_{s}\wh h,(M+\tfrac{1}{2D-1}I)P_{s}\wh h\ran+\calO(\eps_{1}\frac{|\wh h|^{2}}{t}+\eps_{1}^{-2}\eps_{2}^{2}\frac{|\wh h|}{t})\leq-\frac{9}{10t}|P_{s}\wh h|^{2}+\calO(\frac{\eps_{1}^{3}}{t}).
\]
Integrating this from $t_{0}$ and applying \eqref{eq:Sec9-stable-initially-small},
we obtain 
\[
|P_{s}\wh h|^{2}(t)\leq|P_{s}\wh h|^{2}(t_{0})+\int_{t_{0}}^{t}(\tau/t)^{\frac{18}{10}}\cdot\tau^{-1}\eps_{1}^{3}d\tau\aleq o_{t_{0}\to\infty}(1)+\eps_{1}^{3}\aleq\eps_{1}^{3}.\qedhere
\]
\end{proof}
\begin{lem}[Control of unstable directions; closing bootstrap]
There exists an initial data $u(t_{0})$ in \eqref{eq:Sec9-def-calO}
such that $T=T'=+\infty$.
\end{lem}

\begin{proof}
We use the Brouwer argument. Suppose not; for any initial data $u(t_{0})$
in \eqref{eq:Sec9-def-calO}, we have $T<+\infty$. By \eqref{eq:Sec9-control-g(t)},
we have $T'>T$. By \eqref{eq:Sec9-control-stable} and \eqref{eq:Sec9-def-bootstrap},
we have $|P_{u}\wh h(T)|=\eps_{1}$ or $|\wh a(T)|=\eps_{2}$. Now,
observe that for any time $t\in[t_{0},T]$ with $|P_{u}\wh h(t)|=\eps_{1}$,
we have 
\begin{align}
\frac{d}{dt}\frac{1}{2}|P_{u}\wh h|^{2} & =\frac{1}{t}\lan P_{u}\wh h,(M+\tfrac{1}{2D-1}I)P_{u}\wh h\ran+\calO(\eps_{1}\frac{|\wh h|^{2}}{t}+\eps_{1}^{-2}\eps_{2}^{2}\frac{|\wh h|}{t})\nonumber \\
 & \geq\frac{1}{(2D-1)t}|P_{u}\wh h|^{2}-\calO(\frac{\eps_{1}^{3}}{t})\geq\frac{\eps_{1}^{2}}{2Dt}.\label{eq:Sec9-unstab-deriv-positive-1}
\end{align}
Observe also that for any time $t\in[t_{0},T]$ with $|\wh a(t)|=\eps_{2}$,
we have 
\begin{equation}
\frac{d}{dt}\frac{1}{2}|\wh a|^{2}=\Big(\frac{e_{0}}{\lmb^{2}}|\wh a|+\calO(\frac{\eps_{2}^{2}}{\lmb^{2}}+\frac{1}{t})\Big)|\wh a|\geq\Big(\frac{e_{0}}{\lmb^{2}}\cdot\frac{\eps_{2}}{2}\Big)\frac{\eps_{2}}{2}.\label{eq:Sec9-unstab-deriv-positive-2}
\end{equation}
In particular, $T$ is the first time $t\in[t_{0},T]$ such that either
$|P_{u}\wh h(t)|=\eps_{1}$ or $|\wh a(t)|=\eps_{2}$.

Consider the map 
\[
\Phi:(h(t_{0}),a(t_{0}))\in\br{B_{V_{u}}(0;\eps_{1})}\times\br{B_{\bbR}(0;\eps_{2})}\mapsto(P_{u}\wh h(T),\wh a(T))\in\rd(B_{V_{u}}(0;\eps_{1})\times B_{\bbR}(0;\eps_{2})).
\]
First, as $\wh h(t_{0})=h(t_{0})$ and $\wh a(t_{0})=a(t_{0})$, $\Phi$
is the identity map on the boundary $\rd(B_{V_{u}}(0;\eps_{1})\times B_{\bbR}(0;\eps_{2}))$.
Next, we claim that $\Phi$ is continuous. It suffices to show that
$u(t_{0})\mapsto(\wh h(T),\wh a(T))$ is continuous. By the continuity
statement of Lemma~\ref{lem:sec9-refined-mod-est} and the local
well-posedness of \eqref{eq:NLH} in $\dot{H}^{1}$, $(u(t_{0}),t)\mapsto(\wh h(t),\wh a(t))$
and $(u(t_{0}),t)\mapsto(\frac{d}{dt}\wh h(t),\frac{d}{dt}\wh a(t))$
are continuous. The positivity \eqref{eq:Sec9-unstab-deriv-positive-1}
and \eqref{eq:Sec9-unstab-deriv-positive-2} ensures that $u(t_{0})\mapsto T$
is continuous, so $u(t_{0})\mapsto(\wh h(T),\wh a(T))$ is continuous
as desired.

By the previous paragraph, the map $\Phi:\br{B_{V_{u}}(0;\eps_{1})}\times\br{B_{\bbR}(0;\eps_{2})}\to\rd(B_{V_{u}}(0;\eps_{1})\times B_{\bbR}(0;\eps_{2}))$
is continuous and equal to the identity map on the boundary. Such
a map cannot exist by the Brouwer fixed point theorem, so we get a
contradiction. Therefore, $T=+\infty$ (and hence $T=T'=+\infty$)
for some special initial data $u(t_{0})$ in \eqref{eq:Sec9-def-calO}. 
\end{proof}

\subsection{Proof of Proposition~\ref{prop:Sec9-main}}

Having closed the bootstrap, we need to show $\|g(t)\|_{\dot{H}^{1}}\to0$
and derive sharp asymptotics of $\lmb(t)$. Note that we already know
$\lmb(t)\aeq|\vec{z}(t)-\vec{z}^{\ast}|\aeq t^{-\frac{1}{2D-1}}$
for all large $t$ from the bootstrap.
\begin{lem}[$\dot{H}^{1}$-decay of $g(t)$]
We have 
\begin{equation}
\|g(t)\|_{\dot{H}^{1}}\to0\qquad\text{as }t\to+\infty.\label{eq:9.13}
\end{equation}
\end{lem}

\begin{proof}
\uline{Step 1: Proof of \mbox{$a(t)\to0$}}. We begin with 
\[
a_{t}-\frac{e_{0}}{\lmb^{2}}a=\calO(\|g\|_{\dot{H}^{2}}^{2}+t^{-1}).
\]
In particular, we have 
\[
(e^{-t}a^{2})_{t}\geq-e^{-t}|a|\cdot\calO(\|g\|_{\dot{H}^{2}}^{2}+t^{-1}).
\]
Since $e^{-t}a^{2}(t)\to0$ as $t\to+\infty$, we can integrate from
$+\infty$ to obtain 
\begin{align*}
a^{2}(t) & \aleq\int_{t}^{\infty}e^{t-\tau}|a(\tau)|\cdot(\|g(\tau)\|_{\dot{H}^{2}}^{2}+\tau^{-1})d\tau\\
 & \aleq\sup_{\tau\ge t}|a(\tau)|\cdot\Big(\int_{t}^{\infty}\|g(\tau)\|_{\dot{H}^{2}}^{2}d\tau+t^{-1}\Big)\leq o_{t\to\infty}(1)\cdot\sup_{\tau\geq t}|a(\tau)|.
\end{align*}
This concludes the proof of $a(t)\to0$ as $t\to+\infty$.

\uline{Step 2: Proof of \mbox{$\|g(t_{n})\|_{\dot{H}^{1}}\to0$}
for some \mbox{$t_{n}\to+\infty$}}. If not, we have 
\begin{equation}
\liminf_{t\to+\infty}\|g(t)\|_{\dot{H}^{1}}\eqqcolon\eps_{\ast}>0.\label{eq:9.11}
\end{equation}
Since $u(t_{0})\in L^{2}$, $u(t)$ is a global $H^{1}$-solution
to \eqref{eq:NLH} and we can access to the identity 
\[
\int u^{2}(t_{1})=\int u^{2}(t_{0})-\int_{t_{0}}^{t_{1}}\int_{\bbR^{N}}\{|\nabla u|^{2}-|u|^{p+1}\}dxdt
\]
for any $t_{1}\geq t_{0}$. Taking $t_{1}\to+\infty$ and using $\int u^{2}(t_{1})\geq0$,
we conclude 
\begin{equation}
\limsup_{t_{1}\to+\infty}\int_{t_{0}}^{t_{1}}\int_{\bbR^{N}}\{|\nabla u|^{2}-|u|^{p+1}\}dxdt<\infty.\label{eq:9.12}
\end{equation}
We use \eqref{eq:9.11} to derive a contradiction to \eqref{eq:9.12}.
Observe that 
\begin{align*}
 & \tint{}{}\{|\nabla u(t)|^{2}-|u(t)|^{p+1}\}\\
 & =\tint{}{}\{|\nabla U|^{2}-|U|^{p+1}\}+\tint{}{}(-\Dlt U+f(U))g+\tint{}{}p(f(U)-f(\calW))g+\tint{}{}pf(\calW)g\\
 & \quad+\tint{}{}|\nabla g|^{2}-p(p-1)\tint{}{}\{|U|^{p-1}-|\calW|^{p-1}\}g^{2}-p(p-1)\tint{}{}|\calW|^{p-1}g^{2}+\tint{}{}\calO(|g|^{p+1})\\
 & =o_{t\to\infty}(1)+o_{t\to\infty}(1)\cdot\|g\|_{\dot{H}^{1}}+o_{t\to\infty}(1)\cdot\|g\|_{\dot{H}^{1}}+\calO(\lmb\|g\|_{\dot{H}^{2}})\\
 & \quad+\|g\|_{\dot{H}^{1}}^{2}+o_{t\to\infty}(1)\cdot\|g\|_{\dot{H}^{1}}^{2}-\calO(\lmb^{2}\|g\|_{\dot{H}^{2}}^{2})+\calO(\|g\|_{\dot{H}^{1}}^{p+1})\\
 & \geq(1-\calO(\eps_{2}^{p-1})-o_{t\to\infty}(1))\tint{}{}|\nabla g|^{2}-o_{t\to\infty}(1)-\calO(\|g\|_{\dot{H}^{2}}^{2}).
\end{align*}
By \eqref{eq:9.11}, for all large $t$, we have 
\[
\int\{|\nabla u(t)|^{2}-|u(t)|^{p+1}\}\geq\frac{1}{2}\eps_{\ast}^{2}-\calO(\|g(t)\|_{\dot{H}^{2}}^{2}).
\]
Note that $\|g(t)\|_{\dot{H}^{2}}^{2}$ is integrable, but $\frac{1}{2}\eps_{\ast}^{2}$
is not integrable as $t\to+\infty$. This contradicts to \eqref{eq:9.12}.

\uline{Step 3: Proof of \mbox{\eqref{eq:9.13}}}. By Step~2, we
have $E[u(t_{n})]\to JE[W]$. Since $E[u(t)]$ is decreasing in time,
we have $E[u(t)]\to JE[W]$. This together with \eqref{eq:Sec9-energy-ineq}
and $a(t)\to0$ of Step~1 gives 
\[
c\|g(t)\|_{\dot{H}^{1}}^{2}\leq E[u(t)]-JE[W]+o_{t\to\infty}(1)\leq o_{t\to\infty}(1).
\]
This completes the proof of \eqref{eq:9.13}.
\end{proof}
To complete the proof of Proposition~\ref{prop:Sec9-main}, it remains
to derive sharp asymptotics of $\lmb(t)$. 
\begin{proof}[Proof of Proposition~\ref{prop:Sec9-main}]
One might directly integrate the modulation equations from $t=+\infty$,
but let us resort to the classification theorem (the minimally degenerate
case of Theorem~\ref{thm:main-lmb-to-zero-classification}) because
we know that all $\|g(t)\|_{\dot{H}^{1}}$, $\lmb(t)$, and $|\vec{z}(t)-\vec{z}^{\ast}|$
go to zero as $t\to+\infty$. This completes the proof.
\end{proof}
\begin{rem}[Construction of other examples]
\label{rem:other-examples}The construction of minimal examples for
Theorem~\ref{thm:main-one-bubble-tower-classification} and the first
item of Theorem~\ref{thm:main-lmb-to-zero-classification} requires
only $J=2$ with $\iota_{1}=-\iota_{2}$. As in the proof of Proposition~\ref{prop:Sec9-main},
one linearizes the formal dynamical system and identifies stable and
unstable directions (as in Section~\ref{subsec:Linearization-formal-ODE}),
and then setup bootstrap hypotheses (as in Section~\ref{subsec:sec9-setup}).
One can control stable directions forwards in time, whereas for unstable
directions one performs a Brouwer argument to choose a special initial
data (as in Section~\ref{subsec:sec9-closing-bootstrap}). We skip
the construction of these standard (or expected) minimal examples
as it is not the main focus of this paper.
\end{rem}

\appendix

\section{\label{subsec:proof-of-lin-coer}Proof of Proposition~\ref{prop:linear-coer}}
\begin{proof}
\uline{Proof of \mbox{\eqref{eq:calL-H1-bdd}} and \mbox{\eqref{eq:calL-H2-bdd}}}.
We begin with (recall $y_{i}=(x-z_{i})/\lmb_{i}$)
\[
|f'(\calW)|\aleq\tsum{i=1}J|W_{;i}|^{p-1}\aeq\tsum{i=1}J\lmb_{i}^{-2}\lan y_{i}\ran^{-4}\aleq\tsum{i=1}J|x-z_{i}|^{-2}.
\]
By Hardy's inequality and Rellich's inequality, we get 
\begin{align*}
\|f'(\calW)g\|_{(\dot{H}^{1})^{\ast}} & \aleq\tsum i{}\||x-z_{i}|\cdot|x-z_{i}|^{-2}g\|_{L^{2}}\aleq\|g\|_{\dot{H}^{1}},\\
\|f'(\calW)g\|_{L^{2}} & \aleq\tsum i{}\||x-z_{i}|^{-2}g\|_{L^{2}}\aleq\|g\|_{\dot{H}^{2}}.
\end{align*}
This proves \eqref{eq:calL-H1-bdd} and \eqref{eq:calL-H2-bdd}. 

\uline{Proof of \mbox{\eqref{eq:calL-H1-coer-quad-form}} and \mbox{\eqref{eq:calL-H2-coer}}}.
We proceed by induction on $J$. If $J=1$, \eqref{eq:calL-H1-coer-quad-form}
and \eqref{eq:calL-H2-coer} follow from \eqref{eq:calL-one-bubble-coer-H1-quad-form}
and \eqref{eq:calL-one-bubble-coer}, respectively, after scaling
and spatial translation. 

Let $J\geq2$ and assume the $J-1$ case. We show the $J$ case of
\eqref{eq:calL-H1-coer-quad-form} and \eqref{eq:calL-H2-coer} for
some small constant $\dlt=\dlt_{J}>0$ to be chosen later. Let $(\vec{\iota},\vec{\lmb},\vec{z})\in\calP_{J}(\dlt_{J})$.
Permuting indices $1,\dots,J$, we may assume $\lmb_{J}=\min_{i\in\setJ}\lmb_{i}$.
For a constant $K=K(g)\in[\dlt_{J}^{-1/20},\dlt_{J}^{-1/10}]\subset[1,\frac{R}{10}]$
to be chosen later, decompose 
\[
g(x)=g(x)\chi_{K}(y_{J})+g(x)(1-\chi_{K}(y_{J}))\eqqcolon g^{\rmin}(x)+g^{\rmex}(x).
\]
Write 
\begin{equation}
\begin{aligned}-\lan\calL_{\calW}g,g\ran & =-\lan\calL_{\calW}g^{\rmin},g^{\rmin}\ran-2\lan\calL_{\calW}g^{\rmin},g^{\rmex}\ran-\lan\calL_{\calW}g^{\rmex},g^{\rmex}\ran,\\
\lan\calL_{\calW}g,\calL_{\calW}g\ran & =\lan\calL_{\calW}g^{\rmin},\calL_{\calW}g^{\rmin}\ran+2\lan\calL_{\calW}g^{\rmin},\calL_{\calW}g^{\rmex}\ran+\lan\calL_{\calW}g^{\rmex},\calL_{\calW}g^{\rmex}\ran.
\end{aligned}
\label{eq:2.22}
\end{equation}

First, we consider the interior zone. We write 
\begin{equation}
\calL_{\calW}g^{\rmin}=\calL_{W_{;J}}g^{\rmin}+\{f'(\calW)-f'(W_{;J})\}g^{\rmin}.\label{eq:2.22-1}
\end{equation}
For $|y_{J}|\leq2K$, we have $|W_{;i}|\aeq\lmb_{J}^{-D}R_{Ji}^{-2D}$
for $i\neq J$, so 
\[
f'(\calW)-f'(W_{;J})=\calO(|\calW-W_{;J}|^{p-1})=\calO(\lmb_{J}^{-2}R_{J}^{-4})=K^{2}R_{J}^{-4}\calO(|x-z_{J}|^{-2}).
\]
This implies 
\begin{align*}
\|\{f'(\calW)-f'(W_{;J})\}g^{\rmin}\|_{(\dot{H}^{1})^{\ast}} & \aleq K^{2}R_{J}^{-4}\||x-z_{J}|^{-1}g\|_{L^{2}}\aleq K^{2}R_{J}^{-4}\|g\|_{\dot{H}^{1}},\\
\|\{f'(\calW)-f'(W_{;J})\}g^{\rmin}\|_{L^{2}} & \aleq K^{2}R_{J}^{-4}\||x-z_{J}|^{-2}g\|_{L^{2}}\aleq K^{2}R_{J}^{-4}\|g\|_{\dot{H}^{2}}.
\end{align*}
Substituting these controls and $R_{J}^{-1}<\dlt_{J}$ into \eqref{eq:2.22-1},
we conclude 
\begin{align*}
-\lan\calL_{\calW}g^{\rmin},g^{\rmin}\ran & =-\lan\calL_{W_{;J}}g^{\rmin},g^{\rmin}\ran-\calO(K^{2}\dlt_{J}^{4}\|g\|_{\dot{H}^{1}}^{2}),\\
\lan\calL_{\calW}g^{\rmin},\calL_{\calW}g^{\rmin}\ran & =\lan\calL_{W_{;J}}g^{\rmin},\calL_{W_{;J}}g^{\rmin}\ran-\calO(K^{2}\dlt_{J}^{4}\|g\|_{\dot{H}^{2}}^{2}).
\end{align*}
Applying the $J=1$ case of \eqref{eq:calL-H1-coer-quad-form} and
\eqref{eq:calL-H2-coer} to the first term gives 
\begin{align*}
-\lan\calL_{W_{;J}}g^{\rmin},g^{\rmin}\ran & \geq\dlt_{1}\|g^{\rmin}\|_{\dot{H}^{1}}^{2}-\calO(\lan\calY_{\ul{;J}},g^{\rmin}\ran^{2}+\tsum a{}\lan\calZ_{a\ul{;J}},g^{\rmin}\ran^{2}),\\
\lan\calL_{W_{;J}}g^{\rmin},\calL_{W_{;J}}g^{\rmin}\ran & \geq\dlt_{1}\|g^{\rmin}\|_{\dot{H}^{2}}^{2}-\calO(\lmb_{J}^{-2}\tsum a{}\lan\calZ_{a\ul{;J}},g^{\rmin}\ran^{2}),
\end{align*}
for some constant $\dlt_{1}>0$. Now, we we would like to replace
$g^{\rmin}$ in the error terms by $g$. In view of $|\calZ_{a}(y)|+|\calY(y)|\aleq\lan y\ran^{-(N-1)}$,
we have 
\begin{align*}
\tsum a{}|\lan\calZ_{a\ul{;J}},g^{\rmex}\ran|+|\lan\calY_{\ul{;J}},g^{\rmex}\ran| & \aleq K^{-\frac{N-4}{2}}\||x-z_{J}|^{-1}g^{\rmex}\|_{L^{2}}\aleq K^{-\frac{N-4}{2}}\|g\|_{\dot{H}^{1}},\\
\lmb_{J}^{-1}\{\tsum a{}|\lan\calZ_{a\ul{;J}},g^{\rmex}\ran|+|\lan\calY_{\ul{;J}},g^{\rmex}\ran|\} & \aleq K^{-\frac{N-6}{2}}\||x-z_{J}|^{-2}g^{\rmex}\|_{L^{2}}\aleq K^{-\frac{N-6}{2}}\|g\|_{\dot{H}^{2}}.
\end{align*}
Therefore, we have proved 
\begin{equation}
\begin{aligned}-\lan\calL_{W_{;J}}g^{\rmin},g^{\rmin}\ran & \geq\dlt_{1}\|g^{\rmin}\|_{\dot{H}^{1}}^{2}-\calO(\lan\calY_{\ul{;J}},g\ran^{2}+\tsum a{}\lan\calZ_{a\ul{;J}},g\ran^{2})-\calO(K^{2}\dlt_{J}^{4}+K^{-(N-4)})\|g\|_{\dot{H}^{1}}^{2},\\
\lan\calL_{W_{;J}}g^{\rmin},\calL_{W_{;J}}g^{\rmin}\ran & \geq\dlt_{1}\|g^{\rmin}\|_{\dot{H}^{2}}^{2}-\calO(\lmb_{J}^{-2}\tsum a{}\lan\calZ_{a\ul{;J}},g\ran^{2})-\calO(K^{2}\dlt_{J}^{4}+K^{-(N-6)})\|g\|_{\dot{H}^{2}}^{2}.
\end{aligned}
\label{eq:2.23}
\end{equation}

Next, we consider the exterior zone. Let us denote $\calW'\coloneqq\tsum{i<J}{}W_{;i}$
and write 
\begin{equation}
\calL_{\calW}g^{\rmex}=\calL_{\calW'}g^{\rmex}+\{f'(\calW)-f'(\calW')\}g^{\rmex}.\label{eq:2.22-2}
\end{equation}
For $|y_{J}|\geq K$, we have $|W_{;J}|^{p-1}\aleq\lmb_{J}^{-2}|y_{J}|^{-4}\aleq K^{-2}\lmb_{J}^{-2}|y_{J}|^{-2}$,
so 
\[
f'(\calW)-f'(\tsum{i<J}{}W_{;i})=\calO(|W_{;J}|^{p-1})\aleq K^{-2}|x-z_{J}|^{-2}.
\]
This implies 
\begin{align*}
\|\chf_{|y_{J}|\geq K}\{f'(\calW)-f'(W_{;J})\}g\|_{(\dot{H}^{1})^{\ast}} & \aleq K^{-2}\||x-z_{J}|^{-1}g\|_{L^{2}}\aleq K^{-2}\|g\|_{\dot{H}^{1}},\\
\|\chf_{|y_{J}|\geq K}\{f'(\calW)-f'(W_{;J})\}g\|_{L^{2}} & \aleq K^{-2}\||x-z_{J}|^{-2}g\|_{L^{2}}\aleq K^{-2}\|g\|_{\dot{H}^{2}}.
\end{align*}
Substituting these controls into \ref{eq:2.22-2} gives 
\begin{align*}
-\lan\calL_{\calW}g^{\rmex},g^{\rmex}\ran & =-\lan\calL_{\calW'}g^{\rmex},g^{\rmex}\ran+\calO(K^{-2}\|g\|_{\dot{H}^{1}}^{2}),\\
\lan\calL_{\calW}g^{\rmex},\calL_{\calW}g^{\rmex}\ran & =\lan\calL_{\calW'}g^{\rmex},\calL_{\calW'}g^{\rmex}\ran+\calO(K^{-2}\|g\|_{\dot{H}^{2}}^{2}).
\end{align*}
Applying the $J-1$ case of \eqref{eq:calL-H1-coer} and \eqref{eq:calL-H2-coer}
to the first term gives 
\begin{align*}
-\lan\calL_{\calW'}g^{\rmex},g^{\rmex}\ran & \geq\dlt_{J-1}\|g^{\rmex}\|_{\dot{H}^{1}}^{2}-\calO(\tsum{i\neq J}{}\{\lan\calY_{a\ul{;i}},g^{\rmex}\ran^{2}+\tsum a{}\lan\calZ_{a\ul{;i}},g^{\rmex}\ran^{2}\}),\\
\lan\calL_{\calW'}g^{\rmex},\calL_{\calW'}g^{\rmex}\ran & \geq\dlt_{J-1}\|g^{\rmex}\|_{\dot{H}^{2}}^{2}-\calO(\tsum{i\neq J}{}\lmb_{i}^{-2}\tsum a{}\lan\calZ_{a\ul{;i}},g^{\rmex}\ran^{2}),
\end{align*}
for some constant $\dlt_{J-1}>0$. We would like to replace $g^{\rmex}$
in the error terms by $g$. For $i\neq J$, since 
\[
\chf_{|y_{J}|\leq2K}\{\tsum a{}|\calZ_{a\ul{;i}}|+|\calY_{\ul{;i}}|\}\aleq\chf_{|y_{J}|\leq2K}\lmb_{i}^{-\frac{N+2}{2}}\lan y_{i}\ran^{-(N-1)}\aleq\lmb_{i}^{-\frac{3}{2}}\lmb_{J}^{-\frac{N-1}{2}}R_{Ji}^{-(N-1)}
\]
and $\lmb_{J}\leq\lmb_{i}$, we get 
\begin{align*}
\{\tsum a{}|\lan\calZ_{a\ul{;i}},g^{\rmin}\ran|+|\lan\calY_{\ul{;i}},g^{\rmin}\ran|\} & \aleq(\lmb_{J}/\lmb_{i})^{\frac{3}{2}}R_{Ji}^{-(N-1)}K^{\frac{N+2}{2}}\||x-z_{J}|^{-1}g\|_{L^{2}}\aleq K^{\frac{N+2}{2}}\dlt_{J}^{N-1}\|g\|_{\dot{H}^{1}},\\
\lmb_{i}^{-1}\{\tsum a{}|\lan\calZ_{a\ul{;i}},g^{\rmin}\ran|+|\lan\calY_{\ul{;i}},g^{\rmin}\ran|\} & \aleq(\lmb_{J}/\lmb_{i})^{\frac{5}{2}}R_{Ji}^{-(N-1)}K^{\frac{N+4}{2}}\||x-z_{J}|^{-2}g\|_{L^{2}}\aleq K^{\frac{N+4}{2}}\dlt_{J}^{N-1}\|g\|_{\dot{H}^{2}}.
\end{align*}
Therefore, we have proved 
\begin{equation}
\begin{aligned}-\lan\calL_{\calW'}g^{\rmex},g^{\rmex}\ran & \geq\dlt_{J-1}\|g^{\rmex}\|_{\dot{H}^{1}}^{2}-\calO(\tsum{i\neq J}{}\{\lan\calY_{a\ul{;i}},g\ran^{2}+\tsum a{}\lan\calZ_{a\ul{;i}},g\ran^{2}\})-\calO(K^{-2}+K^{\frac{N+2}{2}}\dlt_{J}^{N-1})\|g\|_{\dot{H}^{1}}^{2},\\
\lan\calL_{\calW'}g^{\rmex},\calL_{\calW'}g^{\rmex}\ran & \geq\dlt_{J-1}\|g^{\rmex}\|_{\dot{H}^{2}}^{2}-\calO(\tsum{i\neq J}{}\lmb_{i}^{-2}\tsum a{}\lan\calZ_{a\ul{;i}},g\ran^{2})-\calO(K^{-2}+K^{\frac{N+4}{2}}\dlt_{J}^{N-1})\|g\|_{\dot{H}^{2}}^{2}.
\end{aligned}
\label{eq:2.24}
\end{equation}

Finally, we estimate the mixed terms of \eqref{eq:2.22}: 
\begin{align*}
|\lan\calL_{\calW}g^{\rmin},g^{\rmex}\ran| & \aleq\|\chf_{K\leq|y_{J}|\leq2K}\{|\nabla g|+|x-z_{J}|^{-1}|g|\}\|_{L^{2}}^{2},\\
|\lan\calL_{\calW}g^{\rmin},\calL_{\calW}g^{\rmex}\ran| & \aleq\|\chf_{K\leq|y_{J}|\leq2K}\{|\Dlt g|+|x-z_{J}|^{-2}|g|\}\|_{L^{2}}^{2}.
\end{align*}
Since $\|\nabla g\|_{L^{2}}+\||x-z_{J}|^{-1}g\|_{L^{2}}\aleq\|g\|_{\dot{H}^{1}}$
and $\|\Dlt g\|_{L^{2}}+\||x-z_{J}|^{-2}g\|_{L^{2}}\aleq\|g\|_{\dot{H}^{2}}$,
we can choose $K=K(g)\in[\dlt_{J}^{-1/20},\dlt_{J}^{-1/10}]$ such
that 
\begin{equation}
\begin{aligned}\|\chf_{K\leq|y_{J}|\leq2K}\{|\nabla g|+|x-z_{J}|^{-1}|g|\}\|_{L^{2}}^{2} & \aleq|\log(\dlt_{J}^{-1/20})|^{-1}\|g\|_{\dot{H}^{1}}^{2},\\
\|\chf_{K\leq|y_{J}|\leq2K}\{|\Dlt g|+|x-z_{J}|^{-2}|g|\}\|_{L^{2}}^{2} & \aleq|\log(\dlt_{J}^{-1/20})|^{-1}\|g\|_{\dot{H}^{2}}^{2}.
\end{aligned}
\label{eq:2.25}
\end{equation}

Now, the proof of \eqref{eq:calL-H1-coer} and \eqref{eq:calL-H2-coer}
follows from combining \eqref{eq:2.23}, \eqref{eq:2.24}, and \eqref{eq:2.25},
and taking $\dlt_{J}>0$ sufficiently small.

\uline{Proof of \mbox{\eqref{eq:calL-H1-coer}}}. We use \eqref{eq:calL-H1-coer-quad-form}
to prove \eqref{eq:calL-H1-coer}. Let $\dlt>0$ be a small constant
to be chosen later; let $(\vec{\iota},\vec{\lmb},\vec{z})\in\calP_{J}(\dlt)$.
Let $c_{i}\coloneqq\lan\calY_{\ul{;i}},g\ran$ and $\wh g\coloneqq g-\tsum j{}c_{j}\calY_{;j}$.
We begin with 
\begin{align}
-\lan\calL_{\calW}g,\wh g\ran+\lan\calL_{\calW}g,\tsum j{}c_{j}\calY_{;j}\ran & =-\{\lan\calL_{\calW}\wh g,\wh g\ran+\tsum j{}c_{j}\lan\calL_{\calW}\calY_{;j},\wh g\ran\}+\tsum j{}c_{j}\lan g,\calL_{\calW}\calY_{;j}\ran\nonumber \\
 & =-\lan\calL_{\calW}\wh g,\wh g\ran+\tsum{i,j}{}c_{i}c_{j}\lan\calY_{;i},\calL_{\calW}\calY_{;j}\ran.\label{eq:2.14-1}
\end{align}

For the first term of \eqref{eq:2.14-1}, we apply \eqref{eq:calL-H1-coer-quad-form}
to have 
\begin{align*}
-\lan\calL_{\calW}\wh g,\wh g\ran & \geq\dlt_{J}\|\wh g\|_{\dot{H}^{1}}^{2}-\calO(\tsum i{}\lan\calY_{\ul{;i}},\wh g\ran^{2}+\tsum{a,i}{}\lan\calZ_{a\ul{;i}},\wh g\ran^{2})\\
 & \geq\dlt_{J}\|\wh g\|_{\dot{H}^{1}}^{2}-\calO(\tsum{a,i}{}\lan\calZ_{a\ul{;i}},g\ran^{2})-\calO(\tsum i{}\lan\calY_{\ul{;i}},\wh g\ran^{2}+\tsum{a,i,j}{}c_{j}^{2}\lan\calZ_{a\ul{;i}},\calY_{;j}\ran^{2}),
\end{align*}
for some fixed constant $\dlt_{J}>0$. The terms in the last $\calO$
are small (we used $\|\calY\|_{L^{2}}=1$ and $\lan\calZ_{a},\calY\ran=0$):
\begin{align*}
\lan\calY_{\ul{;i}},\wh g\ran & =\tsum{j\neq i}{}c_{j}\lan\calY_{\ul{;i}},\calY_{;j}\ran=o_{\dlt\to0}(1)\cdot|\vec{c}|,\\
c_{j}\lan\calZ_{a\ul{;i}},\calY_{;j}\ran & =c_{j}\chf_{i\neq j}\lan\calZ_{a\ul{;i}},\calY_{;j}\ran=o_{\dlt\to0}(1)\cdot|\vec{c}|.
\end{align*}
Thus we have proved 
\begin{equation}
-\lan\calL_{\calW}\wh g,\wh g\ran\geq\dlt_{J}\|\wh g\|_{\dot{H}^{1}}^{2}-\calO(\tsum{a,i}{}\lan\calZ_{a\ul{;i}},g\ran^{2})-o_{\dlt\to0}(1)\cdot|\vec{c}|^{2}.\label{eq:2.14-2}
\end{equation}

We turn to the second term of \eqref{eq:2.14-1}. Using $\calL_{W_{;j}}\calY_{;j}=e_{0}\calY_{\ul{;j}}$,
we have 
\begin{align*}
\tsum{i,j}{}c_{i}c_{j}\lan\calY_{;i},\calL_{\calW}\calY_{;j}\ran & =\tsum{i,j}{}c_{i}c_{j}(e_{0}\lan\calY_{;i},\calY_{\ul{;j}}\ran+\lan\calY_{;i},\{f'(\calW)-f'(W_{;j})\}\calY_{;j}\ran)\\
 & =e_{0}\tsum i{}c_{i}^{2}+\tsum{i,j}{}\chf_{i\neq j}c_{i}c_{j}e_{0}\lan\calY_{;i},\calY_{\ul{;j}}\ran+\tsum{i,j}{}c_{i}c_{j}\lan\calY_{;i},\{f'(\calW)-f'(W_{;j})\}\calY_{;j}\ran.
\end{align*}
The second term of the right hand side is of size $o_{\dlt\to0}(1)\cdot|\vec{c}|^{2}$
as before. For the last term, we use $\calY=\calO(W)$, \eqref{eq:f(a,b)-2},
and \eqref{eq:f(W1,W2)-H1dual-est} to estimate 
\[
\chf_{i\neq j}|\lan\calY_{;i},\{f'(\calW)-f'(W_{;j})\}\calY_{;j}\ran|\aleq\|W_{;j}\{f'(\calW)-f'(W_{;j})\}\|_{L^{(2^{\ast})'}}\|W_{;i}\|_{L^{2^{\ast}}}\aleq R_{j}^{-\frac{N+2}{2}}=o_{\dlt\to0}(1).
\]
Thus, we have proved that 
\begin{equation}
\tsum{i,j}{}c_{i}c_{j}\lan\calY_{;i},\calL_{\calW}\calY_{;j}\ran=e_{0}\tsum i{}c_{i}^{2}-o_{\dlt\to0}(1)\cdot|\vec{c}|^{2}.\label{eq:2.14-3}
\end{equation}

Substituting \eqref{eq:2.14-2} and \eqref{eq:2.14-3} into \eqref{eq:2.14-1}
gives 
\[
-\lan\calL_{\calW}g,\wh g\ran+\lan\calL_{\calW}g,\tsum j{}c_{j}\calY_{;j}\ran\geq\dlt_{J}\|\wh g\|_{\dot{H}^{1}}^{2}+e_{0}\tsum i{}c_{i}^{2}-\calO(\tsum{a,i}{}\lan\calZ_{a\ul{;i}},g\ran^{2})-o_{\dlt\to0}(1)\cdot|\vec{c}|^{2}.
\]
Since $\|\wh g\|_{\dot{H}^{1}}^{2}+\tsum i{}c_{i}^{2}\aeq\|g\|_{\dot{H}^{1}}^{2}$,
taking $\dlt>0$ sufficiently small completes the proof of \eqref{eq:calL-H1-coer}.
\end{proof}

\section{\label{subsec:Proof-of-lemma-integrals}Proof of Lemma~\ref{lem:f(W1,W2)-est}}
\begin{lem}
Assume $\lmb_{1}\leq\lmb_{2}$. Then, we have 
\begin{align}
|x-z_{1}|\leq\tfrac{1}{2}\sqrt{\lmb_{1}\lmb_{2}}R_{12}\quad & \implies\quad\lan y_{2}\ran\aeq\sqrt{\lmb_{1}/\lmb_{2}}R_{12},\label{eq:2.10-1}\\
|x-z_{2}|\leq\tfrac{1}{2}\sqrt{\lmb_{1}\lmb_{2}}R_{12}\leq|x-z_{1}|\quad & \implies\quad|y_{1}|\aeq\sqrt{\lmb_{2}/\lmb_{1}}R_{12},\label{eq:2.10-2}\\
\min\{|x-z_{1}|,|x-z_{2}|\}\geq\tfrac{1}{2}\sqrt{\lmb_{1}\lmb_{2}}R_{12}\quad & \implies\quad|x-z_{1}|\aeq|x-z_{2}|\ageq\sqrt{\lmb_{1}\lmb_{2}}R_{12}.\label{eq:2.10-3}
\end{align}
We also have 
\begin{align}
|y_{1}|\leq\tfrac{1}{2}R_{12}\quad & \implies\quad|W_{;1}|\ageq|W_{;2}|\aeq\lmb_{1}^{-\frac{N-2}{2}}R_{12}^{-(N-2)},\label{eq:2.13-1}\\
|y_{1}|\geq\tfrac{1}{2}R_{12}\quad & \implies\quad|W_{;1}|\aleq|W_{;2}|.\label{eq:2.13-2}
\end{align}
\end{lem}

\begin{proof}
Note that $R_{12}=\max\{\sqrt{\frac{\lmb_{2}}{\lmb_{1}}},\frac{|z_{1}-z_{2}|}{\sqrt{\lmb_{1}\lmb_{2}}}\}$
due to $\lmb_{1}\leq\lmb_{2}$.

\uline{Proof of \mbox{\eqref{eq:2.10-1}}}. Assume $|x-z_{1}|\leq\frac{1}{2}\sqrt{\lmb_{1}\lmb_{2}}R_{12}$.
This implies $|x-z_{1}|\leq\frac{1}{2}(\lmb_{2}+|z_{1}-z_{2}|)$,
so $\lmb_{2}+|x-z_{2}|\aeq\lmb_{2}+|z_{1}-z_{2}|$. Dividing both
sides by $\lmb_{2}$, we obtain $\lan y_{2}\ran\aeq\sqrt{\lmb_{1}/\lmb_{2}}R_{12}$.

\uline{Proof of \mbox{\eqref{eq:2.10-2}}}. Assume $|x-z_{2}|\leq\tfrac{1}{2}\sqrt{\lmb_{1}\lmb_{2}}R_{12}\leq|x-z_{1}|$.
The latter inequality gives a lower bound $|y_{1}|\geq\frac{1}{2}\sqrt{\lmb_{2}/\lmb_{1}}R_{12}$.
The former inequality gives $|x-z_{1}|\leq\tfrac{1}{2}\sqrt{\lmb_{1}\lmb_{2}}R_{12}+|z_{1}-z_{2}|\leq\frac{3}{2}\sqrt{\lmb_{1}\lmb_{2}}R_{12}$,
which gives an upper bound $|y_{1}|\leq\frac{3}{2}\sqrt{\lmb_{2}/\lmb_{1}}R_{12}$. 

\uline{Proof of \mbox{\eqref{eq:2.10-3}}}. The proof easily follows
from $|z_{1}-z_{2}|\leq\sqrt{\lmb_{1}\lmb_{2}}R_{12}$.

\uline{Proof of \mbox{\eqref{eq:2.13-1}}}. Assume $|y_{1}|\leq\frac{1}{2}R_{12}$.
Note that $|W_{;1}|\aeq\lmb_{1}^{-\frac{N-2}{2}}\lan y_{1}\ran^{-(N-2)}\ageq\lmb_{1}^{-\frac{N-2}{2}}R_{12}^{-(N-2)}$.
By \eqref{eq:2.10-1}, we have $|W_{;2}|\aeq\lmb_{2}^{-\frac{N-2}{2}}\lan y_{2}\ran^{-(N-2)}\aeq\lmb_{1}^{-\frac{N-2}{2}}R_{12}^{-(N-2)}$
as desired.

\uline{Proof of \mbox{\eqref{eq:2.13-2}}}. Assume $|y_{1}|\geq\frac{1}{2}R_{12}$.
If $|x-z_{2}|\leq2(\lmb_{2}+|z_{1}-z_{2}|)$, then $|W_{;2}|\ageq\lmb_{2}^{\frac{N-2}{2}}(\lmb_{2}+|z_{1}-z_{2}|)^{-(N-2)}\aeq\lmb_{1}^{-\frac{N-2}{2}}R_{12}^{-(N-2)}$
and $|W_{;1}|\aeq\lmb_{1}^{-\frac{N-2}{2}}|y_{1}|^{-(N-2)}\aleq\lmb_{1}^{-\frac{N-2}{2}}R_{12}^{-(N-2)}$,
so $|W_{;2}|\ageq|W_{;1}|$. If $|x-z_{2}|\geq2(\lmb_{2}+|z_{1}-z_{2}|)$,
then $|x-z_{1}|\aeq|x-z_{2}|\ageq\lmb_{2}+|z_{1}-z_{2}|$ so $|W_{;1}|\aeq\lmb_{1}^{\frac{N-2}{2}}|x-z_{1}|^{-(N-2)}\aleq\lmb_{2}^{\frac{N-2}{2}}|x-z_{2}|^{-(N-2)}\aeq|W_{;2}|$.
This completes the proof.
\end{proof}
\begin{proof}[Proof of Lemma~\ref{lem:f(W1,W2)-est}]
\uline{Proof of \mbox{\eqref{eq:W1W2-Hdot1}}}. We omit the $\eps$-dependence
of implicit constants in the proof. It suffices to estimate $\int g_{\eps}(W_{;1},W_{;2})$,
where 
\begin{equation}
g_{\eps}(W_{;1},W_{;2})\coloneqq\begin{cases}
|W_{;1}|^{\frac{N-\eps}{N-2}}|W_{;2}|^{\frac{N+\eps}{N-2}} & \text{if }|W_{;1}|\geq|W_{;2}|,\\
|W_{;2}|^{\frac{N-\eps}{N-2}}|W_{;1}|^{\frac{N+\eps}{N-2}} & \text{if }|W_{;1}|\leq|W_{;2}|.
\end{cases}\label{eq:def-g_eps}
\end{equation}
By symmetry, we may assume 
\[
\lmb_{1}\leq\lmb_{2}.
\]
We separate into a few regions to prove the estimates. Note that $\bbR^{N}$
can be written as an almost disjoint union 
\[
\bbR^{N}=\Omg_{1}\cup\Omg_{2}\cup\Omg_{3}\cup\Omg_{4},
\]
where 
\begin{align*}
\Omg_{1} & =\{x:|y_{1}|\leq\tfrac{1}{2}R_{12}\},\\
\Omg_{2} & =\{x:\tfrac{1}{2}\lmb_{1}R_{12}\leq|x-z_{1}|\leq\tfrac{1}{2}\sqrt{\lmb_{1}\lmb_{2}}R_{12}\},\\
\Omg_{3} & =\{x:|x-z_{2}|\leq\tfrac{1}{2}\sqrt{\lmb_{1}\lmb_{2}}R_{12}\leq|x-z_{1}|\},\\
\Omg_{4} & =\{x:\min\{|x-z_{1}|,|x-z_{2}|\}\geq\tfrac{1}{2}\sqrt{\lmb_{1}\lmb_{2}}R_{12}\}.
\end{align*}
Note that \eqref{eq:2.13-1} and \eqref{eq:2.13-2} imply 
\begin{align*}
g_{\eps}(W_{;1},W_{;2}) & \aeq|W_{;1}|^{\frac{N-\eps}{N-2}}|W_{;2}|^{\frac{N+\eps}{N-2}}\qquad\text{in }\Omg_{1},\\
g_{\eps}(W_{;1},W_{;2}) & \aeq|W_{;2}|^{\frac{N-\eps}{N-2}}|W_{;1}|^{\frac{N+\eps}{N-2}}\qquad\text{in }\Omg_{2}\cup\Omg_{3}\cup\Omg_{4}.
\end{align*}

In $\Omg_{1}=\{|y_{1}|\leq\frac{1}{2}R_{12}\}$, we have $|W_{;2}|^{\frac{N+\eps}{N-2}}\aeq\lmb_{1}^{-\frac{N+\eps}{2}}R_{12}^{-(N+\eps)}$
by \eqref{eq:2.13-1}, so 
\[
g_{\eps}(W_{;1},W_{;2})\aeq|W_{;1}|^{\frac{N-\eps}{N-2}}|W_{;2}|^{\frac{N+\eps}{N-2}}\aeq\lmb_{1}^{-N}R_{12}^{-(N+\eps)}\lan y_{1}\ran^{-(N-\eps)}.
\]
Hence, 
\[
\int_{\Omg_{1}}g_{\eps}(W_{;1},W_{;2})\aeq R_{12}^{-(N+\eps)}\int_{|y_{1}|\leq\frac{1}{2}R_{12}}\lan y_{1}\ran^{-(N-\eps)}dy_{1}\aeq\begin{cases}
R_{12}^{-N} & \text{if }\eps>0,\\
R_{12}^{-N}\log(1+R_{12}) & \text{if }\eps=0,\\
R_{12}^{-(N+\eps)} & \text{if }\eps<0.
\end{cases}
\]

In $\Omg_{2}=\{\frac{1}{2}\lmb_{1}R_{12}\leq|x-z_{1}|\leq\frac{1}{2}\sqrt{\lmb_{1}\lmb_{2}}R_{12}\}$,
we have $|W_{;2}|^{\frac{N-\eps}{N-2}}\aeq\lmb_{1}^{-\frac{N-\eps}{2}}R_{12}^{-(N-\eps)}$
by \eqref{eq:2.10-1}, so 
\[
g_{\eps}(W_{;1},W_{;2})\aeq|W_{;2}|^{\frac{N-\eps}{N-2}}|W_{;1}|^{\frac{N+\eps}{N-2}}\aeq\lmb_{1}^{-N}R_{12}^{-(N-\eps)}\lan y_{1}\ran^{-(N+\eps)}.
\]
Hence, 
\[
\int_{\Omg_{2}}g_{\eps}(W_{;1},W_{;2})\aeq R_{12}^{-(N-\eps)}\int_{\frac{1}{2}R_{12}\leq|y_{1}|\leq\frac{1}{2}\sqrt{\lmb_{2}/\lmb_{1}}R_{12}}\lan y_{1}\ran^{-(N+\eps)}dy_{1}\aleq\begin{cases}
R_{12}^{-N} & \text{if }\eps>0,\\
R_{12}^{-N}\log(\frac{\lmb_{2}}{\lmb_{1}}) & \text{if }\eps=0,\\
(\frac{\lmb_{2}}{\lmb_{1}})^{-\frac{\eps}{2}}R_{12}^{-N} & \text{if }\eps<0.
\end{cases}
\]

In $\Omg_{3}=\{|x-z_{2}|\leq\frac{1}{2}\sqrt{\lmb_{1}\lmb_{2}}R_{12}\leq|x-z_{1}|\}$,
we have $|W_{;1}|^{\frac{N+\eps}{N-2}}\aeq\lmb_{2}^{-\frac{N+\eps}{2}}R_{12}^{-(N+\eps)}$
by \eqref{eq:2.10-2}, so 
\[
g_{\eps}(W_{;1},W_{;2})\aeq|W_{;2}|^{\frac{N-\eps}{N-2}}|W_{;1}|^{\frac{N+\eps}{N-2}}\aeq\lmb_{2}^{-N}R_{12}^{-(N+\eps)}\lan y_{2}\ran^{-(N-\eps)}.
\]
Hence, 
\[
\int_{\Omg_{3}}g_{\eps}(W_{;1},W_{;2})\aleq R_{12}^{-(N+\eps)}\int_{|y_{2}|\leq\frac{1}{2}\sqrt{\lmb_{1}/\lmb_{2}}R_{12}}\lan y_{2}\ran^{-(N-\eps)}dy_{2}\aleq\begin{cases}
(\frac{\lmb_{1}}{\lmb_{2}})^{\frac{\eps}{2}}R_{12}^{-N} & \text{if }\eps>0,\\
R_{12}^{-N}\log(1+\sqrt{\frac{\lmb_{1}}{\lmb_{2}}}R_{12}) & \text{if }\eps=0,\\
R_{12}^{-(N+\eps)} & \text{if }\eps<0.
\end{cases}
\]

In $\Omg_{4}=\{\min\{|x-z_{1}|,|x-z_{2}|\}\geq\tfrac{1}{2}\sqrt{\lmb_{1}\lmb_{2}}R_{12}\}$,
we use \eqref{eq:2.10-3} to obtain 
\[
g_{\eps}(W_{;1},W_{;2})\aeq|W_{;2}|^{\frac{N-\eps}{N-2}}|W_{;1}|^{\frac{N+\eps}{N-2}}\aeq\lmb_{2}^{\frac{N-\eps}{2}}\lmb_{1}^{\frac{N+\eps}{2}}|x-z_{1}|^{-2N}.
\]
Hence, 
\[
\int_{\Omg_{4}}g_{\eps}(W_{;1},W_{;2})\aeq\lmb_{2}^{\frac{N-\eps}{2}}\lmb_{1}^{\frac{N+\eps}{2}}\int_{|x-z_{1}|\geq\frac{1}{2}\sqrt{\lmb_{1}\lmb_{2}}R_{12}}|x-z_{1}|^{-2N}dx\aleq\Big(\frac{\lmb_{1}}{\lmb_{2}}\Big)^{\frac{\eps}{2}}R_{12}^{-N}.
\]
This completes the proof of \eqref{eq:W1W2-Hdot1}.

\uline{Proof of \mbox{\eqref{eq:f(W1,W2)-H1dual-est}}}. Since
$N>6$, we can choose $\eps>0$ such that $(2^{\ast})'=\frac{N+\eps}{N-2}$.
With this $\eps$, notice that $f(W_{;1},W_{;2})^{(2^{\ast})'}=g_{\eps}(W_{;1},W_{;2})$,
where $g_{\eps}(W_{;1},W_{;2})$ is defined in \eqref{eq:def-g_eps}.
By \eqref{eq:W1W2-Hdot1}, we get
\[
\|f(W_{;1},W_{;2})\|_{L^{(2^{\ast})'}}\aeq\|g_{\eps}(W_{;1},W_{;2})\|_{L^{1}}^{1/(2^{\ast})'}\aleq R_{12}^{-\frac{N+2}{2}}
\]
as desired.

\uline{Proof of \mbox{\eqref{eq:f(W1,W2)-L2-est}} and \mbox{\eqref{eq:f(W1,W2)-H2dual-est}}}.
As in the proof of \eqref{eq:W1W2-Hdot1}, we write $\bbR^{N}$ as
\[
\bbR^{N}=\Omg_{1}\cup\Omg_{2}\cup\Omg_{3}\cup\Omg_{4}.
\]

In $\Omg_{1}=\{|y_{1}|\leq\frac{1}{2}R_{12}\}$, we have $|W_{;2}|\aeq\lmb_{1}^{-\frac{N-2}{2}}R_{12}^{-(N-2)}$
by \eqref{eq:2.13-1}, so 
\[
f(W_{;1},W_{;2})\aeq|W_{;1}|^{p-1}|W_{;2}|\aeq\lmb_{1}^{-\frac{N+2}{2}}R_{12}^{-(N-2)}\lan y_{1}\ran^{-4}.
\]
Hence, 
\begin{align*}
 & \|\chf_{\Omg_{1}}f(W_{;1},W_{;2})\|_{L_{x}^{2}}\aeq\lmb_{1}^{-1}R_{12}^{-(N-2)}\|\chf_{|y_{1}|\leq\frac{1}{2}R_{12}}\lan y_{1}\ran^{-4}\|_{L_{y_{1}}^{2}}\\
 & \qquad\aeq\lmb_{1}^{-1}R_{12}^{-(N-2)}\chf_{N=7}+\lmb_{1}^{-1}R_{12}^{-(N-2)}(1+\log R_{12})^{\frac{1}{2}}\chf_{N=8}+\lmb_{1}^{-1}R_{12}^{-\frac{N}{2}-2}\chf_{N\geq9},\\
 & \|\chf_{\Omg_{1}}f(W_{;1},W_{;2})\|_{L_{x}^{(2^{\ast\ast})'}}\aeq\lmb_{1}R_{12}^{-(N-2)}\|\chf_{|y_{1}|\leq\frac{1}{2}R_{12}}\lan y_{1}\ran^{-4}\|_{L_{y_{1}}^{(2^{\ast\ast})'}}\aeq\lmb_{1}R_{12}^{-\frac{N}{2}}.
\end{align*}

In $\Omg_{2}=\{\frac{1}{2}\lmb_{1}R_{12}\leq|x-z_{1}|\leq\frac{1}{2}\sqrt{\lmb_{1}\lmb_{2}}R_{12}\}$,
we have $|W_{;2}|^{p-1}\aeq\lmb_{1}^{-2}R_{12}^{-4}$ by \eqref{eq:2.10-1},
so 
\[
f(W_{;1},W_{;2})\aeq|W_{;2}|^{p-1}|W_{;1}|\aeq\lmb_{1}^{-\frac{N+2}{2}}R_{12}^{-4}|y_{1}|^{-(N-2)}.
\]
Hence, 
\begin{align*}
 & \|\chf_{\Omg_{2}}f(W_{;1},W_{;2})\|_{L_{x}^{2}}\aleq\lmb_{1}^{-1}R_{12}^{-4}\|\chf_{|y_{1}|\geq\frac{1}{2}R_{12}}|y_{1}|^{-(N-2)}\|_{L_{y_{1}}^{2}}\aleq\lmb_{1}^{-1}R_{12}^{-\frac{N}{2}-2},\\
 & \|\chf_{\Omg_{2}}f(W_{;1},W_{;2})\|_{L_{x}^{(2^{\ast\ast})'}}\aleq\lmb_{1}R_{12}^{-4}\|\chf_{\frac{1}{2}R_{12}\leq|y_{1}|\leq\frac{1}{2}\sqrt{\lmb_{2}/\lmb_{1}}R_{12}}|y_{1}|^{-(N-2)}\|_{L_{y_{1}}^{(2^{\ast\ast})'}}\\
 & \qquad\aleq\lmb_{1}\sqrt{\lmb_{2}/\lmb_{1}}^{\frac{1}{2}}R_{12}^{-\frac{N}{2}}\chf_{N=7}+\lmb_{1}(\log(\lmb_{2}/\lmb_{1}))^{\frac{1}{(2^{\ast\ast})'}}R_{12}^{-\frac{N}{2}}\chf_{N=8}+\lmb_{1}R_{12}^{-\frac{N}{2}}\chf_{N\geq9}.
\end{align*}

In $\Omg_{3}=\{|x-z_{2}|\leq\frac{1}{2}\sqrt{\lmb_{1}\lmb_{2}}R_{12}\leq|x-z_{1}|\}$,
we have $|W_{;1}|\aeq\lmb_{2}^{-\frac{N-2}{2}}R_{12}^{-(N-2)}$ by
\eqref{eq:2.10-2}, so 
\[
f(W_{;1},W_{;2})\aeq|W_{;2}|^{p-1}|W_{;1}|\aeq\lmb_{2}^{-\frac{N+2}{2}}R_{12}^{-(N-2)}\lan y_{2}\ran^{-4}.
\]
Hence, 
\begin{align*}
 & \|\chf_{\Omg_{3}}f(W_{;1},W_{;2})\|_{L_{x}^{2}}\aeq\lmb_{2}^{-1}R_{12}^{-(N-2)}\|\chf_{|y_{2}|\leq\frac{1}{2}\sqrt{\lmb_{1}/\lmb_{2}}R_{12}}\lan y_{2}\ran^{-4}\|_{L_{y_{2}}^{2}}\\
 & \qquad\aeq\lmb_{2}^{-1}R_{12}^{-(N-2)}\chf_{N=7}+\lmb_{2}^{-1}R_{12}^{-(N-2)}(1+\log(\sqrt{\tfrac{\lmb_{1}}{\lmb_{2}}}R_{12}))^{\frac{1}{2}}\chf_{N=8}+\lmb_{2}^{-1}\sqrt{\tfrac{\lmb_{1}}{\lmb_{2}}}^{\frac{N}{2}-4}R_{12}^{-\frac{N}{2}-2}\chf_{N\geq9},\\
 & \|\chf_{\Omg_{3}}f(W_{;1},W_{;2})\|_{L_{x}^{(2^{\ast\ast})'}}\aeq\lmb_{2}R_{12}^{-(N-2)}\|\chf_{|y_{2}|\leq\frac{1}{2}\sqrt{\lmb_{1}/\lmb_{2}}R_{12}}\lan y_{2}\ran^{-4}\|_{L_{y_{2}}^{(2^{\ast\ast})'}}\aeq\lmb_{2}\sqrt{\tfrac{\lmb_{1}}{\lmb_{2}}}^{\frac{N}{2}-2}R_{12}^{-\frac{N}{2}}.
\end{align*}

In $\Omg_{4}=\{\min\{|x-z_{1}|,|x-z_{2}|\}\geq\tfrac{1}{2}\sqrt{\lmb_{1}\lmb_{2}}R_{12}\}$,
we use \eqref{eq:2.10-3} to obtain 
\[
f(W_{;1},W_{;2})\aeq|W_{;2}|^{p-1}|W_{;1}|\aeq\lmb_{2}^{2}\lmb_{1}^{\frac{N-2}{2}}|x-z_{1}|^{-(N+2)}.
\]
Hence, 
\begin{align*}
\|\chf_{\Omg_{4}}f(W_{;1},W_{;2})\|_{L_{x}^{2}} & \aleq\lmb_{2}^{2}\lmb_{1}^{\frac{N-2}{2}}\|\chf_{|x-z_{1}|\geq\tfrac{1}{2}\sqrt{\lmb_{1}\lmb_{2}}R_{12}}|x-z_{1}|^{-(N+2)}\|_{L_{x}^{2}}\aleq\lmb_{2}^{-1}\sqrt{\tfrac{\lmb_{1}}{\lmb_{2}}}^{\frac{N}{2}-4}R_{12}^{-\frac{N}{2}-2},\\
\|\chf_{\Omg_{4}}f(W_{;1},W_{;2})\|_{L_{x}^{(2^{\ast\ast})'}} & \aleq\lmb_{2}^{2}\lmb_{1}^{\frac{N-2}{2}}\|\chf_{|x-z_{1}|\geq\tfrac{1}{2}\sqrt{\lmb_{1}\lmb_{2}}R_{12}}|x-z_{1}|^{-(N+2)}\|_{L_{x}^{(2^{\ast\ast})'}}\aleq\lmb_{2}\sqrt{\tfrac{\lmb_{1}}{\lmb_{2}}}^{\frac{N}{2}-2}R_{12}^{-\frac{N}{2}}.
\end{align*}
This completes the proof.
\end{proof}

\section{\label{sec:Proof-of-inner-prod}Proof of \eqref{eq:2.48} and \eqref{eq:2.49}}
\begin{proof}
\uline{Proof of \mbox{\eqref{eq:2.48}} when \mbox{$\lmb_{i}\leq\lmb_{j}$}}.
Introducing $\lmb=\lmb_{j}/\lmb_{i}\geq1$ and $z=(z_{j}-z_{i})/\lmb_{i}$,
we get 
\[
\tint{}{}[\Lmb W]_{;i}f'(W_{;i})W_{;j}=\iota_{i}\iota_{j}\tint{}{}\Lmb Wf'(W)W_{\lmb,z}.
\]
For $|y|\leq\frac{1}{2}(\lmb+|z|)$, we have $\lmb+|y-z|\aeq\lmb+|z|$
so 
\begin{align*}
W_{\lmb,z}(y) & =W_{\lmb,z}(0)+y\cdot\nabla W_{\lmb,z}(0)+\calO(|y|^{2}\lmb^{\frac{N-2}{2}}(\lmb+|z|)^{-N})\\
 & =W_{\lmb,z}(0)+y\cdot\nabla W_{\lmb,z}(0)+R_{ij}^{-(N-2)}\cdot\calO(|y|^{2}(\lmb+|z|)^{-2}).
\end{align*}
This gives 
\begin{align*}
 & \tint{|y|\leq\frac{1}{2}(\lmb+|z|)}{}\Lmb Wf'(W)W_{\lmb,z}\\
 & =W_{\lmb,z}(0)\tint{|y|\leq\frac{1}{2}(\lmb+|z|)}{}\Lmb Wf'(W)+\tint{|y|\leq\frac{1}{2}(\lmb+|z|)}{}R_{ij}^{-(N-2)}\cdot\calO(\lan y\ran^{-N}(\lmb+|z|)^{-2})\\
 & =W_{\lmb,z}(0)\tint{}{}\Lmb Wf'(W)+R_{ij}^{-(N-2)}\cdot\calO\big((\lmb+|z|)^{-2}(1+\log(\lmb+|z|))\big).
\end{align*}
For $|y|>\frac{1}{2}(\lmb+|z|)$, we have $W(y)\aleq|y|^{-(N-2)}\aleq(\lmb+|z|+|y-z|)^{-(N-2)}$,
so 
\begin{align*}
|\tint{|y|>\frac{1}{2}(\lmb+|z|)}{}\Lmb Wf'(W)W_{\lmb,z}| & \aleq\tint{}{}(\lmb+|z|+|y-z|)^{-(N+2)}\lmb^{\frac{N-2}{2}}(\lmb+|y-z|)^{-(N-2)}\\
 & \aleq\lmb^{\frac{N-2}{2}}(\lmb+|z|)^{-N}\aleq R_{ij}^{-(N-2)}(\lmb+|z|)^{-2}.
\end{align*}
Combining the previous two displays, we conclude 
\[
\tint{}{}\Lmb Wf'(W)W_{\lmb,z}=\tint{}{}\Lmb Wf'(W)\cdot W_{\lmb,z}(0)+R_{ij}^{-(N-2)}\cdot\calO((\lmb+|z|)^{-2}(1+\log(\lmb+|z|))).
\]
Since $\int\Lmb Wf'(W)=-\frac{N-2}{2}\int f(W)$ and 
\[
(\lmb+|z|)^{-2}(1+\log(\lmb+|z|))\aeq(\sqrt{\tfrac{\lmb_{j}}{\lmb_{i}}}R_{ij})^{-2}(1+\log(\sqrt{\tfrac{\lmb_{j}}{\lmb_{i}}}R_{ij})),
\]
we obtain \eqref{eq:2.48} when $\lmb_{i}\leq\lmb_{j}$.

\uline{Proof of \mbox{\eqref{eq:2.48}} when \mbox{$\lmb_{i}\geq\lmb_{j}$}}.
In this case, we first integrate by parts: 
\[
\tint{}{}[\Lmb W]_{;i}f'(W_{;i})W_{;j}=-\tint{}{}\Dlt[\Lmb W]_{;i}W_{;j}=-\tint{}{}[\Lmb W]_{;i}\Dlt W_{;j}=\tint{}{}f(W_{;j})[\Lmb W]_{;i}.
\]
Introducing $\lmb=\lmb_{i}/\lmb_{j}\geq1$ and $z=(z_{i}-z_{j})/\lmb_{j}$,
we get 
\[
\tint{}{}f(W_{;j})[\Lmb W]_{;i}=\iota_{i}\iota_{j}\tint{}{}f(W)[\Lmb W]_{\lmb,z}.
\]
Proceeding as in the previous paragraph, we conclude 
\[
\tint{}{}f(W)[\Lmb W]_{\lmb,z}=\tint{}{}f(W)\cdot[\Lmb W]_{\lmb,z}(0)+R_{ij}^{-(N-2)}\cdot\calO((\lmb+|z|)^{-2}(1+\log(\lmb+|z|))).
\]
As $\lmb+|z|\aeq\sqrt{\lmb}R_{ij}\aeq\sqrt{\frac{\lmb_{i}}{\lmb_{j}}}R_{ij}$,
we get \eqref{eq:2.48} when $\lmb_{i}\geq\lmb_{j}$.

\uline{Proof of \mbox{\eqref{eq:2.49}} when \mbox{$\lmb_{i}\leq\lmb_{j}$}}.
First, we manipulate the integral as 
\[
\tint{}{}[\nabla W]_{;i}f'(W_{;i})W_{;j}=-\tint{}{}\Dlt[\nabla W]_{;i}W_{;j}=-\tint{}{}[\nabla W]_{;i}\Dlt W_{;j}=\tint{}{}[\nabla W]_{;i}f(W_{;j}).
\]
We perform a further integration by parts as 
\[
\tint{}{}[\nabla W]_{;i}f(W_{;j})=-\lmb_{i}\tint{}{}\nabla W_{;i}\,\Dlt W_{;j}=\lmb_{i}\tint{}{}\Dlt W_{;i}\,\nabla W_{;j}=-\tfrac{\lmb_{i}}{\lmb_{j}}\tint{}{}f(W_{;i})[\nabla W]_{;j}.
\]
Introducing $\lmb=\lmb_{j}/\lmb_{i}\geq1$ and $z=(z_{j}-z_{i})/\lmb_{i}$,
we get 
\[
-\tfrac{\lmb_{i}}{\lmb_{j}}\tint{}{}f(W_{;i})[\nabla W]_{;j}=-\iota_{i}\iota_{j}\lmb^{-1}\tint{}{}f(W)[\nabla W]_{\lmb,z}=-\iota_{i}\iota_{j}\tint{}{}f(W)\nabla W_{\lmb,z}.
\]
Now, proceeding as in the proof of \eqref{eq:2.48}, one obtains 
\[
-\tint{}{}f(W)\nabla W_{\lmb,z}=-\tint{}{}f(W)\cdot\nabla W_{\lmb,z}(0)+\calO(R_{ij}^{-(N-2)}(\lmb+|z|)^{-3}(1+\log(\lmb+|z|))).
\]
Expressing this in terms of $(\lmb_{i},z_{i})$ and $(\lmb_{j},z_{j})$
and applying $\lmb+|z|\geq R_{ij}$ give \eqref{eq:2.49} when $\lmb_{i}\leq\lmb_{j}$.

\uline{Proof of \mbox{\eqref{eq:2.49}} when \mbox{$\lmb_{i}\geq\lmb_{j}$}}.
Introducing $\lmb=\lmb_{i}/\lmb_{j}\geq1$ and $z=(z_{i}-z_{j})/\lmb_{j}$,
we get 
\[
\tint{}{}[\nabla W]_{;i}f'(W_{;i})W_{;j}=\tint{}{}[\nabla W]_{;i}f(W_{;j})=\tint{}{}f(W)[\nabla W]_{\lmb,z}=\lmb\tint{}{}f(W)\nabla W_{\lmb,z}.
\]
By the previous paragraph, we observe 
\[
\lmb\tint{}{}f(W)\nabla W_{\lmb,z}=\tint{}{}f(W)\cdot\lmb\nabla W_{\lmb,z}(0)+\calO(R_{ij}^{-(N-2)}\lmb(\lmb+|z|)^{-3}(1+\log(\lmb+|z|))).
\]
Expressing this in terms of $(\lmb_{i},z_{i})$ and $(\lmb_{j},z_{j})$
and applying $\lmb+|z|\geq R_{ij}$ give \eqref{eq:2.49} when $\lmb_{i}\geq\lmb_{j}$.
\end{proof}

\section{\label{sec:Proof-of-Proposition-3.1}Proof of Proposition~\ref{prop:distance-multi-bubbles}}
\begin{proof}[Proof of Proposition~\ref{prop:distance-multi-bubbles}]
The first item follows from Lemma~\ref{lem:2.2} below and the triangle
inequality. The second item follows from Lemma~\ref{lem:2.4} and
Lemma~\ref{lem:2.5} below. 
\end{proof}
\begin{lem}[Distance between one-bubbles]
\label{lem:2.2}For any $(\iota,\lmb,z),(\iota',\lmb',z')\in\calP_{1}$,
we have 
\begin{equation}
\|\calW(\iota,\lmb,z)-\calW(\iota',\lmb',z')\|_{\dot{H}^{1}}\aeq d_{\calP}((\iota,\lmb,z),(\iota',\lmb',z'))\aeq\|\calW(\iota,\lmb,z)-\calW(\iota',\lmb',z')\|_{L^{2^{\ast}}}\label{eq:2.3}
\end{equation}
\end{lem}

\begin{proof}
The proof is straightforward and thus omitted.
\end{proof}
\begin{lem}
\label{lem:2.4}Let $J\geq1$. There exists $\dlt>0$ such that if
$(\vec{\iota},\vec{\lmb},\vec{z}),(\vec{\iota}',\vec{\lmb}',\vec{z}')\in\calP_{J}(\dlt)$
satisfy $\|\calW(\vec{\iota},\vec{\lmb},\vec{z})-\calW(\vec{\iota}',\vec{\lmb}',\vec{z}')\|_{L^{2^{\ast}}}\leq\dlt$,
then there exists unique permutation $\pi:\setJ\to\setJ$ satisfying
$d_{\calP}((\vec{\iota},\vec{\lmb},\vec{z}),(\vec{\iota}'\circ\pi,\vec{\lmb}'\circ\pi,\vec{z}'\circ\pi))\leq\frac{1}{10}$.
\end{lem}

\begin{proof}
We proceed by induction on $J$. Note that the $J=1$ case follows
from \eqref{eq:2.3}.

\uline{Inductive step}. Let $J\geq2$ and assume the $J-1$ case.
Let $\dlt_{J-1}>0$ be the constant found in the $J-1$ case. For
a small constant $0<\dlt_{J}\ll1$ to be chosen in the proof, assume
$(\vec{\iota},\vec{\lmb},\vec{z}),(\vec{\iota}',\vec{\lmb}',\vec{z}')\in\calP_{J}(\dlt_{J})$
satisfy $\|\calW(\vec{\iota},\vec{\lmb},\vec{z})-\calW(\vec{\iota}',\vec{\lmb}',\vec{z}')\|_{L^{2^{\ast}}}\leq\dlt_{J}$.
We may assume $\lmb_{J}=\min_{j\in\setJ}\lmb_{j}\leq\min_{j\in\setJ}\lmb_{j}'$. 

First, we claim that there exists $\pi(J)\in\setJ$ such that 
\begin{equation}
\frac{\lmb_{\pi(J)}'+|z_{\pi(J)}'-z_{J}|}{\sqrt{\lmb_{\pi(J)}'\lmb_{J}}}=\calO(1).\label{eq:2.7-1}
\end{equation}
Denote $y_{J}=(x-z_{J})/\lmb_{J}$. We begin with 
\begin{align*}
\|\chf_{|x|\leq10}W\|_{L^{2^{\ast}}} & =\|\chf_{|y_{J}|\leq10}\calW(\iota_{J},\lmb_{J},z_{J})\|_{L^{2^{\ast}}}\\
 & \leq\dlt_{J}+\tsum{j\neq J}{}\|\chf_{|y_{J}|\leq10}\calW(\iota_{j},\lmb_{j},z_{j})\|_{L^{2^{\ast}}}+\tsum{j\in\setJ}{}\|\chf_{|y_{J}|\leq10}\calW(\iota_{j}',\lmb_{j}',z_{j}')\|_{L^{2^{\ast}}}.
\end{align*}
Using the minimality of $\lmb_{J}$, the last two terms of the right
hand side can be estimated as 
\begin{align*}
\tsum{j\neq J}{}\|\chf_{|y_{J}|\leq10}\calW(\iota_{j},\lmb_{j},z_{j})\|_{L^{2^{\ast}}} & \aleq\tsum{j\neq J}{}R_{jJ}^{-(N-2)}\aleq\dlt_{J}^{N-2},\\
\tsum{j\in\setJ}{}\|\chf_{|y_{J}|\leq10}\calW(\iota_{j}',\lmb_{j}',z_{j}')\|_{L^{2^{\ast}}} & \aleq\tsum{j\in\setJ}{}\Big(\frac{\lmb_{j}'+|z_{j}'-z_{J}|}{\sqrt{\lmb_{j}'\lmb_{J}}}\Big)^{-(N-2)}.
\end{align*}
Since $\|\chf_{|x|\leq10}W\|_{L^{2^{\ast}}}\ageq1$ and $\dlt_{J}>0$
is small, the claim \eqref{eq:2.7-1} follows.

Next, introduce a large constant $R_{0}\gg1$ and assume $\dlt_{J}$
is small depending on $R_{0}$. We begin with 
\begin{align*}
 & \|\chf_{|y_{J}|\leq R_{0}}\{\calW(\iota_{J},\lmb_{J},z_{J})-\calW(\iota_{\pi(J)}',\lmb_{\pi(J)}',z_{\pi(J)}')\}\|_{L^{2^{\ast}}}\\
 & \leq\dlt_{J}+\tsum{j\neq J}{}\|\chf_{|y_{J}|\leq R_{0}}\calW(\iota_{j},\lmb_{j},z_{j})\|_{L^{2^{\ast}}}+\tsum{j\neq\pi(J)}{}\|\chf_{|y_{J}|\leq R_{0}}\calW(\iota_{j}',\lmb_{j}',z_{j}')\|_{L^{2^{\ast}}}.
\end{align*}
For the left hand side, by \eqref{eq:2.7-1}, we have 
\begin{align*}
 & \|\chf_{|y_{J}|\leq R_{0}}\{\calW(\iota_{J},\lmb_{J},z_{J})-\calW(\iota_{\pi(J)}',\lmb_{\pi(J)}',z_{\pi(J)}')\}\|_{L^{2^{\ast}}}\\
 & \geq\|\calW(\iota_{J},\lmb_{J},z_{J})-\calW(\iota_{\pi(J)}',\lmb_{\pi(J)}',z_{\pi(J)}')\|_{L^{2^{\ast}}}-\calO(R_{0}^{-(N-2)/2}).
\end{align*}
For the right hand side, we have 
\[
\tsum{j\neq J}{}\|\chf_{|y_{J}|\leq R_{0}}\calW(\iota_{j},\lmb_{j},z_{j})\|_{L^{2^{\ast}}}\leq\tsum{j\neq J}{}\calO_{R_{0}}(R_{jJ}^{-(N-2)})=\calO_{R_{0}}(\dlt_{J}^{N-2}),
\]
and for the last term we use \eqref{eq:2.7-1} to obtain 
\[
\tsum{j\neq\pi(J)}{}\|\chf_{|y_{J}|\leq R_{0}}\calW(\iota_{j}',\lmb_{j}',z_{j}')\|_{L^{2^{\ast}}}=\tsum{j\neq\pi(J)}{}\calO_{R_{0}}(R_{\pi(J)j}^{\prime-(N-2)})=\calO_{R_{0}}(\dlt_{J}^{N-2}).
\]
Gathering the previous three computations and using the fact that
$\dlt_{J}$ is small depending on $R_{0}$, we get 
\begin{equation}
\|\calW(\iota_{J},\lmb_{J},z_{J})-\calW(\iota_{\pi(J)}',\lmb_{\pi(J)}',z_{\pi(J)}')\|_{L^{2^{\ast}}}\aleq R_{0}^{-(N-2)/2}.\label{eq:2.7-2}
\end{equation}
Combining this with \eqref{eq:2.3}, we get 
\begin{equation}
d_{\calP}((\iota_{J},\lmb_{J},z_{J}),(\iota_{\pi(J)}',\lmb_{\pi(J)}',z_{\pi(J)}'))\leq\calO(R_{0}^{-(N-2)/2)})\leq\frac{1}{10}.\label{eq:2.7-3}
\end{equation}
Finally, we use \eqref{eq:2.7-2} to obtain 
\begin{align*}
 & \|\tsum{j\neq J}{}\calW(\iota_{j},\lmb_{j},z_{j})-\tsum{j\neq\pi(J)}{}\calW(\iota_{j}',\lmb_{j}',z_{j}')\|_{L^{2^{\ast}}}\\
 & \leq\|\calW(\vec{\iota},\vec{\lmb},\vec{z})-\calW(\vec{\iota}',\vec{\lmb}',\vec{z}')\|_{L^{2^{\ast}}}+\|\calW(\iota_{J},\lmb_{J},z_{J})-\calW(\iota_{\pi(J)}',\lmb_{\pi(J)}',z_{\pi(J)}')\|_{L^{2^{\ast}}}\\
 & \leq\dlt_{J}+\calO(R_{0}^{-(N-2)/2})\leq\dlt_{J-1}.
\end{align*}
Applying the $J-1$ case, we can find the remaining components $\pi(1),\dots,\pi(J-1)$
of the permutation satisfying 
\begin{equation}
\max_{j\neq J}d_{\calP}((\iota_{j},\lmb_{j},z_{j}),(\iota_{\pi(j)},\lmb_{\pi(j)},z_{\pi(j)}))\leq\frac{1}{10}.\label{eq:2.10}
\end{equation}
Combining \eqref{eq:2.7-3} and \eqref{eq:2.10}, we conclude $d_{\calP}((\vec{\iota},\vec{\lmb},\vec{z}),(\vec{\iota}'\circ\pi,\vec{\lmb}'\circ\pi,\vec{z}'\circ\pi))\leq\frac{1}{10}$
as desired.

It remains to show the uniqueness of $\pi$. If $\pi_{1}$ and $\pi_{2}$
are such permutations, then $d_{\calP}((\vec{\iota}'\circ\pi_{1},\vec{\lmb}'\circ\pi_{1},\vec{z}'\circ\pi_{1}),(\vec{\iota}'\circ\pi_{2},\vec{\lmb}'\circ\pi_{2},\vec{z}'\circ\pi_{2}))\leq\frac{1}{5}$.
If $\pi_{1}\neq\pi_{2}$, then there is $j_{0}\in\setJ$ such that
$j_{1}\coloneqq\pi_{1}(j_{0})$ and $j_{2}\coloneqq\pi_{2}(j_{0})$
are distinct so that $d_{\calP}((\iota_{j_{1}}',\lmb_{j_{1}}',z_{j_{1}}'),(\iota_{j_{2}}',\lmb_{j_{2}}',z_{j_{2}}'))\leq\frac{1}{5}$.
This contradicts to $(\vec{\iota}',\vec{\lmb}',\vec{z}')\in\calP_{J}(\dlt_{J})$
with $\dlt_{J}$ small. This completes the proof.
\end{proof}
\begin{lem}
\label{lem:2.5}Let $J\geq1$. There exist $\dlt>0$ and $C\geq1$
such that if $(\vec{\iota},\vec{\lmb},\vec{z}),(\vec{\iota}',\vec{\lmb}',\vec{z}')\in\calP_{J}(\dlt)$
are such that $d_{\calP}((\vec{\iota},\vec{\lmb},\vec{z}),(\vec{\iota}',\vec{\lmb}',\vec{z}'))\leq\frac{1}{10}$,
then 
\[
d_{\calP}((\vec{\iota},\vec{\lmb},\vec{z}),(\vec{\iota}',\vec{\lmb}',\vec{z}'))\leq C\|\calW(\vec{\iota},\vec{\lmb},\vec{z})-\calW(\vec{\iota}',\vec{\lmb}',\vec{z}')\|_{L^{2^{\ast}}}.
\]
\end{lem}

\begin{proof}
Suppose not. Then, there exist sequences $(\vec{\iota}_{n},\vec{\lmb}_{n},\vec{z}_{n}),(\vec{\iota}_{n}',\vec{\lmb}_{n}',\vec{z}_{n}')\in\calP_{J}(1/n)$
such that $d_{\calP}((\vec{\iota}_{n},\vec{\lmb}_{n},\vec{z}_{n}),(\vec{\iota}_{n}',\vec{\lmb}_{n}',\vec{z}_{n}'))\leq\frac{1}{10}$
and 
\begin{equation}
\frac{\|\calW(\vec{\iota}_{n},\vec{\lmb}_{n},\vec{z}_{n})-\calW(\vec{\iota}_{n}',\vec{\lmb}_{n}',\vec{z}_{n}')\|_{L^{2^{\ast}}}}{d_{\calP}((\vec{\iota}_{n},\vec{\lmb}_{n},\vec{z}_{n}),(\vec{\iota}_{n}',\vec{\lmb}_{n}',\vec{z}_{n}'))}\to0.\label{eq:2.7}
\end{equation}
Note that $\vec{\iota}_{n}=\vec{\iota}_{n}'$ for all $n$ by $d_{\calP}((\vec{\iota}_{n},\vec{\lmb}_{n},\vec{z}_{n}),(\vec{\iota}_{n}',\vec{\lmb}_{n}',\vec{z}_{n}'))\leq\frac{1}{10}$;
we will omit $\vec{\iota}_{n}$ and $\vec{\iota}_{n}'$ when there
is no confusion. Taking a subsequence, we may assume for all $j\in\setJ$
\[
\frac{d_{\calP}((\lmb_{j,n},z_{j,n}),(\lmb_{j,n}',z_{j,n}'))}{d_{\calP}((\vec{\lmb}_{n},\vec{z}_{n}),(\vec{\lmb}_{n}',\vec{z}_{n}'))}\to\exists\ell_{j}\in(0,1]\quad\text{as }n\to\infty.
\]

Let $\calI\coloneqq\{j\in\setJ:\ell_{j}\neq0\}\neq\emptyset$. For
each $n$, we may choose $j_{0}=j_{0}(n)\in\calI$ such that $\lmb_{j_{0},n}=\min_{j\in\calI}\lmb_{j,n}\leq\min_{j\in\calI}\lmb_{j,n}'$
without loss of generality. Let $R_{0}\geq1$ be a constant to be
chosen later and denote $y_{j_{0}}=(x-z_{j_{0}})/\lmb_{j_{0}}$. Now
observe 
\begin{align*}
\|\calW & (\vec{\iota}_{n},\vec{\lmb}_{n},\vec{z}_{n})-\calW(\vec{\iota}_{n},\vec{\lmb}_{n}',\vec{z}_{n}')\|_{L^{2^{\ast}}}\geq\|\chf_{|y_{j_{0}}|\leq R_{0}}\{\calW(\vec{\iota}_{n},\vec{\lmb}_{n},\vec{z}_{n})-\calW(\vec{\iota}_{n},\vec{\lmb}_{n}',\vec{z}_{n}')\}\|_{L^{2^{\ast}}}\\
 & \geq\|\calW(\lmb_{j_{0},n},z_{j_{0},n})-\calW(\lmb_{j_{0},n}',z_{j_{0},n}')\|_{L^{2^{\ast}}}\\
 & \quad-\|\chf_{|y_{j_{0}}|>R_{0}}\{\calW(\lmb_{j_{0},n},z_{j_{0},n})-\calW(\lmb_{j_{0},n}',z_{j_{0},n}')\}\|_{L^{2^{\ast}}}\\
 & \quad-\tsum{j\notin\calI}{}\|\chf_{|y_{j_{0}}|\leq R_{0}}\{\calW(\lmb_{j,n},z_{j,n})-\calW(\lmb_{j,n}',z_{j,n}')\}\|_{L^{2^{\ast}}}\\
 & \quad-\tsum{j\in\calI\setminus\{j_{0}\}}{}\|\chf_{|y_{j_{0}}|\leq R_{0}}\{\calW(\lmb_{j,n},z_{j,n})-\calW(\lmb_{j,n}',z_{j,n}')\}\|_{L^{2^{\ast}}}.
\end{align*}
By \eqref{eq:2.3}, $j_{0}\in\calI$, and the definition of $\calI$,
we can estimate the first two terms of the above as 
\begin{align*}
\|\calW(\lmb_{j_{0},n},z_{j_{0},n})-\calW(\lmb_{j_{0},n}',z_{j_{0},n}')\|_{L^{2^{\ast}}} & \ageq d_{\calP}((\vec{\lmb}_{n},\vec{z}_{n}),(\vec{\lmb}_{n}',\vec{z}_{n}')),\\
\chf_{j\notin\calI}\|\calW(\lmb_{j,n},z_{j,n})-\calW(\lmb_{j,n}',z_{j,n}')\|_{L^{2^{\ast}}} & \aleq o_{n\to\infty}(1)\cdot d_{\calP}((\vec{\lmb}_{n},\vec{z}_{n}),(\vec{\lmb}_{n}',\vec{z}_{n}')).
\end{align*}
Next, as $d_{\calP}((\vec{\lmb}_{n},\vec{z}_{n}),(\vec{\lmb}_{n}',\vec{z}_{n}'))\leq\frac{1}{10}$,
we have a pointwise bound 
\[
|\calW(\lmb_{j,n},z_{j,n})-\calW(\lmb_{j,n}',z_{j,n}')|\aleq d_{\calP}((\lmb_{j,n},z_{j,n}),(\lmb_{j,n}',z_{j,n}'))\cdot\calW(\lmb_{j,n},z_{j,n}).
\]
Thus we have 
\begin{align*}
 & \|\chf_{|y_{j_{0}}|>R_{0}}\{\calW(\lmb_{j_{0},n},z_{j_{0},n})-\calW(\lmb_{j_{0},n}',z_{j_{0},n}')\}\|_{L^{2^{\ast}}}\\
 & \quad\aleq d_{\calP}((\lmb_{j_{0},n},z_{j_{0},n}),(\lmb_{j_{0},n}',z_{j_{0},n}'))\cdot\|\chf_{|y_{j_{0}}|>R_{0}}\calW(\lmb_{j_{0},n},z_{j_{0},n})\|_{L^{2^{\ast}}}\\
 & \quad\aleq d_{\calP}((\vec{\lmb}_{n},\vec{z}_{n}),(\vec{\lmb}_{n}',\vec{z}_{n}'))\cdot R_{0}^{-(N-2)/2}.
\end{align*}
By the minimality of $\lmb_{j_{0},n}$ and $(\vec{\iota}_{n},\vec{\lmb}_{n},\vec{z}_{n}),(\vec{\iota}_{n}',\vec{\lmb}_{n}',\vec{z}_{n}')\in\calP_{J}(1/n)$,
we also have 
\begin{align*}
 & \chf_{j\in\calI\setminus\{j_{0}\}}\|\chf_{|y_{j_{0}}|\leq R_{0}}\{\calW(\lmb_{j,n},z_{j,n})-\calW(\lmb_{j,n}',z_{j,n}')\}\|_{L^{2^{\ast}}}\\
 & \quad\aleq d_{\calP}((\lmb_{j,n},z_{j,n}),(\lmb_{j,n}',z_{j,n}'))\cdot\chf_{j\in\calI\setminus\{j_{0}\}}\|\chf_{|y_{j_{0}}|\leq R_{0}}\calW(\lmb_{j,n},z_{j,n})\|_{L^{2^{\ast}}}\\
 & \quad\aleq d_{\calP}((\lmb_{j,n},z_{j,n}),(\lmb_{j,n}',z_{j,n}'))\cdot\chf_{j\in\calI\setminus\{j_{0}\}}\calO_{R_{0}}(R_{jj_{0}}^{-(N-2)})\\
 & \quad\aleq d_{\calP}((\lmb_{j,n},z_{j,n}),(\lmb_{j,n}',z_{j,n}'))\cdot o_{n\to\infty}(1).
\end{align*}
Combining the previous displays, we have proved 
\[
\|\calW(\vec{\iota}_{n},\vec{\lmb}_{n},\vec{z}_{n})-\calW(\vec{\iota}_{n},\vec{\lmb}_{n}',\vec{z}_{n}')\|_{L^{2^{\ast}}}\geq d_{\calP}((\vec{\lmb}_{n},\vec{z}_{n}),(\vec{\lmb}_{n}',\vec{z}_{n}'))\cdot\{c-\calO(R_{0}^{-(N-2)/2})-o_{n\to\infty}(1)\}
\]
for some constant $c>0$. Taking $R_{0}$ large and then $n$ large,
we obtain 
\[
\|\calW(\vec{\iota}_{n},\vec{\lmb}_{n},\vec{z}_{n})-\calW(\vec{\iota}_{n},\vec{\lmb}_{n}',\vec{z}_{n}')\|_{L^{2^{\ast}}}\ageq d_{\calP}((\vec{\lmb}_{n},\vec{z}_{n}),(\vec{\lmb}_{n}',\vec{z}_{n}')),
\]
which contradicts to \eqref{eq:2.7}. This completes the proof.
\end{proof}
\bibliographystyle{abbrv}
\bibliography{References}

\end{document}